\documentclass[letterpaper]{article}
\pdfoutput=1
\usepackage{amsmath,amsthm,amssymb,enumerate}
\usepackage{dsfont}
\usepackage{color}
\usepackage{graphicx}
\usepackage{epsfig}
\usepackage{fancyhdr}
\usepackage{subfigure}
\usepackage{hyperref}
\usepackage{graphicx}        
\usepackage{multicol}        
\usepackage[bottom]{footmisc}

\usepackage[section]{placeins}

\begin{document}

\def\sgn{\mathrm{sign}}
\def\E{\mathds{E}}
\def\P{\mathds{P}}
\def\R{\mathds{R}}
\def\C{\mathds{C}}
\def\N{\mathds{N}}
\def\s{\mathfrak{s}}
\def\m{\mathfrak{m}}
\def\I{\mathcal{I}}
\def\M{\mathcal{M}}
\def\U{\mathcal{U}}
\def\T{\mathcal{T}}
\def\Tau{\mathcal{T}}
\def\r{\mathrm{r}}
\def\d{\mathrm{d}}
\def\NCDF{\mathcal{N}}
\def\G{\mathcal{G}}
\def\L{\mathcal{L}}
\def\F{\mathsf{F}}
\def\X{\mathsf{X}}
\def\ds{\displaystyle}
\def\indfunc{\mathds{1}}
\def\b{k}
\def\ind{\mathbb{I}}
\def\disteq{\overset{d}{=}}

\numberwithin{equation}{section}

 \pagestyle{fancy}
 \fancyhead{}
 \fancyhead[C]{G.~Campolieti and Y. Sui, Last Hitting Time Distributions for Solvable Diffusions}
 \fancyhead[R]{\thepage}
 \fancyfoot{}

\newtheorem{theorem}{Theorem} 
\newtheorem{corollary}{Corollary}
\newtheorem{pro}{Proposition}
\newtheorem{lemma}{Lemma}



\title{Last Hitting Time Distributions for Solvable Diffusions}

\author{Giuseppe~Campolieti and Yaode Sui\\
Mathematics Department, Wilfrid Laurier University\\
75 University Avenue West, Waterloo, Ontario, Canada\\
E-mail: \texttt{gcampoli@wlu.ca}}

\date{\today}

\maketitle

\begin{abstract}
By considering any one-dimensional time-homogeneous solvable diffusion process, 
this paper develops a complete analytical framework for computing the distribution of the last hitting time, to any level, and its joint distribution with the process value 
on any {\it finite time horizon}.  Our formalism allows for regular diffusions with any type of endpoint boundaries. The cases with imposed killing at any one or two interior points of the original state space also lead us to analytical formulae for joint distributions of the last hitting time, the process value and its attained maximum and/or minimum value on any finite time horizon. 
A key aspect of the formulation is the inherent link between last and first hitting times, which we exploit in deriving novel general formulae for the distribution of the last hitting time on any finite time horizon. 
The simpler known formulae for an infinite time horizon are then easily recovered as a special limit. 
We further link the joint distribution of the last hitting time and the process value to the transition density of the process, 
transition densities with additionally imposed killing, first hitting time probabilities and first hitting time densities associated to the last hitting level. 
In particular, we derive general formulae for each component of the joint distribution, i.e., the jointly continuous, the partly continuous (defective) and the jointly defective portions. By employing the spectral expansions of the transition densities and the first hitting time distributions, our derivations culminate in spectral expansions for both marginal and joint distributions of the last hitting time and the process value on any finite time horizon.

An additional main contribution of this paper lies in the application of our general formulae which give rise to newly closed-form analytical formulae for several solvable diffusions. In particular, we systematically derive analytical expressions for each portion of the marginal and joint distributions of the last hitting time and the process value on any finite time horizon, without and with imposed killing at one or two interior points, for Brownian motion, Brownian motion with drift (geometric Brownian motion), the squared Bessel, squared radial Ornstein-Uhlenbeck (CIR) and Ornstein-Uhlenbeck processes. Most of the formulae are given in terms of spectral series that are rapidly convergent and efficiently implemented. We demonstrate this by presenting some numerical calculations of marginal and joint distributions using accurately truncated series.
\end{abstract}


\section*{Introduction}
First passage (or exit) times play an important role in the theory of stochastic processes.  
For a one-dimensional (scalar) diffusion process, the first passage time is the first hitting time to a given level. 
First passage time distributions have been extensively studied and applied in many areas of stochastic modeling. For solvable diffusions, standard formulae can be used to produce analytical expressions for Laplace transforms of the first hitting time distributions, e.g., see \cite{BS02} for tables of formulae associated to ``exponential stopping'' for some commonly known solvable diffusions. Solvability implies known analytical expressions (generally in the form of series) for the fundamental solutions associated to the infinitesimal generator of the diffusion. For simple processes, e.g., Brownian motion, 
Laplace inversion readily recovers known simple analytical expressions for the first hitting time distributions. 
For generally solvable diffusions the inversions can be represented analytically in the form of spectral expansions. 
In \cite{Linetsky2004b}, spectral expansions are derived for the distribution of the first hitting time to a given level for scalar solvable diffusions with any type of endpoint boundaries. The approach in \cite{Linetsky2004b} uses the Spectral Theorem for semigroups of self-adjoint contractions in the Hilbert space of real-valued functions that are square-integrable with respect to the diffusion speed measure. 
As part of a self-contained general framework for any solvable scalar diffusion, this paper presents 
a related complementary approach to deriving the same spectral expansions of the first hitting time distributions. 
Our formulation makes direct use of the transition probability density function (PDF) of the process with imposed killing at the first hitting level(s). Moreover, we also include spectral expansions for the distributions of first exit times from an interval, and for hitting an upper or lower level before the other, as these are integral portions of the marginal and joint distributions associated to last passage  (or exit) times. 

Last passage times are also fundamental. 
For a Markov process with continuous paths, such as a diffusion, the last passage time is equivalent to the last hitting time to a given level, within a specified time horizon. The last passage time is not a stopping time (with respect to the natural filtration of a process). It provides a kind of window into possible future scenarios of a process. 
For some early theoretical works on last passage (exit) times, 
see, e.g., \cite{Getoor},\cite{Pitman and Yor(1981)},\cite{Salminen84}. 
The last passage time, in combination with other functionals of a stochastic process, has spawned an increasing number of applications over the years. The mathematical finance literature has especially benefited. 
In \cite{Profeta2010} it is shown how standard European options, under a positive local martingale model for the asset price, are expressible in terms of the distribution of the last passage time to the strike level for an infinite time horizon. 
In \cite{Egami}, the last passage time of a conditioned diffusion on an infinite time horizon is used to model credit risk. 
More recently, research concerning the last passage time has also focused on Lévy processes. 
For instance, in \cite{Sato} the moments of last exit times are studied. In \cite{chiu2005passage}, the joint Laplace transform of the first and last passage time distributions are derived for spectrally negative Lévy processes. 
In \cite{landriault2023bridging} two random times, occupation-type and Parisian-type first–last passage time, are introduced to bridge first and last passage times within a risk model based on spectrally negative Lévy processes. In excursion theory, many concepts are intrinsically related to the last passage time, such as meander processes or excursions straddling one or two levels within some finite time.  Parisian options, first introduced in \cite{Chesney}, have payoffs that depend upon the accumulated excursion time of the stock price staying above (or below) a prescribed barrier level for a given amount of time that exceeds a specified option time window.  The introduction of a Parisian ruin time has also emerged in modeling financial default as the time when a process meanders below a prescribed default threshold for an extended duration. 

It is particularly important to distinguish between a last passage (hitting) time for a finite time horizon, $T\in (0,\infty)$, and that for the infinite time horizon, $T=\infty$, which is a limiting case of a finite time horizon. 
Analytical formulae for the distribution of the last passage time, and of its associated joint distributions, have so far been derived for a  limited class of processes and mostly for $T=\infty$. 
For finite $T\in (0,\infty)$, the known analytical formulae in the literature are predominantly given as Laplace transforms of the distributions. Moreover, for scalar diffusions most of the analytical formulae for the distributions have so far remained focused on the simplest processes, e.g., standard Brownian motion (with drift). 
Hence, generally for $T\in (0,\infty)$, there has not been a comprehensive analytical theory for computing the marginal and joint distributions associated to the last hitting time to any given level. This paper directly addresses this important gap by presenting a thorough self-contained analytical treatment for generally time-homogeneous solvable scalar diffusions on any finite time horizon. 
For example, the distribution of the last hitting time to a given level (for nontrivial transcient diffusions) has long been known for $T=\infty$, e.g., see \cite{Pitman and Yor(1981)}. Yet, to date, for a finite time horizon only the Laplace transform (w.r.t. $T$) of the distribution has been derived, e.g., see \cite{BS02}. In Theorem~\ref{last-passage-propn-time-t} we perform the Laplace inversion analytically and hence derive a general new formula for the last hitting time density for all time $t\in (0,T)$, which is reduced to a product involving the time--$t$ transition PDF, evaluated at the last hitting level, and a difference of initial value (left and right) derivatives, evaluated at the last hitting level, of the time--$(T-t)$ distribution of the first hitting time to the last hitting level. The discrete (defective) portion of the last hitting time distribution is simply the tail (past $T$) probability of first hitting the last hitting level. Moreover, the limit $T\to\infty$ trivially recovers the known formula for infinite time horizon. Theorems \ref{last-passage-propn-time-t-kill-ab} and \ref{last-passage-propn-time-t-kill-b} further provide analogous new formulae for diffusions with imposed killing at either one or two interior points of the original state space of the process. The discrete portions of the last hitting time distributions are also given by tail probabilities of appropriate first hitting times (for hitting one level before another) on any finite time horizon. The distributions of the last hitting time for infinite time horizon are again recovered as $T\to\infty$. 

Furthermore, Theorems \ref{joint_last-passage-propn-time-t}, \ref{joint_last-passage-propn-time-t_ab} and \ref{joint_last-passage-propn-time-t-b} provide new general formulae for the joint densities of the last hitting time, to a given level, and the process value for any finite time horizon $T$, and for all respective cases without and with imposed killing at interior points. 
The formulae provide new insights and are readily implementable. The joint densities are given as products of the respective time--$t$ transition PDF and a first hitting time density for hitting the last hitting level within a remaining time window $T-t$. Our formulae also display an equivalent (``dual") interpretation in a backward time setting with transition and joint densities relative to the diffusion speed density. 
The key ingredients in deriving the densities in Theorems~\ref{last-passage-propn-time-t}--\ref{joint_last-passage-propn-time-t-b} are conditioning on the intermediate time--$t$ value of the process, applying the Markov property, time homogeneity and the forward and backward Kolmogorov PDEs. 
In a recent paper\cite{Camp_Sui_excursion}, part of this methodology was used in deriving simplified formulae for joint densities  associated to excursions straddling two arbitrary levels within a finite time $T$.

We complete the determination of the joint distributions by also deriving formulae for the corresponding partly and doubly discrete (or continuous) portions. These portions are simply given by either the time--$T$ transition PDFs with imposed killing at the hitting level or by a probability of entering the cemetery  state (for nonconservative endpoints of the process) before reaching the last hitting level within the time horizon $T$. Based on spectral formulae for the transition PDFs and first hitting time distributions, the relations established in Theorems~\ref{last-passage-propn-time-t}--\ref{joint_last-passage-propn-time-t-b} lead us to the derivation of general spectral expansion formulae, given in Propositions \ref{last-passage-propn-time-t-spectral}--\ref{marginal_joint_last-passage-propn-kill-b}, for all respective marginal and joint densities associated to the last hitting time to any given level and finite time horizon. Spectral expansions for the defective portions of the marginal and joint distributions are also derived. The formulae are readily applicable to any analytically tractable, i.e., solvable, diffusion process. We demonstrate this by applying our formulae to some known families of solvable diffusions and thereby generate several new analytical  expressions for the marginal and joint distributions associated to the last hitting time under such processes. 
Moreover, we perform efficient and accurate numerical calculations of some distributions by appropriately truncating the analytical spectral series.

%
The related sections of this paper are organized as follows. 
In Section~\ref{sect1} we summarize the basic fundamentals for a general time-homogeneous scalar solvable diffusion. 
This section also sets the basic notation and assumptions in the paper. Section~\ref{sect_spectral_transPDF} contains the spectral theory for transition PDFs, explicitly also treating the essential cases with imposed killing at one or two interior points of the state space. Section~\ref{sect_spectral_FHT} links transition PDFs, with appropriately imposed killing, to the distribution of first hitting times. 
Propositions~\ref{prop_spec_first_hit_1}--\ref{prop_spec_first_hit_ab} provide the general spectral expansion formulae for first hitting time distributions. 
Section~\ref{last_hitting_time_sect} is central to the paper, containing the complete theoretical treatment of the marginal 
and joint distributions associated to last hitting times for any solvable scalar diffusion. In particular, this section includes Theorems~\ref{last-passage-propn-time-t}--\ref{joint_last-passage-propn-time-t-b} and Propositions \ref{last-passage-propn-time-t-spectral}--\ref{marginal_joint_last-passage-propn-kill-b} for computing the continuous (density) portions of the distributions, as well as separate derivations of general formulae for the partly or doubly discrete (defective) portions of the distributions. Section~\ref{last_hitting_time_sect} also serves to connect the main formalism for the first and last hitting times. 
Section~\ref{subsect_spectral_formulae} applies the formalism in Section~\ref{last_hitting_time_sect} to several known solvable diffusions. In particular, we provide in-depth systematic derivations of new analytically closed-form formulae (mainly as spectral expansions) for both marginal, and joint distributions associated to last hitting times, for Brownian motion, Brownian motion with drift (geometric Brownian motion), the squared Bessel, squared radial Ornstein-Uhlenbeck (CIR) and Ornstein-Uhlenbeck processes. Our formulae also include all the respective cases for the killed diffusions, i.e., with additonally imposed killing at interior points of the original diffusions. These cases then also automatically give new analytically closed-form formulae for the multiply joint distributions of the last hitting time to a level, the process value and its attained maximum and/or minimum value (equivalently first hitting times to one or two interior levels) on any finite time horizon $T$. 
Section~\ref{sectNumerical} contains numerical calculations.  All relevant proofs are given in the Appendix.

\section{Basic Fundamentals and Notation}\label{sect1}
We consider a one-dimensional time-homogeneous regular diffusion $\{X_t, t\geq 0\}$ on
the state space $\I\equiv (l,r)$ with endpoints $l$ and $r$, $-\infty\leq l<r\leq\infty$. If the left and/or right endpoint is in the state space then $\I\equiv [l,r)$, or $(l,r]$, or $[l,r]$, accordingly. This process has infinitesimal generator $\mathcal G$ defined by
\begin{equation}\label{eq:X_Gen}
\mathcal G f(x):= {1\over 2}\nu^2(x)f^{\prime\prime}(x) + \alpha(x)f^\prime(x) 
= \frac{1}{\m(x)}\left(\frac{f^\prime(x)}{\s(x)}\right)^\prime\,,\quad   x \in \I\,.
\end{equation}
Throughout we assume an identically zero instantaneous killing measure, although this is readily incorporated in the formalism. The domain of $\mathcal G$ consists of all bounded functions on $\I$ s.t. $\mathcal G f$ are bounded on $\I$ and satisfying appropriate boundary conditions as described below. We assume that the respective infinitesimal drift and diffusion coefficient functions $\alpha(x)$ and $\nu(x)>0$ are continuous on $\I$, i.e., with $\m(x)$ and $\s'(x)$ continuous on $\I$. The respective scale and speed density functions are defined as
\begin{equation}\label{speed_scale_dens}
\s(x):=\exp\left(-\int^x\frac{2\alpha(z)}{\nu^2(z)}\,dz\right)
   \mbox{,\,\,\,   } \m(x):=\frac{2}{\nu^2(x)\s(x)}.
   \end{equation}
As a stochastic differential equation, $dX_t = \alpha(X_t)dt + \nu (X_t)dW_t$, where $\{W_t\}_{t\geq 0}$ is  a standard scalar Brownian motion.

The homogeneous second order linear ordinary differential equation:
\begin{equation} \label{eq:phi}
\mathcal{G} u (x) = \lambda\,u(x),\; \lambda\in \C,\; x\in\I,
\end{equation}
admits two fundamental solutions, denoted by 
$\varphi^+_\lambda(x)$ and $\varphi^-_\lambda(x)$. 
Every solution of (\ref{eq:phi}), with appropriate boundary conditions, is a linear combination of these two fundamental solutions. For positive real values of $\lambda$, $\varphi^+_\lambda(x)$ and $\varphi^-_\lambda(x)$ are respectively {\it strictly increasing and decreasing} convex real-valued positive functions for real $x\in\I$. Their Wronskian is
\begin{equation} \label{wronskian} W[\varphi^-_\lambda,\varphi^+_\lambda](x)
   := \varphi^-_\lambda(x) \varphi^{+ \,\prime}_\lambda(x) - 
\varphi^+_\lambda(x) \varphi^{- \,\prime}_\lambda(x)
      = w_\lambda\s(x)\,,
 \end{equation}
$\varphi^{\pm \,\prime}_\lambda(x) \equiv 
\frac{d \varphi^{\pm}_\lambda(x)}{dx}$, with Wronskian factor $w_\lambda$ as a function of $\lambda$ and independent of $x$. With this definition, $w_\lambda > 0$ for any positive real $\lambda$. The fundamental functions $\varphi^\pm_\lambda$ are unique, up to a multiplication constant, 
with $\varphi^+_\lambda$ as increasing and $\varphi^-_\lambda$ as decreasing for all real $\lambda>0$, and by further imposing boundary conditions in the case of non-singular (regular) endpoints. The boundary conditions of $\varphi_\lambda^+$ at the left endpoint are as follows for real $\lambda>0$: If $l$ is entrance-not-exit, then
$\varphi_\lambda^+(l+) > 0, \frac{1}{\s(l+)}\frac{d\varphi_\lambda^+(l+)}{dx}=0$. 
If $l$ is exit-not-entrance, then $\varphi_\lambda^+(l+) = 0, \frac{1}{\s(l+)}\frac{d\varphi_\lambda^+(l+)}{dx}>0$. 
If $l$ is a natural boundary, then $\varphi_\lambda^+(l+) = 0, \frac{1}{\s(l+)}\frac{d\varphi_\lambda^+(l+)}{dx}=0$. 
If $l$ is regular, then different boundary conditions can be applied, e.g., $\varphi_\lambda^+(l+) = 0$ if $l\notin \I$
is a killing boundary or $\frac{1}{\s(l+)}\frac{d\varphi_\lambda^+(l+)}{dx}=0$ if $l\in\I$ is a reflecting boundary.
Analogous conditions hold for the right boundary $r$ involving $\varphi_\lambda^-(r-)$
and $\frac{1}{\s(r-)}\frac{d\varphi_\lambda^-(r-)}{dx}$. 
The classification of $l (r)$ as either natural, entrance-not-exit, exit-not-entrance or regular is given by the standard Feller conditions, e.g., see \cite{BS02,Karlin}.

The Green function (resolvent kernel)
\footnote{Throughout the text we define the transition PDF w.r.t. the Lebesgue measure as opposed to the speed measure. The Green function w.r.t. the speed measure, $G_\lambda(x,y) := G(\lambda;x,y)/\m(y)$, 
which is symmetric in $x$ and $y$, is also commonly used (e.g., see \cite{BS02}).} 
for the $X$-diffusion on $\I$ is given by
\begin{equation}
G(\lambda;x,y) = w_\lambda^{-1}\m(y)\varphi^+_\lambda(x \wedge y)\varphi^-_\lambda(x \vee y),\,\,\lambda \in\C, \,x,y\in\I,
\label{greenfunc}
\end{equation}
where $x \wedge y:=\min\{x,y\}$ and $x \vee y:=\max\{x,y\}$, and $w_\lambda$
is a function of $\lambda$ given by (\ref{wronskian}).

The Green function the $X$-diffusion on the lower subinterval $\I^{\,-}_b := (l,b)\subset \I$ (or $\I^{\,-}_b:=[l,b)$
if $l$ is regular instantaneously reflecting), with {\it imposed
killing at a prescribed upper level} $b\in \I$, is given by
\begin{equation}
G_b^-(\lambda;x,y) = \m(y) \frac{\varphi_\lambda^+(x \wedge y)\,\phi(x \vee y,b;\lambda)}{w_\lambda \varphi_\lambda^+(b)}, \,\,\lambda \in\C, \,x,y\in\I^{\,-}_b.
\label{greenfunc_up}
\end{equation}
Similarly, for the $X$-diffusion on the upper subinterval $\I^{\,+}_b := (b,r)$ (or $(b,r]$ if $r$ is
regular instantaneously reflecting), {\it imposed killing at a prescribed lower level} $b\in \I$, the Green function 
is given by 
\begin{equation}
G_b^+(\lambda;x,y) = \m(y) \frac{\phi(b,x \wedge y;\lambda)\,\varphi_\lambda^-(x \vee y)}{w_\lambda \varphi_\lambda^-(b)}, \,\,\lambda \in\C, \,x,y\in\I^{\,+}_b.
\label{greenfunc_down}
\end{equation}
Within (\ref{greenfunc_up})-(\ref{greenfunc_down}), and throughout, we define the {\it generalized cylinder function}
(associated to a given pair of fundamental solutions $\varphi_\lambda^\pm$):
\begin{equation}
\phi(x,y;\lambda) := \varphi_\lambda^-(x)\varphi_\lambda^+(y) - \varphi_\lambda^-(y)\varphi_\lambda^+(x),
\label{phi_function}
\end{equation}
Note that $\phi$ is antisymmetric, $\phi(x,y;\lambda) = -\phi(y,x;\lambda)$. For real values of $\lambda > 0$, 
$b\in \I$, $\phi(b,x;\lambda)$ ($\phi(x,b;\lambda)$) is a positive increasing (decreasing) function of $x$ on 
$(b,r)$ ($(l,b)$).

For the $X$-diffusion on $(a,b)$, $l < a < b < r$, {\it with imposed killing at both prescribed lower and upper levels}
$a,b\in \I$, the Green function is given by
\begin{equation}
G_{(a,b)}(\lambda;x,y) = \m(y) \frac{\phi(a,x \wedge y;\lambda)\,\phi(x \vee y,b;\lambda)}{ w_\lambda \phi(a,b;\lambda)}, \,\,\lambda \in\C, \,x,y\in(a,b).
\label{greenfunc_double}
\end{equation}

We denote the diffusion on
 $\I^{\,-}_b$ with Green function in (\ref{greenfunc_up}), by 
$X_{b}\equiv \{X_{b,\,t}, t\ge 0\}$, where
\begin{align}\label{X-Killed-b}
X_{b,\,t} := \begin{cases} X_t & \text{for } t < \T_b,
\\
\partial^\dagger & \text{for }  t \geq \T_b. 
\end{cases}
\end{align}
This process is started below level $b$, $X_{b,\,0} = X_0 < b$, and is sent to the cemetery state,   
$\partial^\dagger$, upon hitting level $b$ at the first hitting time $\T_b := \inf\{t \ge 0 : X_t = b\}$. 
The diffusion on $\I^{\,+}_b$ is also defined as in
(\ref{X-Killed-b}); however, the diffusion is started above level $b$, $X_{b,\,0} = X_0 > b$. 
Similarly, $X_{(a,b)}\equiv \{X_{(a,b),\,t}, t\ge 0\}$ denotes the diffusion on $(a,b)$, with imposed killing at $a$ and $b$ and Green function in (\ref{greenfunc_double}), where
\begin{align}\label{X-Killed-ab}
X_{(a,b),\,t} := \begin{cases} X_t & \text{for } t < \T_{(a,b)},
\\
\partial^\dagger & \text{for }  t \geq \T_{(a,b)}. \end{cases}
\end{align}
The process is started at $X_{(a,b),\,0} = X_0 \in (a,b)$ and $\T_{(a,b)}:= \inf\{t \ge 0 : X_t \notin (a,b)\} = \T_a \wedge \T_b$ is the first exit time from the interval $(a,b)$.

The respective Green functions and transition PDFs are Laplace transform pairs. In particular, we have the conditional
\footnote{Throughout, $\P_x(A)$ denotes the probability of event $A$ conditional on $X_0=x$.}
 probability $\P_x(X_t \in dy) =  p(t;x,y) dy$ with $p(t;x,y)$ as the time-$t$ transition PDF of the $X$-process started at $X_0 = x$, $x,y\in\I, t > 0$. With $G$ given in (\ref{greenfunc}), we have
\begin{align}
G(\lambda;x,y) &= {\mathcal L}_t\{p(t;x,y)\}(\lambda) = \int_0^\infty p(t;x,y) e^{-\lambda t}  dt\,,
\\
p(t;x,y) &=  {\mathcal L}^{-1}_\lambda\{G(\lambda;x,y)\}(t) = {1\over 2\pi i}\int_{c-i\infty}^{c+i\infty} G(\lambda;x,y) e^{\lambda t}d\lambda\,.
\end{align}
The latter is the Bromwich contour integral on the complex $\lambda$-plane with all singularities of $G$ 
(as function of $\lambda$) lying to the left of the contour line $\text{Re}(\lambda) = c$.
Similarly, for the transition PDFs of the process with imposed killing on respective intervals $\I_b^\pm$, denoted by $p_b^\pm(t;x,y)$, 
$x,y\in\I_b^\pm, t > 0$, we have
\begin{align}\label{joint_transPDF_kill_b}
\P_x( X_{b,\,t} \in dy) = \P_x(\T_b > t, X_t \in dy)
= \begin{cases} 
\P_x(m_t > b , X_t  \in dy)  = p_{b}^+(t;x,y)dy, \,\,\,x > b,
\\
\P_x(M_t < b, X_t  \in dy) = p_{b}^-(t;x,y)dy, \,\,\,x < b.
\end{cases}
\end{align}
The respective densities $p_{b}^\pm$ are zero for values of $x,y\notin \I_b^\pm$ and satisfy zero boundary condition at the killing level, i.e., $p_{b}^\pm(t;b,y) = p_{b}^\pm(t;x,b)\equiv 0$. Throughout we define
\begin{equation}
M_t := \sup\limits_{0\le s \le t} X_s\,\,,\,\,\,\,m_t := \inf\limits_{0\le s \le t} X_s\,.
\label{sampled-sup-inf-defn}
\end{equation}
The respective Laplace transform pairs are 
\begin{align}
G_b^\pm(\lambda;x,y) = {\mathcal L}_t\{p_b^\pm(t;x,y)\}(\lambda)\,,\,\,\,\,\,\,
p_b^\pm(t;x,y) =  {\mathcal L}^{-1}_\lambda\{G_b^\pm(\lambda;x,y)\}(t).
\end{align}
We denote the transition PDF of $X_{(a,b)}$ by $p_{(a,b)}(t;x,y)$ where
\begin{align}\label{joint_transPDF_kill_a_b}
\P_x( X_{(a,b),t} \in dy) = \P_x(\T_{(a,b)} > t, X_t \in dy)
= \P_x(m_t > a, M_t < b, X_t  \in dy)  = p_{(a,b)}(t;x,y) dy\,,
\end{align}
$x,y\in (a,b), t > 0$, and
\begin{align}
G_{(a,b)}(\lambda;x,y) = {\mathcal L}_t\{p_{(a,b)}(t;x,y)\}(\lambda)\,,\,\,\,\,\,\,
p_{(a,b)}(t;x,y) =  {\mathcal L}^{-1}_\lambda\{G_{(a,b)}(\lambda;x,y)\}(t).
\end{align}

\section{Spectral Series for Transition Densities and First Hitting Time Distributions}
\subsection{Transition Densities}
\label{sect_spectral_transPDF}
Laplace inversions of the respective Green function, by closing the Bromwich integral and applying Cauchy's Residue Theorem, produces analyically closed-form spectral expansions of the respective transition PDF. 
Residues from the simple poles give rise to the discrete part of the spectrum, while continuous parts of the spectral expansion arise as integrals over branch cut discontinuities of the Green function. 

The qualitative nature of the spectrum is also intimitely related to the classification of the boundaries of the Sturm-Liouville (SL) ODE associated to (\ref{eq:phi}). 
See, for example, \cite{Linetsky2004a} for a general discussion and summary of the possible spectral categories and their relation to the boundary classification of the endpoints for a one-dimensional time-homogeneous diffusion. Here we simply re-state some basic facts about the connection between the boundary classification and the spectrum. 
We recall the (non-negative) Sturm-Liouville (SL) operator ${\mathcal H} := -{\mathcal G}$, where 
${\mathcal G}$ is the non-positive generator for the $X$-diffusion. Letting $\epsilon=-\lambda$, equation (\ref{eq:phi}) is a Sturm-Liouville second order linear ODE ${\mathcal H} u(x) = \epsilon u(x)$, i.e., 
${\mathcal H} \varphi^\pm_{-\epsilon}(x) = \epsilon\, \varphi^\pm_{-\epsilon}(x)$, $x\in (e_1,e_2)$. 
As endpoints we consider either $e_1 = l, e_2 = r$, or $e_1 = l, e_2 = b$, or $e_1 = b, e_2 = r$, or $e_1 = a, e_2 = b$ where $-\infty \le l < a < b < r \le \infty$. 
For a given $\epsilon$ (or $\lambda$), the SL ODE is either oscillatory or non-oscillatory at an endpoint $e_i$, $i=1,2$. 
In particular, for any given {\it real} value of $\epsilon$ (or $\lambda$) the SL ODE (or (\ref{eq:phi})) 
is oscillatory at $e_i$ if and only if every solution of the SL ODE (or (\ref{eq:phi})) has infinitely many zeros clustering at $e_i$. Otherwise, it is non-oscillatory at $e_i$. This classification at a given boundary $e_i$ is mutually exclusive for each real value of $\epsilon$ (or $\lambda$) and, of course, can vary with $\epsilon$ (or $\lambda$). 
This leads to the fact that an endpoint $e_i$ falls into only one of two cases: 
(i) NONOSC, where the SL ODE (or (\ref{eq:phi})) is non-oscillatory for all real $\epsilon$ (or $\lambda$); 
(ii) O-NO with some cutoff $\Lambda\ge 0$, where the SL ODE (or (\ref{eq:phi})) is oscillatory at
$e_i$ for all real values of $\lambda < -\Lambda$ ($\epsilon > \Lambda$) and non-oscillatory at $e_i$ for all real values
$\lambda > -\Lambda$ ($\epsilon < \Lambda$). 

Assuming smoothness conditions on the derivatives of the speed and scale densities, there are three main spectral categories that can arise (e.g., see Theorems 2-4 in \cite{Linetsky2004a}):
\vskip0.05in
\noindent{\it Spectral Category} I. This is the case when the Green function is a meromorphic function, analytic in $\lambda$ with the exception of a countable number of isolated simple poles, i.e., both endpoints are NONOSC and the {\it eigenspectrum is simple, nonnegative and purely discrete}. 
We recall the fact that non-natural boundaries are NONOSC. Hence, any diffusion with both endpoints as either regular or exit-not-entrance or entrance-not-exit is necessarily in Spectral Category I. Moreover, diffusions having one non-natural boundary and one NONOSC natural boundary, 
or having both NONOSC natural boundaries, are also in Spectral Category I. 

\noindent{\it Spectral Category} II: If one boundary is NONOSC and the other is O-NO with cutoff $\Lambda \ge 0$, then the spectrum of ${\mathcal H}$ is simple and nonnegative, with essential spectrum in 
$[\Lambda,\infty)$. Moreover, 
${\mathcal H}$ has purely absolutely continuous spectrum in $(\Lambda,\infty)$. If the O-NO endpoint is non-oscillatory for $\epsilon = \Lambda \ge 0$, then there exists a finite set (it may be empty) of simple eigenvalues in $[0,\Lambda]$. If the O-NO endpoint is oscillatory for $\epsilon = \Lambda > 0$, then there exists an infinite sequence of simple eigenvalues in $[0,\Lambda)$ clustering at $\Lambda$.

\noindent {\it Spectral Category} III: If the left boundary $e_1$ is O-NO with a cutoff $\Lambda_1 \ge 0$ and the right boundary $e_2$ is O-NO with cutoff $\Lambda_2 \ge 0$, then the essential spectrum is
contained in $[\Lambda_<,\infty)$, $\Lambda_< :=\min\{\Lambda_1,\Lambda_2\}$.
Moreover, ${\mathcal H}$ has purely absolutely continuous spectrum in $(\Lambda_<,\infty)$.
The part of the spectrum below $\Lambda_>:=\max\{\Lambda_1,\Lambda_2\}$ is simple (with multiplicity one) and the part above $\Lambda_>$ has multiplicity two. If the SL equation is non-oscillatory for $\epsilon = \Lambda_< \ge 0$, then there is a finite set (it may be empty) of simple eigenvalues in $[0, \Lambda_<]$. If the SL equation is oscillatory for 
$\epsilon = \Lambda_< > 0$, then there is an infinite sequence of simple eigenvalues in $[0, \Lambda_<)$ clustering at 
$\Lambda_<$.

For regular diffusions with imposed killing at one endpoint, there can at most be 
one O-NO endpoint and therefore Spectral Category III does not apply, i.e., we are either in Spectral Category I or II. The spectral expansions in this paper make use of either Spectral Category I or II.
\footnote{The only case considered here which corresponds to Spectral Category III is the most trivial diffusion, i.e., Brownian motion on $\R$ with both endpoints $\pm\infty$ as O-NO natural with spectral cutoff 
$\Lambda_1 = \Lambda_2 = 0$. In this case the Laplace inverse of the Green function trivially gives the known Gaussian transition PDF.} 
We now summarize the relevant spectral expansion formulae for each case. The derivations are based on standard manipulations. 

%
%
Consider the diffusion $X_b\in \I^{\,-}_b$ with killing at upper level $b\in \I$ with left endpoint $l$ assumed as either NONOSC natural, regular, exit-not-entrance or entrance-not-exit. Both $l$ and $b$ are NONOSC, i.e., we are in Spectral Category I where $G_b^-$ in (\ref{greenfunc_up}) is meromorphic. The transition PDF is given by
\begin{eqnarray}
p_b^-(t;x,y) = \m(y)\sum_{n=1}^\infty e^{-\lambda_n t}\phi_n(x)\phi_n(y), \,\,\,x,y\in \I^{\,-}_b, t >0.
\label{u_spectral_1}
\end{eqnarray}
The eigenvalues $\{\lambda_n \equiv \lambda^-_{n,b},\,n=1,2,...\}$ ($\nearrow \infty$ as $n\nearrow\infty$) 
are all the nonnegative simple zeros solving 
\begin{equation}\label{eigen_trans_pdf_1}
\varphi^+_{\lambda}(b)\big\vert_{\lambda = -\lambda_n}  \equiv \varphi^+_{-\lambda_n}(b) = 0.
\end{equation}
The eigenfunctions $\phi_n(x)\equiv \phi^{(b)}_n(x)$ on $\I^{\,-}_b$, 
where ${\mathcal H}\phi_n(x) = \lambda_n \phi_n(x)$, have the equivalent expressions:
\begin{equation}
\phi_n(x) = \pm \sqrt{A_n\over C_n}\varphi^+_{-\lambda_n}\!(x)
= \pm {\phi(b,x;-\lambda_n) \over \sqrt{A_nC_n}},
\label{eigenfunc1}
\end{equation}
$$C_n := \big[w_\lambda{\partial\over \partial\lambda}\varphi_\lambda^+(b)\big]_{\lambda=-\lambda_n} 
\equiv - w_{_{-\lambda_n}}\big[{\partial\over \partial\lambda}\varphi_{-\lambda}^+(b)\big]_{\lambda=\lambda_n}, \,\,\,
A_n = -\varphi^-_{-\lambda_n}\!(b).$$
[Note: $\m(y)\phi_n(x)\phi_n(y) = \text{Res}\,G_b^-(\lambda=-\lambda_n; x,y)$.] From (\ref{eigenfunc1}), the product of eigenfunctions in (\ref{u_spectral_1}) takes the useful form:
\begin{eqnarray}
\phi_n(x)\phi_n(y) = {\varphi^-_{-\lambda_n}\!(b) \over w_{_{\!-\lambda_n}}
\big[{\partial\over \partial\lambda}\varphi_{-\lambda}^+(b)\big]_{\lambda=\lambda_n}} 
\varphi^+_{-\lambda_n}\!(x)\varphi^+_{-\lambda_n}\!(y).
\label{spectral_1_product_eigen}
\end{eqnarray}
Similarly, for $X_b\in \I^{\,+}_b$ with killing at lower level $b$ and right endpoint $r$ 
assumed to be either NONOSC natural, regular, exit-not-entrance or entrance-not-exit, 
we have the transition PDF
\begin{eqnarray}
p_b^+(t;x,y) = \m(y)\sum_{n=1}^\infty e^{-\lambda_n t}\phi_n(x)\phi_n(y), \,\,\,x,y\in \I^{\,+}_b, t >0.
\label{u_spectral_2}
\end{eqnarray}
The eigenvalues $\{\lambda_n \equiv \lambda^+_{n,b},\,n=1,2,...\}$ ($\nearrow \infty$ as $n\nearrow\infty$) are all the nonnegative simple zeros solving 
\begin{equation}\label{eigen_trans_pdf_2}
\varphi^-_{\lambda}(b)\big\vert_{\lambda = -\lambda_n}  \equiv \varphi^-_{-\lambda_n}(b) = 0.
\end{equation}
The eigenfunctions $\phi_n(x)\equiv \phi^{(b)}_n(x)$ on $\I^{\,+}_b$, where 
${\mathcal H}\phi_n(x) = \lambda_n \phi_n(x)$, have the equivalent expressions:
\begin{equation}
\phi_n(x) = \pm \sqrt{A_n\over C_n}\phi(b,x;-\lambda_n)
= \pm {\varphi^-_{-\lambda_n}\!(x) \over \sqrt{A_nC_n}},
\label{eigenfunc2}
\end{equation}
$$C_n := \big[w_\lambda{\partial\over \partial\lambda}\varphi_\lambda^-(b)\big]_{\lambda=-\lambda_n}
\equiv - w_{_{-\lambda_n}}\big[{\partial\over \partial\lambda}\varphi_{-\lambda}^-(b)\big]_{\lambda=\lambda_n},\,\,\,
A_n = -[\varphi^+_{-\lambda_n}\!(b)]^{-1}.$$
[Note: $\m(y)\phi_n(x)\phi_n(y) = \text{Res}\,G_b^+(\lambda=-\lambda_n; x,y)$.] 
Using (\ref{eigenfunc2}) gives the product of eigenfunctions in (\ref{u_spectral_2}):
\begin{eqnarray}
\phi_n(x)\phi_n(y) = {\varphi^+_{-\lambda_n}\!(b) \over w_{_{\!-\lambda_n}}
\big[{\partial\over \partial\lambda}\varphi_{-\lambda}^-(b)\big]_{\lambda=\lambda_n}} 
\varphi^-_{-\lambda_n}\!(x)\varphi^-_{-\lambda_n}\!(y).
\label{spectral_2_product_eigen}
\end{eqnarray}

%
%
For the diffusion $X_{(a,b)}$ we automatically have Spectral Category I since both endpoints $a$ and $b$ are
regular killing (i.e., NONOSC). The transition PDF has a discrete spectral expansion,  
\begin{eqnarray}
p_{(a,b)}(t;x,y) = \m(y)\sum_{n=1}^\infty e^{-\lambda_n t}\phi_n(x)\phi_n(y),  \,\,\,x,y\in (a,b), t >0.
\label{u_spectral_3}
\end{eqnarray}
The eigenvalues $\{\lambda_n \equiv \lambda_n^{(a,b)},\,n=1,2,...\}$ ($\nearrow \infty$ as $n\nearrow\infty$) are all the nonnegative simple zeros solving 
\begin{equation}\label{eigen_trans_pdf_3}
\phi(a,b;-\lambda_n) = 0.
\end{equation}
The eigenfunctions $\phi_n(x) \equiv \phi_n^{(a,b)}(x)$ on $(a,b)$, where 
${\mathcal H}\phi_n(x) = \lambda_n \phi_n(x)$, are given by
\begin{equation}
\phi_n(x) = \pm \sqrt{A_n\over C_n}\phi(a,x;-\lambda_n)
= \pm {\phi(x,b;-\lambda_n) \over \sqrt{A_nC_n}},
\label{eigenfunc3}
\end{equation}
$$
C_n = \big[w_\lambda{\partial\over \partial\lambda} \phi(a,b;\lambda)\big]_{\lambda=-\lambda_n} 
\equiv -w_{_{-\lambda_n}} \Delta(a,b;\lambda_n).
$$ 
[Note: $\m(y)\phi_n(x)\phi_n(y) = \text{Res}\,G_{(a,b)}(\lambda=-\lambda_n; x,y)$.] 
Throughout, we define 
\begin{equation}\label{Delta_derivative}
\Delta(a,b;\lambda_n) := {\partial\over \partial\lambda}\phi(a,b;-\lambda)\big\vert_{\lambda=\lambda_n} 
\equiv - {\partial\over \partial\lambda}\phi(a,b;\lambda)\big\vert_{\lambda=-\lambda_n}.
\end{equation}
Since $\phi(x,b;-\lambda_n) = A_n \phi(a,x;-\lambda_n)$, we have 
$A_n = -{\varphi^+_{-\lambda_n}\!(b)\over \varphi^+_{-\lambda_n}\!(a)}$.
From (\ref{eigenfunc3}), we have the product of eigenfunctions in (\ref{u_spectral_3}) given by
\begin{align}
\phi_n(x)\phi_n(y) = {\varphi^+_{-\lambda_n}\!(b) \over \varphi^+_{-\lambda_n}\!(a)} 
{\phi(a,x;-\lambda_n)\phi(a,y;-\lambda_n) \over w_{_{-\lambda_n}} \Delta(a,b;\lambda_n)}
= {\phi(a,x;-\lambda_n)\phi(b,y;-\lambda_n) \over w_{_{-\lambda_n}} \Delta(a,b;\lambda_n)}.
\label{spectral_3_product_eigen}
\end{align}
The second expression is clearly more efficiently computed. The equivalence of both expressions follows from (\ref{eigen_trans_pdf_3}), i.e., 
${\varphi^+_{-\lambda_n}\!(b) \over \varphi^+_{-\lambda_n}\!(a)} 
= {\varphi^-_{-\lambda_n}\!(b) \over \varphi^-_{-\lambda_n}\!(a)}$.

For $X_b\in \I^{\,-}_b$ ($\I^{\,+}_b$) with endpoint $l$ ($r$) as O-NO natural with cutoff $\Lambda_- \ge 0$ 
($\Lambda_+ \ge 0$), the respective transition PDF has the form (i.e., Spectral Category II):
\begin{eqnarray}
p^\pm_b(t;x,y) \!\!&=& \!\!\sum_{n\ge 1} e^{-\lambda_n t} \,\text{Res}\,G^\pm_b(\lambda=-\lambda_n;x,y)
+ {1 \over \pi}\int_{\Lambda_\pm}^\infty \!e^{-\epsilon t}\, \text{Im}\,G^\pm_b(\epsilon e^{-i\pi};x,y)\,d\epsilon
\nonumber \\
\!&\equiv& \!\m(y)\bigg[\sum_{n\ge 1} e^{-\lambda_n t}\phi^{(b)}_n(x)\phi^{(b)}_n(y) +
\int_{\Lambda_\pm}^\infty e^{-\epsilon t}\Psi_b(x,\epsilon)\Psi_b(y,\epsilon) \,d\Omega^{(b)\,\pm}_{ac}(\epsilon)\bigg].
\label{u_spectral_2_b}
\end{eqnarray}
The sum is over a discrete set of increasing eigenvalues $\{\lambda_n\}_{n\ge 1}$ (if any exist) respectively given by 
(\ref{eigen_trans_pdf_1}) or (\ref{eigen_trans_pdf_2}) with respective product eigenfunctions 
given by (\ref{spectral_1_product_eigen}) or (\ref{spectral_2_product_eigen}). 
The integral in the second expression is over a 
continuous spectrum with real function $\Psi_b(x,\epsilon)$ as a nonzero multiple of the cylinder function 
$\phi(b,x;-\epsilon)$ where ${\mathcal H} \Psi_b(x,\epsilon) = \epsilon \Psi_b(x,\epsilon)$ 
on either respective interval $\I_b^\pm$ with $\Psi_b(b,\epsilon) = 0$. 
The respective absolutely continuous spectral
function on $[\Lambda_\pm,\infty)$ (normalized relative to $\Psi_b(x,\epsilon)$) is 
denoted by $\Omega^{(b)\,\pm}_{ac}(\epsilon)$ and arises by integrating along the branch cut 
discontinuity of the respective Green function $G^\pm_b$ with respective branch point $\lambda=-\Lambda_\pm$, i.e., 
$\text{Im}\,G^\pm_b(\epsilon e^{-i\pi};x,y) = [G^\pm_b(\epsilon e^{-i\pi};x,y) - G^\pm_b(\epsilon e^{i\pi};x,y)]/2i$. 
[Note: throughout we simply take the principal branch cut, although other convenient branches can be chosen.] 
In some cases of Spectral Category II the spectrum may be purely continuous with no summation term. Generally, the spectral expansion in \eqref{u_spectral_2_b} can be cast as a single Stieltjes integral on $[0,\infty)$.

We remark that the transition PDFs in (\ref{u_spectral_2_b}), with assumed nonzero integral component, can be accurately approximated by employing (\ref{u_spectral_3}) in the respective limit of either $a\searrow l$ or $b\nearrow r$. In practice, to compute $p^+_a(t;x,y)$ we can use (\ref{u_spectral_3}), with an appropriately truncated number of terms in the series, and for progressively larger values of $b$ until an acceptable error tolerance is achieved. Similarly, $p^-_b(t;x,y)$ can be approximated using (\ref{u_spectral_3}) for progressively smaller values of $a$ until an acceptable error tolerance is achieved.

%
%
Lastly, consider the regular diffusion $X\in \I$. 
If both endpoints $l$ and $r$ are NONOSC natural, regular, exit-not-entrance or 
entrance-not-exit, then we have Spectral Category I with transition PDF: 
\begin{eqnarray}
p(t;x,y) = \m(y)\sum_{n=1}^\infty e^{-\lambda_n t}\phi_n(x)\phi_n(y) \,\,\,x,y\in \I, t >0.
\label{u_spectral_4}
\end{eqnarray}
The eigenvalues $\{\lambda_n,\,n=1,2,...\}$ 
($\nearrow \infty$ as $n\nearrow\infty$) are all nonnegative simple zeros solving $w_{{-\!\lambda_n}}=0$, with  
$w_{\lambda}$ as the Wronskian factor in (\ref{wronskian}). 
The eigenfunctions are equally given by
\begin{equation}
\phi_n(x) = \pm \sqrt{A_n\over C_n}\varphi^+_{-\!\lambda_n}\!(x)
= \pm { \varphi^-_{-\!\lambda_n}\!(x) \over \sqrt{A_nC_n}}
\label{eigenfunc_regular}
\end{equation}
where $C_n := {d\over d\lambda}w_\lambda\big\vert_{\lambda=-\lambda_n}$ 
and $\varphi^-_{-\!\lambda_n}\!(x) = A_n \varphi^+_{-\!\lambda_n}\!(x)$ for any $x\in\I$ with nonzero constant $A_n$. 
[Note: $\m(y)\phi_n(x)\phi_n(y) = \text{Res}\,G(\lambda=-\lambda_n; x,y)$.]
The product of eigenfunctions in (\ref{u_spectral_4}) is hence equally given by
\begin{align}
\phi_n(x)\phi_n(y) = {A_n\over C_n}\varphi^+_{-\!\lambda_n}\!(x)\varphi^+_{-\!\lambda_n}\!(y) 
=  {1\over A_n C_n}\varphi^-_{-\!\lambda_n}\!(x)\varphi^-_{-\!\lambda_n}\!(y)
= {1\over C_n} \varphi^\pm_{-\!\lambda_n}\!(x)\varphi^\mp_{-\!\lambda_n}\!(y).
\label{spectral_4_product_eigen}
\end{align}

If one endpoint ($l$ or $r$) is NONOSC natural, regular, exit-not-entrance or entrance-not-exit and the other endpoint is O-NO natural with a cutoff $\Lambda \ge 0$ (i.e., Spectral Category II), then    
\begin{eqnarray}
p(t;x,y) \!\!&=& \!\!\sum_{n\ge 1} e^{-\lambda_n t} \,\text{Res}\,G(\lambda=-\lambda_n;x,y)
+ {1 \over \pi}\int_\Lambda^\infty \!e^{-\epsilon t}\, \text{Im}\,G (\epsilon e^{-i\pi};x,y)\,d\epsilon
\nonumber \\
\!&\equiv& \!\m(y)\bigg[\sum_{n\ge 1} e^{-\lambda_n t}\phi_n(x)\phi_n(y) +
\int_\Lambda^\infty e^{-\epsilon t}\Psi(x,\epsilon)\Psi(y,\epsilon) \,d\Omega_{ac}(\epsilon)\bigg]
\label{u_spectral_5}
\end{eqnarray}
with product eigenfunctions in (\ref{spectral_4_product_eigen}). 
For a purely continuous spectrum there is no summation term in (\ref{u_spectral_5}). 
In the second expression, the integral along the branch cut is recast in terms of an integral involving a real solution to the SL equation, ${\mathcal H} \Psi(x,\epsilon) = \epsilon \Psi(x,\epsilon)$, $x\in\I$, 
(e.g., $\Psi(x,\epsilon)$ is a nonzero multiple of $\varphi^+_{-\epsilon}(x)$ if $l$ is NONOSC) and an absolutely continuous spectral function $\Omega_{ac}(\epsilon)$ normalized relative to $\Psi(x,\epsilon)$.

If both endpoints $l$ and $r$ are O-NO natural, then we have Spectral Category III where the spectral expansion may generally involve a combination of different integral terms (from the continuous portions of the spectrum) and a summation term from the discrete spectrum. We simply refer to \cite{Linetsky2004a} 
for the general spectral representation of the transition PDF. We remark also that the transition PDFs for Spectral Category III can be accurately approximated by employing (\ref{u_spectral_3}) in the respective limits $a\searrow l$ and $b\nearrow r$. In practice, we can use (\ref{u_spectral_3}), with a truncated number of terms in the series, for progressively smaller value of $a$ and larger values of $b$ until an acceptable error tolerance is achieved.

\subsection{First Hitting Time Distributions}\label{sect_spectral_FHT}
We recall the definition of the respective {\it first passage time up and down at level $b$},
\begin{equation}\label{FHT-up-b}
{\Tau}^+_b := \inf \,\{ t \ge 0 \,\big|\, X_t \ge b \}
\end{equation}
and
\begin{equation}\label{FHT-down-b}
{\Tau}^-_b := \inf \,\{ t \ge 0 \,\big|\, X_t \le b \}.
\end{equation}
For any regular diffusion, as defined in Section \ref{sect1}, the first passage times are equivalent to the respective first hitting times up or down. Clearly, we have ${\Tau}^+_b = \Tau_b$ for $X_0 \le b$ and ${\Tau}^+_b \equiv 0$ for $X_0 \ge b$. Similarly, ${\Tau}^-_b = \Tau_b$ for $X_0 \ge b$ and ${\Tau}^-_b \equiv 0$ for $X_0 \le b$. In the respective trivial cases, where ${\Tau}^\pm_b \equiv 0$, observe that ${\Tau}^\pm_b \ne {\Tau}_b$.\footnote{Throughout this paper we obviously consider the nontrivial cases. However, it proves convenient to display our formulae in terms of ${\Tau}^+_b$ or ${\Tau}^-_b$, even if it suffices to use $\Tau_b$.} From standard formulae, e.g., see \cite{Karlin},
the probability of the diffusion (starting at $X_0=x$) ever hitting $b\in (l,r)$ in finite time is given by 
\begin{align}\label{FHT-b-up-infinity}
\Phi^{+}_b(x) := \P_x({\Tau}^+_b < \infty) =
\left\{\begin{array}{ll}
1&, l \in E_l^c,\\
\displaystyle{\mathcal{S}(l,x] \over \mathcal{S}(l,b]}&, l \in E_l\,,
   \end{array}\right.
\end{align}
for $x\in (l,b]$, and 
\begin{align}\label{FHT-b-down-infinity}
\Phi^{-}_b(x) := \P_x({\Tau}^-_b < \infty) =
\left\{\begin{array}{ll}
1&,  r \in E_r^c,\\[10pt]
\displaystyle{\mathcal{S}[x,r) \over \mathcal{S}[b,r)}&, r \in E_r\,,
   \end{array}\right.
\end{align}
for $x\in [b,r)$, with scale function $\mathcal{S}[x,y]:= \int_{x}^{y} \s(z)dz$, 
$\mathcal{S}(l,y]:= \lim\limits_{x\to l+}\mathcal{S}[x,y]$, $\mathcal{S}[x,r):= \lim\limits_{y\to r-}\mathcal{S}[x,y]$. 
Throughout, we define the set $E_e := \{\text{$e$ is attracting natural, or exit-not-entrance or regular killing} \}$ 
and $E_e^c = \{\text{$e$ is non-attracting natural, or entrance-not-exit or regular reflecting} \}$ for either boundary point $e\in\{l,r\}$. 
Since $\P_x({\Tau}^\pm_b = \infty) = 1 - \P_x({\Tau}^\pm_b < \infty)$, we have (respectively for $x\le b$ and $x \ge b$)
\begin{align}\label{FHT-b-up-equal_infinity}
\P_x({\Tau}^+_b = \infty) = {\mathcal{S}[x,b] \over \mathcal{S}(l,b]}\cdot \ind_{E_l}\,\,\,\,\text{and} \,\,\,\,
\P_x({\Tau}^-_b = \infty) = {\mathcal{S}[b,x] \over \mathcal{S}[b,r)}\cdot \ind_{E_r}
\end{align}
where $\ind_{E_e} \equiv 1$ if $e \in E_e$ and $\ind_{E_e} \equiv 0$ if $e \in E^c_e$.

The cumulative distribution function (CDF) of the respective first hitting time (up or down) is hence expressible in terms of the  respective tail probability,
\begin{equation} \label{X_full_cdf}
\P_{x}(\Tau^{\pm}_{b} \le t) = \Phi^{\pm}_b(x) - \P_{x}(t < \Tau^{\pm}_{b} < \infty),
\end{equation}
for all $t \in [0,\infty)$. Hence,
\begin{equation} \label{FHT_cdf_complement}
\P_{x}(\Tau^{\pm}_{b} > t) = 1 - \P_{x}(\Tau^{\pm}_{b} \le t) = 
\P_x({\Tau}^\pm_b = \infty) + \P_{x}(t < \Tau^{\pm}_{b} < \infty).
\end{equation}
Note that \eqref{FHT_cdf_complement} gives the tail probability including infinite time. 
The following Lemma links the tail probabilities in \eqref{X_full_cdf}, and the first hitting time PDFs, to the corresponding transition PDFs in \eqref{joint_transPDF_kill_b}.  
\begin{lemma} \label{Lemma_hitting_time}
The tail probabilities in (\ref{X_full_cdf}) are given by 
\begin{equation} \label{X_hit_lemma_1}
\P_{x}(t < \Tau^{\pm}_{b} < \infty) = \int_{\I_b^{\mp}} \Phi^{\pm}_b(y) \,p_b^\mp(t;x,y)\,d y
\end{equation}
where $p_b^\mp(t;x,y)$ are the transition PDFs on the respective intervals $\I_b^{\mp}$ with imposed killing at upper (lower) level $b\in (l,r)$. The respective first hitting time PDFs $f^\pm(t;x,b) := {\partial\over \partial t}\P_{x}({\Tau}^\pm_b \le t)$, 
$t\in (0,\infty)$, are equivalently given by the following left (right) limit derivatives at $b$:
\begin{equation}\label{fhit_up_down_pdf}
f^\pm(t;x,b) = \mp \displaystyle\frac{1}{\s(b)} \frac{\partial}{\partial y}
\left( \frac{p_b^\mp(t;x,y)}{\m(y)}\right)\bigg|_{y=b\mp} 
= \mp \frac{1}{\m(x)}\displaystyle\frac{1}{\s(b)} \frac{\partial}{\partial y}
p_b^\mp(t;y,x)\bigg|_{y=b\mp}.
\end{equation}
\end{lemma}
\begin{proof} See \cite{CM21} which is reproduced in \ref{sect_Lemma1_proof}. 
\end{proof}
\noindent Remark: In case $l$ is a conservative boundary, i.e., if $l \in E^c_l$ or if $l$ is attracting natural, then it readily follows that \eqref{X_full_cdf} gives 
$\P_{x}(\Tau^{+}_{b} \le t) = 1 - \int_l^b p_{b}^-(t;x,y)dy = 1 - \P_{x}(M_t < b)$, i.e., $\P_{x}(\Tau^{+}_{b} > t)  \equiv \P_{x}(M_t < b) = \int_l^b p_{b}^-(t;x,y)dy$. Similarly, if $r$ is conservative then \eqref{X_full_cdf} recovers  
$\P_{x}(\Tau^{-}_{b} \le t) = 1 - \int_b^r p_{b}^+(t;x,y)dy \equiv 1 - \P_{x}(m_t > b)$, 
i.e., $\P_{x}(\Tau^{-}_{b} > t) = \P_{x}(m_t > b) = \int_b^r p_{b}^+(t;x,y)dy$.

%
%
By making use of the spectral representations of $p_b^\mp$, Lemma~\ref{Lemma_hitting_time} 
gives rise to spectral representations of the densities, tail probabilities and hence CDFs of the respective first hitting times. 
Note that only Spectral Category I or II are possible. 
In particular, Proposition \ref{prop_spec_first_hit_1} corresponds to Spectral Category I, while Proposition \ref{prop_spec_first_hit_2} corresponds to Spectral Category II.
In what follows it proves convenient to define
\begin{align}\label{FHT_eigenfunctions_1}
\psi_n^+(x;b) := {\varphi^+_\lambda (x) \over {\partial \over \partial \lambda}\,
\varphi^+_{\lambda}(b)}\bigg\vert_{\lambda=-\lambda_{n,b}^-}
\equiv - {\varphi^+_{-\lambda_{n,b}^-} (x) \over {\partial \over \partial \lambda}\,
\varphi^+_{-\lambda}(b)\vert_{\lambda=\lambda_{n,b}^-}}\,,
\\
\psi_n^-(x;b) := {\varphi^-_\lambda (x) \over {\partial \over \partial \lambda}\,
\varphi^-_{\lambda}(b)}\bigg\vert_{\lambda=-\lambda_{n,b}^+}
\equiv - {\varphi^-_{-\lambda_{n,b}^+} \over {\partial \over \partial \lambda}\,
\varphi^-_{-\lambda}(b)\vert_{\lambda=\lambda_{n,b}^+}}\,,
\label{FHT_eigenfunctions_2}
\end{align}
where $\lambda_{n,b}^-$ and $\lambda_{n,b}^+$ denote the eigenvalues solving (\ref{eigen_trans_pdf_1}) and (\ref{eigen_trans_pdf_2}), respectively.

\begin{pro} \label{prop_spec_first_hit_1}
If the left boundary $l$ is NONOSC (i.e., exit-not-entrance, entrance-not-exit, regular or NONOSC natural), then the first hitting time up at $b\in(l,r)$, where $X_0=x\in \I_b^-$, has probability density given by the spectral expansion, for all $t\in (0,\infty)$:
\begin{equation}
f^{+}(t;x,b) = \sum_{n=1}^\infty e^{-\lambda_{n,b}^- t}\psi_n^+(x;b)
\label{FHT_prop1_1}
\end{equation}
and
\begin{equation}
\P_{x}(t < \Tau^{+}_{b} < \infty) = 
\sum_{n=1}^\infty {e^{- \lambda_{n,b}^- t} \over  \lambda_{n,b}^-}\psi_n^+(x;b)\,.
\label{FHT_prop1_2}
\end{equation}
If the right boundary $r$ is NONOSC, then, the first hitting time down at $b\in(l,r)$, where $X_0=x\in \I_b^+$, has probability density given by the spectral expansion, for all $t\in (0,\infty)$:
\begin{equation}
f^{-}(t;x,b) = \sum_{n=1}^\infty e^{-\lambda_{n,b}^+ t}\psi_n^-(x;b)
\label{FHT_prop2_1}
\end{equation}
and
\begin{equation}
\P_{x}(t < \Tau^{-}_{b} < \infty) = 
\sum_{n=1}^\infty {e^{- \lambda_{n,b}^+ t} \over  \lambda_{n,b}^+}\psi_n^-(x;b)\,.
\label{FHT_prop2_2}
\end{equation}
\end{pro}
\begin{proof} Alternative proofs, taken from \cite{CM21}, are reproduced in \ref{sect_FHT_Prop1_proof}.   
\end{proof}
\begin{pro} \label{prop_spec_first_hit_2}
If the left boundary $l$ is O-NO natural with spectral cutoff $\Lambda_- \ge 0$, then the first hitting time up at $b\in(l,r)$, where $X_0=x\in \I_b^-$, has probability density given by the equivalent spectral expansions, for all $t\in (0,\infty)$:
\begin{align}
f^{+}(t;x,b) &= \sum_{n\ge 1} e^{-\lambda_{n,b}^- t}\psi_n^+(x;b) 
+ {1 \over \pi}\int_{\Lambda_-}^\infty \!e^{-\epsilon t}\, \textup{Im}
\bigg\{{\varphi^+_\lambda (x) \over \varphi^+_\lambda (b)}\bigg\vert_{\lambda = \epsilon e^{-i\pi}}\bigg\}\,d\epsilon
\nonumber \\
\!&= \sum_{n\ge 1} e^{-\lambda_{n,b}^- t}\psi_n^+(x;b) -  
\int_{\Lambda_-}^\infty e^{-\epsilon t}{\Psi_b^\prime(b,\epsilon)\over \s(b)}\Psi_b(x,\epsilon) \,d\Omega^{(b)\,-}_{ac}(\epsilon)
\label{FHT_prop_ONO1_1}
\end{align}
and
\begin{align}
\P_{x}(t < \Tau^{+}_{b} < \infty) &= 
\sum_{n\ge 1} {e^{- \lambda_{n,b}^- t} \over  \lambda_{n,b}^-}\psi_n^+(x;b) 
+ {1 \over \pi} \int_{\Lambda_-}^\infty \!{e^{-\epsilon t}\over \epsilon}\, \textup{Im}
\bigg\{{\varphi^+_\lambda (x) \over \varphi^+_\lambda (b)}\bigg\vert_{\lambda = \epsilon e^{-i\pi}}\bigg\}\,d\epsilon
\nonumber \\
&=\sum_{n\ge 1} {e^{- \lambda_{n,b}^- t} \over  \lambda_{n,b}^-}\psi_n^+(x;b) 
- \int_{\Lambda_-}^\infty {e^{-\epsilon t}\over \epsilon}{\Psi_b^\prime(b,\epsilon)\over \s(b)}\Psi_b(x,\epsilon) \,d\Omega^{(b)\,-}_{ac}(\epsilon).
\label{FHT_prop_ONO1_2}
\end{align}
If the right boundary $r$ is O-NO natural with spectral cutoff $\Lambda_+ \ge 0$, then the first hitting time down at $b\in(l,r)$, where $X_0=x\in \I_b^+$, has probability density given by the equivalent spectral expansions, for all $t\in (0,\infty)$:
\begin{align}
f^{-}(t;x,b) &= \sum_{n\ge 1} e^{-\lambda_{n,b}^+ t}\psi_n^-(x;b) 
+ {1 \over \pi}\int_{\Lambda_+}^\infty \!e^{-\epsilon t}\, \textup{Im}
\bigg\{{\varphi^-_\lambda (x) \over \varphi^-_\lambda (b)}\bigg\vert_{\lambda = \epsilon e^{-i\pi}}\bigg\}\,d\epsilon
\nonumber \\
\!&= \sum_{n\ge 1} e^{-\lambda_{n,b}^+ t}\psi_n^-(x;b) +   
\int_{\Lambda_+}^\infty e^{-\epsilon t}{\Psi_b^\prime(b,\epsilon)\over \s(b)}\Psi_b(x,\epsilon) \,d\Omega^{(b)\,+}_{ac}(\epsilon)
\label{FHT_prop_ONO2_1}
\end{align}
and
\begin{align}
\P_{x}(t < \Tau^{-}_{b} < \infty) &= 
\sum_{n\ge 1} {e^{- \lambda_{n,b}^+ t} \over  \lambda_{n,b}^+}\psi_n^-(x;b) 
+ {1 \over \pi} \int_{\Lambda_+}^\infty \!{e^{-\epsilon t}\over \epsilon}\, \textup{Im}
\bigg\{{\varphi^-_\lambda (x) \over \varphi^-_\lambda (b)}\bigg\vert_{\lambda = \epsilon e^{-i\pi}}\bigg\}\,d\epsilon 
\nonumber \\
&= \sum_{n\ge 1} {e^{- \lambda_{n,b}^+ t} \over  \lambda_{n,b}^+}\psi_n^-(x;b) 
+ \int_{\Lambda_+}^\infty {e^{-\epsilon t}\over \epsilon}{\Psi_b^\prime(b,\epsilon)\over \s(b)}\Psi_b(x,\epsilon) \,d\Omega^{(b)\,+}_{ac}(\epsilon)\,.
\label{FHT_prop_ONO2_2}
\end{align}
\end{pro}
\begin{proof} Alternative proofs, taken from \cite{CM21}, are reproduced in \ref{sect_FHT_Prop2_proof}.   
\end{proof}
\noindent We remark that the expressions in the first line of \eqref{FHT_prop_ONO1_1}--\eqref{FHT_prop_ONO2_2} lead more directly to explicit formulae based on knowledge of the fundamental functions $\varphi^\pm_\lambda$. Also, the summations are zero for a purely continuous spectrum.

%
%
We now set $X_0 = x \in (a,b)$ on any subinterval $(a,b)\subset (l,r)$ and consider the first hitting time up at $b$ for the diffusion $X_{a}$ killed at $a$, i.e., the first time for the $X$-diffusion on $\I$ to hit $b$ before $a$:
\begin{equation}\label{X-fptime-up-b-killa}
{\Tau}^+_b\!(a) := \inf \,\{ t \ge 0 \,:\, X_t = b\,\,,\, m_t > a \} = \inf \,\{ t \ge 0 \,:\, X_{a,\,t} = b\} ,
\end{equation}
Similarly, the first hitting time down at $a$ for the diffusion $X_{b}$ killed at $b$, i.e., the first time for the $X$-diffusion on $\I$ to hit $a$ before $b$:
\begin{equation}\label{X-fptime-down-a-killb}
{\Tau}^-_a\!(b) := \inf \,\{ t \ge 0 \,:\, X_t = a\,\,,\,M_t < b\} = \inf \,\{ t \ge 0 \,:\, X_{b,\,t} = a\}.
\end{equation}
We note that $\P_{x}({\Tau}^+_b\!(a) \le t) = \P_{x}({\Tau}_b \le t, {\Tau}_b < {\Tau}_a)$ and 
$\P_{x}({\Tau}^-_a\!(b) \le t) = \P_{x}({\Tau}_a \le t, {\Tau}_a < {\Tau}_b)$. 
For $t \in (0,\infty)$, we define the respective PDFs of ${\Tau}^+_b\!(a)$ and ${\Tau}^-_a\!(b)$:
\begin{equation}\label{defn:FHT-PDF-b-killa_a-killb}
f^+(t;x,b\vert a) := {\partial\over \partial t}\P_{x}({\Tau}^+_b\!(a) \le t)\,\,\,\text{and}\,\,\,
f^-(t;x,a\vert b) := {\partial\over \partial t}\P_{x}({\Tau}^-_a\!(b) \le t)\,.
\end{equation}
From standard theory:
\begin{align}\label{FHT-b-killa-infinity}
\Phi^{+}_b(x\vert a) &:= \P_x\left({\Tau}^+_b\!(a) < \infty \right) = \P_x\left({\Tau}_b < \Tau_a \right) = {\mathcal{S}[a,x] \over \mathcal{S}[a,b]}\,,
\\
\Phi^{-}_a(x\vert b) &:=\P_x\left({\Tau}^-_a\!(b) < \infty \right) = \P_x\left({\Tau}_a < \Tau_b \right) = {\mathcal{S}[x,b] \over \mathcal{S}[a,b]}\,.
\label{FHT-a-killb-infinity}
\end{align}
As required, these sum to $\P_x\left({\Tau}_{(a,b)} < \infty \right)=1$, with 
$\P_x({\Tau}^+_b\!(a) = \infty) = \P_x\left({\Tau}_a < \Tau_b \right) 
= {\mathcal{S}[x,b] \over \mathcal{S}[a,b]}$ and 
$\P_x({\Tau}^-_a\!(b) = \infty) = \P_x\left({\Tau}_b < \Tau_a \right) 
= {\mathcal{S}[a,x] \over \mathcal{S}[a,b]}$.
Analogous to \eqref{X_full_cdf}--\eqref{FHT_cdf_complement}, we have 
\begin{align} \label{X_full_cdf_tau_ab_1}
\P_x({\Tau}^+_b\!(a) \le t) = {\mathcal{S}[a,x] \over \mathcal{S}[a,b]} - \P_{x}(t < {\Tau}^+_b\!(a) < \infty),
\\
\P_x({\Tau}^-_a\!(b) \le t) = {\mathcal{S}[x,b] \over \mathcal{S}[a,b]} - \P_{x}(t < {\Tau}^-_a\!(b) < \infty),
\label{X_full_cdf_tau_ab_2}
\end{align}
and 
\begin{align} \label{FHT_cdf_complement_tau_ab_1}
\P_x({\Tau}^+_b\!(a) > t) = {\mathcal{S}[x,b] \over \mathcal{S}[a,b]} + \P_{x}(t < {\Tau}^+_b\!(a) < \infty),
\\
\P_x({\Tau}^-_a\!(b) > t) = {\mathcal{S}[a,x] \over \mathcal{S}[a,b]} + \P_{x}(t < {\Tau}^-_a\!(b) < \infty).
\label{FHT_cdf_complement_tau_ab_2}
\end{align}

Analogous to Lemma \ref{Lemma_hitting_time}, the following links the above tail probabilities and first hitting time densities with the transition PDF in \eqref{joint_transPDF_kill_a_b}.
\begin{lemma} \label{Lemma_hitting_time_ab}
The above tail probabilities are given by
\begin{eqnarray} \label{hit_ab_lemma_1}
\P_{x}(t < {\Tau}^+_b\!(a) < \infty) = \int_{a}^{b}  \Phi^{+}_b(y\vert a) \,p_{(a,b)}(t;x,y) \, dy\,,
\\
\P_{x}(t < {\Tau}^-_a\!(b) < \infty) = \int_{a}^{b}  \Phi^{-}_a(y\vert b)  \,p_{(a,b)}(t;x,y) \, dy\,,
\label{hit_ab_lemma_2}
\end{eqnarray}
for all $t>0$, where $p_{(a,b)}$ is the transition PDF of the diffusion on $(a,b)$ with imposed killing  
at both $a,b\in (l,r)$, $ a < b$. 
The first hitting time densities in \eqref{defn:FHT-PDF-b-killa_a-killb} 
are equivalently given by the left (right) limits:
\begin{align}\label{hit_ab_lemma_3}
f^+(t;x,b\vert a) &= - \displaystyle\frac{1}{\s(b)} \frac{\partial}{\partial   y}
\left(
  \frac{p_{(a,b)}(t;x,y)}{\m(y)}\right)\bigg|_{y=b-} 
= - \frac{1}{\m(x)}\frac{1}{\s(b)} \frac{\partial}{\partial   y}p_{(a,b)}(t;y,x)\bigg|_{y=b-}  \,,
\\
f^{-}(t;x,a\vert b) &= \displaystyle\frac{1}{\s(a)} \frac{\partial}{\partial   y}
\left(
  \frac{p_{(a,b)}(t;x,y)}{\m(y)}\right)\bigg|_{y=a+}
= \frac{1}{\m(x)}\frac{1}{\s(a)} \frac{\partial}{\partial   y}p_{(a,b)}(t;y,x)\bigg|_{y=a+}\,.
\label{hit_ab_lemma_4}
\end{align}
\end{lemma}
\begin{proof} See \cite{CM21} which is reproduced in \ref{sect_Lemma2_proof}. 
\end{proof}

By employing \eqref{u_spectral_3}, Lemma~\ref{Lemma_hitting_time_ab}  
gives rise to spectral representations of the densities, tail probabilities and CDFs of the respective first hitting times. 
Note that Spectral Category I applies. In what follows it proves convenient 
(in analogy with \eqref{FHT_eigenfunctions_1}--\eqref{FHT_eigenfunctions_2}) to define
\begin{eqnarray}\label{FHT_eigenfunctions_ab}
\psi_n^+(x;a,b) := {\phi(x,a; -\lambda_n) \over \Delta(a,b; \lambda_n)},\,\,\,\,
\psi_n^-(x;a,b) := {\phi(b,x; -\lambda_n) \over \Delta(a,b; \lambda_n)},
\end{eqnarray}
where $\lambda_n \equiv \lambda_n^{(a,b)}$ are eigenvalues solving \eqref{eigen_trans_pdf_3} and 
$\Delta(a,b; \lambda_n)$ is defined in \eqref{Delta_derivative}.
\begin{pro} \label{prop_spec_first_hit_ab}
The first hitting times defined by \eqref{X-fptime-up-b-killa}--\eqref{X-fptime-down-a-killb}, where $X_0=x \in (a,b)$, have the respective probability density functions and tail probabilities, for all $t\in (0,\infty)$:
\begin{align}
f^+(t;x,b\vert a) &= \sum_{n=1}^\infty e^{-\lambda_n t} \psi_n^+(x;a,b),
\label{FHT_prop3_up_1}
\\
f^{-}(t;x,a\vert b)  &= \sum_{n=1}^\infty e^{-\lambda_n t} \psi_n^-(x;a,b),
\label{FHT_prop3_down_1}
\\
\P_{x}(t < {\Tau}^+_b\!(a) < \infty) &= 
\sum_{n=1}^\infty {e^{- \lambda_n t} \over  \lambda_n}\psi_n^+(x;a,b),
\label{FHT_prop3_up_2}
\\
\P_{x}(t < {\Tau}^-_a\!(b) < \infty) &= 
\sum_{n=1}^\infty {e^{- \lambda_n t} \over  \lambda_n} \psi_n^-(x;a,b).
\label{FHT_prop3_down_2}
\end{align}
where $\lambda_n \equiv \lambda_n^{(a,b)}$ solve \eqref{eigen_trans_pdf_3}.
\end{pro}
\begin{proof} Alternative proofs, taken from \cite{CM21}, are reproduced in \ref{sect_Prop_FHT_ab_proof}.   
\end{proof}
As noted at the end of Section \ref{sect_spectral_transPDF} for transition PDFs with O-NO natural boundaries, the series in Proposition \ref{prop_spec_first_hit_ab} 
can also be used to accurately approximate the first hitting time densities and tail probabilities in Proposition \ref{prop_spec_first_hit_2}. For instance, the probability in \eqref{FHT_prop_ONO1_2} can be accurately approximated by using \eqref{FHT_prop3_up_2} for values of $a$ progressively closer to the left boundary $l$ until an acceptable error tolerance with an appropriately truncated number of terms in the series. 
Similarly, the probability in \eqref{FHT_prop_ONO2_2} (with lower level $b$ now labeled as $a$) can be approximated by using \eqref{FHT_prop3_down_2}, for values of $b$ progressively closer to the right boundary $r$ and an appropriately truncated number of terms in the series.

%
%
%
%
\section{Last Hitting Time}\label{last_hitting_time_sect}
\subsection{Distribution of Last Hitting Time}\label{subsect2}
Consider any time-homogeneous regular diffusion $\{X_t, t\geq 0\}\in \I$ as defined in Section \ref{sect1}. 
The last hitting time to any fixed level $\b\in \I$ within any {\it finite} time interval $[0,T]$, $T>0$, 
is defined by 
\begin{align}\label{last_hiting_time_defn}
g^X_\b(T) \equiv g_\b(T) := \sup\{0\le u \le T: X_u = \b \}
\end{align}
and $g_\b(T) \equiv 0$ if $\{0\le u \le T: X_u = \b \} = \emptyset$. 
Note that $g_\b(T)\in [0,T]$ is generally a mixed random variable having a nonzero probability mass function (discrete portion) at $0$ and a continuous CDF with a probability density function on $(0,T)$. The CDF of $g_\b(T)$ is given by
\begin{align}\label{last_time-t-0}
\P_x(g_\b(T) \le t) \equiv \P_x(0 \le g_\b(T) \le t) = \P_x(g_\b(T) =0) + \P_x(0 < g_\b(T) \le t),
\end{align}
$0 \le t \le T$, for the diffusion started at $X_0 = x \in \I$, 
and where $\P_x(g_\b(T) \le T) \equiv \P_x(g_\b(T) < T) = 1$.

By conditioning on the time-$t$ value of the process, $X_t$, and using the total law of probabilities:
\begin{align}\label{prop_last_time-t-1}
\P_x(g_\b(T) \le t) &= \int_l^r \P_x(g_\b(T) \le t \vert X_t = y)\,  p(t;x,y) dy 
+ \P_x(g_\b(T) \le t \vert X_t =\partial^\dagger) \cdot \P_x(X_t =\partial^\dagger)
\nonumber \\
&= \int_l^r \P_y(\Tau_\b > T - t)\,  p(t;x,y) dy + \P_x(X_t =\partial^\dagger)\,.
\end{align}
Here we used the Markov property and time homogeneity, i.e., given $X_t = y$, the probability that the last 
time the process will hit $\b$ is not greater than $t$ (within the interval $[0,T]$) is equivalent to the probability that the process, restarted at $y$, has a first hitting time $\T_\b$ to level $\b$ that is at least equal to $T - t$. 
We emphasize that we are allowing any type of boundary for $l$ or $r$, i.e., conservative or nonconservative, where $\P_x(X_t =\partial^\dagger)$ is the probability that the process is in the cemetery state $\partial^\dagger$ at time $t$. Moreover, $\P_x(g_\b(T) \le t \vert X_t =\partial^\dagger) = 1$, i.e., the process cannot hit level $\b$ past time $t$ given that it is already in the cemetery state at time $t$. 
By taking $t = T$ in (\ref{prop_last_time-t-1}), and using $\P_y(\Tau_\b > 0) = 1$ for all $y\ne \b$, we 
recover the fact that the CDF at $T$ is unity: 
\begin{equation*}\label{probability_conservation}
\P_x(g_\b(T) \le T) = \int_l^r  p(T;x,y) dy + \P_x(X_T =\partial^\dagger) 
= \P_x(X_T \in \I)  + \P_x(X_T=\partial^\dagger)  = 1\,.
\end{equation*}
In the limit $t \searrow 0$, the transition PDF approaches the Dirac delta $p(0+;x,y) = \delta(x-y)$ and 
$\P_x(X_0 =\partial^\dagger) = 0$, since $X_0 = x\in \I$. Hence, (\ref{prop_last_time-t-1}) recovers 
$\P_x(g_\b(T) = 0) = \P_x(\T_\b > T)$, i.e., the equivalence of events $\{g_\b(T) = 0\} = \{\T_\b > T\}$ that also follows from the definition of $g_\b(T)$. Hence, the discrete portion of the distribution of $g_\b(T)$ is given by 
\eqref{FHT-b-up-equal_infinity} and \eqref{FHT_cdf_complement} for $t=T$:
\begin{align}\label{prop_last_time-t-1-limit}
\P_x(g_\b(T) = 0) = \begin{cases} 
\P_{x}(\Tau^{+}_{\b} > T) = \P_x(\Tau^{+}_{\b} = \infty) + \P_{x}(T < \Tau^{+}_{\b} < \infty)&, x < \b,
\\
\P_{x}(\Tau^{-}_{\b} > T) = \P_x(\Tau^{-}_{\b} = \infty) + \P_{x}(T < \Tau^{-}_{\b} < \infty)&, x > \b.
\end{cases}
\end{align}
The tail probabilities $\P_{x}(T < \Tau^{\pm}_{\b} < \infty)$ have spectral expansions given respectively by \eqref{FHT_prop1_2} and \eqref{FHT_prop2_2} for NONOSC endpoints ($l$ or $r$), or \eqref{FHT_prop_ONO1_2} and \eqref{FHT_prop_ONO2_2} for O-NO natural endpoints ($l$ or $r$). Note the trivial case when $x=\b$: $\P_\b(g_\b(T) = 0) = \P_{\b}(\Tau_{\b} > T)=0$. 

Since $\P_x(X_t =\partial^\dagger) = 1 - \int_l^r p(t;x,y) dy$ 
and $\P_y(\T_\b > T - t) = 1 - \P_y(\T_\b \le T - t)$, (\ref{prop_last_time-t-1}) is written more compactly in terms of the CDF of the first hitting time:
\begin{align}\label{prop_last_time-t-1-prime}
\P_x(g_\b(T) \le t) 
= 1 - \int_l^r \P_y(\T_\b \le T - t)\,  p(t;x,y) dy\,.
\end{align}
The following theorem firstly provides a formula for the Laplace transform (w.r.t. time horizon $T$) of the probability density of 
$g_\b(T)$, i.e., \eqref{last-passage-density-time-t-Laplace}--\eqref{last-passage-density-phi}, which is a known result 
(e.g., see \cite{BS02}). Secondly, the Laplace inversion in \eqref{last-passage-density-phi} is performed to give 
\eqref{last-passage-density-phi-explicit}. Hence, we uncover the general expression in \eqref{last-passage-pdf-explicit} for the time-$t$ PDF of $g_\b(T)$ in terms of the time-$t$ transition PDF and (left/right) derivatives of the CDFs of the first hitting times, at time $T-t$, w.r.t. the initial value evaluated at the (upper/lower) hitting level $\b$. 
 \begin{theorem}\label{last-passage-propn-time-t}
Let $f_{g_\b(T)}(t;x) := {\partial \over \partial t}\P_x(g_\b(T) \le t) $, $t \in (0,T)$, denote the probability density function of the last hitting time $g_\b(T)$, 
for any $x,\b\in \I$. 
Then, 
\begin{equation}\label{last-passage-density-time-t-Laplace}
{\mathcal L}_T\{f_{g_\b(T)}(t;x) \}(\lambda) = {e^{-\lambda t} \over \lambda G(\lambda;\b,\b)} p(t;x,\b).
\end{equation}
In particular, 
\begin{equation}\label{last-passage-density-time-t-1}
f_{g_\b(T)}(t;x) =\xi(T-t;\b)\, p(t;x,\b)
\end{equation}
where
\begin{align}\label{last-passage-density-phi}
\xi(u;\b) &:= {\mathcal L}_\lambda^{-1}\left\{{1 \over \lambda G(\lambda;\b,\b)}\right\}(u)
\\
&= {1 \over \m(\b)\s(\b)} \left[ {\partial \over \partial y} \P_y(\Tau^+_\b \le u)\big\vert_{y=\b-} 
- {\partial \over \partial y} \P_y(\Tau^-_\b \le u)\big\vert_{y=\b+} 
\right].
\label{last-passage-density-phi-explicit}
\end{align}
Hence, 
\begin{equation}\label{last-passage-pdf-explicit}
f_{g_\b(T)}(t;x) = {p(t;x,\b) \over \m(\b)}{1 \over \s(\b)} 
\left[ {\partial \over \partial y} \P_y(\Tau^+_\b \le T-t)\big\vert_{y=\b-} 
-\, {\partial \over \partial y} \P_y(\Tau^-_\b \le T-t)\big\vert_{y=\b+} 
\right].
\end{equation}
\end{theorem}
\begin{proof} The original statement of this theorem and its proof is found in \cite{CM21}. 
See \ref{sect_a1}. 
\end{proof}
\noindent [Remark: $\Tau^\pm_\b$, defined in (\ref{FHT-up-b})-(\ref{FHT-down-b}), can be replaced by $\Tau_\b$ in 
\eqref{last-passage-pdf-explicit}, i.e., $\P_y(\Tau^+_\b \le T-t) = \P_y(\Tau_\b \le T-t)$, for $y<\b$, and 
$\P_y(\Tau^-_\b \le T-t) = \P_y(\Tau_\b \le T-t)$, for $y>\b$.]

\vskip0.1in
%
The density of $g_\b := \sup\{u \ge 0: X_u = \b \} = \lim_{T\to \infty}g_\b(T)$, i.e., the last hitting time for an infinite time horizon, is easily recovered from \eqref{last-passage-pdf-explicit} in the limit $T\to \infty$ where 
$\P_y(\Tau^\pm_\b \le T-t) \to\P_y(\Tau^\pm_\b < \infty)$. 
From \eqref{FHT-b-up-infinity}--\eqref{FHT-b-down-infinity}, 
${1 \over \s(\b)} \big[ {\partial \over \partial y} \P_y(\Tau^+_\b < \infty)\big\vert_{y=\b-} 
- {\partial \over \partial y} \P_y(\Tau^-_\b < \infty)\big\vert_{y=\b+}\big] = {\ind_{E_l} \over \mathcal{S}(l,\b]} 
+ {\ind_{E_r} \over \mathcal{S}[\b,r)}$. Hence, this gives the known expression for the density $f_{g_\b}(t;x) := {\partial \over \partial t}\P_x(g_\b \le t)$:
\begin{equation}\label{last-passage-density-infinite-T}
f_{g_\b}(t;x) = \lim_{T\to \infty} f_{g_\b(T)}(t;x) = \left[ {\ind_{E_l} \over \mathcal{S}(l,\b]} + {\ind_{E_r} \over \mathcal{S}[\b,r)}
\right]{p(t;x,\b) \over \m(\b)},
\end{equation}
$t\in (0,\infty)$. The discrete portion follows from \eqref{prop_last_time-t-1-limit} and \eqref{FHT-b-up-equal_infinity}, i.e.,
\begin{equation}\label{last-passage-discrete-infinite-T}
\P_x(g_\b = 0) = \P_x(\Tau_{\b} = \infty) = {(\mathcal{S}[x,\b])^+ \over \mathcal{S}(l,\b]}\cdot \ind_{E_l} 
+ {(\mathcal{S}[\b,x])^+ \over \mathcal{S}[\b,r)}\cdot \ind_{E_r}.
\end{equation}
Note that, for an infinite time horizon, the nontrivial (nonzero) cases involve only transient diffusions where at least one of the boundaries, $l$ or $r$, is attracting and non-reflecting. For a recurrent diffusion (where $\ind_{E_l} = \ind_{E_r} = 0$) 
we simply have $\P_x(g_\b = \infty)=1$.

The CDF in (\ref{last_time-t-0}) is expressible as
\begin{equation}\label{last-passage-CDF-time_0_to_t}
\P_x(g_\b(T) \le t) = \P_x(\Tau_\b > T) + \int_0^t f_{g_\b(T)}(u;x) du.
\end{equation}
This provides a useful alternative to (\ref{prop_last_time-t-1-prime}) for computing the  CDF of $g_\b(T)$. 
Both (\ref{prop_last_time-t-1-prime}) and (\ref{last-passage-CDF-time_0_to_t}) make use of the first hitting time distributions in different ways. To implement (\ref{prop_last_time-t-1-prime}) we need to integrate (over the state space) the time-$(T-t)$ CDF of the first hitting time (as function of the initial point $y$) against the time-$t$ transition PDF (with $x$ as initial and $y$ as terminal point).  

Theorem~\ref{last-passage-propn-time-t} and Proposition~\ref{prop_spec_first_hit_1}--\ref{prop_spec_first_hit_2} 
lead to a spectral representation of the last hitting time density.
\begin{pro}\label{last-passage-propn-time-t-spectral}
The probability density function of $g_\b(T)$ has the representation:
\begin{equation}\label{last-passage-pdf-spectral}
f_{g_\b(T)}(t;x) = {p(t;x,\b) \over \m(\b)} 
\left[ \mathcal{S}(l,r;\b) 
+  {1 \over \s(\b)}\bigg(\eta^-(T-t;\b) - \eta^+(T-t;\b)\bigg)\right],
\end{equation}
where $\mathcal{S}(l,r;\b) := \frac{\ind_{E_l}}{\mathcal{S}(l,\b]} + \frac{\ind_{E_r}}{\mathcal{S}[\b,r)}$ and 
\begin{align}\label{last_passage_pdf_spectral_1}
 \eta^-(T-t;\b) 
= \begin{cases}
\displaystyle \sum_{n=1}^{\infty}{e^{-\lambda_{n,\b}^+ (T-t)} \over \lambda_{n,\b}^+} 
\hat\psi^-_{n}(\b), &  \text{if $r$ is NONOSC},\\
\displaystyle \sum_{n\ge 1} {e^{-\lambda_{n,\b}^+ (T-t)} \over \lambda_{n,\b}^+}\hat\psi^-_{n}(\b)
+ {1 \over \pi}\int_{\Lambda_+}^\infty \!e^{-\epsilon (T-t)}\, \textup{Im}
\bigg\{{\varphi^{-\prime}_\lambda (\b) \over \varphi^-_\lambda (\b)}\bigg\vert_{\lambda = \epsilon e^{-i\pi}}\bigg\}\,{d\epsilon \over \epsilon} , &  \text{if $r$ is O-NO},
\end{cases}
\end{align}
\begin{align}\label{last_passage_pdf_spectral_2}
\eta^+(T-t;\b) 
= \begin{cases}
\displaystyle \sum_{n=1}^{\infty}{e^{-\lambda_{n,\b}^- (T-t)} \over \lambda_{n,\b}^-} 
\hat\psi^+_{n}(\b), &  \text{if $l$ is NONOSC},\\
\displaystyle \sum_{n\ge 1} {e^{-\lambda_{n,\b}^- (T-t)} \over \lambda_{n,\b}^-} 
\hat\psi^+_{n}(\b)
+ {1 \over \pi}\int_{\Lambda_-}^\infty \!e^{-\epsilon (T-t)}\, \textup{Im}
\bigg\{{\varphi^{+\prime}_\lambda (\b) \over \varphi^+_\lambda (\b)}\bigg\vert_{\lambda = \epsilon e^{-i\pi}}\bigg\}\, {d\epsilon  \over \epsilon}, &  \text{if $l$ is O-NO},
\end{cases}
\end{align}
where 
$\displaystyle \hat\psi^\pm_{n}(\b) 
:= {\partial \over \partial x}\psi_n^\pm(x;\b)\vert_{x=\b}
= {\varphi^{\pm\prime}_\lambda (\b) \over {\partial \over \partial \lambda}\,
\varphi^\pm_{\lambda}(\b)}\bigg\vert_{\lambda=-\lambda_{n,\b}^\mp}
\equiv - {\varphi^{\pm\,\,\prime}_{-\lambda_{n,\b}^\mp} (\b) 
\over {\partial \over \partial \lambda}\,
\varphi^\pm_{-\lambda}(\b)\vert_{\lambda = \lambda_{n,\b}^\mp}}$
and 
$\Lambda_-$($\Lambda_+$) as the spectral cutoff for $l$($r$). 
\end{pro}
\begin{proof} See \ref{sect_proof_last_hitting_spectral}.  
\end{proof}
\noindent 
We remark that for several cases only Spectral Category I (with both $l$ and $r$ as NONOSC) applies and hence only the first expressions in \eqref{last_passage_pdf_spectral_1}-\eqref{last_passage_pdf_spectral_2} apply. 
As noted for \eqref{u_spectral_2_b}, the summations for the O-NO cases may be empty with purely continuous spectral expansions. 
We also observe that the term multiplying $p(t;x,\b)$ in \eqref{last-passage-pdf-spectral} is the spectral representation of 
$\xi(T-t;\b)$ in Theorem \ref{last-passage-propn-time-t}.

%
%
We note that if a process is defined as 
$F_t := {\sf F} (X_t)$, with ${\sf F}:\I \to {\mathcal D}$ as a 
monotonic (increasing or decreasing) mapping\footnote{Note that throughout the text we shall use the sans serif symbols ${\sf F}$ and ${\sf X}$ for the mapping functions.} with unique continuous inverse 
${\sf X} \equiv {\sf F}^{-1}$, i.e., $X_t = {\sf X} (F_t)$, 
then 
\begin{align}\label{last_passage_CDF_X_to_F}
\P_{F_0}(g^F_K(T) \le t) = \P_{x}(g_\b(T) \le t) \,,
\end{align}
where $g^F_K(T) :=\sup\{0 \le u \le T : F_u = K\}$, $K\in {\mathcal D}$, $x= X_0 = {\sf X}(F_0)$, $\b = {\sf X}(K)$. 
Hence, the density $f_{g^F_K(T)}(t;F_0) = f_{g_k(T)}(t;x)$ and $\P_{F_0}(g^F_K(T) = 0) = \P_{x}(g_\b(T) = 0)$. 
Note also the first hitting time $\T^F_K := \inf\{t\ge 0: F_t = K\} = \T^X_\b \equiv \T_\b$, i.e., 
$\P_{F_0}(\T^F_K > T) = \P_{x}(\T_\b > T)$.

%
%
\subsection{Joint Distribution of Last Hitting Time and Process Value}\label{subsect3}

We now derive formulae for the joint distribution of $(g_\b(T),X_T)$, i.e., the last hitting time to level $\b$ and terminal value of the  diffusion on $\I$ on any finite time interval $[0,T]$, $T>0$. 
Letting $X_0\equiv x,z,\b\in \I$, and applying a similar conditioning as in (\ref{prop_last_time-t-1}):
\begin{align}\label{prop_joint_last_time-t-1}
\P_x(g_\b(T) \le t, X_T \in dz) 
&= \int_l^r \P_x(g_\b(T) \le t , X_T \in dz \vert X_t = y)\,  p(t;x,y) dy 
\nonumber \\
&= \int_l^r \P_y(\T_\b > T - t, X_{T-t}  \in dz)\,  p(t;x,y) dy\,.
\end{align}
Again, we used the Markov and time-homogeneity properties. As well,   
$\P(X_T \in \I \vert X_t = \partial^\dagger) = 0$, for $t\in[0,T)$, i.e., $\P_x(g_\b(T) \le t , X_T \in dz \vert X_t = \partial^\dagger) = 0$.

Consistent with the equivalence $\{g_\b(T) = 0\} = \{\T_\b > T\}$, taking $t \searrow 0$ in (\ref{prop_joint_last_time-t-1}) recovers the portion of the joint distribution that is discrete in $g_\b(T)$ and continuous in $X_T$:
\begin{align}\label{prop_joint_last_time-t-2}
\P_x(g_\b(T) = 0, X_T \in dz) &= \P_x(\T_\b > T, X_{T}  \in dz)
\nonumber \\
&= \begin{cases} 
\P_x(m_T > \b , X_{T}  \in dz)  = p_{\b}^+(T;x,z)dz &, x,z > \b,
\\
\P_x(M_T < \b, X_{T}  \in dz) = p_{\b}^-(T;x,z)dz &, x,z < \b.
\end{cases}
\end{align}
Note: for $x=\b$ or $z=\b$, $\P_\b(g_\b(T) = 0, X_{T}  \in dz) \equiv 0$. 
Here we used \eqref{joint_transPDF_kill_b}, where $p_{\b}^\pm$ are the transition densities (that are nonzero on the respective intervals $\I_\b^\pm$) with imposed killing at the last hitting level $\b$. Hence, the distribution in \eqref{prop_joint_last_time-t-2} admits a spectral expansion given by either 
\eqref{u_spectral_1} or \eqref{u_spectral_2} in the respective cases where $l$ is NONOSC or $r$ is NONOSC, or by 
\eqref{u_spectral_2_b} if $l$ or $r$ are O-NO natural. 

As in Section \ref{subsect2} , we are generally considering nonconservative processes 
on the state space $\I$ where  
\begin{align}\label{prop_joint_last_time-t-3}
\P_x(g_\b(T) \le t)  = \P_x(g_\b(T) \le t, X_T \in \I) + \P_x(g_\b(T) \le t, X_T =\partial^\dagger).
\end{align}
If the process is conservative on $\I$, i.e., $\P_x(X_T = \partial^\dagger)
=\P_x(g_\b(T) \le t, X_T =\partial^\dagger) \equiv 0$, then 
\begin{align}\label{prop_joint_last_time-t-conserve}
\P_x(g_\b(T) \le t)  = \P_x(g_\b(T) \le t, X_T \in \I) 
= \int_l^r \P_x(g_\b(T) \le t, X_T \in dz)
\end{align}
with the (marginal) density of $g_\b(T)$ given by taking ${\partial \over \partial t}$ on both sides of (\ref{prop_joint_last_time-t-conserve}):
\begin{align}\label{last_density_time-t-conserve_from_joint}
f_{g_\b(T)}(t;x) = \int_l^r  f_{g_\b(T), X_T}(t,z;x) dz.
\end{align}
In any case (conservative or nonconservative) we shall denote by $f_{g_\b(T), X_T}(t,z;x)$  the {\it joint density} of 
$(g_\b(T), X_T)$ where ${\partial \over \partial t}\P_x(g_\b(T) \le t, X_T \in dz) = f_{g_\b(T), X_T}(t,z;x) dz$, i.e., $\P_x(g_\b(T) \in dt, X_T \in dz) = f_{g_\b(T), X_T}(t,z;x) dtdz$. If the process is nonconservative on $\I$, then in the place of (\ref{last_density_time-t-conserve_from_joint}) we have
\begin{align}\label{last_density_time-t-nonconserve_from_joint}
f_{g_\b(T)}(t;x) = \int_l^r  f_{g_\b(T), X_T}(t,z;x) dz 
+ {\partial \over \partial t}\P_x(g_\b(T) \le t, X_T = \partial^\dagger)\,.
\end{align}
Hence, this relation allows us to also compute the last term in \eqref{last_density_time-t-nonconserve_from_joint},
 i.e., the density of $g_\b(T)$ at $t\in (0,T)$ jointly with the process being in the cemetery state at time $T$, 
by subtracting the marginal (given by \eqref{last-passage-pdf-explicit}) and the integral on $\I$ of the joint density 
(given by \eqref{joint_last_passage_pdf_explicit} below). If the process is conservative on $\I$, then we trivially have ${\partial \over \partial t}\P_x(g_\b(T) \le t, X_T = \partial^\dagger) \equiv 0$.

Moreover, if the process is conservative on $\I$, then $\P_x(g_\b(T) = 0, X_T = \partial^\dagger) = 0$ and hence the discrete portion is recovered by integrating (\ref{prop_joint_last_time-t-2}), i.e., 
$\P_x(g_\b(T) = 0) = \P_x(g_\b(T) = 0, X_T \in \I) = \int_l^r \P_x(\T_\b > T, X_T \in dz)$. 
More generally, we have
$$\P_x(g_\b(T) = 0)  = \P_x(\T_\b > T) = \P_x(g_\b(T) = 0, X_T \in \I) + \P_x(g_\b(T) = 0, X_T = \partial^\dagger).$$
Hence, the nonzero jointly discrete (doubly defective) distribution is given equivalently by
\begin{eqnarray}
\P_x(g_\b(T) = 0, X_T = \partial^\dagger) &=& \P_x(\T_\b > T) - \P_x(\T_\b > T, X_T \in \I)
\label{joint_last_doubly_defective_1}
\\
&=& \begin{cases} 
\P_x(M_T < \b,  X_T = \partial^\dagger) = \P_x(\T_l^-(\b) \le T) &, x < \b,
\\
\P_x(m_T > \b ,  X_T = \partial^\dagger) = \P_x(\T_r^+(\b) \le T) &, x > \b.
\end{cases} 
\label{joint_last_doubly_defective_2}
\end{eqnarray} 
Note: we now assume the boundaries $l$ and $r$ are nonconservative (regular killing or exit-not-entrance) as otherwise the expression for $x \le \b$, or $x \ge \b$, is trivially zero if $l$, or $r$, is conservative, respectively. Both \eqref{joint_last_doubly_defective_1} or \eqref{joint_last_doubly_defective_2} can be used to derive the discrete spectral expansion for the jointly defective distribution: 
\begin{eqnarray}
\P_x(g_\b(T) = 0, X_T = \partial^\dagger) 
= \begin{cases} 
\displaystyle {\mathcal{S}[x,\b] \over \mathcal{S}(l,\b]} 
+ \sum_{n=1}^\infty {e^{- \lambda_n T} \over  \lambda_n}
\psi_n^+(x;\b) {\varphi^-_{-\lambda_n}(\b) \varphi^{+\,\prime}_{-\lambda_n}(l+) 
\over w_{-\lambda_n}\s(l+)}&, x < \b,
\\
\displaystyle {\mathcal{S}[\b,x] \over \mathcal{S}[\b,r)}  
- \sum_{n=1}^\infty {e^{- \lambda_n T} \over  \lambda_n}
\psi_n^-(x;\b) {\varphi^+_{-\lambda_n}(\b) \varphi^{-\,\prime}_{-\lambda_n}(r-) 
\over w_{-\lambda_n}\s(r-)}&, x > \b,
\end{cases}
\label{joint_last_doubly_defective_formula}
\end{eqnarray}
where $\lambda_n \equiv \lambda_{n,\b}^-$ for $x < \b$ and $\lambda_n \equiv \lambda_{n,\b}^+$ for $x > \b$, with 
$\psi_n^\pm$ in \eqref{FHT_eigenfunctions_1}--\eqref{FHT_eigenfunctions_2}. See  
\ref{Proof_doubly_defective} for a detailed derivation of \eqref{joint_last_doubly_defective_formula}. 
Note: for $x=\b$, $\psi_n^\pm(\b;\b)=0$ and $\mathcal{S}[\b,\b]=0$, i.e., 
$\P_\b(g_\b(T) = 0, X_T = \partial^\dagger)=0$, as trivially required.

The integral in (\ref{prop_joint_last_time-t-1}) can be split into two integrals, $y\in (l,\b)$ and $y\in (\b,r)$: 
\begin{align}\label{prop_joint_last_time-t-1_alt}
\P_x(g_\b(T) \le t, X_T \in dz) / dz =
\begin{cases}
\int_l^\b \,  p_{\b}^-(T-t;y,z)\,p(t;x,y)  dy 
&, z \in (l,\b),
\\
\int_\b^r \,  p_{\b}^+(T-t;y,z) \,p(t;x,y) dy 
&, z \in (\b,r),
\end{cases}
\end{align}
where $\P_y(\T_\b > T-t, X_{T-t}  \in dz) = \P_y(m_{T-t} > \b, X_{T-t}  \in dz) = 
p_{\b}^+(T-t;y,z) dz$, for $y,z \in (\b,r)$, and $\P_y(\T_\b > T-t, X_{T-t}  \in dz) = \P_y(M_{T-t} < \b, X_{T-t}  \in dz) = p_{\b}^-(T-t;y,z) dz$, for $y,z \in (l,\b)$.

Theorem \ref{joint_last-passage-propn-time-t} now gives the joint PDF of $(g_\b(T), X_T)$ at $(t,z)$, w.r.t. the speed measure at $z$, as a  product of the time-$t$ transition PDF, w.r.t. the speed measure at $\b$, of the process starting at $x$ and ending at the last hitting level $\b$, and the first hitting time density to hit $\b$ within a time $(T-t)$, starting at the future value $X_T=z$. 
\footnote{The formula for $f_{g_\b(T), X_T}(t,z;x)/\m(z)$ in \eqref{joint_last_passage_pdf_explicit} can also be interpretated in backward time, i.e., multiplying probabilities for paths ``starting'' at $z$ and first hitting $\b$ in time $T-t$ with density 
$f^+(T-t,z;\b)$, or $f^-(T-t,z;\b)$, and subsequently ``ending" at $x$, from $\b$, in time $t$ with density ${p(t;x,\b) \over \m(\b)} \equiv q(t;x,\b) = q(t;\b,x) = {p(t;\b,x) \over \m(x)}$.
}
As in Theorem~\ref{last-passage-propn-time-t}, this holds for any type of left and right boundaries.
\begin{theorem}\label{joint_last-passage-propn-time-t}
The joint PDF of $(g_\b(T), X_T)$ for $t \in (0,T)$, $X_0 = x,\b \in \I$, has the representation:
\begin{equation}\label{joint_last_passage_pdf_explicit}
f_{g_\b(T), X_T}(t,z;x) 
= {p(t;x,\b) \over \m(\b)} \m(z)
\begin{cases}
f^+(T-t,z;\b)
&, z \in (l,\b),
\\
f^-(T-t,z;\b)
&, z \in (\b,r),
\end{cases}
\end{equation}
with first hitting time densities $f^\pm$ given by \eqref{fhit_up_down_pdf}, i.e.,  
$\m(z) f^\pm(T-t,z;\b) = \mp \displaystyle\displaystyle\frac{1}{\s(\b)} \frac{\partial}{\partial y}p_\b^\mp(T-t;y,z)\bigg|_{y=\b\mp}$.
\end{theorem}
\begin{proof} The original statement of this theorem and its proof is found in \cite{CM21}. 
See \ref{sect_a2}. 
\end{proof}
\noindent Remark: The joint PDF is zero at $z=\b$. In the case of a conservative process on $\I$, the joint PDF in \eqref{joint_last_passage_pdf_explicit} 
satisfies \eqref{last_density_time-t-conserve_from_joint}. 
Indeed, by splitting the integral of $f_{g_\b(T), X_T}(t,z;x)$ over $z\in (l,\b)$ and $z\in (\b,r)$, and then reversing the order of differentiation in $y$ and integration in $z$, 
while using $\int_\b^r p_{\b}^+(T-t;y,z)dz = 1 - \P_{y}(\Tau^{-}_{\b} \le T-t)$ and 
$\int_l^\b p_{\b}^-(T-t;y,z)dz = 1 - \P_{y}(\Tau^{+}_{\b} \le T-t)$, recovers the marginal PDF 
$f_{g_\b(T)}(t;x)$ in \eqref{last-passage-pdf-explicit}.

Based on Theorem ~\ref{joint_last-passage-propn-time-t} we directly have the following spectral representations.

%
%
\begin{pro}\label{joint_last-passage-propn-time-t-new-formula}
The joint PDF of $(g_\b(T), X_T)$, $t\in (0,T)$, $x, \b\in \I$, has the following representation. 
For $z \in (l,\b)$:
\begin{align}\label{joint_last_passage_pdf_spectral_1}
&\!\!\!\!\!\!\!f_{g_\b(T), X_T}(t,z;x) = {p(t;x,\b) \over \m(\b)} \m(z)
\nonumber \\
&\,\,\,\,\,\,\,\times\begin{cases}
\displaystyle \sum_{n=1}^{\infty}e^{-\lambda_{n,k}^- (T-t)} \psi^+_{n}(z;k), &  \text{if $l$ is NONOSC},
\\
\displaystyle \sum_{n\ge 1}e^{-\lambda_{n,k}^- (T-t)} \psi^+_{n}(z;k) 
+ {1 \over \pi}\int_{\Lambda_-}^\infty \!e^{-\epsilon (T-t)}\, \textup{Im}
\bigg\{{\varphi^+_\lambda (z) \over \varphi^+_\lambda (\b)}\bigg\vert_{\lambda = \epsilon e^{-i\pi}}\bigg\}\,d\epsilon
, &  \text{if $l$ is O-NO}.
\end{cases}
\end{align}
For $z \in (\b,r)$:
\begin{align}\label{joint_last_passage_pdf_spectral_2}
&\!\!\!\!\!\!\!f_{g_\b(T), X_T}(t,z;x) = {p(t;x,\b) \over \m(\b)} \m(z)
\nonumber \\
&\,\,\,\,\,\,\,\times\begin{cases}
\displaystyle \sum_{n=1}^{\infty}e^{-\lambda_{n,k}^+ (T-t)}  \psi^-_{n}(z;k) , &  \text{if $r$ is NONOSC}, \\
\displaystyle \sum_{n\ge 1}e^{-\lambda_{n,k}^+ (T-t)}  \psi^-_{n}(z;k) 
+ {1 \over \pi}\int_{\Lambda_+}^\infty \!e^{-\epsilon (T-t)}\, \textup{Im}
\bigg\{{\varphi^-_\lambda (z) \over \varphi^-_\lambda (\b)}\bigg\vert_{\lambda = \epsilon e^{-i\pi}}\bigg\}\,d\epsilon, &  \text{if $r$ is O-NO}.
\end{cases}
\end{align}
\end{pro}
\begin{proof} The series follow directly from $f^\pm$ in Propositions~\ref{prop_spec_first_hit_1}.--\ref{prop_spec_first_hit_2}. An alternative proof is given in \ref{spectral_prop_joint_last_proof}. 
\end{proof}
\noindent We remark that, as in Proposition~\ref{last-passage-propn-time-t-spectral}, many applications require only Spectral Category I ($l$ and $r$ as NONOSC) where only the first expressions in  
\ref{joint_last_passage_pdf_spectral_1}-\ref{joint_last_passage_pdf_spectral_2} apply. 
As noted above, the summations for the O-NO cases may be empty with purely continuous spectral expansions.

Following the discussion at the end of Section \ref{subsect2}, 
we also have a simple relation between the joint PDFs of $(g^F_K(T), F_T)$ and $(g^X_\b(T), X_T)$:
\begin{align}\label{joint_last_PDF_X_to_F}
f_{g^F_K(T), F_T}(t,z; F_0) = \vert {\sf X}^\prime(z) \vert f_{g^X_\b(T), X_T}(t,{\sf X}(z);X_0)\,,
\end{align}
$z, K, F_0 \in \mathcal{D}$, $\b = {\sf X}(K)$, $X_0 =  {\sf X}(F_0)$. For the nonzero partly discrete portion we have 
\begin{align}\label{pmf_F_to_X_joint_last_time_1}
&\P_{F_0}(g^F_K(T) = 0, F_T \in dz) = \P_{X_0}(g^X_\b(T) = 0, X_T \in d \,{\sf X}(z))
\nonumber \\
&= \begin{cases} 
\P_{X_0}(m_T > \b , X_{T}  \in d \,{\sf X}(z))  = p_{\b}^+(T;X_0,{\sf X}(z)) {\sf X}^\prime(z) dz &, z,F_0 > K,
\\
\P_{X_0}(M_T < \b, X_{T}  \in d \,{\sf X}(z)) = p_{\b}^-(T;X_0,{\sf X}(z)) {\sf X}^\prime(z) dz &, z,F_0 < K,
\end{cases}
\end{align}
if ${\sf X}^\prime > 0$, and 
\begin{align}\label{pmf_F_to_X_joint_last_time_2}
&\P_{F_0}(g^F_K(T) = 0, F_T \in dz)
\nonumber \\
&= \begin{cases} 
\P_{X_0}(M_T < \b , X_{T}  \in d \,{\sf X}(z))  = p_{\b}^-(T;X_0,{\sf X}(z)) \vert {\sf X}^\prime(z) \vert dz &, z,F_0 > K,
\\
\P_{X_0}(m_T > \b, X_{T}  \in d \,{\sf X}(z)) = p_{\b}^+(T;X_0,{\sf X}(z)) \vert {\sf X}^\prime(z) \vert dz &, z,F_0 < K,
\end{cases}
\end{align}
if ${\sf X}^\prime < 0$. Analogous relations hold for the doubly defective distribution.

%
%
%
%
\subsection{Joint Distributions of Last Hitting Time, Process Value and Extrema}\label{subsect4}

The formulae in previous sections are now extended by considering the joint distribution of 
$(g_\b(T),X_T)$ subject to imposed killing at interior points of the state space $\I$. We begin by deriving the distribution of 
the pair $(g_\b(T),X_T)$ jointly with the sampled minimum and maximum of the process at any time $T>0$. 
By conditioning on $X_t$, where $l < a < b < r$, we have
\begin{align}
&\P_x(g_\b(T) \le t, m_T> a, M_T < b, X_T \in dz) 
\nonumber \\
&=\int_a^b \P_y(g_\b(T) \le t, m_T> a, M_T < b, X_T \in dz \vert X_t = y)\, p_{(a,b)}(t;x,y) dy
\nonumber \\
&=\int_a^b \P_y(\T_\b > T-t, m_{T-t}> a, M_{T-t} < b, X_{T-t} \in dz)\,  p_{(a,b)}(t;x,y) dy
\label{prop_joint_last_time_kill_ab-t-1}
\\
&=\begin{cases}
\int_a^\b \P_y(m_{T-t}> a, M_{T-t} < \b, X_{T-t} \in dz)\,  p_{(a,b)}(t;x,y) dy\,,\,\, z\in (a,\b),
\\
\int_\b^b \P_y(m_{T-t}> \b, M_{T-t} < b, X_{T-t} \in dz)\,  p_{(a,b)}(t;x,y) dy\,,\,\,z \in (\b,b),
\end{cases}
\label{prop_joint_last_time_kill_ab-t-2}
\\
&=\begin{cases}
\left(\int_a^\b p_{(a,\b)}(T-t;y,z)  p_{(a,b)}(t;x,y) dy\right) dz\,,\,\, z\in (a,\b),
\\
\left(\int_\b^b p_{(\b,b)}(T-t;y,z)  p_{(a,b)}(t;x,y) dy\right) dz\,,\,\, z \in (\b,b),
\end{cases}
\label{prop_joint_last_time_kill_ab-t-3}
\end{align}
for all $t\in (0,T)$, $X_0= x,z,\b\in (a,b)$. Note that $\P_x(g_\b(T) \le t, m_T> a, M_T < b, X_T \in dz)=0$ for $z=\b$. 
Equation (\ref{prop_joint_last_time_kill_ab-t-1}) follows from the Markov property and  time-homogeneity of the process. Equation (\ref{prop_joint_last_time_kill_ab-t-2}) holds since  
$\P_y(\T_\b > T-t, m_{T-t} > a, M_{T-t} < b) = \P_y(m_{T-t} > a, M_{T-t} < \b)$, for $y \in (a,\b)$, and similarly 
$\P_y(\T_\b > T-t, m_{T-t} > a, M_{T-t} < b) = \P_y(m_{T-t} > \b, M_{T-t} < b)$, for $y \in (\b,b)$. Equation (\ref{prop_joint_last_time_kill_ab-t-3}) follows directly from the definition of the time-$(T-t)$ transition PDFs, $p_{(a,\b)}$ and $p_{(\b,b)}$, for the process with imposed killing at both endpoints of the intervals $(a,\b)$ and $(\b,b)$, respectively. 

For $t=T$, $\P_x(g_\b(T) \le T, m_T> a, M_T < b, X_T \in dz) = 
\P_x(m_T> a, M_T < b, X_T \in dz) = p_{(a,b)}(T;x,z) dz$, which is trivially recovered as $t\nearrow T$ in \eqref{prop_joint_last_time_kill_ab-t-3}. In the limit $t \searrow 0$, the nonzero partly discrete portion of the distribution is also recovered:
\begin{align}\label{prop_joint_last_time_kill_ab-t-3-zero}
\P_x(g_\b(T) = 0, m_T> a, M_T < b, X_T \in dz) = \P_x(\T_\b > T, m_T> a, M_T < b, X_T \in dz) 
\nonumber \\
= \begin{cases} 
\P_x(m_T > a , M_T < \b, X_T \in dz) =  p_{(a,\b)}(T;x,z) dz &, x,z \in (a,\b),
\\
\P_x(m_T > \b, M_T < b, X_{T}  \in dz) =  p_{(\b,b)}(T;x,z) dz &, x,z \in (\b,b).
\end{cases}
\end{align}
We note the trivial case with zero joint probability when $x=\b$ or $z=\b$.

The joint probabilities in (\ref{prop_joint_last_time_kill_ab-t-1})-(\ref{prop_joint_last_time_kill_ab-t-3-zero}) can be related to joint probabilities of the last hitting time and the process value for the killed process $X_{(a,b)}$ defined in 
(\ref{X-Killed-ab}). In particular, defining the last hitting time to level $\b\in (a,b)$ of this process within time $T$ as 
$$g^{(a,b)}_{\b}(T) := \sup\{0\le u \le T: X_{(a,b),u} = \b \}\,,$$
where $X_{(a,b),0} = X_0 = x \in (a,b)$, gives the equivalence
\begin{align}\label{last-passage-CDF-time_0_to_t_kill_ab}
\P_x(g^{(a,b)}_{\b}(T) \le t, X_{(a,b),T} \in dz) = 
\P_x(g_\b(T) \le t, m_T> a, M_T < b, X_T \in dz) \,.
\end{align}
In analogy with \eqref{prop_joint_last_time-t-3}, we can express the marginal CDF of $g^{(a,b)}_{\b}(T)$ as
\begin{align}\label{prop_joint_last_time-t-3_ab}
\P_x(g^{(a,b)}_{\b}(T) \le t)  &= \P_x(g^{(a,b)}_{\b}(T) \le t, X_{(a,b),T} \in (a,b))
+ \P_x(g^{(a,b)}_{\b}(T) \le t, X_{(a,b),T} = \partial^\dagger),
\end{align}
(note: $\{X_{(a,b),T} = \partial^\dagger\} \equiv \{\T_{(a,b)}\le T\}$) where 
\begin{align}\label{prop_joint_last_time-in_ab}
\P_x(g^{(a,b)}_{\b}(T) \le t, X_{(a,b),T} \in (a,b)) 
&= \P_x(g_\b(T) \le t, m_T> a, M_T < b)
\nonumber \\
&= \int_a^b\,\P_x(g^{(a,b)}_{\b}(T) \le t, X_{(a,b),T} \in dz)\,.
\end{align}
Hence, we note that, from \eqref{prop_joint_last_time-t-3_ab}--\eqref{prop_joint_last_time-in_ab},
$$\P_x(g^{(a,b)}_{\b}(T) \le t) 
> \P_x(g_\b(T) \le t, m_T> a, M_T < b)$$ 
since $\P_x(g^{(a,b)}_{\b}(T) \le t, X_{(a,b),T} = \partial^\dagger) > 0$ gives the nonzero probability of all paths that are killed within time $T$ and having last hitting times within any prior time $t < T$. 
For $t=T$ we must have $\P_x(g^{(a,b)}_{\b}(T) \le T) = 1$, i.e., the last hitting time within time $T$ for the process $X_{(a,b)}$ must be in $[0,T]$. 

For $t=0$ we have 
$$\P_x(g^{(a,b)}_{\b}(T) = 0, X_{(a,b),T} \in dz) = \P_x(g_\b(T) = 0, m_T> a, M_T < b, X_T \in dz)$$
as given by (\ref{prop_joint_last_time_kill_ab-t-3-zero}). Integrating over $z\in (a,b)$ gives 
\begin{align}
\P_x(g^{(a,b)}_{\b}(T) = 0, X_{(a,b),T} \in (a,b)) 
&= \P_x(g_\b(T) = 0, m_T> a, M_T < b)
\nonumber \\
&= \begin{cases} 
\int_a^\b  p_{(a,\b)}(T;x,z) dz &, x \in (a,\b),
\\
\int_\b^b  p_{(\b,b)}(T;x,z) dz &, x \in (\b,b).
\label{prop_joint_last_ab_pmf}
\end{cases}
\end{align}
In analogy with \eqref{prop_last_time-t-1-limit}, we have the nonzero discrete portion of the distribution:
\begin{align}\label{prop_last_time-CDF_discrete_killed_ab}
\P_x(g^{(a,b)}_{\b}(T) = 0) 
&= 
\begin{cases}
\P_x(\T^+_\b (a) > T) = 
\displaystyle{\mathcal{S}[x,\b] \over \mathcal{S}[a,\b]} 
+ \sum_{n=1}^\infty {e^{- \lambda_n^{(a,\b)} T} \over \lambda_n^{(a,\b)}}\psi_n^+(x;a,\b)&, x \in (a,\b),
\\
\P_x(\T^-_\b (b) > T) = 
\displaystyle{\mathcal{S}[\b,x] \over \mathcal{S}[\b,b]}
+ \sum_{n=1}^\infty {e^{- \lambda_n^{(\b,b)} T} \over \lambda_n^{(\b,b)}} \psi_n^-(x;\b,b)&, x \in (\b,b).
\end{cases}
\end{align} 
These discrete spectral expansions follow by appropriately adopting \eqref{FHT_cdf_complement_tau_ab_1}--\eqref{FHT_cdf_complement_tau_ab_2} and \eqref{FHT_prop3_up_2}--\eqref{FHT_prop3_down_2}. 

The nonzero jointly defective portion of the distribution has the discrete spectral expansion:
\begin{align}
&\P_x(g^{(a,b)}_{\b}(T) = 0, X_{(a,b),T} = \partial^\dagger) 
= \P_x(g^{(a,b)}_{\b}(T) = 0) - \P_x(g^{(a,b)}_{\b}(T) = 0, X_{(a,b),T} \in (a,b))
\nonumber 
\\
\label{joint_last_ab_doubly_defective_new}
&= \begin{cases}
 \P_x(\T^-_a (\b) \le T) 
= \displaystyle{\mathcal{S}[x,\b] \over \mathcal{S}[a,\b]} - 
\sum_{n=1}^\infty {e^{- \lambda_n^{(a,\b)} T} \over \lambda_n^{(a,\b)}}\psi_n^-(x;a,\b)&, x \in (a,\b),
\\
 \P_x(\T^+_b (\b) \le T) 
= \displaystyle{\mathcal{S}[\b,x] \over \mathcal{S}[\b,b]}
- \sum_{n=1}^\infty {e^{- \lambda_n^{(\b,b)} T} \over \lambda_n^{(\b,b)}} \psi_n^+(x;\b,b)&, x \in (\b,b).
\end{cases}
\end{align}
The above probability expressions are analogues of \eqref{joint_last_doubly_defective_1} and \eqref{joint_last_doubly_defective_2}. Their equivalence, given by the series in \eqref{joint_last_ab_doubly_defective_new}, is shown in \ref{Proof_doubly_defective_ab}.

Let $f_{g^{(a,b)}_{\b}(T), X_{(a,b),T}}(t,z;x)$ denote the {\it joint density} of 
$(g^{(a,b)}_{\b}(T), X_{(a,b),T})$ for the process started at $X_0 =x$, i.e., by \eqref{last-passage-CDF-time_0_to_t_kill_ab},
\begin{align}
{\partial \over \partial t}\P_x(g^{(a,b)}_{\b}(T) \le t, X_{(a,b),T} \in dz) 
&= {\partial \over \partial t}\P_x(g_\b(T) \le t, m_T> a, M_T < b, X_T \in dz) 
\nonumber \\
&= f_{g^{(a,b)}_{\b}(T), X_{(a,b),T}}(t,z;x) dz,
\label{joint_PDF_last_ab_defn}
\end{align}
for all $x,\b,z \in (a,b)$, $t\in (0,T)$. Equivalently,
\[
f_{\!g^{(a,b)}_{\b}(T), X_{(a,b),T}}\!(t,z;x) dtdz \equiv 
\P_x(g^{(a,b)}_{\b}(T) \!\in\! dt, X_{(a,b),T} \!\in\! dz) = \P_x(g_{\b}(T) \!\in\! dt, m_T> a, M_T < b, X_T \!\in\! dz).
\]
Hence, $f_{g^{(a,b)}_{\b}(T), X_{(a,b),T}}$ is also the joint PDF of the pair $(g_{\b}(T) ,X_T)$ subject to the side condition $\{m_T> a, M_T < b\}$.

Differentiating w.r.t. $t$ on both sides of (\ref{prop_joint_last_time-t-3_ab}) gives the analogue of 
(\ref{last_density_time-t-nonconserve_from_joint}):
\begin{align}\label{prop_marginal_last_time-t_ab}
f_{g^{(a,b)}_{\b}(T)}(t;x) = \int_a^b f_{g^{(a,b)}_{\b}(T), X_{(a,b),T}}(t,z;x) dz
+ {\partial \over \partial t}\P_x(g^{(a,b)}_{\b}(T) \le t, X_{(a,b),T} = \partial^\dagger),
\end{align}
where $f_{g^{(a,b)}_{\b}(T)}(t;x) := {\partial \over \partial t}\P_x(g^{(a,b)}_{\b}(T) \le t)$ is the (marginal) PDF of  $g^{(a,b)}_{\b}(T)$. Here we used (\ref{prop_joint_last_time-in_ab}), where the integral term in (\ref{prop_marginal_last_time-t_ab}) is equal to the PDF of the last hitting time, $g_{\b}(T)$, subject to the side condition $\{m_T > a, M_T < b\}$: 
$${\partial \over \partial t}\P_x(g^{(a,b)}_{\b}(T) \le t, X_{(a,b),T} \in (a,b)) \equiv 
{\partial \over \partial t}\P_x(g_{\b}(T) \le t, m_T> a, M_T < b).$$ 
As seen in (\ref{prop_marginal_last_time-t_ab}), the integral of the joint PDF (over the state space $(a,b)$) does not recover the marginal density of $g^{(a,b)}_{\b}(T)$, as must be the case since the killed process $X_{(a,b)}$ is clearly nonconservative on $(a,b)$.

We remark that the second term on the r.h.s. of (\ref{prop_marginal_last_time-t_ab}) corresponds to the PDF of 
$g^{(a,b)}_{\b}(T)$ subject to the side condition $\{X_{(a,b),T} =\partial^\dagger\}$, which is given by subtracting the integral of the joint PDF in (\ref{prop_marginal_last_time-t_ab}) from the marginal PDF $f_{g^{(a,b)}_{\b}(T)}(t;x)$. 

The (marginal) distribution of $g^{(a,b)}_{\b}(T)$ follows by a simple extension of the analysis in 
Section~\ref{subsect2}, where the killed diffusion $X_{(a,b)}$ is considered in the place of $X$. In particular, by conditioning and employing the total law of probabilities, we have the analogue of 
(\ref{prop_last_time-t-1}):
\begin{align}\label{prop_last_time-CDF_1_killed_ab}
\P_x(g^{(a,b)}_{\b}(T) \le t) &= \int_a^b \P_x(g^{(a,b)}_{\b}(T) \le t \vert X_{(a,b),t} = y)\, p_{(a,b)}(t;x,y) dy 
\nonumber \\
&\,\,\,+ \P_x(g^{(a,b)}_{\b}(T) \le t \vert X_{(a,b),t} =\partial^\dagger) \cdot \P_x(X_{(a,b),t} =\partial^\dagger)\,.
\end{align}
Note: $\P_x(g^{(a,b)}_{\b}(T) \le t \vert X_{(a,b),t} =\partial^\dagger) = 1$, 
$\P_x(X_{(a,b),t} =\partial^\dagger) = 1 - \P_x(X_{(a,b),t} \in (a,b)) = 
1 - \int_a^b p_{(a,b)}(t;x,y) dy$. Moreover, by the Markov property, the conditional probability within the integral is equivalently set to
$$\P_x(g^{(a,b)}_{\b}(T) \le t \vert X_{(a,b),t} = y) = \P_y(\T^-_\b (b) > T - t) \cdot \mathbb {I}_{\{\b < y < b\}} + 
\P_y(\T^+_\b (a) > T - t) \cdot \mathbb {I}_{\{a < y < \b\}},$$  
with $\T^+_\b (a)$ as first hitting time up at $\b$ before $a$ and $\T^-_\b (b)$ 
as first hitting time down at $\b$ before $b$. 
Combining these expressions into (\ref{prop_last_time-CDF_1_killed_ab}) gives the analogue of (\ref{prop_last_time-t-1-prime}):
\begin{align}\label{prop_last_time-CDF_2_killed_ab}
\P_x(g^{(a,b)}_{\b}(T) \le t) &= 1 - \int_a^\b \P_y(\T^+_\b (a) \le T - t)\,  p_{(a,b)}(t;x,y) dy 
- \int_\b^b \P_y(\T^-_\b (b) \le T - t)\,  p_{(a,b)}(t;x,y) dy \,.
\end{align}
For $t=T$, we clearly have $\P_x(g^{(a,b)}_{\b}(T) \le T) = 1$ (as required) and $t \searrow 0$ in (\ref{prop_last_time-CDF_2_killed_ab}) also recovers \eqref{prop_last_time-CDF_discrete_killed_ab}.
Given \eqref{prop_last_time-CDF_discrete_killed_ab} and the (marginal) density, we can also use the analogue of \eqref{last-passage-CDF-time_0_to_t} to compute the CDF of $g^{(a,b)}_{\b}(T)$, i.e., for all $t\in [0,T]$,
\begin{equation}\label{last-passage-CDF-kill_a_b_time_0_to_t}
\P_x(g^{(a,b)}_{\b}(T) \le t) = \P_x(g^{(a,b)}_{\b}(T) = 0) + \int_0^t f_{g^{(a,b)}_{\b}(T)}(u;x) du.
\end{equation}

%
%
We now present general formulae for computing the above marginal and joint PDFs. In particular, they are respective extensions (or analogues) of Theorem \ref{last-passage-propn-time-t} and 
\ref{joint_last-passage-propn-time-t} for the process with imposed killing at two interior points $a$ and $b$. 
In Theorems \ref{last-passage-propn-time-t-kill-ab}--\ref{joint_last-passage-propn-time-t_ab}, we assume $X_{(a,b),0} \equiv X_0 = x \in (a,b)$, with last hitting level $\b\in (a,b)$, $t\in (0,T)$, $T\in (0,\infty)$.
%
%
%
\begin{theorem}\label{last-passage-propn-time-t-kill-ab}
The probability density function of $g^{(a,b)}_{\b}(T)$ is given by
\begin{equation}\label{last-passage-density-time-t-Laplace-kill-ab}
f_{g^{(a,b)}_{\b}(T)}(t;x) =\xi_{(a,b)}(T-t;\b) \, p_{(a,b)}(t;x,\b)\,,
\end{equation}
where $\xi_{(a,b)}(u;\b)$ is given by the Laplace inverse
\begin{align}\label{last-passage-density-phi-killed}
\xi_{(a,b)}(u;\b) &= {\mathcal L}_\lambda^{-1}\left\{{1 \over \lambda G_{(a,b)}(\lambda;\b,\b)}\right\}(u)
\\
&= {1 \over \m(\b)\s(\b)} \left[ {\partial \over \partial y} \P_y(\T_\b^+(a) \le u)\bigg\vert_{y=\b-} 
- {\partial \over \partial y} \P_y(\T^-_\b(b) \le u)\bigg\vert_{y=\b+} \right].
\label{last-passage-density-phi-killed_2}
\end{align}
Hence, explicitly we have
\begin{align}\label{last-passage-pdf-explicit-kill-ab}
f_{g^{(a,b)}_{\b}(T)}(t;x) = {p_{(a,b)}(t;x,\b) \over \m(\b)}{1 \over \s(\b)} \!
\left[ {\partial \over \partial y} \P_y(\T_\b^+(a) \le T-t)\bigg\vert_{y=\b-} 
- {\partial \over \partial y} \P_y(\T^-_\b(b) \le T-t)\bigg\vert_{y=\b+} \right].
\end{align}
\end{theorem}
\begin{proof} The original statement of this theorem and its proof is found in \cite{CM21}. 
See \ref{sect_a3}. 
\end{proof}
For infinite time horizon, the density of $g^{(a,b)}_{\b} :=\sup\{u \ge 0: X_{(a,b),\,u} = \b\}$ is simply recovered by \eqref{last-passage-pdf-explicit-kill-ab} in the limit $T\to\infty$, i.e., $f_{g^{(a,b)}_{\b}}(t;x) = \lim\limits_{T\to\infty}f_{g^{(a,b)}_{\b}(T)}(t;x)$:
\begin{equation}\label{last-passage-density-inf-ab}
f_{g^{(a,b)}_{\b}}(t;x)  = \left[{1 \over {\mathcal S[a,\b]}} + {1 \over {\mathcal S[\b,b]}} \right]{p_{(a,b)}(t;x,\b) \over \m(\b)},\,\,\,t \in (0,\infty).
\end{equation}
Here we used \eqref{FHT-b-killa-infinity}--\eqref{FHT-a-killb-infinity} which also gives
\begin{equation}\label{last-passage-discrete-inf-ab}
\P_x(g^{(a,b)}_{\b} = 0) = {({\mathcal S[x,\b]})^+ \over {\mathcal S[a,\b]}} + {({\mathcal S[\b,x]})^+ \over {\mathcal S[\b,b]}}.
\end{equation}
The expressions in \eqref{last-passage-density-inf-ab}--\eqref{last-passage-discrete-inf-ab} are analogues of \eqref{last-passage-density-infinite-T} and \eqref{last-passage-discrete-infinite-T} 
where $\ind_{E_a} = \ind_{E_b} = 1$. The diffusion killed at both endpoints is obviously transient with 
$\P_x(g^{(a,b)}_{\b} = \infty) = 0$.

%
%
\begin{theorem}\label{joint_last-passage-propn-time-t_ab}
The joint PDF of $(g^{(a,b)}_{\b}(T), X_{(a,b),T})$ has the representation:
\begin{equation}\label{joint_last_passage_pdf_explicit_ab}
f_{g^{(a,b)}_{\b}(T), X_{(a,b),T}}(t,z;x) 
= {p_{(a,b)}(t;x,\b) \over \m(\b)}\m(z)
\begin{cases}
f^+(T-t; z,\b\vert a)
&, z \in (a,\b),
\\
f^-(T-t; z,\b\vert b)
&, z \in (\b,b),
\end{cases}
\end{equation}
with $f^\pm$ as first hitting time densities to hit $\b$ before $a$ or $b$, respectively, i.e., 
$\displaystyle \m(z)f^+(T-t; z,\b\vert a) = -{1 \over \s(\b)}{\partial \over \partial y} p_{(a,\b)}(T-t;y,z)\big\vert_{y=\b-}$ and 
$\displaystyle\m(z)f^-(T-t; z,\b\vert b) = {1 \over \s(\b)}{\partial \over \partial y} p_{(\b,b)}(T-t;y,z)\big\vert_{y=\b+}.$
\end{theorem}
\begin{proof} The original statement of this theorem and its proof is found in \cite{CM21}. 
See \ref{sect_a4}. 
\end{proof}
Note that the joint PDF is trivially zero at $z=\b$. Based on Theorems~\ref{last-passage-propn-time-t-kill-ab}--\ref{joint_last-passage-propn-time-t_ab} and  Proposition~\ref{prop_spec_first_hit_ab}, purely discrete spectral series representations for the marginal and joint densities arise as follows.
%
%
\begin{pro}\label{joint_last-passage-propn-time-t_ab-new-version}
The joint PDF of $(g^{(a,b)}_{\b}(T), X_{(a,b),T})$ has the representation:
\begin{equation}\label{joint_last_passage_pdf_explicit_ab_new}
f_{g^{(a,b)}_{\b}(T), X_{(a,b),T}}(t,z;x) 
= {p_{(a,b)}(t;x,\b) \over \m(\b)}\m(z)
\left\{
\begin{array}{lr}
            \sum_{n=1}^{\infty}e^{-\lambda_n^{(a,\b)} (T-t)} \psi^+_{n}(z;a,\b), &  z\in (a,\b), 
\\
            \sum_{n=1}^{\infty}e^{-\lambda_n^{(\b,b)} (T-t)}  \psi^-_{n}(z;\b,b) , &  z\in (\b,b), 
\end{array}
\right.
\end{equation}
with $\psi^+_{n}(z;a,\b)$ and $\psi^-_{n}(z;\b,b)$ respectively given by \eqref{FHT_eigenfunctions_ab} and 
$p_{(a,b)}(t;x,\b)$ given by \eqref{u_spectral_3} where $y=\b$.
The (marginal) PDF of $g^{(a,b)}_{\b}(T)$ has the representation:
\begin{equation}\label{last_passage_pdf_discrete_spec}
f_{g^{(a,b)}_{\b}(T)}(t;x) = {p_{(a,b)}(t;x,\b) \over \m(\b)} 
\left[ \mathcal{S}(a,b;\b) 
+ 
\sum_{n=1}^{\infty}
\bigg(
e^{-\lambda_n^{(a,\b)} (T-t)} 
{\hat\psi^+_{n}(a,\b) \over \lambda_n^{(a,\b)} }
+  
e^{-\lambda_n^{(\b,b)} (T-t)}  
{\hat\psi^-_{n}(\b,b) \over \lambda_n^{(\b,b)} }
\bigg)\right],
\end{equation}
where $\mathcal{S}(a,b;\b) := \frac{1}{\mathcal{S}[a,\b]} + \frac{1}{\mathcal{S}[\b,b]} = 
{\mathcal{S}[a,b] \over \mathcal{S}[a,\b]\mathcal{S}[\b,b]}$
and
\begin{align}\label{psi_hat_derivatives}
\hat\psi^+_{n}(a,\b) &:= -{1 \over \s(\b)}{\partial \over \partial x}\psi_n^+(x;a,\b)\vert_{x=\b} 
= { w_{-\lambda_n}\over \Delta(a,\b; \lambda_n)}{\varphi^+_{-\lambda_n}\!(a) \over \varphi^+_{-\lambda_n}\!(\b)}
\equiv { w_{-\lambda_n}\over \Delta(a,\b; \lambda_n)}
{\varphi^-_{-\lambda_n}\!(a) \over \varphi^-_{-\lambda_n}\!(\b)},\,\, \lambda_n = \lambda_n^{(a,\b)}, 
\nonumber \\
\hat\psi^-_{n}(\b,b) &:= {1 \over \s(\b)}{\partial \over \partial x}\psi_n^-(x;\b,b)\vert_{x=\b} 
= { w_{-\lambda_n}\over \Delta(\b,b; \lambda_n)}{\varphi^+_{-\lambda_n}\!(b) \over \varphi^+_{-\lambda_n}\!(\b)}
\equiv { w_{-\lambda_n}\over \Delta(\b,b; \lambda_n)}
{\varphi^-_{-\lambda_n}\!(b) \over \varphi^-_{-\lambda_n}\!(\b)}, \,\, \lambda_n = \lambda_n^{(\b,b)}.
\end{align}
\end{pro}
\begin{proof} See \ref{proof_joint_last_spec_ab}. 
\end{proof}
\noindent Observe that the term multiplying $p_{(a,b)}(t;x,\b)$ in \eqref{last_passage_pdf_discrete_spec} is a series  representation of $\xi_{(a,b)}(T-t;\b)$ in Theorem \ref{last-passage-propn-time-t-kill-ab}, which, as shown in \ref{proof_joint_last_spec_ab}, is the Laplace inverse of the function in \ref{last-passage-density-phi-killed}. 

Substituting the discrete spectral expansion for $p_{(a,b)}(t;x,\b)$ within \eqref{joint_last_passage_pdf_explicit_ab_new} and \eqref{last_passage_pdf_discrete_spec} produces an expression involving a double series for the joint and marginal densities. We can, for example, use it within the integral term in \eqref{last-passage-CDF-kill_a_b_time_0_to_t} to obtain a formula in terms of single and double series for the CDF of $g^{(a,b)}_{\b}(T)$ when combined with the series in \eqref{prop_last_time-CDF_discrete_killed_ab} for the discrete portion. In particular, we have the continuous portion of the CDF:
\begin{align}\label{series_CDF_g_K_ab_T}
\int_0^t f_{g^{(a,b)}_{\b}(T)}(u;x) du 
&= \sum_{m=1}^\infty \bigg[\mathcal{S}(a,b;\b){[1 - e^{-\lambda_m t}] \over \lambda_m}
+ \sum_{n=1}^\infty \bigg({\hat\psi^+_{n}(a,\b)  \over \lambda_n^{(a,\b)}} \alpha_+(t,T;m,n) 
\nonumber \\
&\phantom{= \sum_{m=1}^\infty \bigg[} + {\hat\psi^-_{n}(\b,b)  \over \lambda_n^{(\b,b)}} \alpha_-(t,T;m,n) 
\bigg)
\bigg]\phi_m(x)\phi_m(\b)
\end{align}
with $\lambda_m\equiv\lambda_m^{(a,b)}$, $\phi_m = \phi^{(a,b)}_m$ 
given by \eqref{eigen_trans_pdf_3}--\eqref{eigenfunc3}, and coefficients 
\begin{align*}
\alpha_+(t,T;m,n) := e^{-\lambda_n^{(a,\b)} T}
\begin{cases}
{e^{(\lambda_n^{(a,\b)} - \lambda_m)t} - 1 \over \lambda_n^{(a,\b)} - \lambda_m}
& , \lambda_n^{(a,\b)} \ne \lambda_m,
\\
t & , \lambda_n^{(a,\b)} = \lambda_m, 
\end{cases}
\end{align*}
and $\alpha_-(t,T;m,n)$ defined in the same manner where $\lambda_n^{(a,\b)}$ is replaced by $\lambda_n^{(\b,b)}$.

%
%

The above formulae are readily extended to the diffusion $X_b$ defined by \eqref{X-Killed-b} with imposed killing at a single interior point $b\in (l,r)$. This corresponds to a single side condition on either the minimum or maximum of the process. The last hitting time to $\b$ of this process within time $T>0$ is defined as 
$$g^b_{\b}(T) := \sup\{0\le u \le T: X_{b,\,u} = \b \}\,,$$
where $X_{b,\,0} = X_0 =x$. 
The two cases that arise are: (i) $x,\b \in \I_b^-$ and (ii) $x,\b \in \I_b^+$. 

Let us first consider case (i). The analogues 
of (\ref{last-passage-CDF-time_0_to_t_kill_ab})--(\ref{prop_joint_last_time-in_ab}) are
\begin{align}\label{last-passage-CDF-time_0_to_t_kill_b-below}
\P_x(g^{b}_{\b}(T) \le t, X_{b,T} \in dz) &= 
\P_x(g_\b(T) \le t, M_T < b, X_T \in dz) \,,
\\
\P_x(g^{b}_{\b}(T) \le t)  &= \P_x(g^{b}_{\b}(T) \le t, X_{b,T} \in \I_b^-)
+ \P_x(g^{b}_{\b}(T) \le t, X_{b,T} = \partial^\dagger),
\label{prop_joint_last_time-t-3_b-below}
\\
\P_x(g^{b}_{\b}(T) \le t, X_{b,T} \in \I_b^-) 
&= \int_l^b\, \P_x(g_\b(T) \le t, M_T < b, X_T \in dz)\,.
\label{prop_joint_last_time-in_b-below}
\end{align}
From (\ref{prop_joint_last_time-t-3_b-below}) we see that $\P_x(g^{b}_{\b}(T) \le t) > \P_x(g^{b}_{\b}(T) \le t, X_{b,T} \in \I_b^-)$ where $\P_x(g^{b}_{\b}(T) \le t, X_{b,T} = \partial^\dagger) > 0$ is the probability for paths having last hitting within time $t$ and killed by time $T > t$.

The nonzero partly discrete (defective) portion of the distribution is given by 
\begin{align}
\P_x(g^{b}_{\b}(T) = 0, X_{b,T} \in dz) = \P_x(g_\b(T) = 0, M_T < b, X_T \in dz)
= \begin{cases} 
p^-_{\b}(T;x,z) dz &, x,z \in (l,\b),
\\
p_{(\b,b)}(T;x,z) dz &, x,z \in (\b,b).
\label{prop_joint_last_b-below_pmf}
\end{cases}
\end{align}
The analogue of (\ref{prop_last_time-CDF_discrete_killed_ab}) is
\footnote{Note that $\T^+_\b$ is simply the first hitting time up to level $\b$, which also corresponds to $\T^+_\b(l)$ 
in case $l$ is exit-non-entrance or regular killing, i.e., the first hitting time up to $\b$ before the process exits or is killed at $l$. Similarly, the first hitting time down to level $\b$, $\T^-_\b$, also corresponds to $\T^-_\b(r)$ in case $r$ is exit-non-entrance or regular killing, i.e., the first hitting time down to $\b$ before the process exits or is killed at $r$.}
\begin{align}\label{last_time-CDF_discrete_killed_below_b}
\P_x(g^b_{\b}(T) = 0) = 
\begin{cases}
\P_x(\T^+_\b > T) &, x \in (l,\b),
\\
\P_x(\T^-_\b (b) > T) &, x \in (\b,b).
\end{cases}
\end{align}
For $x \in (l,\b)$, this admits the same spectral expansion as in \ref{prop_last_time-t-1-limit}. For $x \in (\b,b)$, the spectral expansion is given as in \eqref{prop_last_time-CDF_discrete_killed_ab}. 
The nonzero jointly defective distribution is given by
\begin{align}\label{doubly_defective_kill_up_b}
\P_x(g^{b}_{\b}(T) = 0, X_{b,T} = \partial^\dagger) &= 
\P_x(g^{b}_{\b}(T) =0)  - \P_x(g^{b}_{\b}(T) =0, X_{b,T} \in \I_b^-)
\nonumber \\
&= \begin{cases}
 \P_x(\T^-_l (\b) \le T) &, x \in (l,\b),
\\
 \P_x(\T^+_b (\b) \le T) &, x \in (\b,b).
\end{cases}
\end{align}
Hence, the spectral expansion for this defective distribution is given by \eqref{joint_last_doubly_defective_formula} 
for $x \in (l,\b)$ if endpoint $l$ is nonconservative (and is zero if $l$ is conservative) and by  \eqref{joint_last_ab_doubly_defective_new} for $x \in (\b,b)$.

Let $f_{g^{b}_{\b}(T), X_{b,T}}(t,z;x)$ denote the {\it joint density} of 
$(g^{b}_{\b}(T), X_{b,T})$, i.e., 
$$\P_x(g^{b}_{\b}(T) \in dt, X_{b,T} \in dz) = f_{g^{b}_{\b}(T), X_{b,T}}(t,z;x) dtdz$$
where
\begin{align}
{\partial \over \partial t}\P_x(g^{b}_{\b}(T) \le t, X_{b,T} \in dz) 
= {\partial \over \partial t}\P_x(g_\b(T) \le t, M_T < b, X_T \in dz) 
= f_{g^{b}_{\b}(T), X_{b,T}}(t,z;x) dz,
\label{joint_PDF_last_b-below_defn}
\end{align}
for all $x,\b,z \in (l,b)$, $t\in (0,T)$. That is, $f_{g^{b}_{\b}(T), X_{b,T}}$ is also the joint PDF of the pair 
$(g_{\b}(T) ,X_T)$ subject to the side condition on the maximum, $\{M_T < b\}$. The analogue of 
(\ref{prop_marginal_last_time-t_ab}) reads:
\begin{align}\label{prop_marginal_last_time-t_b-below}
f_{g^{b}_{\b}(T)}(t;x) = \int_l^b f_{g^{b}_{\b}(T), X_{b,T}}(t,z;x) dz
+ {\partial \over \partial t}\P_x(g^{b}_{\b}(T) \le t, X_{b,T} = \partial^\dagger),
\end{align}
where $f_{g^{b}_{\b}(T)}(t;x) := {\partial \over \partial t}\P_x(g^{b}_{\b}(T) \le t)$ is the (marginal) 
PDF of $g^{b}_{\b}(T)$. From (\ref{prop_marginal_last_time-t_b-below}), we note that the integral of the joint PDF does not recover the PDF of $g^{b}_{\b}(T)$ since the $X_{b}$-process is killed at $b$, i.e., it is clearly nonconservative on 
$(l,b)$.

By similar steps leading to \eqref{prop_last_time-CDF_2_killed_ab}, we have the CDF of $g^b_{\b}(T)$:
\begin{align}\label{prop_last_time-CDF_2_killed_below_b}
\P_x(g^b_{\b}(T) \le t) = 1 - \int_l^\b \P_y(\T^+_\b \le T - t)\,  p_b^-(t;x,y) dy 
- \int_\b^b \P_y(\T^-_\b (b) \le T - t)\,  p_b^-(t;x,y) dy \,.
\end{align}

For case (ii), the above analysis follows in similar fashion. 
In particular, we have the analogues of 
(\ref{last-passage-CDF-time_0_to_t_kill_b-below})-(\ref{prop_joint_last_time-in_b-below}) with side condition 
$\{M_T < b\}$ replaced by $\{m_T > b\}$, $\I_b^-$ replaced by $\I_b^+$ and integral over $(l,b)$ replaced by $(b,r)$. In the place of 
(\ref{prop_joint_last_b-below_pmf}) we have:
\begin{align}
\P_x(g^{b}_{\b}(T) = 0, X_{b,T} \in dz) &= \P_x(g_\b(T) = 0, m_T > b, X_T \in dz)
= \begin{cases} 
p_{(b,\b)}(T;x,z) dz &, x,z \in (b,\b),
\\
p^+_{\b}(T;x,z) dz &, x,z \in (\b,r).
\label{prop_joint_last_b-above_pmf}
\end{cases}
\end{align}
The joint density $f_{g^{b}_{\b}(T), X_{b,T}}(t,z;x)$ is now nonzero for $x,\b,z \in (b,r)$, $x\ne \b$, $t\in (0,T)$, where $\{m_T > b\}$ replaces $\{M_T < b\}$ in the analogue of (\ref{joint_PDF_last_b-below_defn}) and the 
integral in (\ref{prop_marginal_last_time-t_b-below}) is now over $(b,r)$. 
In the place of \eqref{last_time-CDF_discrete_killed_below_b}-\eqref{doubly_defective_kill_up_b} we now have:
\begin{align}\label{prop_last_time-CDF_discrete_killed_above_b}
\P_x(g^b_{\b}(T) = 0) = 
\begin{cases}
\P_x(\T^+_\b(b) > T) &, x \in (b,\b),
\\
\P_x(\T^-_\b > T) &, x \in (\b,r),
\end{cases}
\\
\label{doubly_defective_kill_down_b}
\P_x(g^{b}_{\b}(T) = 0, X_{b,T} = \partial^\dagger) 
= \begin{cases}
 \P_x(\T^-_b (\b) \le T) &, x \in (b,\b),
\\
 \P_x(\T^+_r (\b) \le T) &, x \in (\b,r).
\end{cases}
\end{align}
The probability in \eqref{prop_last_time-CDF_discrete_killed_above_b} has spectral expansion given by 
the first equation line in \ref{prop_last_time-t-1-limit} for $x \in (\b,r)$ and by 
the first equation line in \eqref{prop_last_time-CDF_discrete_killed_ab}, with parameter $a$ replaced by $b$, 
for $x \in (b,\b)$. 
The probability in \eqref{doubly_defective_kill_down_b} has spectral expansion given by the first equation line in \eqref{joint_last_ab_doubly_defective_new}, with parameter $a$ replaced by $b$, for $x \in (b,\b)$, and by \eqref{joint_last_doubly_defective_formula} for $x \in (\b,r)$ if endpoint $r$ is nonconservative (and is zero if $r$ is conservative).

The expression for the CDF in \eqref{prop_last_time-CDF_2_killed_below_b} is now replaced by
\begin{align}\label{prop_last_time-CDF_2_killed_above_b}
\P_x(g^b_{\b}(T) \le t) = 1 - \int_b^\b \P_y(\T^+_\b(b) \le T - t)\,  p_b^+(t;x,y) dy 
- \int_\b^r \P_y(\T_\b^- \le T - t)\,  p_b^+(t;x,y) dy \,
\end{align}

Theorems \ref{last-passage-propn-time-t-kill-b}--\ref{joint_last-passage-propn-time-t-b} below are analogues of Theorems \ref{last-passage-propn-time-t-kill-ab}--\ref{joint_last-passage-propn-time-t_ab} where $X_{b,0} \equiv X_0 = x$, $t\in (0,T)$, $T\in (0,\infty)$.

\begin{theorem}\label{last-passage-propn-time-t-kill-b}
The probability density function of $g^b_{\b}(T)$ is given as follows.

(i) For $x,\b \in \I_b^-$, 
\begin{equation}\label{last-passage-density-time-t-Laplace-kill-b-below}
f_{g^b_{\b}(T)}(t;x) =\xi_{b}(T-t;\b) \, p^-_{b}(t;x,\b)\,,
\end{equation}
$t\in (0,T)$, where $\xi_{b}(u;\b)$ is given by the Laplace inverse
\begin{align}\label{last-passage-density-phi-b-below}
\xi_{b}(u;\b) = {\mathcal L}_\lambda^{-1}\!\left\{\!{1 \over \lambda G^-_{b}(\lambda;\b,\b)}\!\right\}\!(u)
= {1 \over \m(\b)\s(\b)} \left[ {\partial \over \partial y} \P_y(\T_\b^+ \le u)\bigg\vert_{y=\b-} 
- {\partial \over \partial y} \P_y(\T^-_\b(b) \le u)\bigg\vert_{y=\b+} \right].
\end{align}
Hence, explicitly we have
\begin{align}\label{last-passage-pdf-explicit-kill-b-below}
f_{g^b_{\b}(T)}(t;x) = {p^-_{b}(t;x,\b) \over \m(\b)}{1 \over \s(\b)} 
 \left[ {\partial \over \partial y} \P_y(\T_\b^+ \le T-t)\bigg\vert_{y=\b-} 
 - {\partial \over \partial y} \P_y(\T^-_\b(b) \le T-t)\bigg\vert_{y=\b+} \right].
\end{align}

(ii) For $x,\b \in \I_b^+$, 
\begin{equation}\label{last-passage-density-time-t-Laplace-kill-b-above}
f_{g^b_{\b}(T)}(t;x) =\xi_{b}(T-t;\b) \, p^+_{b}(t;x,\b)\,,
\end{equation}
$t\in (0,T)$, where $\xi_{b}(u;\b)$ is given by the Laplace inverse
\begin{align}\label{last-passage-density-phi-b-above}
\xi_{b}(u;\b) &= {\mathcal L}_\lambda^{-1}\!\left\{\!{1 \over \lambda G^+_{b}(\lambda;\b,\b)}\!\right\}\!(u)
= {1 \over \m(\b)\s(\b)} \left[ {\partial \over \partial y} \P_y(\T_\b^+(b) \le u)\bigg\vert_{y=\b-} 
- {\partial \over \partial y} \P_y(\T^-_\b \le u)\bigg\vert_{y=\b+} \right].
\end{align}
Hence, explicitly we have
\begin{align}\label{last-passage-pdf-explicit-kill-b-above}
f_{g^b_{\b}(T)}(t;x) = {p^+_{b}(t;x,\b) \over \m(\b)}{1 \over \s(\b)} 
 \left[ {\partial \over \partial y} \P_y(\T_\b^+(b) \le T-t)\bigg\vert_{y=\b-} 
 - {\partial \over \partial y} \P_y(\T^-_\b \le T-t)\bigg\vert_{y=\b+} \right].
\end{align}
\end{theorem}
\begin{proof} The original statement of this theorem and its proof is found in \cite{CM21}. 
The steps are similar to those in the proof of  
Theorem~\ref{last-passage-propn-time-t-kill-ab}. For $x,\b \in \I_b^-$, (\ref{prop_last_time-CDF_2_killed_below_b}) is used, whereas for $x,\b \in \I_b^+$, (\ref{prop_last_time-CDF_2_killed_above_b}) is used.
\end{proof}

We remark that the density of $g^b_{\b} :=\sup\{u \ge 0: X_{b,u} = \b\}$ is also simply recovered by taking the limit 
$T\to\infty$ of \eqref{last-passage-density-time-t-Laplace-kill-b-below} and \eqref{last-passage-pdf-explicit-kill-b-above}, i.e.,
$f_{g^b_{\b}}(t;x) = \lim\limits_{T\to\infty}f_{g^b_{\b}(T)}(t;x)$ gives, for $t\in (0,\infty)$:
\begin{equation}\label{last-passage-density-inf-b}
 f_{g^b_{\b}}(t;x) =
\begin{cases}
\displaystyle \big[{\ind_{E_l} \over {\mathcal S(l,\b]}} + {1 \over {\mathcal S[\b,b]}}\big]
{p_b^-(t;x,\b) \over \m(\b)}
&, x,\b \in \I_b^-,
\\
\displaystyle \big[{1 \over {\mathcal S[b,\b]}} + {\ind_{E_r} \over {\mathcal S[\b,r)}} \big]
{p_b^+(t;x,\b) \over \m(\b)}
&, x,\b \in \I_b^+.
\end{cases}
\end{equation}
The respective discrete portion is given by 
\begin{equation}\label{last-passage-discrete-inf-b}
\P_x(g^b_{\b} = 0) =
\begin{cases}
\displaystyle {({\mathcal S[x,\b]})^+ \over {\mathcal S(l,\b]}}\ind_{E_l} + {({\mathcal S[\b,x]})^+ \over {\mathcal S[\b,b]}}
&, x,\b \in \I_b^-,
\\
\displaystyle {({\mathcal S[x,\b]})^+ \over {\mathcal S[b,\b]}} + {({\mathcal S[\b,x]})^+ \over {\mathcal S[\b,r)}}\ind_{E_r}
&, x,\b \in \I_b^+.
\end{cases}
\end{equation}
In both cases the diffusion is transient with $\P_x(g^b_{\b} = \infty) = 0$.

%
%
\begin{theorem}\label{joint_last-passage-propn-time-t-b}
The joint PDF of $(g^{b}_{\b}(T), X_{b,T})$ has the following representation.

\!\!\!\!\!\!(i) For $x, \b \in \I_b^-$:

\begin{equation}\label{joint_last_passage_pdf_explicit_b_1}
f_{g^{b}_{\b}(T), X_{b,T}}(t,z;x) 
= {p^-_{b}(t;x,\b) \over \m(\b)}\m(z) 
\begin{cases}
f^+(T-t; z,\b)   &, z \in (l,\b),
\\
f^-(T-t; z,\b|b) &, z \in (\b,b).
\end{cases}
\end{equation}
(ii) For $x,\b \in \I_b^+$:

\begin{equation}\label{joint_last_passage_pdf_explicit_b_2}
f_{g^{b}_{\b}(T), X_{b,T}}(t,z;x) 
= {p^+_{b}(t;x,\b) \over \m(\b)}\m(z) 
\begin{cases}
f^+(T-t; z,\b | b)  &, z \in (b,\b),
\\
f^-(T-t; z,\b)  &, z \in (\b,r).
\end{cases}
\end{equation}
The first hitting time densities $f^\pm(T-t; z,\b)$ are as in \eqref{joint_last_passage_pdf_explicit}, 
where $\m(z)f^-(T-t; z,\b|b)$ is as in \eqref{joint_last_passage_pdf_explicit_ab} and 
$\m(z)f^+(T-t; z,\b | b) = -{1 \over \s(\b)}{\partial \over \partial y} p_{(b,\b)}(T-t;y,z)\big\vert_{y=\b-}$.
\end{theorem}
\begin{proof} An original proof is given in \cite{CM21}. 
The steps are as in the proof of Theorem~\ref{joint_last-passage-propn-time-t_ab} 
where we have analoguous relations to (\ref{prop_joint_last_time_kill_ab-t-3}), but now with only one side condition on either the minimum or the maximum of the process treated separately. 
\end{proof}

Theorems \ref{last-passage-propn-time-t-kill-b} and \ref{joint_last-passage-propn-time-t-b} 
lead directly to Proposition \ref{marginal_joint_last-passage-propn-kill-b}.

%
%
\begin{pro}\label{marginal_joint_last-passage-propn-kill-b}
The PDF of $g^{b}_{\b}(T)$ and joint PDF of $(g^{b}_{\b}(T), X_{b,T})$ have the following representations.

\vskip0.1in
(i) For $x, \b \in \I_b^-$:
\begin{align}\label{marginal_spectral_kill-b_1}
&f_{g^{b}_\b(T)}(t;x) = {p^-_{b}(t;x,\b) \over \m(\b)} 
\left[ \frac{\ind_{E_l}}{\mathcal{S}(l,\b]} + \frac{1}{\mathcal{S}[\b,b]} 
-  {1 \over \s(\b)}\eta^+(T-t;\b) 
+  \sum_{n=1}^{\infty} e^{-\lambda_n^{(\b,b)} (T-t)} {\hat\psi^-_{n}(\b,b) \over \lambda_n^{(\b,b)} }  \right],
\\ 
\label{joint_last_passage_pdf_spectral_b_1}
&f_{g^{b}_{\b}(T), X_{b,T}}(t,z;x) 
= {p^-_{b}(t;x,\b) \over \m(\b)} \m(z)
\begin{cases}
f^+(T-t,z;\b) &, z \in (l,\b),
\\
\displaystyle \sum_{n=1}^{\infty} e^{-\lambda_n^{(\b,b)} (T-t)} \psi^-_{n}(z;\b,b)  &, z \in (\b,b).
\end{cases}
\end{align}
(ii) For $x,\b \in \I_b^+$:
\begin{align}\label{marginal_spectral_kill-b_2}
&f_{g^{b}_{\b}(T)}(t;x) = {p^+_{b}(t;x,\b) \over \m(\b)} 
\left[ \frac{\ind_{E_r}}{\mathcal{S}[\b,r)} + \frac{1}{\mathcal{S}[b,\b]} 
+ {1 \over \s(\b)}\eta^-(T-t;\b) 
+  \sum_{n=1}^{\infty} e^{-\lambda_n^{(b,\b)} (T-t)} {\hat\psi^+_{n}(b,\b) \over \lambda_n^{(b,\b)} } \right],
\\ 
\label{joint_last_passage_pdf_spectral_b_2}
&f_{g^{b}_{\b}(T), X_{b,T}}(t,z;x) 
= {p^+_{b}(t;x,\b) \over \m(\b)} \m(z)
\begin{cases}
\displaystyle \sum_{n=1}^{\infty} e^{-\lambda_n^{(b,\b)} (T-t)} \psi^+_{n}(z;b,\b)  &, z \in (b,\b),
\\
f^-(T-t,z;\b)  &, z \in (\b,r).
\end{cases}
\end{align}
In \eqref{marginal_spectral_kill-b_1} and \eqref{marginal_spectral_kill-b_2}, $\eta^\pm(T-t;\b)$ are given by \eqref{last_passage_pdf_spectral_1}--\eqref{last_passage_pdf_spectral_2}, while $\hat\psi^-_{n}(\b,b)$ and 
$\hat\psi^+_{n}(b,\b)$ are given by \eqref{psi_hat_derivatives}. 
In \eqref{joint_last_passage_pdf_spectral_b_1} and \eqref{joint_last_passage_pdf_spectral_b_2}, 
$f^\pm(T-t,z;\b)$ have the spectral representations given by \eqref{joint_last_passage_pdf_spectral_1} and \eqref{joint_last_passage_pdf_spectral_2}, while $\psi^-_{n}(z;\b,b)$ and $\psi^+_{n}(z;b,\b)$ are respectively given by \eqref{FHT_eigenfunctions_ab}.
\end{pro}
\begin{proof} The terms within the square brackets in \eqref{marginal_spectral_kill-b_1} and \eqref{marginal_spectral_kill-b_2} arise by applying \eqref{last-passage-pdf-explicit-kill-b-below} and \eqref{last-passage-pdf-explicit-kill-b-above} in Theorem \ref{last-passage-propn-time-t-kill-b} where we simply identify the respective terms with those in  
Theorems \ref{last-passage-propn-time-t} and \ref{last-passage-propn-time-t-kill-ab}, 
whose series are already given in Proposition \ref{last-passage-propn-time-t-spectral} and in  \eqref{last_passage_pdf_discrete_spec} of Proposition \ref{joint_last-passage-propn-time-t_ab-new-version}, 
where level $a$ is replaced by $b$. Similarly, the series within \eqref{joint_last_passage_pdf_spectral_b_1} and \eqref{joint_last_passage_pdf_spectral_b_2} are obtained by identifying each term in 
Theorem \ref{joint_last-passage-propn-time-t-b} with the respective ones in 
Theorems \ref{joint_last-passage-propn-time-t} and \ref{joint_last-passage-propn-time-t_ab}, whose series are already given in Proposition \ref{joint_last-passage-propn-time-t-new-formula} and in \eqref{joint_last_passage_pdf_explicit_ab_new} of Proposition \ref{joint_last-passage-propn-time-t_ab-new-version}, 
where level $a$ is replaced by $b$.
\end{proof}
As described at the end of Sections \ref{sect_spectral_transPDF} and \ref{sect_spectral_FHT}, we note that all marginal and joint distributions involving Spectral Category II, i.e., with one O-NO boundary, can be accurately approximated by using discrete spectral series where an extra killing is imposed at one or two interior points. For instance, by accurately truncating the series and taking appropriate limiting values of the respective lower or upper killing levels $a$ or $b$, the series in Proposition \ref{joint_last-passage-propn-time-t_ab-new-version} can be used to accurately approximate the spectral expansions in Propositions \ref{last-passage-propn-time-t-spectral}, \ref{joint_last-passage-propn-time-t-new-formula} and \ref{marginal_joint_last-passage-propn-kill-b} for the case of O-NO boundaries.

We close this section by noting that, for processes with imposed killing at interior points, we also have obvious, yet useful, relations that extend those presented at the end of Sections \ref{subsect2} and \ref{subsect3}. In particular, 
consider an $F$-process with imposed killing at the endpoints of $(A,B)$ defined by 
$F_{(A,B),t} := {\sf F}(X_{(a,b),t})$, with smooth monotonic map ${\sf F}$ and its inverse ${\sf X}$, as above, and where 
there is a one-to-one mapping of the intervals $(a,b)\in \I$ and $(A,B)\in \mathcal{D}$:  
$a = \min\{{\sf X}(A),{\sf X}(B)\}$, $b = \max\{{\sf X}(A),{\sf X}(B)\}$, e.g., 
$a = {\sf X}(A), b = {\sf X}(B)$ if ${\sf X}^\prime > 0$. 
For any $F_{(A,B),0} = F_0\in (A,B)$ we have $X_0\equiv x = {\sf X}(F_0) \in (a,b)$. 
The last hitting time of the $F$-process to level $K\in (A,B)$ is defined as 
$g^{(A,B), F}_K(T) :=\sup\{0 \le u \le T : F_{(A,B),u} = K\}$. We set $\b = {\sf X}(K)$. 
Analogous to \eqref{last_passage_CDF_X_to_F} we have 
\begin{align}\label{last_passage_CDF_X_to_F_kill_A_B}
\P_{F_0}(g^{(A,B), F}_K(T) \le t) = \P_{x}(g^{(a,b)}_{\b}(T) \le t) \,,
\end{align}
with PDF $f_{g^{(A,B), F}_K(T)}(t;F_0) = f_{g^{(a,b)}_{\b}(T)}(t;x)$. 
For the joint PDF of $(g^{(A,B), F}_K(T),F_{(A,B),T})$,
\begin{align}\label{joint_last_PDF_X_to_F_kill_A_B}
f_{g^{(A,B), F}_K(T),F_{(A,B),T}}(t,z; F_0) 
= \vert {\sf X}^\prime(z) \vert f_{g^{(a,b)}_{\b}(T), X_{(a,b),T}}(t,{\sf X}(z);x)\,,
\end{align}
$z, K, F_0 \in(A,B)$. Similarly, for the partly discrete distribution we have 
\begin{align}\label{joint_last_discrete_X_to_F_kill_A_B}
\P_{F_0}(g^{(A,B), F}_K(T) = 0, F_{(A,B),T} \in dz) = \P_x(g^{(a,b)}_{\b}(T) = 0, X_{(a,b),T} \in d \,{\sf X}(z))
\end{align}
and $\P_{F_0}(g^{(A,B), F}_K(T) = 0, F_{(A,B),T} = \partial^\dagger) = \P_x(g^{(a,b)}_{\b}(T) = 0, 
X_{(a,b),T} = \partial^\dagger)$ for the jointly defective distribution. Lastly, we note that obvious analogous relations hold for an $F$-process with imposed killing at one interior point, say $B \in \mathcal{D}$, defined by 
$F_{B,t} := {\sf F}(X_{b,t})$, where $b={\sf X}(B)$.

%
%

\section{Explicit formulae for Some Known Solvable Diffusions}\label{subsect_spectral_formulae}
All expressions derived in each subsection are valid for $T\in (0,\infty)$. For the distributions of the last hitting times, 
the limit $T\to\infty$ produces the respective analytical expressions for infinite time horizon. 
For each case, the reader can easily verify that the resulting expressions also follow by substituting the given respective transition PDFs and scale functions within either \eqref{last-passage-density-infinite-T}--\eqref{last-passage-discrete-infinite-T}, 
\eqref{last-passage-density-inf-ab}--\eqref{last-passage-discrete-inf-ab} or \eqref{last-passage-density-inf-b}--\eqref{last-passage-discrete-inf-b}. In the interest of space, we shall not write down these expressions for infinite time horizon.
\subsection{Brownian motion}\label{subsect_standard_BM}
The simplest diffusion is Brownian motion $X_t := X_0 + W_t \in \R$, where $X_0=0$ for standard Brownian motion.  
The scale and speed densities are $\s(x) = 1$ and $\m(x) = 2$ and 
$\varphi^\pm_\lambda(x) = e^{\pm\sqrt{2\lambda}x}$ is a pair of fundamental solutions, i.e., the Green function $G(\lambda;x,y) = {1\over \sqrt{2\lambda}}e^{-\sqrt{2\lambda}\vert x - y \vert}$, $x,y\in\R$, where $G(\lambda;\b,\b) = {1\over \sqrt{2\lambda}}$. 
By a standard Laplace transform identity, (\ref{last-passage-density-phi}) gives 
$\xi(u;\b) = \sqrt{2}{\mathcal L}_\lambda^{-1}\left\{{1 \over \sqrt{\lambda}}\right\}(u) 
= \sqrt{2\over \pi u}$. [Note: there is no dependence on $\b$ in this simple case.] 
The well-known transition density also follows by a Laplace transform identity,  
$p(t;x,y) = {\mathcal L}_\lambda^{-1}\{G(\lambda;x,y) \}(t) = {1\over \sqrt{2\pi t}}e^{- {(x-y)^2\over 2t}}$, $x,y\in\R$, $t>0$. Using these expressions within (\ref{last-passage-density-time-t-1}), for $X_0 \equiv x$, recovers the known formula:
\begin{equation}\label{last_hitting_PDF_BM}
f_{g_\b(T)}(t;x) = p(t;x,\b) \, \xi(T-t;\b)
= {1\over \pi\sqrt{t(T-t)}} e^{- (x - \b)^2/2t}\,,
\end{equation}
$t\in (0,T)$, $x,\b \in \R$. Alternatively, \eqref{last_hitting_PDF_BM} also follows by applying \eqref{last-passage-density-phi-explicit} where the respective first hitting time CDFs for Brownian motion are (for $y < \b$ and $y > \b$):
\[
\P_y(\Tau^+_\b \le u) = 2 \,\NCDF\left({ y - \b \over \sqrt{u}} \right)\,,\,\,\,\,
\P_y(\Tau^-_\b \le u) = 2 \,\NCDF\left({ \b - y \over \sqrt{u}} \right)\,,
\]
where $\NCDF(x) := {1\over \sqrt{2\pi}}\int_{-\infty}^x e^{-{z^2\over 2}}d z$ denotes the standard normal CDF. 
Hence, ${\partial \over \partial y} \P_y(\Tau^+_\b \le u)\vert_{y=\b-} = 
- {\partial \over \partial y} \P_y(\Tau^-_\b \le u)\vert_{y=\b+}= \sqrt{2 \over \pi u}$ and therefore 
(\ref{last-passage-pdf-explicit}) recovers \eqref{last_hitting_PDF_BM}.

We further remark that Proposition \ref{last-passage-propn-time-t-spectral} provides us with a direct alternative. In particular, both $l=-\infty$ and $r=\infty$ are nonattracting natural and O-NO with common purely continuous spectrum with $\Lambda_\pm = 0$. Hence, the density in \eqref{last_hitting_PDF_BM} is given by 
\eqref{last-passage-pdf-spectral}, where $\ind_{E_l} = \ind_{E_r}=0$ and 
$\eta^\pm(T-t;\b)$ are given by \eqref{last_passage_pdf_spectral_1}--\eqref{last_passage_pdf_spectral_2} with no  summation term. 
In this case we have $\textup{Im}{\varphi^{\pm\prime}_\lambda (\b) \over \varphi^\pm_\lambda (\b)}
\vert_{\lambda = \epsilon e^{-i\pi}} 
= \mp\sqrt{2\epsilon}$, i.e., $\eta^-(T-t;\b)-\eta^+(T-t;\b) = {2\sqrt{2}\over \pi}\int_0^\infty e^{-\epsilon (T-t)} 
{d \epsilon \over \sqrt{\epsilon}} = {4\over \pi \sqrt{T-t}}\int_0^\infty e^{-y^2/2}  dy 
= {2\sqrt{2}\over \sqrt{\pi(T-t)}}$. 

The discrete part of the distribution, given by \eqref{prop_last_time-t-1-limit}, is simply written as 
\begin{equation}\label{std_BM_defective_last}
\P_x(g_\b(T) = 0) = 2 \,\NCDF\left({ \vert x - \b \vert \over \sqrt{T}} \right) - 1,
\end{equation}
$x,\b\in\R$. Hence, using \eqref{last_hitting_PDF_BM} and \eqref{std_BM_defective_last} within 
\eqref{last-passage-CDF-time_0_to_t} gives the CDF expressed as
\begin{align}\label{last_hitting_CDF_BM}
\P_x(g_\b(T) \le t) = 2 \,\NCDF\left({\vert x - \b \vert \over \sqrt{T}} \right) - 1 
+ \int_0^t {1\over \pi\sqrt{u(T-u)}} e^{- (x-\b)^2/2u} \,du\,.
\end{align}
Note: setting $x=\b$ recovers the well-known arcsine law for the zeros of Brownian motion 
on $[0,T]$.
 
The joint PDF, $f_{g_\b(T), X_T}(t,z;x)$, for $t \in (0,T)$, $X_0 \equiv x,z,\b \in \R$, is a simple application of \eqref{joint_last_passage_pdf_explicit} since we can directly use the known transition PDF for Brownian motion killed at level $\b\in\R$, i.e.,
\begin{equation}\label{killed_BM_trans_PDF}
p_\b^\pm(T - t;y,z) = {1\over \sqrt{2\pi (T-t)}}\left( e^{-{(y-z)^2\over 2(T-t)}} - e^{-{(y + z - 2\b)^2\over 2(T-t)}} \right)\ind_{\{y,z \,\in \,\I_\b^\pm\}}.
\end{equation}
Differentiating and casting as one expression for all $z\in\R$ gives
\begin{equation*}
\m(z)f^\pm(T-t,z;\b) = 
\mp {\partial \over \partial y} p_{\b}^\mp(T-t;y,z)\bigg\vert_{y = \mp\b} =  
\sqrt{2 \over \pi (T-t)^3} \,\vert z - \b \vert e^{-(z - \b)^2/2(T-t)}\,.
\end{equation*}
By combining this with $p(t;x,\b)$ above, (\ref{joint_last_passage_pdf_explicit}) 
then recovers the known formula for the joint density:
\begin{equation}\label{last_hitting_joint_PDF_BM}
f_{g_\b(T), X_T}(t,z;x) = {\vert z - \b \vert \over 2\pi\sqrt{t (T-t)^3}} e^{-{(x-\b)^2\over2t} 
- {(z - \b)^2\over 2(T-t)}},
\end{equation}
$t\in (0,T)$, $x,\b \in \R$. 
Brownian motion is conservative on $\R$, i.e., \eqref{last_density_time-t-conserve_from_joint} is readily verified.
From (\ref{prop_joint_last_time-t-2}) and \eqref{killed_BM_trans_PDF} we also have the nonzero partly discrete distribution:
\begin{align}\label{last_hitting_defective_joint_BM}
\P_x(g_\b(T) = 0, X_T \in dz) = 
\begin{cases} 
{1\over \sqrt{2\pi T}}\left( e^{-{(x-z)^2\over 2T}} - e^{-{(x + z - 2\b)^2\over 2T}} \right)
dz &, x,z > \b,
\\
{1\over \sqrt{2\pi T}}\left( e^{-{(x-z)^2\over 2T}} - e^{-{(x + z - 2\b)^2\over 2T}} \right) 
dz &, x,z < \b.
\end{cases}
\end{align}
It is also easy to show that integrating \eqref{last_hitting_defective_joint_BM} over $z\in\R$ 
recovers \eqref{std_BM_defective_last}, i.e., $\P_x(g_\b(T) = 0, X_T \in \R) = \P_x(g_\b(T) = 0)$ and 
$\P_x(g_\b(T) = 0, X_T \in \partial^\dagger) = 0$ since Brownian motion is conservative on $\R$.

As an alternative, Proposition \ref{joint_last-passage-propn-time-t-new-formula} can be used to derive 
\eqref{last_hitting_joint_PDF_BM} where $\pm\infty$ are O-NO with purely continuous spectrum, i.e., $\Lambda_\pm = 0$ where \eqref{joint_last_passage_pdf_spectral_1} and \eqref{joint_last_passage_pdf_spectral_2} give the known first hitting time density:
\begin{align*}
f^\pm(T-t,z;\b) 
= {1 \over \pi} \textup{Im}\big\{\!
\int_0^\infty e^{-{1\over 2} (T-t)y^2 + i \vert z - \b \vert y} y d y \big\} 
= {\vert z - \b \vert  \over\sqrt{2 \pi (T-t)^3}} \,e^{-(z - \b)^2/2(T-t)}.
\end{align*}
[Here we used the identity $\int_0^\infty e^{-{A\over 2} y^2 + i B y} y d y 
= {e^{-B^2/2A}\over A}[1 + iB\sqrt{\pi \over 2A}]$, $A>0,B\in\R$, with $A=T-t, B = \vert z - \b \vert$.] 
This recovers \eqref{last_hitting_joint_PDF_BM} for all $x,z,\b\in\R$, $t\in (0,T)$.

%
%
\subsection{Drifted Brownian motion}\label{subsect_drifted_BM}
Consider $X_t := X_0 + \mu t + W_t \in \R$ with 
constant $\mu \in \R$. The scale and speed densities are $\s(x) = e^{-2\mu x}$ and $\m(x) = 2e^{2\mu x}$. A pair of fundamental solutions is $\varphi^+_\lambda(x) = e^{(\sqrt{2\lambda + \mu^2} - \mu)x},\,  \varphi^-_\lambda(x) = e^{-(\sqrt{2\lambda + \mu^2} + \mu)x}$ where  
$G(\lambda;x,y) = e^{\mu (y-x)} {1\over \sqrt{2\lambda + \mu^2}}e^{-\sqrt{2\lambda + \mu^2}\vert x - y \vert}$, $x,y\in\R$. 
As in Section \ref{subsect_standard_BM}, we shall first derive the distribution of $g_\b(T)$ 
followed by the joint distribution of $(g_\b(T),X_T)$ for any $X_0 \equiv x,\b \in \R$.
 
To implement (\ref{last-passage-density-time-t-1}) we use 
$G(\lambda;\b,\b) = {1 \over \sqrt{2\lambda + \mu^2}}$. Since  
$
{\mathcal L}_\lambda^{-1}\left\{{\sqrt{\lambda + a} \over \lambda} \right\}(u) =
{e^{-a u} \over \sqrt{\pi u}} + \sqrt{a\pi}\,\text{Erf}(\sqrt{a u}),\, a > 0,
$
with error function $\text{Erf}(x) =  {2 \over \sqrt{\pi}}\left[\NCDF(\sqrt{2}x) - {1\over 2} \right]$, 
(\ref{last-passage-density-phi}) gives
\begin{align*}
\xi(u;\b) := \sqrt{2}{\mathcal L}_\lambda^{-1}\left\{{\sqrt{\lambda + \mu^2/2} \over \lambda}  \right\}(u)
= \sqrt{2 \over \pi u}e^{-{\mu^2 \over 2} u} + \mu \left[2\NCDF(\mu \sqrt{u})  - 1\right]\,,\,u>0,
\end{align*}
where $\vert \mu \vert \left[2\NCDF(\vert \mu \vert  \sqrt{u})  - 1 \right]
= \mu \left[2\NCDF(\mu \sqrt{u})  - 1\right]$.
Using the known density $p(t;x,\b) = {1\over \sqrt{2\pi t}}e^{- {(\b - x - \mu t)^2\over 2t}}$ and $\xi(T-t;\b)$ within
(\ref{last-passage-density-time-t-1}) gives the density of $g_\b(T)$, for $t\in(0,T)$, $x,\b\in\R$:
\begin{equation}\label{last_hitting_PDF_drifted_BM}
f_{g_\b(T)}(t;x) =
\left( \sqrt{2 \over \pi (T-t)}e^{-{\mu^2 \over 2} (T-t)} +  \mu \left[2\NCDF(\mu \sqrt{T-t})  - 1 \right] \right)
{1\over \sqrt{2\pi t}}e^{- {(\b - x - \mu t)^2\over 2t}}.
\end{equation}
The known first hitting time distribution for drifted Brownian motion gives the discrete portion of the distribution:
\begin{equation}\label{last_hitting_discrete_drifted_BM}
\P_x(g_\b(T) =0) = \P_x(\T_\b > T) = 
\begin{cases}
\NCDF\left({\b - x - \mu T) \over \sqrt{T}} \right) - e^{2\mu (\b - x)} \NCDF\left({x - \b - \mu T \over \sqrt{T}} \right)
&, x < \b,
\\
\NCDF\left({x - \b + \mu T) \over \sqrt{T}} \right) - e^{2\mu (\b - x)} \NCDF\left({\b - x + \mu T \over \sqrt{T}} \right)
&, x > \b.
\end{cases}
\end{equation}
Note: $\P_{x}(g_\b(T) = 0)=0$ when $x=\b$.

As an alternative derivation of \eqref{last_hitting_PDF_drifted_BM}, we can readily apply (\ref{last-passage-density-phi-explicit}) while using the above formula, where 
$\P_y(\T^+_\b \le u) = 1 - \P_y(\T_\b > u)$, for $y \le \b$, and $\P_y(\T^-_\b \le u) = 1 - \P_y(\T_\b > u)$, for $y \ge \b$. Upon differentiating, we again have 
\[
\xi(u;\b) =
{1\over 2}\left[{\partial \over \partial y} \P_y(\T^+_\b \le u)\big\vert_{y=\b-}
- {\partial \over \partial y} \P_y(\T^-_\b \le u)\big\vert_{y=\b+} \right]
= \sqrt{2 \over \pi u}e^{-{\mu^2 \over 2} u} + \mu \left[2\NCDF(\mu \sqrt{u})  - 1\right]\,.
\]
Hence, by (\ref{last-passage-pdf-explicit}), \eqref{last_hitting_PDF_drifted_BM} is recovered. 
Proposition \ref{last-passage-propn-time-t-spectral} can also be used as another alternative derivation.

The joint PDF, $f_{g_\b(T), X_T}(t,z;x)$, for $t \in (0,T)$, $X_0 \equiv x,z,\b \in \R$, again follows as a  simple application of \eqref{joint_last_passage_pdf_explicit} where we directly use the known transition PDF for drifted Brownian motion killed at $\b$, 
\begin{equation}\label{driftedBM_single_kill_transPDF}
p_\b^{\pm}(T - t;y,z) = {1\over \sqrt{2\pi (T-t)}}\left(e^{-{(z-y - \mu (T-t))^2\over 2(T-t)}}
- e^{2\mu (\b - y)} e^{-{(z+y - 2\b - \mu (T-t))^2\over 2(T-t)}} \right)\ind_{\{y,z \,\in \,\I_\b^\pm\}}. 
\end{equation}
Since $\m(\b)\s(\b) = 2$, we have
\begin{equation}\label{driftedBM_single_kill_transPDF_der}
{\m(z) \over \m(\b)}f^\pm(T-t; z , \b) = 
\mp {1\over 2}{\partial \over \partial y} p_{\b}^\mp(T-t;y,z)\bigg\vert_{y = \mp\b} =    
\sqrt{1 \over 2 \pi (T-t)^3} \,\vert z - \b \vert e^{-{(z - \b)^2\over 2(T-t)} + \mu (z - \b) - {\mu^2 \over 2}(T-t)}\,.
\end{equation}
Combining this with $p(t;x,\b)$, \eqref{joint_last_passage_pdf_explicit} gives
\begin{equation}\label{last_hitting_joint_PDF_drfited_BM}
f_{g_\b(T), X_T}(t,z;x) = {\vert z - \b \vert \over 2\pi\sqrt{t (T-t)^3}} 
e^{ -{(\b - x - \mu t)^2\over 2t} - {(z - \b- \mu (T-t))^2 \over 2(T-t)}}.
\end{equation}
Drifted Brownian motion is conservative on $\R$ and \eqref{last_density_time-t-conserve_from_joint} is readily verified. 
Employing \eqref{driftedBM_single_kill_transPDF} within (\ref{prop_joint_last_time-t-2}) gives the nonzero partly discrete distribution:
\begin{align}\label{last_hitting_defective_joint_drifted_BM}
\P_x(g_\b(T) = 0, X_T \in dz) = 
\begin{cases} 
{1\over \sqrt{2\pi T}}\left( e^{-{(z-x - \mu T)^2\over 2T}} 
- e^{2\mu (\b - x)} e^{-{(z+x - 2\b - \mu T)^2\over 2T}} \right)
dz &, x,z > \b,
\\
{1\over \sqrt{2\pi T}}\left( e^{-{(z-x - \mu T)^2\over 2T}} 
- e^{2\mu (\b - x)} e^{-{(z+x - 2\b - \mu T)^2\over 2T}} \right) 
dz &, x,z < \b.
\end{cases}
\end{align}
Integrating over $z$ recovers \eqref{last_hitting_discrete_drifted_BM} since the process is conservative on $\R$ 
with $\P_x(g_\b(T) = 0, X_T \in \partial^\dagger) = 0$. 
Proposition \ref{joint_last-passage-propn-time-t-new-formula} also offers an alternative to deriving the formula in \eqref{last_hitting_joint_PDF_drfited_BM}. The steps are as in Section \ref{subsect_standard_BM}.

%
%
\subsubsection{Geometric Brownian motion}\label{subsect_GBM}
Consider geometric Brownian motion (GBM) with generator 
${\mathcal G}f(x) := {1\over 2}\sigma^2 x^2 f''(x) + \mu xf'(x)$, $x\in (0,\infty)$, $\mu\in\R$, $\sigma > 0$. 
We now exploit the one-to-one mapping from GBM to Brownian motion. We denote GBM by $\{F_t\}_{t \ge 0}$, where 
$F_t = F_0e^{(\mu - \sigma^2/2)t + \sigma W_t}$, 
$F_0 > 0$. Letting $X_t = \nu t + W_t$, $\nu := (\mu - \sigma^2/2)/\sigma$, 
gives $F_t = F_0e^{\sigma X_t}$, i.e., $F_t = {\sf F}(X_t)$, with increasing map ${\sf F}(x) := F_0e^{\sigma x}$, $x\in\R$, and unique inverse ${\sf X}(y) := {1\over \sigma}\ln {y \over F_0}$, $y\in (0,\infty)$, i.e., 
$X_t = {\sf X}(F_t) = {1\over \sigma}\ln {F_t \over F_0}$, $X_0 = 0$. From (\ref{last_passage_CDF_X_to_F}), 
$\P_{F_0}(g^F_K(T) \le t) = \P_0(g^{X}_\b(T) \le t)$, $\b = {\sf X}(K) = {1 \over\sigma}\ln{K \over F_0}$, $K\in (0,\infty)$, which is the CDF of the last hitting time for Brownian motion with drift $\nu$ and started at $X_0 \equiv x = 0$. 
Hence, $f_{g^F_K(T)}(t;F_0) = f_{g^{X}_\b(T)}(t;0)$, where the latter is given by \eqref{last_hitting_PDF_drifted_BM}:
\begin{equation}\label{last_hitting_PDF_GBM}
f_{g^F_K(T)}(t;F_0)  = 
\left( \sqrt{2 \over \pi (T-t)}e^{-{\nu^2 \over 2} (T-t)} +  \nu \left[2\NCDF(\nu \sqrt{T-t})  - 1 \right] \right) 
{1\over \sqrt{2\pi t}}e^{- {({1 \over\sigma}\ln{K \over F_0} - \nu t)^2\over 2t}}.
\end{equation}
Using $\P_{F_0}(g^F_K(T) = 0) = \P_{0}(g^X_\b(T) = 0)$ and \eqref{last_hitting_discrete_drifted_BM} gives
\begin{equation}
\P_{F_0}(g^F_K(T) = 0)  
=
\begin{cases} 
\NCDF\left({\ln {K\over F_0} - (\mu - {\sigma^2 \over 2}) T) \over \sigma\sqrt{T}} \right) 
- ({K\over F_0})^{{2\mu \over \sigma^2} - 1} \NCDF\left({\ln {F_0\over K} - (\mu - {\sigma^2 \over 2}) T \over \sigma\sqrt{T}} \right)
&, F_0 \le K,
\\
\NCDF\left({\ln {F_0\over K} + (\mu - {\sigma^2 \over 2}) T\over \sigma\sqrt{T}} \right)
 - ({K\over F_0})^{{2\mu \over \sigma^2} - 1} \NCDF\left({\ln {K\over F_0} + (\mu - {\sigma^2 \over 2}) T\over \sigma\sqrt{T}} \right)
&, F_0 \ge K.
\end{cases}
\end{equation}
Note: $\P_{F_0}(g^F_K(T) = 0)=0$ when $F_0=K$.

The joint PDF is obtained by simply applying \eqref{joint_last_PDF_X_to_F}, 
where ${\sf X}(z) := {1 \over\sigma}\ln{z \over F_0}$, ${\sf X}'(z) := {1 \over\sigma z}$, i.e., 
$f_{g^F_K(T), F_T}(t,z; F_0) = {1 \over\sigma z} f_{g^X_\b(T), X_T}(t,{\sf X}(z);0)$. 
Hence, using \eqref{last_hitting_joint_PDF_drfited_BM},
\begin{align}\label{last_hitting_joint_PDF_drfited_GBM}
f_{g^F_K(T), F_T}(t,z; F_0) = {\vert \ln{z \over K} \vert \over 2\pi \sigma^2 z\sqrt{t (T-t)^3}}
e^{ - (\ln{K \over F_0} - (\mu - {\sigma^2\over 2}) t)^2/2\sigma^2 t -
(\ln{z \over K} - (\mu - {\sigma^2\over 2}) (T-t))^2/2\sigma^2 (T-t)},
\end{align}
$z, K, F_0 \in (0,\infty)$, $t\in (0,T)$. Since ${\sf X}'(z) > 0$, \eqref{pmf_F_to_X_joint_last_time_1} gives
\begin{align}
\P_{F_0}(g^F_K(T) = 0, F_T \in dz) = {1 \over\sigma z}
\begin{cases}\label{joint_last_hitting_discrete_GBM}
p_\b^{+}(T;0,{\sf X}(z)) dz &, z, F_0 > K,
\\
p_{\b}^{-}(T;0,{\sf X}(z)) dz &, z, F_0 < K,
\end{cases}
\end{align}
and zero otherwise, where (for $z,F_0 > K$)
$$p_k^{+}(T;0,{\sf X}(z)) = {1\over \sqrt{2\pi T}}
\left[e^{-(\ln{z \over F_0} - (\mu - {\sigma^2\over 2}) T)^2 / 2\sigma^2T}
-  \bigg({K\over F_0}\bigg)^{{2\mu \over \sigma^2} - 1} e^{-(\ln{z F_0 \over K^2} - (\mu - {\sigma^2\over 2}) T)^2 / 2\sigma^2T} \right]$$
and $p_\b^{-}(T;0,{\sf X}(z))$ given by the same expression (for $z,F_0 < K$). Note: $\P_{F_0}(g^F_K(T) = 0, F_T \in \partial^\dagger) = 0$ since GBM is conservative on $(0,\infty)$.

%
%

\subsection{Drifted Brownian motion killed at either of two endpoints}\label{section_drfited_BM_killed_a_b}
Consider drifted Brownian motion on $(a,b)$ with imposed killing at the endpoints $-\infty < a < b < \infty$. 
The respective marginal and joint distributions of $g^{(a,b)}_{\b}(T)$ and 
$(g^{(a,b)}_{\b}(T),X_{(a,b),T})$ are readily derived by applying 
\eqref{prop_joint_last_time_kill_ab-t-3-zero}, \eqref{prop_last_time-CDF_discrete_killed_ab}--\eqref{joint_last_ab_doubly_defective_new} and \eqref{joint_last_passage_pdf_explicit_ab_new}--\eqref{last_passage_pdf_discrete_spec} of Proposition 
\ref{joint_last-passage-propn-time-t_ab-new-version}. 
Hence, we compute $\psi_n^\pm$ defined in \eqref{FHT_eigenfunctions_ab}. 
Using $\varphi^\pm_\lambda$ in Section \ref{subsect_drifted_BM}, the cylinder function in 
\eqref{phi_function} is given by
\begin{equation}\label{drifted_BM_cylinder}
\phi(x,y;\lambda) = 2e^{-\mu(x+y)}\sinh(\sqrt{2\lambda + \mu^2} (y-x)).
\end{equation}
All positive eigenvalues solving $\phi(a,b;-\lambda_n^{(a,b)}) = 0$ are 
$\lambda_n^{(a,b)} = {n^2\pi^2 \over 2(b-a)^2} + {\mu^2\over 2}$, $n \ge 1$. Now, using 
\begin{equation*}
{\partial \over \partial \lambda}\phi(a,b;\lambda) 
= 2e^{-\mu(a+b)}{(b-a)\over \sqrt{2\lambda + \mu^2}}\cosh(\sqrt{2\lambda + \mu^2} (b-a)),
\end{equation*}
gives the derivative in \eqref{Delta_derivative} which simplifies to 
$\Delta(a,b; \lambda_n^{(a,b)}) 
= 2i (-1)^{n} e^{-\mu(a+b)}{(b-a)^2\over n\pi}$. 
[Note: $\sinh(ix)=i\sin x$, $\cosh(ix)=\cos x$, $\cos(n\pi)=(-1)^n$ 
and the eigenvalue equation $\sqrt{2\lambda_n^{(a,b)} - \mu^2} = {n\pi \over b-a}$.] 
From \eqref{drifted_BM_cylinder} we have $\phi(a,x;-\lambda_n^{(a,b)}) = 2ie^{-\mu(a+x)}\sin\big({n\pi (x-a)\over b-a} \big)$ and hence
\begin{align}\label{drifted_BM_Psi_plus_n}
\psi_n^+(x;a,b) = -{\phi(a,x; -\lambda_n^{(a,b)}) \over \Delta(a,b; \lambda_n^{(a,b)})}
= {n\pi \over (b-a)^2}e^{\mu(b-x)}\sin\big({n\pi (b-x)\over b-a}\big).
\end{align}
This expression follows using $(-1)^{n+1}\sin \theta = \sin (-\theta + n\pi) $. 
Similarly,
\begin{equation}\label{drifted_BM_minus_n}
\psi_n^-(x;a,b) = {n\pi \over (b-a)^2}e^{\mu(a-x)} (-1)^{n+1} \sin\big({n\pi (b-x)\over b-a}\big)
= {n\pi \over (b-a)^2}e^{\mu(a-x)}\sin\big({n\pi (x-a)\over b-a}\big).
\end{equation}
The corresponding expressions for $\psi_n^\pm(x;a,\b), \psi_n^\pm(x;\b,b)$, $a < \b < b$, are obvious with eigenvalues  
$\lambda_n^{(a,\b)} = {n^2\pi^2 \over 2(\b-a)^2} + {\mu^2\over 2}$ and $\lambda_n^{(\b,b)} = {n^2\pi^2 \over 2(b-\b)^2} + {\mu^2\over 2}$.

The scale function is given by 
${\mathcal S}[x,y] = \int_x^y e^{-2\mu z} d z = {1\over 2\mu}(e^{-2\mu x} - e^{-2\mu y})$, if $\mu\ne 0$ and 
${\mathcal S}[x,y] = y-x$ if $\mu = 0$. 
Hence, by \eqref{prop_last_time-CDF_discrete_killed_ab} we have the explicit spectral series for the defective marginal distribution:
\begin{align}\label{drifted_BM_prop_last_time_discrete_killed_ab}
\P_x(g^{(a,b)}_{\b}(T) = 0) 
= \begin{cases}
R(x;a,\b)
+ e^{\mu(\b-x) - {\mu^2\over 2}T} 
\displaystyle 
\sum_{n=1}^\infty {2\pi ne^{- {n^2\pi^2 T\over 2(\b-a)^2}} \over n^2\pi^2 + \mu^2 (\b-a)^2}
\sin\big({n\pi (\b-x)\over \b - a}\big)&, x \in (a,\b),
\\
R(x;b,\b)
+ e^{\mu(\b-x) - {\mu^2\over 2}T} 
\displaystyle 
\sum_{n=1}^\infty {2\pi ne^{- {n^2\pi^2 T\over 2(b-\b)^2}} \over n^2\pi^2 + \mu^2 (b-\b)^2}
\sin\big({n\pi (x-\b)\over b-\b}\big)&, x \in (\b,b).
\end{cases}
\end{align} 
Here, and in Section \ref{subsect_drifted_BM_kill_one}, we define $R(x;a,\b) := (e^{2\mu(\b-x)} - 1)/(e^{2\mu(\b-a)} - 1)$, if $\mu\ne 0$, and 
$R(x;a,\b) := (\b-x)/(\b-a)$, if $\mu = 0$, and similarly for $R(x;b,\b)$ with $b$ replacing $a$. Employing 
$\psi_n^-(x;a,\b)$ and $\psi_n^+(x;\b,b)$ within \eqref{joint_last_ab_doubly_defective_new} gives 
the explicit spectral series for the jointly defective distribution:
\begin{align}\label{drifted_BM_joint_last_ab_doubly_defective}
&\P_x(g^{(a,b)}_{\b}(T) = 0, X_{(a,b),T} = \partial^\dagger) 
\nonumber \\
&
= \begin{cases}
R(x;a,\b)
- e^{\mu(a-x) - {\mu^2\over 2}T} 
\displaystyle 
\sum_{n=1}^\infty {2\pi ne^{- {n^2\pi^2 T\over 2(\b-a)^2}} \over n^2\pi^2 + \mu^2 (\b-a)^2}
\sin\big({n\pi (x-a)\over \b - a}\big)&, x \in (a,\b),
\\
R(x;b,\b)
- e^{\mu(b-x) - {\mu^2\over 2}T} 
\displaystyle 
\sum_{n=1}^\infty {2\pi ne^{- {n^2\pi^2 T\over 2(b-\b)^2}} \over n^2\pi^2 + \mu^2 (b-\b)^2}
\sin\big({n\pi (b - x)\over b - \b}\big)&, x \in (\b,b).
\end{cases}
\end{align}

By using the above functions and eigenvalues within \eqref{u_spectral_3}, we readily obtain the well-known spectral series for the transition density,
\begin{eqnarray}\label{drfited_BM_trans_killed_a_b}
p_{(a,b)}(t;x,y) = {2\over b-a}e^{\mu(y-x) - {\mu^2\over 2}t}\sum_{n=1}^\infty e^{-{n^2\pi^2\over 2(b-a)^2}t}
\sin\big({n\pi (x-a)\over b-a}\big)\sin\big({n\pi (y-a)\over b-a}\big)\,,
\end{eqnarray}
$x,y\in (a,b)$, $t > 0$. Hence, the explicit spectral series for the partly discrete distribution in \eqref{prop_joint_last_time_kill_ab-t-3-zero} follows immediately where 
$p_{(a,\b)}(T;x,z)$ and $p_{(\b,b)}(T;x,z)$ are expressed using \eqref{drfited_BM_trans_killed_a_b}, i.e.,
\begin{align}\label{partial_joint_last_time_kill_ab_driftedBM}
&\P_x(g^{(a,b)}_{\b}(T) = 0, X_{(a,b),T} \in dz) 
\nonumber \\
&= \begin{cases} 
\displaystyle{2\over \b-a}e^{\mu(z-x) - {\mu^2\over 2}T}\sum_{n=1}^\infty e^{-{n^2\pi^2\over 2(\b-a)^2}T}
\sin\big({n\pi (x-a)\over \b-a}\big)\sin\big({n\pi (z-a)\over \b-a}\big) dz 
&, x,z \in (a,\b),
\\ 
\displaystyle {2\over b-\b}e^{\mu(z-x) - {\mu^2\over 2}T}\sum_{n=1}^\infty e^{-{n^2\pi^2\over 2(b-\b)^2}T}
\sin\big({n\pi (x-\b)\over b-\b}\big)\sin\big({n\pi (z-\b)\over b-\b}\big) dz
&, x,z \in (\b,b).
\end{cases}
\end{align}

Given $\psi_n^+(x;a,\b)$ and $\psi_n^-(x;\b,b)$, \eqref{joint_last_passage_pdf_explicit_ab_new} now directly gives us the joint density:
\begin{align}\label{drifted_BM_kill_a_b_joint_last_passage}
f_{g^{(a,b)}_{\b}(T), X_{(a,b),T}}(t,z;x) 
= p_{(a,b)}(t;x,\b) e^{\mu(z-\b) - {\mu^2\over 2}(T-t)}
\!\left\{
\begin{array}{lr}
\!\displaystyle \sum_{n=1}^\infty { n\pi e^{- {n^2\pi^2 (T-t)\over 2(\b-a)^2}} \over (\b-a)^2}
\sin\big({n\pi (\b - z)\over \b - a}\big), &\!\!\!\!  z\in (a,\b), 
\\
\!\displaystyle \sum_{n=1}^\infty { n\pi  e^{- {n^2\pi^2 (T-t)\over 2(b - \b)^2}} \over (b - \b)^2}
\sin\big({n\pi (z - \b)\over b - \b}\big), &\!\!\!\!  z\in (\b,b).
\end{array}
\right.
\end{align}
Using \eqref{drfited_BM_trans_killed_a_b} for $p_{(a,b)}(t;x,\b)$ produces a double (sine) series representation of the joint PDF.

We now employ \eqref{last_passage_pdf_discrete_spec}. From the above scale function and $\m(\b) = 2e^{2\mu \b}$, we have (for $\mu\ne 0$):
\begin{align*}
\widehat{\mathcal{S}}(a,b;\b) := {1 \over \m(\b)} \mathcal{S}(a,b;\b) = 
\mu {e^{-2\mu a} - e^{-2\mu b} \over (e^{-2\mu a} - e^{-2\mu \b} )(1 - e^{-2\mu (b-\b)})} 
= {\mu \over 2}{\sinh(\mu(b-a)) \over \sinh(\mu(\b-a))\sinh(\mu(b-\b))}.
\end{align*}
For $\mu=0$ we simply have 
$\widehat{\mathcal{S}}(a,b;\b) = {1 \over 2}{b-a \over (\b-a)(b-\b)}$. 
Computing the derivatives in \eqref{psi_hat_derivatives} gives
\begin{equation*}
{1 \over \m(\b)}\hat\psi^+_{n}(a,\b) = {1\over \b - a}{n^2\pi^2 \over 2(\b-a)^2}
 \equiv {1\over \b - a} \widehat{\lambda}_n^{(a,\b)},\,\,\,\,
{1 \over \m(\b)}\hat\psi^-_{n}(\b,b) = {1\over b - \b}{n^2\pi^2 \over 2(b - \b)^2}
 \equiv {1\over b - \b} \widehat{\lambda}_n^{(\b,b)}.
\end{equation*}
Hence, \eqref{last_passage_pdf_discrete_spec} gives the spectral representation:
\begin{align}\label{drifted_BM_last_passage_pdf_spec}
f_{g^{(a,b)}_{\b}(T)}(t;x) &= p_{(a,b)}(t;x,\b)
\left[ \widehat{\mathcal{S}}(a,b;\b) 
+ 
\sum_{n=1}^{\infty}
\bigg(
{1\over \b - a}{\widehat{\lambda}_n^{(a,\b)} \over \lambda_n^{(a,\b)} }e^{-\lambda_n^{(a,\b)} (T-t)}
+  
{1\over b - \b}{\widehat{\lambda}_n^{(\b,b)} \over \lambda_n^{(\b,b)} } e^{-\lambda_n^{(\b,b)} (T-t)}
\bigg)\right],
\nonumber \\
&\equiv p_{(a,b)}(t;x,\b)
\bigg[ \widehat{\mathcal{S}}(a,b;\b) 
+ 
\sum_{n=1}^{\infty}
\bigg(
{1\over \b - a}{n^2\pi^2 \over n^2\pi^2 + \mu^2 (\b - a)^2}
e^{-({n^2\pi^2 \over 2(\b-a)^2} + {\mu^2\over 2}) (T-t)}
\nonumber \\
&\phantom{\equiv p_{(a,b)}(t;x,\b)
\bigg[ \widehat{\mathcal{S}}(a,b;\b) + 
\sum_{n=1}^{\infty}}
+  
{1\over b - \b}{n^2\pi^2 \over n^2\pi^2 + \mu^2 (b - \b)^2}
e^{-({n^2\pi^2 \over 2(b - \b)^2} + {\mu^2\over 2}) (T-t)}
\bigg)\bigg].
\end{align}
The above formulae complete the specification of the marginal and joint distributions (including the defective portions) of $g^{(a,b)}_{\b}(T)$ and $(g^{(a,b)}_{\b}(T),X_{(a,b),T})$.

We remark that the quantity in square brackets in \eqref{drifted_BM_last_passage_pdf_spec}, i.e., 
$\xi_{(a,b)}(T-t;\b)$ is also readily computed via \eqref{last-passage-density-phi-killed}. In particular, the Green function in \eqref{greenfunc_double} gives 
$H(\lambda) := {1 \over \lambda G_{(a,b)}(\lambda;\b,\b)}$, i.e.,
\begin{align*}\label{drifted_BM_last-passage-density-phi-killed}
\xi_{(a,b)}(u;\b) &= {\mathcal L}_\lambda^{-1}\{H(\lambda) \}(u)
= {\mathcal L}_\lambda^{-1}\big\{
\frac{1}{2\lambda} \frac{\sqrt{2\lambda +\mu^2} \sinh(\sqrt{2\lambda +\mu^2}(b - a))}{\sinh(\sqrt{2\lambda +\mu^2}(\b - a)) \sinh(\sqrt{2\lambda +\mu^2}(b - \b))}\big\}(u).
\end{align*}
Observe that $H(\lambda)$ is meromorphic in $\lambda$ with isolated simple poles at 
$\lambda=0$, $\lambda=-\lambda_n^{(a,\b)}$ and $\lambda=-\lambda_n^{(\b,b)}$, $n\ge 1$. The residues are easily calculated: ${\rm Res}\{H(\lambda);\lambda = 0\} = \widehat{\mathcal{S}}(a,b;\b)$ and
\[
{\rm Res}\{H(\lambda);\lambda  = -\lambda_n^{(a,\b)}\} = {1\over \b - a}{\widehat{\lambda}_n^{(a,\b)} \over \lambda_n^{(a,\b)} },\,\,\,
{\rm Res}\{H(\lambda);\lambda  = -\lambda_n^{(\b,b)}\} = {1\over b - \b}{\widehat{\lambda}_n^{(\b,b)} \over \lambda_n^{(\b,b)} }.
\]
Hence, substituting these residues into the Laplace inversion formula (with $u=T-t$) recovers  \eqref{drifted_BM_last_passage_pdf_spec}.

The CDF of $g^{(a,b)}_{\b}(T)$ can be explicitly expressed as a sum of single and double series by adding 
the series in \eqref{drifted_BM_prop_last_time_discrete_killed_ab} with the series in \eqref{series_CDF_g_K_ab_T} where  
$
\phi_m(x)\phi_m(\b) = {e^{-\mu(x+\b)} \over b - a}
\sin\big({m\pi (x-a)\over b - a}\big)\sin\big({m\pi (\b - a)\over b - a}\big)
$
and with $\mathcal{S}(a,b;\b)$, $\hat\psi^+_{n}(a,\b)$, $\hat\psi^-_{n}(\b,b)$, $\lambda_n^{(a,\b)}$, 
$\lambda_n^{(\b,b)}$, $\lambda_m \equiv \lambda_m^{(a,b)}$ given just above. 

Consider the GBM process, with generator as in \ref{subsect_GBM} and with imposed killing at the endpoints of $(A,B) 
\subset (0,\infty)$. We denote this process by $\{F_{(A,B),t}, t\ge 0\}$. 
Employing \eqref{last_passage_CDF_X_to_F_kill_A_B}--\eqref{joint_last_discrete_X_to_F_kill_A_B} in all of the above relations, while setting $x=0$, 
$\b = {1 \over\sigma}\ln{K \over F_0}$, ${\sf X}(z) = {1 \over\sigma}\ln{z \over F_0}$, 
$a = {1 \over\sigma}\ln{A \over F_0}$, $b = {1 \over\sigma}\ln{B \over F_0}$, and replacing the drift parameter 
$\mu$ by $\nu \equiv (\mu - \sigma^2/2)/\sigma$, leads to all the respective expressions for the marginal and joint distributions of $g^{(A,B), F}_K(T)$ and $(g^{(A,B), F}_K(T), F_{(A,B),T})$, respectively.

%
%
\subsection{Drifted Brownian motion killed at one endpoint}\label{subsect_drifted_BM_kill_one}
Let us now consider drifted Brownian motion with imposed killing at only one endpoint $b\in\R$.  
The respective marginal and joint distributions of $g^{b}_{\b}(T)$ and $(g^{b}_{\b}(T),X_{b,T})$ 
are readily derived in both cases: (i) $x,\b\in \I_b^- \equiv (-\infty,b)$ and (ii) $x,\b\in \I_b^+ \equiv (b,\infty)$. 
Note that $l=-\infty$ and $r=\infty$. 
For the defective portions we apply \eqref{prop_joint_last_b-below_pmf}--\eqref{doubly_defective_kill_up_b} for case (i) and \eqref{prop_joint_last_b-above_pmf}--\eqref{doubly_defective_kill_down_b} for case (ii). For the continuous distributions we shall emply \eqref{last-passage-pdf-explicit-kill-b-below}, \eqref{last-passage-pdf-explicit-kill-b-above}, \eqref{joint_last_passage_pdf_explicit_b_1} and \eqref{joint_last_passage_pdf_explicit_b_2}. Equivalently,  Proposition~\ref{marginal_joint_last-passage-propn-kill-b} can also be used directly. We can simply reuse some of the expressions of Section \ref{section_drfited_BM_killed_a_b} for the transition densities and first hitting time CDFs. 

Consider case (i). Using the expressions for the transition densities $p^-_{\b}(T;x,z)$ and $p_{(\b,b)}(T;x,z)$ within \eqref{prop_joint_last_b-below_pmf} gives
\begin{align}
&\P_x(g^{b}_{\b}(T) = 0, X_{b,T} \in dz) 
\nonumber \\
&= \begin{cases} 
{1\over \sqrt{2\pi T}}\left( e^{-{(z-x - \mu T)^2\over 2T}} 
- e^{2\mu (\b - x)} e^{-{(z+x - 2\b - \mu T)^2\over 2T}} \right) dz &, x,z \in (-\infty,\b),
\\
\displaystyle {2\over b - \b}e^{\mu(z-x) - {\mu^2\over 2}T}\sum_{n=1}^\infty e^{-{n^2\pi^2\over 2(b-\b)^2}T}
\sin\big({n\pi (x-\b)\over b-\b}\big)\sin\big({n\pi (z-\b)\over b-\b}\big) dz &, x,z \in (\b,b).
\label{prop_joint_last_b-below_pmf_driftedBM}
\end{cases}
\end{align}
Substituting the expression for $\P_x(\T^+_\b > T)$, given in \eqref{last_hitting_discrete_drifted_BM} for $x < \b$, and 
$\P_x(\T^-_\b (b) > T)$, given in \eqref{drifted_BM_prop_last_time_discrete_killed_ab} for $x \in (\b,b)$, within 
\eqref{last_time-CDF_discrete_killed_below_b} gives
\begin{align}\label{discrete_killed_below_b_driftedBM}
\P_x(g^b_{\b}(T) = 0) = 
\begin{cases}
\NCDF\left({\b - x - \mu T) \over \sqrt{T}} \right) - e^{2\mu (\b - x)} \NCDF\left({x - \b - \mu T \over \sqrt{T}} \right)
&, x \in (-\infty,\b),
\\
R(x;b,\b)
+ e^{\mu(\b-x) - {\mu^2\over 2}T} 
\displaystyle 
\sum_{n=1}^\infty {2\pi ne^{- {n^2\pi^2 T\over 2(b-\b)^2}} \over n^2\pi^2 + \mu^2 (b-\b)^2}
\sin\big({n\pi (x-\b)\over b-\b}\big)&, x \in (\b,b).
\end{cases}
\end{align}
Note that $\P_x(\T^-_l (\b) \le T) = 0$ since $l=-\infty$ is a natural boundary. 
Moreover, $\P_x(\T^+_b (\b) \le T)$ is given by \eqref{drifted_BM_joint_last_ab_doubly_defective} for $x \in (\b,b)$. 
Hence, by \eqref{doubly_defective_kill_up_b}, we have $\P_x(g^{b}_{\b}(T) = 0, X_{b,T} = \partial^\dagger)=0$ if $x \in (-\infty,\b)$ and $\P_x(g^{b}_{\b}(T) = 0, X_{b,T} = \partial^\dagger)$ is given by \eqref{drifted_BM_joint_last_ab_doubly_defective} for $x \in (\b,b)$.

For case (ii), we use the expressions for $p^+_{\b}(T;x,z)$ and $p_{(b,\b)}(T;x,z)$ within 
\eqref{prop_joint_last_b-above_pmf} to give
\begin{align}
&\P_x(g^{b}_{\b}(T) = 0, X_{b,T} \in dz) 
\nonumber \\
&= \begin{cases} 
\displaystyle {2\over \b - b}e^{\mu(z-x) - {\mu^2\over 2}T}\sum_{n=1}^\infty e^{-{n^2\pi^2\over 2(\b - b)^2}T}
\sin\big({n\pi (x - b)\over \b - b}\big)\sin\big({n\pi (z - b)\over \b - b}\big) dz &, x,z \in (b,\b),
\\
{1\over \sqrt{2\pi T}}\left( e^{-{(z-x - \mu T)^2\over 2T}} 
- e^{2\mu (\b - x)} e^{-{(z+x - 2\b - \mu T)^2\over 2T}} \right) dz &, x,z \in (\b,\infty).
\label{prop_joint_last_b-above_pmf_driftedBM}
\end{cases}
\end{align}
Substituting the expression for $\P_x(\T^-_\b > T)$, given in \eqref{last_hitting_discrete_drifted_BM} for $x > \b$, and 
$\P_x(\T^+_\b(b) > T)$, given by the first expression in 
\eqref{drifted_BM_prop_last_time_discrete_killed_ab} with $b$ replacing $a$, i.e., $x \in (b,\b)$, within 
\eqref{prop_last_time-CDF_discrete_killed_above_b} gives
\begin{align}\label{discrete_killed_above_b_driftedBM}
\P_x(g^b_{\b}(T) = 0) = 
\begin{cases}
R(x;b,\b)
+ e^{\mu(\b-x) - {\mu^2\over 2}T} 
\displaystyle 
\sum_{n=1}^\infty {2\pi ne^{- {n^2\pi^2 T\over 2(\b-b)^2}} \over n^2\pi^2 + \mu^2 (\b-b)^2}
\sin\big({n\pi (\b-x)\over \b - b}\big)&, x \in (b,\b),
\\
\NCDF\left({x - \b + \mu T) \over \sqrt{T}} \right) - e^{2\mu (\b - x)} \NCDF\left({\b - x + \mu T \over \sqrt{T}} \right)
&, x \in (\b,\infty).
\end{cases}
\end{align}
Note that $\P_x(\T^+_r (\b) \le T) = 0$ since $r=\infty$ is a natural boundary. 
Moreover, $\P_x(\T^-_b (\b) \le T)$ is given by the first expression in \eqref{drifted_BM_joint_last_ab_doubly_defective} with $b$ replacing $a$, i.e., for $x \in (b,\b)$. 
Hence, by \eqref{doubly_defective_kill_down_b}, we have $\P_x(g^{b}_{\b}(T) = 0, X_{b,T} = \partial^\dagger)=0$ if $x \in (\b,\infty)$ and for $x \in (b,\b)$,
\begin{align}\label{discrete_killed_above_b_driftedBM}
\P_x(g^{b}_{\b}(T) = 0, X_{b,T} = \partial^\dagger) 
= 
R(x;b,\b)
- e^{\mu(b-x) - {\mu^2\over 2}T} 
\displaystyle 
\sum_{n=1}^\infty {2\pi ne^{- {n^2\pi^2 T\over 2(\b-b)^2}} \over n^2\pi^2 + \mu^2 (\b-b)^2}
\sin\big({n\pi (x-b)\over \b - b}\big).
\end{align}

The marginal and joint densities are now computed. From Section \ref{subsect_drifted_BM} we have 
\begin{align*}
\pm{1\over 2}{\partial \over \partial y} \P_y(\T^\pm_\b \le T -t)\big\vert_{y=\b\mp}
= {e^{-{\mu^2 \over 2} (T-t)}\over \sqrt{2 \pi (T-t)}} \mp \mu \NCDF(\mp\mu \sqrt{T-t}).
\end{align*}
For the first hitting times $\T^-_\b (b)$ and $\T^-_\b (b)$ we have the respective CDFs:
\begin{align*}
\P_y(\T^-_\b (b) \le T - t) = R(y;\b,b) 
- e^{\mu(\b-y) - {\mu^2\over 2}(T-t)} 
\displaystyle 
\sum_{n=1}^\infty {2\pi n e^{- {n^2\pi^2 (T-t)\over 2(b-\b)^2}} \over n^2\pi^2 + \mu^2 (b-\b)^2}
\sin\big({n\pi (y-\b)\over b-\b}\big)&, y \in (\b,b),
\\
\P_y(\T^+_\b (b) \le T - t) = R(y;\b,b) 
- e^{\mu(\b-y) - {\mu^2\over 2}(T-t)} 
\displaystyle 
\sum_{n=1}^\infty {2\pi n e^{- {n^2\pi^2 (T-t)\over 2(\b-b)^2}} \over n^2\pi^2 + \mu^2 (\b-b)^2}
\sin\big({n\pi (\b-y)\over \b-b}\big)&, y \in (b,\b).
\end{align*}
Hence, 
\begin{align*}
{1\over 2}{\partial \over \partial y} \P_y(\T^\pm_\b (b) \le T-t)\bigg\vert_{y=\b\mp}
 = \hat{R}(\b,b) 
+ {e^{-{\mu^2\over 2}(T-t)} \over \b - b}
\displaystyle 
\sum_{n=1}^\infty { n^2\pi^2 e^{- {n^2\pi^2 (T-t)\over 2(\b - b)^2}} \over n^2\pi^2 + \mu^2 (\b - b)^2}
\end{align*}
where $\hat{R}(\b,b):={\mu \over e^{2\mu(\b - b)} - 1}$ if $\mu\ne 0$ and 
$\hat{R}(\b,b):={1 \over 2(\b - b)}$ if $\mu=0$. Note that $\m(\b)\s(\b)=2$. 
Inserting the above respective derivative expressions into \eqref{last-passage-pdf-explicit-kill-b-below} 
and \eqref{last-passage-pdf-explicit-kill-b-above}, along with $p^\pm_{b}(t;x,\b)$, gives 
\begin{align}\label{last-passage-pdf-explicit-kill-b-below_driftedBM_1}
&f_{g^b_{\b}(T)}(t;x) 
= {1\over \sqrt{2\pi t}}\bigg(e^{-{(\b - x - \mu t)^2\over 2t}}
- e^{2\mu (b - x)} e^{-{(\b+x - 2b - \mu t)^2\over 2t}} \bigg)
\\
&\times\begin{cases}
\displaystyle {e^{-{\mu^2 \over 2} (T-t)}\over \sqrt{2 \pi (T-t)}} - \mu \NCDF(-\mu \sqrt{T-t}) 
 - \hat{R}(\b,b) + {e^{-{\mu^2\over 2}(T-t)} \over b - \b}
\sum_{n=1}^\infty {n^2\pi^2 e^{- {n^2\pi^2 (T-t)\over 2(b - \b)^2}} \over n^2\pi^2 + \mu^2 (b - \b)^2}  
&, x,\b\in(-\infty,b),
\nonumber \\
\displaystyle {e^{-{\mu^2 \over 2} (T-t)}\over \sqrt{2 \pi (T-t)}} + \mu \NCDF(\mu \sqrt{T-t}) 
 + \hat{R}(\b,b) + {e^{-{\mu^2\over 2}(T-t)} \over \b - b}
\sum_{n=1}^\infty {n^2\pi^2 e^{- {n^2\pi^2 (T-t)\over 2(\b - b)^2}} \over n^2\pi^2 + \mu^2 (\b - b)^2}
&, x,\b\in(b,\infty).
\end{cases}
\end{align}

The joint PDF now follows from \eqref{joint_last_passage_pdf_explicit_b_1} and \eqref{joint_last_passage_pdf_explicit_b_2}. In particular,
\[
p_{(b,\b)}(T-t;y,z) = {2\over (\b - b)}e^{\mu(z-y) - {\mu^2\over 2}(T-t)}
\sum_{n=1}^\infty e^{-{n^2\pi^2\over 2(\b - b)^2}(T-t)}
\sin\big({n\pi (y-b)\over \b - b}\big)\sin\big({n\pi (z-b)\over \b - b}\big)\,.
\]
For $z\in (b,\b)$, ${\m(z) \over \m(\b)}f^+(T-t; z,\b| b) = 
 -{1\over 2}{\partial \over \partial y} p_{(b,\b)}(T-t;y,z)\big\vert_{y=\b-}$ since $\m(\b)\s(\b)=2$, i.e.,
\begin{align*}
{\m(z) \over \m(\b)}f^+(T-t; z,\b| b) 
 = {\pi\over (\b - b)^2}e^{\mu(z-\b) - {\mu^2\over 2}(T-t)}
\sum_{n=1}^\infty n e^{-{n^2\pi^2\over 2(\b - b)^2}(T-t)}\sin\big({n\pi (\b - z)\over \b - b}\big).
\end{align*}
Similarly, ${\m(z) \over \m(\b)}f^-(T-t; z,\b| b)$ is given by the same series expression for $z\in (\b,b)$.
Note that the series also follow directly from \eqref{joint_last_passage_pdf_spectral_b_1} 
and \eqref{joint_last_passage_pdf_spectral_b_2} using $\psi^-_{n}(z;\b,b)$ and $\psi^+_{n}(z;b,\b)$ of the previous section. 
Combining the above series with \eqref{driftedBM_single_kill_transPDF_der} and 
$p^\pm_{b}(t;x,\b)$, within \eqref{joint_last_passage_pdf_explicit_b_1} and \eqref{joint_last_passage_pdf_explicit_b_2} gives
\begin{align}\label{joint_last_passage_pdf_b_1_driftedBM}
f_{g^{b}_{\b}(T), X_{b,T}}(t,z;x) 
&= {e^{\mu(z-\b) - {\mu^2\over 2}(T-t)}\over \sqrt{2\pi t}}\bigg(e^{-{(\b - x - \mu t)^2\over 2t}}
- e^{2\mu (b - x)} e^{-{(\b+x - 2b - \mu t)^2\over 2t}} \bigg) 
\nonumber \\
&\times\begin{cases}
\displaystyle {\b - z \over \sqrt{2\pi (T-t)^3}} e^{-{(z - \b)^2\over 2(T-t)}} 
&, z \in (-\infty,\b),
\\
{\pi\over (b - \b)^2}
\displaystyle \sum_{n=1}^\infty n e^{-{n^2\pi^2\over 2(b - \b)^2}(T-t)}\sin\big({n\pi (z - \b)\over b - \b}\big)
&, z \in (\b,b),
\end{cases}
\end{align}
for $x,\b\in(-\infty,b)$, and 
\begin{align}\label{joint_last_passage_pdf_b_2_driftedBM}
f_{g^{b}_{\b}(T), X_{b,T}}(t,z;x) 
&= {e^{\mu(z-\b) - {\mu^2\over 2}(T-t)}\over \sqrt{2\pi t}}\bigg(e^{-{(\b - x - \mu t)^2\over 2t}}
- e^{2\mu (b - x)} e^{-{(\b+x - 2b - \mu t)^2\over 2t}} \bigg) 
\nonumber \\
&\times\begin{cases}
 {\pi\over (\b - b)^2}
\displaystyle \sum_{n=1}^\infty n e^{-{n^2\pi^2\over 2(\b - b)^2}(T-t)}\sin\big({n\pi (\b-z)\over \b - b}\big)
&, z \in (b,\b),
\\
\displaystyle {z - \b \over \sqrt{2\pi (T-t)^3}} e^{-{(z - \b)^2\over 2(T-t)}} 
&, z \in (\b,\infty),
\end{cases}
\end{align}
for $x,\b\in (b,\infty)$.

The GBM process with imposed killing at only one point $B\in (0,\infty)$ is given by 
$F_{B,t} = {\sf F}(X_{b,t})$, where the underlying Brownian motion has drift 
parameter $\nu \equiv (\mu - \sigma^2/2)/\sigma$, 
with mapping ${\sf F}$ and its inverse ${\sf X}$ defined in Section~\ref{subsect_GBM}, 
i.e., ${\sf X}(z) = {1 \over\sigma}\ln{z \over F_0}$, where 
$b = {1 \over\sigma}\ln{B \over F_0}$, $\b = {1 \over\sigma}\ln{K \over F_0}$, with initial value $F_0 = F_{B,0}$. The marginal distribution of the lasting hitting time to level $K\in (0,\infty)$, denoted by $g^{B, F}_K(T)$, and the joint distribution of the lasting hitting time and process value, $(g^{B, F}_K(T), F_{B,T})$, follow in the obvious analogous manner stated at the end of Section~\ref{section_drfited_BM_killed_a_b} .

%
%

\subsection{Squared Bessel Process}\label{subsect_SQB}

Consider the squared Bessel (SQB) diffusion $\{X_t, t \geq 0\} \in (0, \infty)$ with generator
\[
G f(x) := 2xf''(x) + 2(\mu + 1)f'(x),\,\,x\in (0, \infty),
\]
i.e., with scale and speed densities $\s(x) = x^{-1-\mu}$ and $\m(x) = {1\over 2}x^\mu$, $\mu\in\R$. [Note: this process is closely related to that satisfying the stochastic differential equation 
$dX_t = \alpha_0 dt + \nu_0\sqrt{X_t}dW_t$, where $\alpha_0 = {\nu_0^2 \over 2}(\mu + 1)$
 with choice $\nu_0=2$.] 
The left endpoint $l = 0$ is entrance-not-exit if $\mu \geq 0$, regular if $\mu \in (-1,0)$ , and exit-not-entrance if $\mu \leq -1$. The right endpoint $r = \infty$ is natural (attracting only for $\mu > 0$). A pair of fundamental solutions 
to \eqref{eq:phi} are (e.g., see \cite{BS02})  
\begin{equation}\label{SQB_fundamental_funcs}
    \varphi_\lambda^+(x) := x^{-\mu/2} I_\nu(\sqrt{2\lambda x});
\ \ \varphi_\lambda^-(x) := x^{-\mu/2} K_\nu(\sqrt{2\lambda x}),
\end{equation}
where $\nu := \mu$ if $\mu \ge 0$ or if $\mu\in(-1,0)$ and $l=0$ is specified as reflecting, 
and $\nu := \vert\mu\vert$ if $\mu \le -1$ or if $\mu\in(-1,0)$ and $l=0$ is specified as killing. 
The functions $I_\nu(x)$ and $K_\nu(x)$ denote, respectively, the modified Bessel functions 
of the first and second kind of order $\nu$, e.g., see \cite{AS72}. 
[Note: $K_\nu(z) \equiv K_{\mu}(z)$ since $K_{-\mu}(z) \equiv K_{\mu}(z)$.]  
The Wronskian factor is simply $w_\lambda = \frac{1}{2}$. 
We consider all possible 
values for $\mu$ and boundary specification (reflecting or killing at $0$) when $\mu\in(-1,0)$. 

The boundary $l=0$ is NONOSC and $r = \infty$ is O-NO natural with spectral cutoff at $0$, i.e., Spectral Category II. The Green function, $G(\lambda;x,y) = (y/x)^{\mu/2} I_\nu(\!\sqrt{2\lambda (x\wedge y)})K_\nu(\!\sqrt{2\lambda (x\vee y)})$, $x,y > 0$, has a branch point at $\lambda=0$ as its only singularity. In particular, we have
\begin{align*}
{1\over \pi}\text{Im}\,G (\epsilon e^{-i\pi};x,y)
= \big({y\over x}\big)^{\mu/2} {1\over \pi} \text{Im}\,\big\{ I_{\nu}(-i\sqrt{2\epsilon (x\wedge y)})
K_{\nu}(-i\sqrt{2\epsilon (x\vee y)})\big\}
= {1\over 2}\big({y\over x}\big)^{\mu/2} J_{\nu}(\sqrt{2\epsilon x})
J_{\nu}(\sqrt{2\epsilon y})\,.
\end{align*}
Here we used ${1\over \pi i} \big[ I_{\nu}(-ia)K_{\nu}(-ib) - I_{\nu}(ia)K_{\nu}(ib)\big] = J_{\nu}(a)J_{\nu}(b)$. 
Throughout, $J_\nu(x)$ and $Y_\nu(x)$ denote, respectively, the Bessel functions of the first and second kind of order 
$\nu$. Substituting the above expression into \eqref{u_spectral_5}, with empty summation, produces the (purely continuous) known spectral representation and closed form expression for the transition density:
\footnote
{
This follows by the integral identity: 
$
\int_0^\infty e^{-\alpha \epsilon} J_{\nu}(2\beta\sqrt{\epsilon}) J_{\nu}(2\gamma\sqrt{\epsilon})\,d \epsilon 
= {e^{-(\beta^2 + \gamma^2)/\alpha}\over \alpha}I_{\nu}\big({2\beta\gamma\over \alpha}\big)$, valid for $\text{Re}\, \nu > -1$ and real $\alpha,\beta,\gamma > 0$, 
while setting $\alpha = t, \beta = \sqrt{x/2}, \gamma = \sqrt{y/2}$.
}
\begin{eqnarray}\label{SQB_trans_PDF_no_kill}
p(t;x,y) = {1\over 2}\big({y\over x}\big)^{\mu \over 2}\!\! \int_0^\infty\! e^{-\epsilon t}J_{\nu}(\sqrt{2\epsilon x}) 
J_{\nu}(\sqrt{2\epsilon y}) d \epsilon
= \left({y\over x}\right)^{\frac{\mu}{2}}
    {e^{-{x + y \over 2t}} \over 2t} I_{\nu}\big({\sqrt{xy}\over t}\big), \,\,x,y\in(0,\infty), t>0.
\end{eqnarray}

Let us first derive the discrete part of the distribution of $g_\b(T)$, $\b\in (0,\infty)$, by implementing \eqref{prop_last_time-t-1-limit}.  Hence, we now compute the quantities required in \eqref{FHT_cdf_complement}. 
Since $l=0$ is NONOSC, we use \eqref{FHT_prop1_2}. The eigenvalues $\lambda_n \equiv \lambda^-_{n,\b}$ 
solve $\varphi_{-\lambda_n}^+(\b) = 0$, i.e., $I_\nu(i\sqrt{2\lambda_n \b}) = 0 \implies J_\nu(\sqrt{2\lambda_n \b}) = 0 
\implies \lambda_n = {1\over 2\b}j_{n,\nu}^2$, where $j_{n,\nu}, n \ge 1,$ 
are all the positive zeros of $J_\nu(z)$.
\footnote{
The zeros are efficiently computed via a root finding algorithm. 
The zeros $z_{n,\nu} = (2n + \nu){\pi\over 2} - {\pi\over 4}$ of the asymptotic ($z\to \infty$) form
$J_\nu(z) \sim \sqrt{2/\pi z}\cos(z - \nu\pi/2 - \pi/4)$ may be used as initial estimates, i.e. $j_{\nu,n}\sim z_{n,\nu}$. The eigenvalues grow roughly like $n^2$ with increasing $n$. Hence, the resulting series converges rapidly and can be truncated using a small number of terms (particularly for larger values of time and will depend on the value of $\b$.)
}
By again using the basic property $I_\nu(iz) =i^\nu J_\nu(z)$ and the differential recurrence, 
$J_\nu'(z) = (\nu/z)J_{\nu}(z) - J_{\nu + 1}(z)$, along with the eigenvalue equation, i.e.,  
$J_{\nu}(\sqrt{2\lambda_n \b}) = 0$, within \eqref{FHT_eigenfunctions_1}, with $\b$ in the place of $b$, gives
\begin{eqnarray}\label{SQB_Psi_plus_func}
\psi_n^+(x;\b) = 
{j_{n,\nu} \over \b}
\frac{(\b/x)^{\mu/2}J_\nu(j_{n,\nu} \sqrt{x/\b})}
{J_{\nu + 1}(j_{n,\nu})}.
\end{eqnarray} 
Hence, by \eqref{FHT_prop1_2},
\begin{eqnarray}\label{SQB_FHT_up_tail}
\P_{x}(T < \Tau^{+}_{\b} < \infty) = 
2\big({\b\over x}\big)^{\mu \over 2}\sum_{n=1}^\infty e^{- {1\over 2\b}j_{n,\nu}^2T} {J_\nu(j_{n,\nu} \sqrt{x/\b}) \over
j_{n,\nu} \, J_{\nu + 1}(j_{n,\nu})},\,\,0 < x < \b < \infty.
\end{eqnarray}
The scale function is given by
\begin{equation}\label{SQB_scale_func}
{\mathcal S}[x,y]=\ln (y/ x) \ind_{\{\mu = 0\}} + {1\over \mu} (x^{-\mu} - y^{-\mu})\ind_{\{\mu \ne 0\}}.
\end{equation}
We have $\ind_{E_0}=0$, if $\mu \ge 0$; $\ind_{E_0}=1$, if $\mu \le -1$ or if $\mu\in(-1,0)$ and 
$\nu=\vert\mu\vert \equiv -\mu$ ($0$ is killing); otherwise, 
$\ind_{E_0}=0$, if $\mu\in(-1,0)$ and $\nu=\mu$ ($0$ is reflecting). From \eqref{FHT-b-up-equal_infinity}, 
$\P_x({\Tau}^+_\b = \infty) = [1 - (\b/x)^\mu] \cdot \ind_{E_0}$, giving
\begin{eqnarray}\label{SQB_FHT_up_all_tail}
\P_{x}(\Tau^{+}_{\b} > T) =  [1 - (\b/x)^\mu] \cdot \ind_{E_0} + 
2\big({\b\over x}\big)^{\mu \over 2}\sum_{n=1}^\infty e^{- {1\over 2\b}j_{n,\nu}^2T} {J_\nu(j_{n,\nu} \sqrt{x/\b}) \over
j_{n,\nu} \, J_{\nu + 1}(j_{n,\nu})},\,\,\,0 < x < \b < \infty.
\end{eqnarray}

We next compute the tail probability in \eqref{FHT_prop_ONO2_2}, where $\Lambda_+=0$ and no summation since the spectrum is now purely continuous. In particular, by using the identity
$$
{1\over 2 i}\bigg[{K_\mu(-ix)\over K_\mu(-iy)} - {K_\mu(ix)\over K_\mu(iy)}\bigg]
=\! \!{1\over 2 i}\bigg[{J_\mu(x) + i Y_\mu(x)\over J_\mu(y) + i Y_\mu(y)} - {J_\mu(x) - i Y_\mu(x)\over J_\mu(y) - i Y_\mu(y)}\bigg]\!
= {Y_\mu(x) J_\mu(y) - J_\mu(x) Y_\mu(y) \over J^2_\mu(y) + Y^2_\mu(y)}
$$ 
we have
\begin{align}\label{SQB_imaginary_ratio}
&\textup{Im}
\bigg\{{\varphi^-_\lambda (x) \over \varphi^-_\lambda (\b)}\bigg\vert_{\lambda = \epsilon e^{-i\pi}}\bigg\}
= \bigg({\b\over x}\bigg)^{\mu \over 2}
{1\over 2 i}\bigg[{K_{\mu}(-i\sqrt{2\epsilon x}) \over K_\mu(-i\sqrt{2\epsilon \b})} 
-{K_{\mu}(i\sqrt{2\epsilon x}) \over K_\mu(i\sqrt{2\epsilon \b})} \bigg]
= \bigg({\b\over x}\bigg)^{\mu \over 2}
\Psi_\mu(\b,x;\epsilon) \rho(\epsilon;\b)\,,
\end{align}
where $\rho(\epsilon;\b):=  [J^2_\mu(\sqrt{2\epsilon \b}) + Y^2_\mu(\sqrt{2\epsilon \b})]^{-1}$ is a spectral density for any given level $\b \in (0,\infty)$. 
Throughout, we also conveniently define the associated Bessel cylinder function:
\begin{equation}\label{cylinder_SQB}
\Psi_{\mu}(x,y;\lambda) := J_{\mu}\big(\sqrt{2\lambda x}\,\big) Y_{\mu}\big(\sqrt{2\lambda y}\,\big)
- Y_{\mu}\big(\sqrt{2\lambda x}\,\big) J_{\mu}\big(\sqrt{2\lambda y}\,\big) 
\end{equation}
for all $\mu \in \R$, $x,y > 0$, $\lambda \in \C$. From the definition of $Y_\mu$ in terms of $J_\mu$ and 
$J_{-\mu}$, it follows that $\Psi_{-\mu}(x,y;\lambda) = \Psi_{\mu}(x,y;\lambda)$. Hence, inserting the expression in  \eqref{SQB_imaginary_ratio} within the integral in \eqref{FHT_prop_ONO2_2}, with $\b$ replacing $b$, gives the purely  continuous spectral (integral) representation:
\begin{eqnarray}\label{SQB_FHT_down_tail}
\P_{x}(T < \Tau^{-}_{\b} < \infty) 
= \bigg({\b\over x}\bigg)^{\mu \over 2}{1\over \pi}
\int_0^\infty {e^{-\epsilon T} \over \epsilon}\Psi_\mu(\b,x;\epsilon) \rho(\epsilon;\b)\,d \epsilon 
\,,\,\,\,\,\,\,\,\,0 < \b < x < \infty,
\end{eqnarray}
Moreover, since $\ind_{E_\infty}=\ind_{\{\mu > 0\}}$, \eqref{FHT-b-up-equal_infinity} gives 
$\P_x({\Tau}^-_\b = \infty) = [1 - (\b/x)^\mu] \cdot \ind_{\{\mu > 0\}}$, i.e., 
\begin{eqnarray}\label{SQB_FHT_down_all_tail}
\P_{x}(\Tau^{-}_{\b} > T) 
= [1 - (\b/x)^\mu]\! \cdot\! \ind_{\mu > 0} + \bigg({\b\over x}\bigg)^{\mu \over 2}{1\over \pi}
\!\int_0^\infty {e^{-\epsilon T} \over \epsilon}\Psi_\mu(\b,x;\epsilon) \rho(\epsilon;\b) 
d \epsilon \,,\,\,\,\,\,\,\,\,0 < \b < x < \infty.
\end{eqnarray}
Hence, combining \eqref{SQB_FHT_up_all_tail} and \eqref{SQB_FHT_down_all_tail} within 
\eqref{prop_last_time-t-1-limit} gives 
\begin{align}\label{prop_last_time_discrete_SQB}
\P_x(g_\b(T) = 0) = \begin{cases} 
\displaystyle [1 - (\b/x)^\mu]\! \cdot\! \ind_{\mu > 0} + \bigg({\b\over x}\bigg)^{\mu \over 2}{1\over \pi}
\!\int_0^\infty {e^{-\epsilon T} \over \epsilon}\Psi_\mu(\b,x;\epsilon) \rho(\epsilon;\b) 
d \epsilon \,,\,\,\,0 < \b < x < \infty,
\\
\displaystyle [1 - (\b/x)^\mu] \cdot \ind_{E_0} + 
2\bigg({\b\over x}\bigg)^{\mu \over 2}\sum_{n=1}^\infty e^{- {1\over 2\b}j_{n,\nu}^2T} {J_\nu(j_{n,\nu} \sqrt{x/\b}) \over
j_{n,\nu} \, J_{\nu + 1}(j_{n,\nu})},\,\,\,\,\,\,\,0 < x < \b < \infty.
\end{cases}
\end{align}

To derive the density of $g_\b(T)$, we can either use \eqref{last-passage-pdf-explicit} or 
Proposition \ref{last-passage-propn-time-t-spectral}. In particular, by direct use of \eqref{SQB_FHT_up_tail} and \eqref{SQB_FHT_down_tail}, we have
\begin{align}\label{eta_plus_k_SQB}
\eta^+(T-t;\b) &\equiv  {\partial \over \partial x} \P_x(T-t < \Tau^+_\b < \infty)\big\vert_{x=\b-}
= - {1\over \b}\sum_{n=1}^\infty e^{- {1\over 2\b}j_{n,\nu}^2(T-t)},
\\
\eta^-(T-t;\b) &\equiv  {\partial \over \partial x} \P_x(T-t < \Tau^-_\b < \infty)\big\vert_{x=\b+}
= {1\over \b\pi^2}
\int_0^\infty  {e^{-\epsilon (T-t)} \over \epsilon} \rho(\epsilon;\b) \,d \epsilon.
\label{eta_minus_k_SQB}
\end{align}
The integral expression was simplified since ${\partial \over \partial x}\Psi_{\mu}(\b,x;\epsilon)\vert_{x=\b} = {1\over \pi\b}$. This obtains by applying the differential recurrence relations for $J_{\mu}$ and $Y_{\mu}$, while differentiating the cylinder function, canceling two terms, and using the Bessel Wronskian relation, 
$J_{\mu + 1}(z)Y_{\mu}(z) - Y_{\mu + 1}(z) J_{\mu}(z) = {2\over \pi z}$, for $z=\sqrt{2\epsilon \b}$. A direct alternative to deriving $\eta^\pm(T-t;\b)$ is to use \eqref{last_passage_pdf_spectral_1} and \eqref{last_passage_pdf_spectral_2}. Since $l=0$ is NONOSC, $\eta^+$ is given by the first expression in \eqref{last_passage_pdf_spectral_2} where it follows from \eqref{SQB_Psi_plus_func} 
that $\widehat\psi_n^+(\b) \equiv {\partial \over \partial x}\psi_n^+(x;\b)\vert_{x=\b} = -{\lambda_{n,\b}^-\over \b}$. 
For $\eta^-$, it is given solely by the integral in \eqref{last_passage_pdf_spectral_1}, with $\Lambda_+=0$, where the integrand is computed by using \eqref{SQB_imaginary_ratio}, ${\partial \over \partial x}\Psi_{\mu}(\b,x;\epsilon)\vert_{x=\b} = {1\over \pi\b}$ and $\Psi_{\mu}(\b,\b;\epsilon)=0$, i.e., 
$\textup{Im} \big\{{\varphi^{-\prime}_\lambda (\b) \over \varphi^-_\lambda (\b)}\big\vert_{\lambda = \epsilon e^{-i\pi}}\big\} = \textup{Im} {\partial \over \partial x}\big\{{\varphi^-_\lambda (x) \over \varphi^-_\lambda (\b)}\big\vert_{\lambda = \epsilon e^{-i\pi}}\big\}\big\vert_{x=\b} = 
{1\over \b\pi}\rho(\epsilon;\b)$.
From the above scale function we have 
\begin{equation*}\label{S_lrk_SQB}
\mathcal{S}(l,r;\b) \equiv \mathcal{S}(0,\infty;\b) = {\ind_{E_0} \over \mathcal{S}(0,\b]} + {\ind_{E_\infty} \over \mathcal{S}[\b,\infty)}
= \mu \b^\mu \big[\ind_{\{\mu > 0\}} - \ind_{\{E_0\}}\big].
\end{equation*}
Using \eqref{last-passage-pdf-spectral} with $\eta^\pm(T-t;\b)$ and \eqref{SQB_trans_PDF_no_kill} above, 
where $\m(\b) = {1\over 2}\b^\mu$, $1/\s(\b) = \b^{1+\mu}$, gives
\begin{align}\label{last-passage-pdf-SQB}
f_{g_\b(T)}(t;x) = \left({\b \over x}\right)^{\!\frac{\mu}{2}}\!
    {e^{-{x + \b \over 2t}} \over t} I_{\nu}\big({\sqrt{x\b}\over t}\big)
\bigg[\mu \big(\ind_{\{\mu > 0\}} - \ind_{E_0}\big)
+ {1\over \pi^2}\!\int_0^\infty  {e^{-\epsilon (T-t)} \over \epsilon} \rho(\epsilon;\b) \,d \epsilon 
+ \sum_{n=1}^\infty e^{- {1\over 2\b}j_{n,\nu}^2(T-t)}\!
\bigg]\!,
\end{align}
$t\in (0,T), x,\b \in (0,\infty)$.

The joint PDF now follows from Proposition~\ref{joint_last-passage-propn-time-t-new-formula}, by using the first line expression in \eqref{joint_last_passage_pdf_spectral_1} and only the integral term with 
$\Lambda_+ = 0$ in  \eqref{joint_last_passage_pdf_spectral_2} and simply substituting \eqref{SQB_Psi_plus_func} and  \eqref{SQB_imaginary_ratio}, with $x$ replaced by $z$, and \eqref{SQB_trans_PDF_no_kill} for $y=\b$, within 
\eqref{joint_last_passage_pdf_spectral_1}--\eqref{joint_last_passage_pdf_spectral_2}. In particular,
\begin{align}\label{eta_plus_z_k_SQB}
f^+(T-t;z,\b) &= \left({\b \over z}\right)^{\!\frac{\mu}{2}}
 {1 \over \b}\displaystyle \sum_{n=1}^\infty e^{- {1\over 2\b}j_{n,\nu}^2(T-t)} \,
{j_{n,\nu}\, J_\nu(j_{n,\nu} \sqrt{z/\b}) \over J_{\nu + 1}(j_{n,\nu})},
\\
f^-(T-t;z,\b) &= \displaystyle {1 \over \pi}\left({\b \over z}\right)^{\!\frac{\mu}{2}}
\int_{0}^\infty e^{-\epsilon (T-t)}\Psi_\mu(\b,z;\epsilon) \rho(\epsilon;\b) d\epsilon.
\label{eta_minus_z_k_SQB}
\end{align}
Hence, 
%
%
\begin{align}\label{joint_last_passage_pdf_SQB_spectral}
f_{g_\b(T), X_T}(t,z;x) 
= \left({z \over x}\right)^{\!\frac{\mu}{2}}{e^{-{x + \b \over 2t}} \over 2t} I_{\nu}\big({\sqrt{x\b}\over t}\big) 
\begin{cases}
\displaystyle {1 \over \b} \sum_{n=1}^\infty e^{- {1\over 2\b}j_{n,\nu}^2(T-t)} \,
{j_{n,\nu}\, J_\nu(j_{n,\nu} \sqrt{z/\b}) \over J_{\nu + 1}(j_{n,\nu})}, & z\in (0,\b),
\\
\displaystyle {1 \over \pi}\int_{0}^\infty e^{-\epsilon (T-t)}\Psi_\mu(\b,z;\epsilon) \rho(\epsilon;\b) d\epsilon, & z\in (\b,\infty),
\end{cases}
\end{align}
$t\in (0,T), x,\b \in (0,\infty)$.

The partly discrete joint distribution is given by \eqref{prop_joint_last_time-t-2}. Employing the transition PDFs given by \eqref{trans_PDF_spectral_SQB_b} and \eqref{trans_PDF_spectral_SQB_b_down} in Section \ref{subsect_SQB_kill_b} within \eqref{prop_joint_last_time-t-2}, for time $T$, level $\b$ and endpoint $z$, gives
\begin{align}\label{prop_joint_last_time-t-2_SQB}
\P_x(g_\b(T) = 0, X_T \in dz) = \bigg({z\over x}\bigg)^{\mu \over 2}
\begin{cases} 
\displaystyle \frac{1}{\b} \sum_{n=1}^\infty \frac{e^{- {1\over 2\b}j_{n,\nu}^2 T}}
{J^2_{\nu + 1}(j_{\nu,n})}
J_{\nu}\big(j_{\nu,n}\sqrt{x/ \b}\big) J_{\nu}\big(j_{\nu,n}\sqrt{z/ \b}\big) dz &, x,z \in (0,\b),
\\
\displaystyle {1\over 2} \int_0^\infty 
e^{-\epsilon T}\Psi_\mu(\b,x;\epsilon) \Psi_\mu(\b,z;\epsilon) \rho(\epsilon;\b) \,d \epsilon \, dz &, x,z \in (\b,\infty).
\end{cases}
\end{align}
For the jointly discrete portion of the distribution we implement \eqref{joint_last_doubly_defective_formula}. Since $r=\infty$ is a natural boundary, $\P_x(g_\b(T) = 0, X_T = \partial^\dagger) = 0$ for $x \in (\b,\infty)$. 
For $x < \b$, note that $\P_x(g_\b(T) = 0, X_T = \partial^\dagger)$ is nonzero only when $l=0$ is nonconservative, i.e., either $\mu \le -1$ or $\mu\in(-1,0)$ and $l=0$ is specified as killing. Hence, we now assume $\mu < 0$ and $\nu = \vert \mu \vert$. To employ \eqref{joint_last_doubly_defective_formula} we consider the $z\to 0+$ asymptotic 
$J_\nu(z) \sim (z/2)^\nu/\Gamma (\nu + 1)$, i.e., 
$\varphi^{+}_{-\lambda_n}(x) = x^{-{\mu\over 2}} I_\nu(i\sqrt{2\lambda_n x}) = 
x^{{\nu\over 2}} i^\nu J_\nu(\sqrt{2\lambda_n x}) 
\sim x^{\nu} i^\nu (\lambda_n/2)^{\nu\over 2}/\Gamma (\nu + 1)$, as $x\to 0+$. Hence, 
$\varphi^{+\,\prime}_{-\lambda_n}(x) \sim x^{\nu - 1} i^\nu (\lambda_n/2)^{\nu\over 2}/\Gamma (\nu)$, as 
$x\to 0+$. Since 
$\s(x) = x^{-\mu - 1} = x^{\nu - 1}$, we have 
$\varphi^{+\,\prime}_{-\lambda_n}(0+) / \s(0+) = {i^\nu \over \Gamma (\nu)} (\lambda_n/2)^{\nu\over 2} = 
{i^\nu \over \Gamma (\nu)}({j_{n,\nu}\over 2\sqrt{\b}})^\nu$. Moreover, 
$\varphi^{-}_{-\lambda_n}(\b) = \b^{-{\mu\over 2}} K_\nu(i\sqrt{2\lambda_n \b}) = 
\b^{{\nu\over 2}} K_\nu(i j_{n,\nu}) = -\b^{{\nu\over 2}}i^{-\nu}{\pi\over 2}Y_\nu(j_{n,\nu})$. The last expression follows by using  the Bessel function identity $i^\nu K_\nu(ix) = -{\pi\over 2}[Y_\nu(x) + i J_\nu(x)]$, $x\in\R$, and setting $x=j_{n,\nu}$. Combining the above with the expression in \eqref{SQB_Psi_plus_func}, and 
${\mathcal{S}[x,\b] \over \mathcal{S}(0,\b]} = 1 - (\b/x)^\mu = 1 - (x/\b)^\nu$, within \eqref{joint_last_doubly_defective_formula} gives
\begin{eqnarray}
\P_x(g_\b(T) = 0, X_T = \partial^\dagger) 
= 1 - (x/\b)^\nu 
- {2^{1-\nu}\pi x^{\nu\over 2}\over  \b^{\nu\over 2}\Gamma (\nu)}
\sum_{n=1}^\infty e^{- {1\over 2\b}j_{n,\nu}^2 T}
{(j_{n,\nu})^{\nu - 1}Y_\nu(j_{n,\nu}) \over J_{\nu + 1}(j_{n,\nu})}J_\nu(j_{n,\nu}\sqrt{x/\b}).
\label{joint_last_doubly_defective_formula_SQB}
\end{eqnarray}
Throughout, $\Gamma(\cdot)$ denotes the standard Gamma function.

%
%
%
%
\subsection{Squared Bessel Process killed at either of two interior points}\label{subsect_SQB_kill_a_b}
Consider the SQB diffusion, $X_{(a,b),t}$, with imposed killing at either level $a$ or $b$, 
$0 < a < b < \infty$. We now derive the marginal and joint distributions in the last hitting time for this process by applying \eqref{prop_last_time-CDF_discrete_killed_ab}--\eqref{joint_last_ab_doubly_defective_new} and Proposition~\ref{joint_last-passage-propn-time-t_ab-new-version}. 
From \eqref{cylinder_SQB} we have the cylinder function 
$\phi(x,y;-\lambda) = {\pi\over 2}(xy)^{-\mu/2}\Psi_{\mu}(x,y;\lambda)$. 
To compute the functions in \eqref{FHT_eigenfunctions_ab} we have $\Delta(a,b;\lambda_n) = {\pi\over 2}(ab)^{-\mu/2} {\partial\over \partial\lambda}\Psi_{\mu}(a,b;\lambda)\big\vert_{\lambda=\lambda_n}$, 
with eigenvalues $\lambda_n = \lambda_n^{(a,b)}$ solving $\Psi_{\mu}(a,b;\lambda_n) = 0$, i.e.,
\footnote
{The eigenvalues can be numerically computed by using a root finding (e.g., bisection) algorithm. 
Combining the $z \to \infty$ asymptotic for $J_\mu(z)$ with $Y_\mu(z) \sim \sqrt{2/\pi z}\sin(z - \mu\pi/2 - \pi/4)$ within the sine addition formula gives $\Psi_\mu(a,b;\lambda) \sim {\sqrt{2}(ab)^{-{1\over 4}}\over \pi\sqrt{\lambda}}
\sin\big(\sqrt{2\lambda}(\sqrt{b} - \sqrt{a})\,\big)$, as $\lambda\to\infty$, with positive zeros in $\lambda$ given by 
$z_n = {n^2\pi^2 \over 2(\sqrt{b} - \sqrt{a})^2}$, $n=1,\ldots$. Hence, $\lambda_n \sim z_n$ may be used as initial estimates. The eigenvalues grow roughly like $n^2$ with increasing $n$. The relevant spectral expansions are rapidly convergent and may be truncated to a relatively small number of terms (particularly for larger values of time) in order to achieve a prescribed accuracy. The number of terms needed to attain a set accuracy will also depend on the relative values of $a$ and $b$.
}
\begin{equation}\label{eigen_SQB_kill_ab}
J_{\mu}(\sqrt{2\lambda_n a})Y_{\mu}(\sqrt{2\lambda_n b}) - 
Y_{\mu}(\sqrt{2\lambda_n a})J_{\mu}(\sqrt{2\lambda_n b}) = 0.
\end{equation}
By using differential recurrence relations, $J_\mu'(z) = (\mu/z)J_{\mu}(z) - J_{\mu + 1}(z)$, 
$Y_\mu'(z) = (\mu/z)Y_{\mu}(z) - Y_{\mu + 1}(z)$, we have
\footnote{Here we also use \eqref{eigen_SQB_kill_ab}, i.e., 
${J_{\mu}(\sqrt{2\lambda_n a}) \over Y_{\mu}(\sqrt{2\lambda_n a})} = 
{J_{\mu}(\sqrt{2\lambda_n b}) \over Y_{\mu}(\sqrt{2\lambda_n b})}$, 
within both terms in the square brackets and then employ the Wronskian identity in the form of  
${J_{\mu}(z) \over Y_{\mu}(z)} - {J_{\mu + 1}(z) \over Y_{\mu + 1}(z)} 
= -{2\over \pi z} {1\over Y_{\mu}(z)Y_{\mu + 1}(z)}$.
}
\begin{align*}
{\partial\over \partial\lambda}\Psi_{\mu}(a,b;\lambda)\big\vert_{\lambda=\lambda_n}
&= {1\over 2\lambda_n}\bigg\{
\sqrt{2\lambda_n a} \, Y_{\mu + 1}(\sqrt{2\lambda_n a})Y_{\mu}(\sqrt{2\lambda_n b}) 
\bigg[ {J_{\mu}(\sqrt{2\lambda_n b}) \over Y_{\mu}(\sqrt{2\lambda_n b})} - 
{J_{\mu + 1}(\sqrt{2\lambda_n a}) \over Y_{\mu + 1}(\sqrt{2\lambda_n a})} \bigg]
\\
&+ \sqrt{2\lambda_n b}\,Y_{\mu + 1}(\sqrt{2\lambda_n b}) Y_{\mu}(\sqrt{2\lambda_n a})
\bigg[ {J_{\mu + 1}(\sqrt{2\lambda_n b}) \over Y_{\mu + 1}(\sqrt{2\lambda_n b})} 
- {J_{\mu}(\sqrt{2\lambda_n a}) \over Y_{\mu}(\sqrt{2\lambda_n a})} \bigg]
\bigg\}
\\
&= {1\over \pi \lambda_n} \bigg[ {Y_{\mu}(\sqrt{2\lambda_n a}) \over Y_{\mu}(\sqrt{2\lambda_n b})} - 
{Y_{\mu}(\sqrt{2\lambda_n b}) \over Y_{\mu}(\sqrt{2\lambda_n a})} \bigg]
\equiv {1\over \pi \lambda_n} \bigg[ {J_{\mu}(\sqrt{2\lambda_n a}) \over J_{\mu}(\sqrt{2\lambda_n b})} - 
{J_{\mu}(\sqrt{2\lambda_n b}) \over J_{\mu}(\sqrt{2\lambda_n a})} \bigg]\,,
\end{align*}
i.e.,
\begin{align}\label{Delta_SQB_kill_ab}
\Delta(a,b;\lambda_n) &= {(ab)^{-\mu/2}\over 2\lambda_n}
\bigg[ {J_{\mu}(\sqrt{2\lambda_n a}) \over J_{\mu}(\sqrt{2\lambda_n b})} - 
{J_{\mu}(\sqrt{2\lambda_n b}) \over J_{\mu}(\sqrt{2\lambda_n a})} \bigg]\,.
\end{align}
Hence, from \eqref{FHT_eigenfunctions_ab} we have
\footnote
{The eigenvalue equation \eqref{eigen_SQB_kill_ab} gives  
$\Psi_{\mu}(b,x; \lambda_n) = {J_{\mu}(\sqrt{2\lambda_n b}) \over J_{\mu}(\sqrt{2\lambda_n a})}
\Psi_{\mu}(a,x; \lambda_n)$. Combining this identity with the antisymmetry 
$\Psi_{\mu}(x,y; \lambda) = - \Psi_{\mu}(y,x; \lambda)$ allows us to also re-express these functions in various ways.
}
\begin{eqnarray}\label{FHT_SQB_eigenfunctions_ab_plus}
\psi_n^+(x;a,b) = \lambda_n^{(a,b)}\pi\alpha_n(a,b)
\bigg({b\over x}\bigg)^{\mu\over 2}\Psi_{\mu}(x,a;\lambda_n)\,,
\\
\psi_n^-(x;a,b) =  \lambda_n^{(a,b)}\pi\alpha_n(a,b)
\bigg({a\over x}\bigg)^{\mu\over 2}\Psi_{\mu}(b,x;\lambda_n),
\label{FHT_SQB_eigenfunctions_ab_minus}
\end{eqnarray}
where throughout we define the coefficients $\alpha_n(a,b) := \bigg[ {J_{\mu}(\sqrt{2\lambda_n a}) \over J_{\mu}(\sqrt{2\lambda_n b})} - 
{J_{\mu}(\sqrt{2\lambda_n b}) \over J_{\mu}(\sqrt{2\lambda_n a})} \bigg]^{-1}$, $\lambda_n \equiv \lambda_n^{(a,b)}$.

Substituting the expressions in \eqref{FHT_SQB_eigenfunctions_ab_plus}--\eqref{FHT_SQB_eigenfunctions_ab_minus}, for the respective intervals $(a,\b)$ and $(\b,b)$, within 
\eqref{prop_last_time-CDF_discrete_killed_ab}--\eqref{joint_last_ab_doubly_defective_new}, with scale function in \eqref{SQB_scale_func}, gives us the explicit spectral series for the nonzero discrete parts of the marginal and joint distributions:
\begin{align}\label{last_time-CDF_discrete_killed_ab_SQB}
&\P_x(g^{(a,b)}_{\b}(T) = 0) 
\\
\nonumber
&= 
\begin{cases}
\displaystyle {\ln (\b/ x)\over \ln (\b/ a)} \ind_{\{\mu = 0\}} 
+ {(\b/ x)^{\mu} - 1 \over (\b/ a)^{\mu} - 1}\ind_{\{\mu \ne 0\}}  
+ \pi\bigg({\b\over x}\bigg)^{\mu\over 2}
\sum_{n=1}^\infty e^{- \lambda_n^{(a,\b)} T} \alpha_n(a,\b) \Psi_{\mu}(x,a;\lambda_n^{(a,\b)})&, x \in (a,\b),
\\
\displaystyle{\ln (\b/ x)\over \ln (\b/ b)} \ind_{\{\mu = 0\}} 
+ {(\b/ x)^{\mu} - 1 \over (\b/ b)^{\mu} - 1}\ind_{\{\mu \ne 0\}}  
+ \pi\bigg({\b\over x}\bigg)^{\mu\over 2}
\sum_{n=1}^\infty e^{- \lambda_n^{(\b,b)} T} \alpha_n(\b,b) \Psi_{\mu}(b,x;\lambda_n^{(\b,b)})&, x \in (\b,b),
\end{cases}
\end{align}
\begin{align}\label{joint_last_ab_doubly_defective_SQB}
&\P_x(g^{(a,b)}_{\b}(T) = 0, X_{(a,b),T} = \partial^\dagger) 
\\
\nonumber
&= \begin{cases}
\displaystyle {\ln (\b/ x)\over \ln (\b/ a)} \ind_{\{\mu = 0\}} 
+ {(\b/ x)^{\mu} - 1 \over (\b/ a)^{\mu} - 1}\ind_{\{\mu \ne 0\}}
 + \pi\bigg({a\over x}\bigg)^{\mu\over 2}\sum_{n=1}^\infty e^{- \lambda_n^{(a,\b)} T} \alpha_n(a,\b) 
\Psi_{\mu}(x,\b;\lambda_n^{(a,\b)})&, x \in (a,\b),
\\
\displaystyle{\ln (\b/ x)\over \ln (\b/ b)} \ind_{\{\mu = 0\}} 
+ {(\b/ x)^{\mu} - 1 \over (\b/ b)^{\mu} - 1}\ind_{\{\mu \ne 0\}}
 + \pi\bigg({b\over x}\bigg)^{\mu\over 2}\sum_{n=1}^\infty e^{- \lambda_n^{(\b,b)} T} \alpha_n(\b,b) 
\Psi_{\mu}(\b,x;\lambda_n^{(\b,b)})&, x \in (\b,b).
\end{cases}
\end{align}
Note: here and below, the respective eigenvalues solve $\Psi_{\mu}(a,\b;\lambda_n^{(a,\b)}) = 0$ and 
$\Psi_{\mu}(\b,b;\lambda_n^{(\b,b)}) = 0$.

Using the above Bessel cylinder function and \eqref{Delta_SQB_kill_ab} within \eqref{spectral_3_product_eigen} also gives the explicit spectral expansion for transition PDF of the SQB process on $(a,b)$ according to (\ref{u_spectral_3}):
\begin{eqnarray}
p_{(a,b)}(t;x,y) = \frac{1}{2}\bigg({y\over x}\bigg)^{{\mu\over 2}} 
\sum_{n=1}^\infty e^{-\lambda_n t} N_{n,(a,b)}^2\, \Psi_{\mu}(a,x;\lambda_n)\Psi_{\mu}(a,y;\lambda_n)\,,\,\,\,
x,y\in (a,b), \, t>0,
\label{trans_PDF_spectral_SQB_ab}
\end{eqnarray}
where $N^2_{n,(a,b)} := \pi^2 \lambda_n \alpha_n(a,b) 
{J_{\mu}(\sqrt{2\lambda_n b}) \over J_{\mu}(\sqrt{2\lambda_n a})} \equiv \pi^2 \lambda_n
\bigg[\bigg({J_{\mu}(\sqrt{2\lambda_n a}) \over J_{\mu}(\sqrt{2\lambda_n b})}\bigg)^2 - 1\bigg]^{-1}$, 
$\lambda_n = \lambda_n^{(a,b)}$, i.e., $N_{n,(a,b)}$ denotes the normalization constant for the eigenfunctions on 
$(a,b)$. Hence, using \eqref{trans_PDF_spectral_SQB_ab}, for time $t=T$, within 
\eqref{prop_joint_last_time_kill_ab-t-3-zero} on the respective intervals gives
\begin{align}\label{joint_last_ab_partly_defective_SQB}
&\P_x(g^{(a,b)}_{\b}(T) = 0, X_{(a,b),T} \in d z) 
\nonumber \\
&= \frac{1}{2}\bigg({z\over x}\bigg)^{{\mu\over 2}}
\begin{cases}
\sum_{n=1}^\infty e^{-\lambda_n^{(a,\b)} T} N_{n,(a,\b)}^2\, \Psi_{\mu}(a,x;\lambda_n^{(a,\b)})
\Psi_{\mu}(a,z;\lambda_n^{(a,\b)}) \,dz
&, x,z \in (a,\b),
\\
\sum_{n=1}^\infty e^{-\lambda_n^{(\b,b)} T} N_{n,(\b,b)}^2\, \Psi_{\mu}(\b,x;\lambda_n^{(\b,b)})
\Psi_{\mu}(\b,z;\lambda_n^{(\b,b)}) \,dz
&, x,z \in (\b,b).
\end{cases}
\end{align}

%
%

To employ Proposition~\ref{joint_last-passage-propn-time-t_ab-new-version} we also need  
$\hat\psi^+_{n}(a,\b)$ and $\hat\psi^-_{n}(\b,b)$. These follow directly by 
using \eqref{Delta_SQB_kill_ab} and the functions $\varphi^\pm_{-\lambda_n}$, evaluated at the respective arguments for $\lambda_n = \lambda_n^{(a,\b)}$ and $\lambda_n = \lambda_n^{(\b,b)}$, i.e., \eqref{psi_hat_derivatives} gives
\begin{eqnarray}\label{FHT_SQB_hat_plus_ak}
{\hat\psi^+_{n}(a,\b) \over \lambda_n} = \b^\mu \alpha_n(a,\b) 
{J_{\mu}(\sqrt{2\lambda_n a}) \over J_{\mu}(\sqrt{2\lambda_n \b})} 
\equiv \b^\mu \bigg[1 - \bigg({J_{\mu}(\sqrt{2\lambda_n \b}) \over J_{\mu}(\sqrt{2\lambda_n a})}\bigg)^2 \bigg]^{-1}
\,,\,\,\,\,\,\,\lambda_n = \lambda_n^{(a,\b)},
\\
{\hat\psi^-_{n}(\b,b) \over \lambda_n} = \b^\mu \alpha_n(\b,b) 
{J_{\mu}(\sqrt{2\lambda_n b}) \over J_{\mu}(\sqrt{2\lambda_n \b})}
\equiv \b^\mu \bigg[\bigg({J_{\mu}(\sqrt{2\lambda_n \b}) \over J_{\mu}(\sqrt{2\lambda_n b})}\bigg)^2 - 1\bigg]^{-1}
\,,\,\,\,\,\,\,\lambda_n = \lambda_n^{(\b,b)}.
\label{FHT_SQB_hat_minus_kb}
\end{eqnarray}
Combining \eqref{FHT_SQB_eigenfunctions_ab_plus} and \eqref{FHT_SQB_eigenfunctions_ab_minus} on the respective intervals $(a,\b)$ and $(\b,b)$, with \eqref{FHT_SQB_hat_plus_ak}, \eqref{FHT_SQB_hat_minus_kb}, and \eqref{trans_PDF_spectral_SQB_ab} for $y=\b$, within \eqref{joint_last_passage_pdf_explicit_ab_new} and \eqref{last_passage_pdf_discrete_spec} produces the explicit spectral series for the joint and marginal densities for 
$t\in (0,T)$, $x,\b\in (a,b)$:
\begin{align}\label{joint_last_passage_pdf_explicit_ab_SQB}
f_{\!g^{(a,b)}_{\b}(T), X_{(a,b),T}}\!(t,z;x) 
= \big({z\over \b}\big)^{\mu \over 2}\pi p_{(a,b)}(t;x,\b) \!
\left\{\!\!
\begin{array}{lr}
        \displaystyle  \sum_{n=1}^{\infty}e^{-\lambda_n^{(a,\b)} (T-t)} \lambda_n^{(a,\b)} \alpha_n(a,\b) 
\Psi_{\mu}(z,a;\lambda_n^{(a,\b)}), &\!\!\!\!\!\!  z\in (a,\b), 
\\
          \displaystyle  \sum_{n=1}^{\infty}e^{-\lambda_n^{(\b,b)} (T-t)}\lambda_n^{(\b,b)} \alpha_n(\b,b) 
\Psi_{\mu}(b,z;\lambda_n^{(\b,b)}), &\!\!\!\!\!\!\!  z\in (\b,b), 
\end{array}
\right.
\end{align}
\begin{align}\label{last_passage_pdf_explicit_ab_SQB}
f_{g^{(a,b)}_{\b}(T)}(t;x) = 2p_{(a,b)}(t;x,\b) 
\bigg[ \b^{-\mu}\mathcal{S}(a,b;\b) 
&+ 
\sum_{n=1}^{\infty}
e^{-\lambda_n^{(a,\b)} (T-t)}
\alpha_n(a,\b) 
{J_{\mu}\big(\scriptstyle{\sqrt{2\lambda_n^{(a,\b)} a}}\big) 
\over J_{\mu}\big(\scriptstyle{\sqrt{2\lambda_n^{(a,\b)} \b}}\big)}
\nonumber \\
&+  
\sum_{n=1}^{\infty}
e^{-\lambda_n^{(\b,b)} (T-t)}  
\alpha_n(\b,b) 
{J_{\mu}\big(\scriptstyle{\sqrt{2\lambda_n^{(\b,b)} b}}\big) 
\over J_{\mu}\big(\scriptstyle{\sqrt{2\lambda_n^{(\b,b)} \b}}\big)}
\bigg],
\end{align}
where $\mathcal{S}(a,b;\b) = {\ln(b/a) \over \ln(\b/a)\ln(b/\b)} \ind_{\{\mu=0\}} 
+ \mu {b^\mu - a^\mu \over (1 - (a/\b)^\mu)((b/\b)^\mu - 1)}\ind_{\{\mu\ne 0\}}$. 
Using the series in \eqref{trans_PDF_spectral_SQB_ab} within the above expressions produces a double series representation for the marginal and joint densities.

%
%
\subsection{Squared Bessel Process killed at one interior point}\label{subsect_SQB_kill_b}
Consider the SQB process with imposed killing at $b\in (0,\infty)$. 
The two cases are: (i) $x,\b\in \I_b^- \equiv (0,b)$, or $\I_b^- \equiv [0,b)$
if the left endpoint $l=0$ is specified as reflecting, and (ii) $x,\b\in \I_b^+ \equiv (b,\infty)$. 
%
%

For case (i) the transition PDF $p_b^-$ is given by \eqref{u_spectral_1}. 
By the Bessel function identities used to derive \eqref{SQB_Psi_plus_func}, and the Wronskian property 
$J_\nu(z)Y_{\nu+1}(z) - J_{\nu+1}(z)Y_{\nu}(z) = -{2\over \pi x}$ for $z=\sqrt{2\lambda_n b} = j_{\nu,n}$, where 
$J_\nu(j_{\nu,n}) = 0$, we have the explicit product eigenfunction in \eqref{spectral_1_product_eigen}, i.e., we recover the spectral series:
\begin{eqnarray}
p_b^-(t;x,y) = \frac{1}{b}\bigg({y\over x}\bigg)^{\mu \over 2} 
\sum_{n=1}^\infty \frac{e^{- {1\over 2b}j_{n,\nu}^2 t}}
{J^2_{\nu + 1}(j_{\nu,n})}
J_{\nu}\big(j_{\nu,n}\sqrt{x/ b}\big) J_{\nu}\big(j_{\nu,n}\sqrt{y/ b}\big),\,\,\,\,\,x,y\in (0,b), \, t>0.
\label{trans_PDF_spectral_SQB_b}
\end{eqnarray}
Note that, throughout, $\nu$ is specified as in Section \ref{subsect_SQB}.

%
%
For case (ii) we have Spectral Category II with absolutely continuous eigenspectrum $(0,\infty)$. The Green function in \eqref{greenfunc_down} takes the form
$
G_b^+(\lambda;x,y) = (y/x)^{\mu  \over 2}\varphi_\mu(b,x\wedge y;\lambda)
{K_\mu(\sqrt{2\lambda (x\vee y)}) \over  K_\mu(\sqrt{2\lambda b})},
$
$\varphi_\mu(b,x\wedge y;\lambda) := K_{\mu}\big(\sqrt{2\lambda b}\,\big) 
I_{\mu}\big(\sqrt{2\lambda (x\wedge y)}\,\big) - I_{\mu}\big(\sqrt{2\lambda b}\,\big) 
K_{\mu}\big(\sqrt{2\lambda (x\wedge y)}\,\big)$, $\mu\in\R$, $x,y \in (b,\infty)$, $\lambda\in\C$. Hence, 
it is analytic in $\lambda$ except along a branch cut with branch point $\lambda=0$. 
Note that $\varphi_\mu(b,x\wedge y;\lambda)$ has no jump discontinuity along any branch cut about  
$\lambda=0$, i.e., 
$\varphi_\mu(b,x\wedge y;\epsilon e^{-i \pi}) = \varphi_\mu(b,x\wedge y;\epsilon e^{i \pi}) 
= \varphi_\mu(b,x\wedge y; -\epsilon) = \Psi_\mu(b,x\wedge y;\epsilon)$, 
for real $\epsilon > 0$. Hence, using the same steps as in the derivation of \eqref{SQB_imaginary_ratio} 
gives 
\begin{align*}
&\text{Im}\,G_b^+ (\epsilon e^{-i\pi};x,y) 
= \bigg({y\over x}\bigg)^{\mu  \over 2}\varphi_\mu(b,x\wedge y;-\epsilon )
{1\over 2 i}\bigg[{K_{\mu}(-i\sqrt{2\epsilon (x\vee y)}) \over K_\mu(-i\sqrt{2\epsilon b})} 
-{K_{\mu}(i\sqrt{2\epsilon (x\vee y)}) \over K_\mu(i\sqrt{2\epsilon b})} \bigg]
\\
&= {\pi \over 2}\bigg({y\over x}\bigg)^{\mu  \over 2}
{\Psi_\mu(b,x\wedge y);\epsilon) \Psi_\mu(b,x\vee y;\epsilon)
\over J^2_\mu(\sqrt{2\epsilon b}) + Y^2_\mu(\sqrt{2\epsilon b})}
= {\pi \over 2}\bigg({y\over x}\bigg)^{\mu  \over 2}
\Psi_\mu(b,x;\epsilon) \Psi_\mu(b,y;\epsilon) \rho(\epsilon;b) \,.
\end{align*}
Substituting this expression within (\ref{u_spectral_5}), with only the integral term as nonzero, gives the purely continuous spectral expansion for the transition density:
\begin{eqnarray}
p^+_{b}(t;x,y) = {1\over 2}\bigg({y\over x}\bigg)^{\mu \over 2}
\int_0^\infty 
e^{-\epsilon t}\Psi_\mu(b,x;\epsilon) \Psi_\mu(b,y;\epsilon) \rho(\epsilon;b)\,d \epsilon\,,\,\,x,y\in(b,\infty), \,t>0.
\label{trans_PDF_spectral_SQB_b_down}
\end{eqnarray}

For case (i) the defective portions are given by \eqref{prop_joint_last_b-below_pmf}--\eqref{doubly_defective_kill_up_b}. 
By making use of \eqref{SQB_FHT_up_all_tail} and 
\eqref{last_time-CDF_discrete_killed_ab_SQB} for $x >\b$, within \eqref{last_time-CDF_discrete_killed_below_b} gives
\begin{align}\label{last_time-CDF_discrete_killed_below_b_SQB}
&\P_x(g^b_{\b}(T) = 0) 
\\
\nonumber 
&=\begin{cases}
\displaystyle [1 - (\b/x)^\mu] \cdot \ind_{E_0} + 
2\big({\b\over x}\big)^{\mu \over 2}\sum_{n=1}^\infty e^{- {1\over 2\b}j_{n,\nu}^2T} {J_\nu(j_{n,\nu} \sqrt{x/\b}) \over
j_{n,\nu} \, J_{\nu + 1}(j_{n,\nu})},\,\,x \in (0,\b),
\\
\displaystyle{\ln (\b/ x)\over \ln (\b/ b)} \ind_{\{\mu = 0\}} 
+ {(\b/ x)^{\mu} - 1 \over (\b/ b)^{\mu} - 1}\ind_{\{\mu \ne 0\}}  
+ \pi\bigg({\b\over x}\bigg)^{\mu\over 2}
\sum_{n=1}^\infty e^{- \lambda_n^{(\b,b)} T} \alpha_n(\b,b) \Psi_{\mu}(b,x;\lambda_n^{(\b,b)}), \,\, x \in (\b,b).
\end{cases}
\end{align}
Using the transition PDF in \eqref{trans_PDF_spectral_SQB_b}, with $t,b,y$ replaced by $T,\b,z$, and \eqref{trans_PDF_spectral_SQB_ab}, with $t,a,y$ replaced by $T,\b,z$, i.e., \eqref{joint_last_ab_partly_defective_SQB} for $x \in (\b,b)$, within \eqref{prop_joint_last_b-below_pmf} gives the partly discrete joint distribution
\begin{align}
\P_x(g^{b}_{\b}(T) = 0, X_{b,T} \in dz) 
= \!\bigg({z\over x}\bigg)^{\!\!{\mu \over 2}} \!
\begin{cases} 
\displaystyle\frac{1}{\b} 
\sum_{n=1}^\infty \frac{e^{- {1\over 2\b}j_{n,\nu}^2 T}}
{J^2_{\nu + 1}(j_{\nu,n})}
J_{\nu}\big(j_{\nu,n}\sqrt{x/ \b}\big) J_{\nu}\big(j_{\nu,n}\sqrt{z/ \b}\big)\, dz , \,\,\,\,\,\, x,z \in (0,\b),
\\
\displaystyle\frac{1}{2} 
\sum_{n=1}^\infty e^{-\lambda_n^{(\b,b)} T} N_{n,(\b,b)}^2\, \Psi_{\mu}(\b,x;\lambda_n^{(\b,b)})
\Psi_{\mu}(\b,z;\lambda_n^{(\b,b)})
\, dz , x,z \in (\b,b).
\label{prop_joint_last_b-below_pmf_SQB}
\end{cases}
\end{align}
The jointly discrete distribution is given by \eqref{doubly_defective_kill_up_b}. 
Hence, for $x < \b$ we have $\P_x(g^{b}_{\b}(T) = 0, X_{b,T} = \partial^\dagger)$ is nonzero only if $l=0$ is nonconservative where 
$\mu \le -1$ or $\mu\in(-1,0)$ and $l=0$ is specified as killing, i.e., $\mu < 0$ and $\nu = \vert \mu \vert$ with $\P_x(g^{b}_{\b}(T) = 0, X_{b,T} = \partial^\dagger)=\P_x(\T^-_0 (\b) \le T)$ given by \eqref{joint_last_doubly_defective_formula_SQB}. 
For $x \in (\b,b)$ we have $\P_x(g^{b}_{\b}(T) = 0, X_{b,T} = \partial^\dagger) = \P_x(\T^+_b (\b) \le T)$ which is given by the second expression in \eqref{joint_last_ab_doubly_defective_SQB}.

For case (ii) the defective portions are given by \eqref{prop_joint_last_b-above_pmf}--\eqref{doubly_defective_kill_down_b}. Using the time-$T$ transition PDFs in \eqref{trans_PDF_spectral_SQB_ab} and \eqref{trans_PDF_spectral_SQB_b_down} on the appropriate intervals within \eqref{prop_joint_last_b-above_pmf} gives
\begin{align}
\P_x(g^{b}_{\b}(T) = 0, X_{b,T} \in dz) 
= \frac{1}{2} \!\bigg({z\over x}\bigg)^{\!\!{\mu \over 2}} \!
\!\begin{cases} 
\displaystyle \sum_{n=1}^\infty e^{-\lambda_n^{(b,\b)} T} N_{n,(b,\b)}^2\, \Psi_{\mu}(b,x;\lambda_n^{(b,\b)})
\Psi_{\mu}(b,z;\lambda_n^{(b,\b)}) 
\, dz , \, x,z \in (b,\b),
\\
\,\,\,\,\displaystyle \int_0^\infty e^{-\epsilon T}\Psi_\mu(\b,x;\epsilon) \Psi_\mu(\b,z;\epsilon) \rho(\epsilon;\b)\,d \epsilon
\, dz , \,\,\,\,\,\, x,z \in (\b,\infty).
\label{prop_joint_last_b-above_pmf_SQB}
\end{cases}
\end{align}
Using \eqref{SQB_FHT_down_all_tail}, and \eqref{last_time-CDF_discrete_killed_ab_SQB} with $a$ replaced by 
$b$, within \eqref{prop_last_time-CDF_discrete_killed_above_b} gives
\begin{align}\label{prop_last_time-CDF_discrete_killed_above_b_SQB}
&\P_x(g^b_{\b}(T) = 0)  
\\
\nonumber
&=\begin{cases}
\displaystyle {\ln (\b/ x)\over \ln (\b/ b)} \ind_{\{\mu = 0\}} 
+ {(\b/ x)^{\mu} - 1 \over (\b/ b)^{\mu} - 1}\ind_{\{\mu \ne 0\}}  
+ \pi\bigg({\b\over x}\bigg)^{\mu\over 2}
\sum_{n=1}^\infty e^{- \lambda_n^{(b,\b)} T} \alpha_n(b,\b) \Psi_{\mu}(x,b;\lambda_n^{(b,\b)})&, x \in (b,\b),
\\
\displaystyle [1 - (\b/x)^\mu]\! \cdot\! \ind_{\{\mu > 0\}} + \bigg({\b\over x}\bigg)^{\mu \over 2}{1\over \pi}
\!\int_0^\infty {e^{-\epsilon T} \over \epsilon}\Psi_\mu(\b,x;\epsilon) \rho(\epsilon;\b) 
d \epsilon &, x \in (\b,\infty).
\end{cases}
\end{align}
The jointly discrete distribution is given by \eqref{doubly_defective_kill_down_b}. 
For $x\in (\b,\infty)$, $\P_x(g^{b}_{\b}(T) = 0, X_{b,T} = \partial^\dagger)=0$ since $r=\infty$ is a natural (conservative) boundary. For $x \in (b,\b)$, $\P_x(g^{b}_{\b}(T) = 0, X_{b,T} = \partial^\dagger) = 
\P_x(\T^-_b (\b) \le T)$, is given by the first expression in \eqref{joint_last_ab_doubly_defective_SQB} 
with $a$ replaced by $b$, i.e.,
\begin{align}\label{joint_last_b_doubly_defective_SQB}
&\P_x(g^{b}_{\b}(T) = 0, X_{b,T} = \partial^\dagger) 
\\
\nonumber
&= 
\displaystyle {\ln (\b/ x)\over \ln (\b/ b)} \ind_{\{\mu = 0\}} 
+ {(\b/ x)^{\mu} - 1 \over (\b/ b)^{\mu} - 1}\ind_{\{\mu \ne 0\}}
 + \pi\bigg({b\over x}\bigg)^{\mu\over 2}\sum_{n=1}^\infty e^{- \lambda_n^{(b,\b)} T} \alpha_n(b,\b) 
\Psi_{\mu}(x,\b;\lambda_n^{(b,\b)}),\,x \in (b,\b).
\end{align}

%
%

We now apply Proposition~\ref{marginal_joint_last-passage-propn-kill-b} by directly using \eqref{eta_plus_k_SQB},\eqref{eta_minus_k_SQB}, \eqref{eta_plus_z_k_SQB}, \eqref{eta_minus_z_k_SQB}, and \eqref{SQB_scale_func}, \eqref{FHT_SQB_eigenfunctions_ab_plus}, \eqref{FHT_SQB_eigenfunctions_ab_minus}, \eqref{FHT_SQB_hat_plus_ak}, \eqref{FHT_SQB_hat_minus_kb} on the appropriate intervals $(\b,b)$ and $(b,\b)$. In particular, combining \eqref{marginal_spectral_kill-b_1} and \eqref{marginal_spectral_kill-b_2} gives the marginal density
\begin{align}\label{marginal_spectral_kill-b_SQB}
&f_{g^{b}_\b(T)}(t;x) 
\nonumber \\
\!\!\!&= \!\begin{cases}
\displaystyle \! 2p^-_{b}(t;x,\b) 
\bigg[ \! R_-(b,\b) + \sum_{n=1}^\infty e^{- {1\over 2\b}j_{n,\nu}^2(T-t)} 
+ \sum_{n=1}^\infty e^{- \lambda_n^{(\b,b)} (T-t)} \alpha_n(\b,b) 
{J_{\mu}\big(\scriptstyle{\sqrt{2\lambda_n^{(\b,b)} b}}\big) \over J_{\mu}\big(\scriptstyle{\sqrt{2\lambda_n^{(\b,b)} \b}}\big)}
 \bigg] , \,\,\,\,\,\,\,\,\,\,\,\,\,\,x,\b \in (0,b),
\\
\displaystyle \! 2p^+_{b}(t;x,\b) 
\bigg[ \! R_+(b,\b) \!+\! {1\over \pi^2}\!\!\int_0^\infty \!\!e^{-\epsilon (T-t)} \rho(\epsilon;\b) {d \epsilon \over \epsilon}
+ \sum_{n=1}^\infty e^{- \lambda_n^{(b,\b)} (T-t)} \alpha_n(b,\b) 
{J_{\mu}\big(\scriptstyle{\sqrt{2\lambda_n^{(b,\b)} b}}\big) \over J_{\mu}\big(\scriptstyle{\sqrt{2\lambda_n^{(b,\b)} \b}}\big)}
 \bigg]\! , x,\!\b \!\in\! (b,\infty),
\end{cases}
\end{align}
where $ R_-(b,\b) := {\b^{-\mu} \over \ln (b/ \b)} \ind_{\{\mu = 0\}} 
+ {\mu \over 1 -  (\b/b)^\mu} \ind_{\{\mu \ne 0\}} - \mu \ind_{E_0}$, 
$R_+(b,\b) := {\b^{-\mu} \over \ln (\b/ b)} \ind_{\{\mu = 0\}} 
+ {\mu \over (\b/b)^\mu - 1} \ind_{\{\mu \ne 0\}} + \mu \ind_{\{\mu>0\}}$, and \eqref{joint_last_passage_pdf_spectral_b_1} and \eqref{joint_last_passage_pdf_spectral_b_2} gives the joint density
\begin{align}\label{joint_last_passage_pdf_b_1_SQB}
f_{g^{b}_{\b}(T), X_{b,T}}(t,z;x) 
= p^-_{b}(t;x,\b) \bigg({z\over \b}\bigg)^{\!\mu\over 2}\!
\begin{cases}
\displaystyle {1\over \b} \sum_{n=1}^\infty e^{- {1\over 2\b}j_{n,\nu}^2(T-t)} j_{n,\nu} 
{J_\nu(j_{n,\nu}\sqrt{z/\b}) \over J_{\nu + 1}(j_{n,\nu})}
&, z \in (0,\b),
\\
\displaystyle \pi \sum_{n=1}^\infty e^{- \lambda_n^{(\b,b)} (T-t)} \lambda_n^{(\b,b)} \alpha_n(\b,b) 
\Psi_\mu(b,z;\lambda_n^{(\b,b)})
&, z \in (\b,b),
\end{cases}
\end{align}
for $x,\b\in(0,b)$, and 
\begin{align}\label{joint_last_passage_pdf_b_2_SQB}
f_{g^{b}_{\b}(T), X_{b,T}}(t,z;x) 
= p^+_{b}(t;x,\b) \bigg({z\over \b}\bigg)^{\!\mu\over 2}\!
\begin{cases}
\displaystyle \pi \sum_{n=1}^\infty e^{- \lambda_n^{(b,\b)} (T-t)} \lambda_n^{(b,\b)} \alpha_n(b,\b) 
\Psi_\mu(z,b;\lambda_n^{(b,\b)})
&, z \in (b,\b),
\\
\displaystyle {1\over \pi} \!\int_0^\infty e^{-\epsilon (T-t)} \Psi_\mu(\b,z;\epsilon) \rho(\epsilon;\b) d \epsilon
&, z \in (\b,\infty),
\end{cases}
\end{align}
for $x,\b\in(b,\infty)$, $t\in (0,T)$. In \eqref{marginal_spectral_kill-b_SQB}--\eqref{joint_last_passage_pdf_b_2_SQB}, 
$p^\pm_{b}(t;x,\b)$ are given by \eqref{trans_PDF_spectral_SQB_b} and \eqref{trans_PDF_spectral_SQB_b_down} with $y=\b$.

%
%
%
%
\subsection{Squared Radial Ornstein-Uhlenbeck (or CIR) Process}\label{subsect_CIR}
The squared radial Ornstein-Uhlenbeck process $\{X_t, t \geq 0\} \in (0,\infty)$ has the generator 
\begin{eqnarray}\label{radial OU gen}
\mathcal{G} f(x):= 2xf''(x) + 2(\mu + 1 - \kappa x)f'(x),\,\, x\in (0,\infty),
\end{eqnarray}
with $\kappa \ne 0$. This is a two-parameter family of diffusions with scale and speed densities  
$\s(x) = x^{-1-\mu}e^{\kappa x}$ and $\m(x) = {1\over 2}x^\mu e^{-\kappa x}$. 
As is well known, both endpoints are NONOSC. The left endpoint $l=0$ is entrance-not-exit if $\mu \ge 0$, regular if $\mu\in (-1,0)$, and exit-not-entrance if $\mu \le -1$. The right endpoint $r=\infty$ is natural (attracting only for $\kappa < 0$). 

We remark that the CIR (Cox-Ingerssol-Ross), or Feller, process formally obeys the SDE:
$$dX_t = (\gamma_0 - \gamma_1 X_t)dt + \nu_0\sqrt{X_t}dW_t, \,\, \gamma_0,\gamma_1 \in \R,\,\gamma_1 \ne 0, \nu_0 > 0.$$ 
The scale and speed densities are $\s(x) = x^{-1-\mu}e^{\kappa x}$ and $\m(x) = {2\over \nu_0^2}x^\mu e^{-\kappa x} \equiv {\kappa\over \gamma_1}x^\mu e^{-\kappa x}$, where $\mu := {2\gamma_0\over \nu_0^2} - 1$, $\kappa := {2\gamma_1\over \nu_0^2}$. 
Hence, the squared radial Ornstein-Uhlenbeck process is a special (standardized) case of the CIR process with volatility parameter choice $\nu_0 = 2$, and where $\gamma_0 =2( \mu + 1)$, 
$\gamma_1 = 2\kappa$ . The generator of the CIR process is that of the squared radial Ornstein-Uhlenbeck  process multiplied by ${\nu_0^2\over 4}$:
$$\mathcal G f(x) := {1\over 2}\nu_0^2 [xf''(x) + (\mu + 1 - \kappa x)f'(x)] \equiv 
{\nu_0^2\over 2} xf''(x) + (\gamma_0 - \gamma_1 x)f'(x), \,x \in (0,\infty).$$ Hence, it follows trivially that a transition PDF of the CIR process $p^{\text{CIR}}(t;x,y) = p({\nu_0^2 \over 4}t;x,y)$, 
where $p$ is the transition PDF of the corresponding squared radial OU process.
\footnote{The well-known constant-elasticity-of-variance (CEV) process $\{F_t\}_{t\ge 0}$ 
with generator $\mathcal{G}^{F} f(y):= {1\over 2}\delta^2 y^{2(\beta + 1)} f''(y) + \nu y f'(y)$, $y\in (0,\infty)$, 
$\delta > 0, \beta \ne 0, \nu \ne 0$, i.e., satisfying the SDE $d F_t = \nu F_t dt + \delta F_t^{\beta + 1}d W_t$, 
arises directly from the squared raidial OU process via a smooth monotonic mapping: 
$F_t = {\sf F}(X_t)$, ${\sf F}(x) := (\delta^2\beta^2 x)^{-\mu}$, $x\in (0,\infty)$, where we set $\mu = {1\over 2\beta}$, $\kappa = \nu\beta$. The unique inverse mapping is ${\sf X}(y) = (\delta^2\beta^2)^{-1}y^{-2\beta}$.
}

In what follows we assume the case where $\kappa < 0$ and $\mu < 0$ and with $l=0$ specified as killing 
for $\mu\in(-1,0)$. We note that the other cases are handled in a similar fashion while using different appropriate fundamental solutions (e.g., see \cite{BS02,CM21}).  
A pair of fundamental solutions to \eqref{eq:phi}, with $\mathcal{G}$ in \eqref{radial OU gen}, and satisfying the above assumptions, is given by
\begin{equation}\label{CIR_fundamental_funcs}
    \varphi_\lambda^+(x) := x^{-\mu} e^{\kappa x} M\big(\frac{\lambda}{2|\kappa|} +1, 1-\mu,|\kappa|x \big)\,;
 \quad \varphi_\lambda^-(x) :=  x^{-\mu} e^{\kappa x} U\big(\frac{\lambda}{2|\kappa|} +1, 1-\mu,|\kappa|x \big),
\end{equation} 
with Wronskian factor $w_\lambda = \vert\kappa\vert^{\mu} \frac{\Gamma(1 - \mu) }
{\Gamma(\frac{\lambda}{2|\kappa|} +1)}$. 
The functions $M(\alpha,\beta,z)$ and $U(\alpha,\beta,z)$ are the standard confluent hypergeometric (Kummer) functions of the first and second kind, e.g., see \cite{AS72}. We note that the pair $\varphi_\lambda^\pm(x)$ are entire functions in $\lambda$.
Applying \eqref{u_spectral_4}--\eqref{eigenfunc_regular}, where $w_{-\lambda_n} = 0 \implies 
\Gamma(-\frac{\lambda_n}{2|\kappa|} +1) = \infty \implies \lambda_n = 2|\kappa|n, n=1,\ldots$, and evaluating the coefficients in \eqref{eigenfunc_regular} while using 
$M(-n,1 + \vert\mu\vert,z) = {n!\Gamma(1+ \vert\mu\vert)\over \Gamma(1+ \vert\mu\vert + n)}
L_n^{(\vert\mu\vert)}(z)$ and $U(-n,1 + \vert\mu\vert,z) = (-1)^n n!L_n^{(\vert\mu\vert)}(z)$, gives the known spectral expansion for the transition PDF,
\begin{align}
p(t;x,y) &= \vert\kappa\vert^{1+\vert\mu\vert} x^{-\mu}e^{\kappa x}
e^{- 2\vert\kappa\vert t}
\sum_{m=0}^\infty  {(e^{-2\vert\kappa\vert t})^m m!\over \Gamma(1+\vert\mu\vert + m)}
L_{m}^{(\vert\mu\vert)}(\vert\kappa\vert x)L_{m}^{(\vert\mu\vert)}(\vert\kappa\vert y)
\nonumber 
\\
&= {\kappa e^{\kappa x + (1+\mu)\kappa t}(y/x)^{\mu \over 2} \over 2\sinh(\kappa t)}
 \exp\left(-{\kappa e^{\kappa t}(x + y) \over 2\sinh(\kappa t)}\right)
I_{\vert\mu\vert}\left({\kappa\sqrt{xy}\over \sinh(\kappa t)}\right)\,,
\label{trans_PDF_spectral_CIR_4}
\end{align}
$x,y\in(0,\infty)$, $t>0$. 
Here, $L_m^{(\alpha)}(z)$ are the associated (or generalized) Laguerre polynomials of integer order $m$ and parameter $\alpha\in\R$, e.g., see \cite{AS72}. The second closed-form expression follows by a direct application of the Hille-Hardy summation formula.
\footnote{$\sum\limits_{m=0}^\infty  {t^m \, m! \over \Gamma(1+\alpha + m)}
L_{m}^{(\alpha)}(x)L_{m}^{(\alpha)}(y)
= {e^{-(x+y)t/(1-t)} \over (xyt)^{\alpha/2}(1-t)}I_\alpha\bigg({2\sqrt{xyt} \over 1-t} \bigg)$ 
is valid for $\vert t \vert < 1$ and $\alpha > -1$}

The discrete part of the distribution of $g_\b(T)$, $\b\in (0,\infty)$, is given by \eqref{prop_last_time-t-1-limit}.  
Both endpoints are NONOSC, i.e., Proposition~\ref{prop_spec_first_hit_1} applies. 
The eigenvalues $\lambda_n \equiv \lambda_{n,\b}^-$, $n\ge 1$, are the positive simple zeros solving 
$\varphi^+_{-\lambda_n}(\b) = 0$, i.e.,
\begin{equation}\label{eigenvalues_CIR_4}
M\big(\!-\!\frac{\lambda_n}{2\vert\kappa\vert} + 1, 1 - \mu,\vert\kappa\vert \b\big) = 0,
\end{equation}
and the eigenvalues $\lambda_n \equiv \lambda_{n,\b}^+$, $n\ge 1$, are the positive simple zeros solving 
$\varphi^-_{-\lambda_n}(\b) = 0$, i.e.,
\footnote{All eigenvalues in \eqref{eigenvalues_CIR_4} are readily computed by using a root finding (e.g., bisection) algorithm. We can also use the leading asymptotic 
$M\big(-\nu,\beta,z) \sim \Gamma(\beta)\pi^{-{1\over 2}}e^{z/2}
(({\beta \over 2} + \nu)z)^{{1\over 4} - {\beta\over 2}}
\cos( \sqrt{(2\beta + 4\nu)z} - {\pi \over 2}\beta + {\pi\over 4})$, as $\nu \to \infty$, for $\beta$, $z\in\R$. 
By setting $\nu = \frac{\lambda_n}{2\vert\kappa\vert} - 1$, $\beta = 1 - \mu$, $z= \vert\kappa\vert k$, in the cosine argument, and equating it to $(n-{1\over 2})\pi$, gives a simple expression for the initial estimates of the eigenvalues. 
Observe that the $\lambda_n$ grow roughly in proportion to $n^2$. By the leading asymptotic 
$U\big(-\nu,\beta,z) \sim \Gamma({1\over 2}\beta + \nu + {1\over 4})
\pi^{-{1\over 2}} e^{z/2}z^{{1\over 4} - {\beta\over 2}}\cos( \sqrt{(2\beta + 4\nu)z} - {\pi \over 2}\beta - \nu\pi + {\pi\over 4})$, as $\nu \to \infty$, for $\beta$, $z\in\R$, the eigenvalues in \eqref{eigenvalues_CIR_5} can be computed in similar fashion. Again, using the same assignment of the parameters leads to an initial estimate of the eigenvalues. In this case the $\lambda_n$ approach a linear growth in $n$ for large $n$.
}
\begin{equation}\label{eigenvalues_CIR_5}
 U\big(\!-\!\frac{\lambda_n}{2\vert\kappa\vert} + 1, 1 - \mu,\vert\kappa\vert \b\big) = 0.
\end{equation}
Using \eqref{CIR_fundamental_funcs} within \eqref{FHT_eigenfunctions_1} and \eqref{FHT_eigenfunctions_2} gives
\begin{eqnarray}\label{CIR_Psi_plus_func}
\psi_n^+(x;\b) = 2\vert\kappa\vert \bigg({\b\over x}\bigg)^{\mu}e^{\kappa(x-\b)}
\frac{M(-\frac{\lambda_{n,k}^-}{2|\kappa|} + 1,1 - \mu,|\kappa| x)}
{M_1(-\frac{\lambda_{n,k}^-}{2|\kappa|} + 1,1 - \mu,|\kappa| \b)},
\\
\psi_n^-(x;\b) = 2\vert\kappa\vert \bigg({\b\over x}\bigg)^{\mu}e^{\kappa(x-\b)}
\frac{U(-\frac{\lambda_{n,k}^+}{2|\kappa|} + 1,1 - \mu,|\kappa| x)}
{U_1(-\frac{\lambda_{n,k}^+}{2|\kappa|} + 1,1 - \mu,|\kappa| \b)}.
\label{CIR_Psi_minus_func}
\end{eqnarray} 
Throughout, we denote the derivative w.r.t. the first argument of the Kummer functions as 
$M_1\big(\alpha,\beta,z) := {\partial\over \partial\alpha}M\big(\alpha,\beta,z)$ and $U_1\big(\alpha,\beta,z) :=  {\partial\over \partial\alpha}U\big(\alpha,\beta,z)$. 
In practice, these derivatives are most efficiently and accurately computed by a simple numerical differentiation, e.g., 
$M_1\big(\alpha,\beta,z) \approx (M\big(\alpha + \delta,\beta,z) - M\big(\alpha,\beta,z))/\delta$, $\delta \simeq 10^{-6}$, and similarly for $U_1$. 
Alternatively, we can use the power series
$M(\alpha,\beta,z) = \sum_{n=0}^\infty \frac{(\alpha)_n z^n}{(\beta)_n n!}$,  
where $(\alpha)_n = \Gamma(\alpha+n)/\Gamma(\alpha) = \alpha(\alpha+1)\ldots (\alpha+n-1) ,\quad (\alpha)_0=1$, 
$(\alpha)_n $, is Pochhammer's symbol. Differentiating, termise w.r.t. 
$\alpha$, gives
\begin{align}
	\label{Kummer_M_der}
	M_1(\alpha,\beta,z) = \sum_{n=0}^\infty \frac{(\alpha)_n}{(\beta)_n n!} \Psi(\alpha+n)z^n - \Psi(\alpha)M(\alpha,\beta,z),
\end{align}
where $\Psi(x) := \Gamma '(x)/\Gamma(x)$ is the digamma function. By using the formal definition of $U$,
\begin{align}\label{Kummer_U_defn}
U(\alpha,\beta,z) = {\pi \over \sin(\pi \beta)}\bigg[
{M(\alpha,\beta,z) \over\Gamma(1+\alpha - \beta) \Gamma(\beta)} 
- z^{1-\beta}{M(1+\alpha - \beta,2-\beta,z) \over\Gamma(\alpha) \Gamma(2-\beta)} \bigg],
\end{align}
and differentiating w.r.t. $\alpha$ gives an alternative formula for computing $U_1$ in terms of $M$ and $M_1$:
\begin{eqnarray}\label{Kummer_U_der}
	&U_1(\alpha,\beta,z) = {\pi \over \sin(\pi \beta)}
\bigg\{
{1 \over\Gamma(1+\alpha - \beta) \Gamma(\beta)}\bigg[ M_1(\alpha,\beta,z) - \Psi(1+\alpha - \beta)M(\alpha,\beta,z) \bigg]
\nonumber \\
&\,\,\,\,\,\,\,\,\,- {z^{1-\beta} \over \Gamma(\alpha) \Gamma(2-\beta)}
\bigg[ M_1(1+\alpha - \beta,2-\beta,z) - \Psi(\alpha)M(1+\alpha - \beta,2-\beta,z) \bigg]
 \bigg\}.
\end{eqnarray}

From the leading asymptotic term, as $\nu \to \infty$, we see that $U(-\nu,\beta,z)$, as a function of $\nu > 0$, 
for real $\beta,z$, oscillates between negative and positive values and its amplitude grows very rapidly in proportion to 
$\Gamma({1\over 2}\beta + \nu + {1\over 4})$. Hence, for large values of $\nu$, a numerical evaluation of 
$U(-\nu,\beta,z)$ that uses a standard numerical library routine can fail due to overflow. This can arise when numerically implementing any of the spectral expansions involving the Kummer $U$ function with large values of 
$\nu =\frac{\lambda_{n,k}^+}{2|\kappa|} - 1$, i.e., for relatively larger eigenvalues. We note that such an overflow is artificial since all relevant expressions, e.g., the ratio of $U$ and $U_1$ in \eqref{CIR_Psi_minus_func}, are well-behaved. 
Hence, to remove such numerical artifacts, in what follows we introduce a {\it rescaled} Kummer $U$ function defined as 
$\widetilde{U}(-\nu,\beta,z) := {U(-\nu,\beta,z) \over \Gamma(\nu + 1)}$, i.e.,
\begin{align}\label{rescaled_U}
&\widetilde{U}(-\nu,\beta,z) = {1\over \sin(\pi \beta)}
\Bigg[
{\sin(\pi (\nu + \beta)) \over \Gamma(\beta)}{\Gamma(\nu + \beta) \over \Gamma(\nu + 1)} M(-\nu,\beta,z) 
+ {\sin(\pi \nu) \over \Gamma(2 - \beta)} z^{1 - \beta}M(1 - \nu - \beta,2 - \beta,z) 
\Bigg].
\end{align}
This expression arises by applying the Gamma reflection formula twice in \eqref{Kummer_U_defn} where $\alpha=-\nu$, e.g., 
${1\over \Gamma(-\nu)} = - {1 \over \pi}\Gamma(\nu+1)\sin(\pi\nu)$. 
All the terms in \eqref{rescaled_U} are directly computed without overflow issues except the ratio 
${\Gamma(\nu + \beta) \over \Gamma(\nu + 1)}$ which is computed without overflow as follows. 
For small values of the argument of $\Gamma(z)$, typically $z < 30$, each Gamma function is directly computed without overflow. For larger arguments, $z \ge 30$, we use Stirling's asymptotic formula, 
$\Gamma (z) \sim \sqrt{2\pi}e^{-z}z^{z -1/2}$. For large arguments, the Gamma function ratio is very accurately approximated by 
$\frac{\Gamma(\nu + \beta)}{\Gamma(\nu + 1)} 
\simeq \big({\nu + \beta \over \nu + 1}\big)^{\nu + {1\over 2}}\big({\nu + \beta \over e}\big)^{\beta - 1}$. 

Hence, by using \eqref{rescaled_U} for $\nu = \frac{\lambda_{n,k}^+}{2|\kappa|} - 1$, $\beta = 1-\mu$, $z=\vert\kappa\vert \b$, and the fact that $\Gamma(\frac{\lambda_n}{2\vert\kappa\vert}) > 0$, we see that all eigenvalues $\lambda_n = \lambda_{n,k}^+$ computed via \eqref{eigenvalues_CIR_5} are now equivalently given by
\begin{equation}\label{eigenvalues_CIR_rescaled}
 \widetilde{U}\big(\!-\!\frac{\lambda_n}{2\vert\kappa\vert} + 1, 1 - \mu,\vert\kappa\vert \b\big) = 0
\end{equation}
which avoids numerical overflow. Moreover, taking the derivative w.r.t. the first argument of $\widetilde{U}$, 
$\widetilde{U}_1(\alpha,\beta,z) := {\partial \over \partial \alpha}\widetilde{U}(\alpha,\beta,z)$, while using the definition of $\widetilde{U}$ and \eqref{eigenvalues_CIR_5}, gives  
$\widetilde{U}_1\big(\!-\!\frac{\lambda_n}{2\vert\kappa\vert} + 1, 1 - \mu,\vert\kappa\vert \b\big) = 
{1\over \Gamma(\frac{\lambda_n}{2\vert\kappa\vert})}
U_1\big(\!-\!\frac{\lambda_n}{2\vert\kappa\vert} + 1, 1 - \mu,\vert\kappa\vert \b \big)$. 
Since $\widetilde{U}\big(\!-\!\frac{\lambda_n}{2\vert\kappa\vert} + 1, 1 - \mu,\vert\kappa\vert x\big) = 
{1\over \Gamma(\frac{\lambda_n}{2\vert\kappa\vert})}
U\big(\!-\!\frac{\lambda_n}{2\vert\kappa\vert} + 1, 1 - \mu,\vert\kappa\vert x\big)$, we have an equivalent expression for the ratio in \eqref{CIR_Psi_minus_func}:
\begin{align}\label{equiv_rescaled_ratio}
\frac{U(-\frac{\lambda_{n,k}^+}{2|\kappa|} + 1,1 - \mu,|\kappa| x)}
{U_1(-\frac{\lambda_{n,k}^+}{2|\kappa|} + 1,1 - \mu,|\kappa| \b)} 
= 
\frac{\widetilde{U}(-\frac{\lambda_{n,k}^+}{2|\kappa|} + 1,1 - \mu,|\kappa| x)}
{\widetilde{U}_1(-\frac{\lambda_{n,k}^+}{2|\kappa|} + 1,1 - \mu,|\kappa| \b)}.
\end{align}
The latter ratio avoids numerical overflow where the numerator is computed using \eqref{rescaled_U} and the denominator is computed most efficiently by simply applying numerical differentiation, 
e.g., $\widetilde{U}_1(-\nu,\beta,z) \approx (\widetilde{U}\big(-\nu + \delta,\beta,z) 
- \widetilde{U}\big(-\nu,\beta,z))/\delta$, $\delta \simeq 10^{-6}$. Alternatively, differentiating \eqref{rescaled_U}, where $\widetilde{U}_1(-\nu,\beta,z) = - 
{\partial \over \partial \nu}\widetilde{U}_1(-\nu,\beta,z)$, gives 
\begin{align}\label{Kummer_rescaled_U_der}
	&\widetilde{U}_1(-\nu,\beta,z) = {1 \over \sin(\pi \beta)}
\bigg\{
{\Gamma(\nu + \beta) \over \Gamma(\nu + 1)\Gamma(\beta)}
\bigg[ 
\Big(\sin(\pi(\nu + \beta)) [\Psi(\nu + 1) - \Psi(\nu + \beta)]
\nonumber \\
&\phantom{\widetilde{U}_1(-\nu,\beta,z) = {1 \over \sin(\pi \beta)}
\bigg\{
{\Gamma(\nu + \beta) \over \Gamma(\nu + 1)\Gamma(\beta)}}
- \pi \cos(\pi(\nu + \beta))\Big) M(-\nu,\beta,z)
+ \sin(\pi(\nu + \beta)) M_1(-\nu,\beta,z)
\bigg]
\nonumber \\
&\phantom{\widetilde{U}_1}
- {z^{1-\beta} \over \Gamma(2-\beta)}
\bigg[ \pi \cos(\pi\nu) M(-\nu + 1 - \beta, 2-\beta,z) - \sin(\pi\nu) M_1(-\nu + 1 - \beta, 2-\beta,z) \bigg]
 \bigg\}.
\end{align}

Using \eqref{CIR_Psi_plus_func} within \eqref{FHT_prop1_2} gives, for $x< \b$,
\begin{eqnarray}\label{CIR_FHT_up_tail}
\P_{x}(T < \Tau^{+}_{\b} < \infty) = 
2|\kappa|\bigg(\frac{\b}{x}\bigg)^{\mu}e^{\kappa(x-\b)} \sum_{n=1}^{\infty} 
\frac{e^{-\lambda_{n,k}^{-}T}}{\lambda_{n,k}^{-}}
\frac{M(-\frac{\lambda_{n,k}^{-}}{2|\kappa|} + 1,1 - \mu,|\kappa| x)}
{M_1(-\frac{\lambda_{n,k}^{-}}{2|\kappa|} + 1,1 - \mu,|\kappa| \b)}.
\end{eqnarray}
Similarly, using \eqref{CIR_Psi_minus_func} within \eqref{FHT_prop2_2} gives, for $x\in (\b,\infty)$,
\begin{eqnarray}\label{CIR_FHT_down_tail}
\P_{x}(T < \Tau^{-}_{\b} < \infty) = 
2|\kappa|\bigg(\frac{\b}{x}\bigg)^{\mu}e^{\kappa(x-\b)} \sum_{n=1}^{\infty} 
\frac{e^{-\lambda_{n,k}^{+}T}}{\lambda_{n,k}^{+}}
\frac{U(-\frac{\lambda_{n,k}^{+}}{2|\kappa|} + 1,1 - \mu,|\kappa| x)}
{U_1(-\frac{\lambda_{n,k}^{+}}{2|\kappa|} + 1,1 - \mu,|\kappa| \b)},
\end{eqnarray}
where ${U \over U_1} = {\widetilde{U} \over \widetilde{U}_1}$ according to \eqref{equiv_rescaled_ratio}.

The scale function is given by
\begin{equation}\label{CIR_scale_func}
{\mathcal S}[x,y] = |\kappa|^\mu [\gamma(|\mu|, |\kappa|y) - \gamma(|\mu|, |\kappa|x)]
\end{equation}
where $\gamma(a,z):= \int_0^z u^{a-1}e^{-u}du$ denotes the incomplete Gamma function. Since both endpoints are attracting, \eqref{FHT-b-up-equal_infinity} gives
$\P_x({\Tau}^+_\b = \infty) = 1 - {\gamma(|\mu|, |\kappa|x) \over \gamma(|\mu|, |\kappa|\b)}$, $x\in (0,\b)$, and 
$\P_x({\Tau}^-_\b = \infty) = 1 - {\Gamma(|\mu|, |\kappa|x) \over \Gamma(|\mu|, |\kappa|\b)}$, $x\in (\b,\infty)$, where $\Gamma(a,z):= \int_z^\infty u^{a-1}e^{-u}du$ is the complementary incomplete Gamma function.
Combining the above expressions into \eqref{prop_last_time-t-1-limit} gives
\begin{align}\label{prop_last_time_discrete_CIR}
\P_x(g_\b(T) = 0) = \!
\begin{cases} 
\!
\displaystyle 1 - {\gamma(|\mu|, |\kappa|x) \over \gamma(|\mu|, |\kappa|\b)} 
+ 2|\kappa|\bigg(\frac{\b}{x}\bigg)^{\mu}e^{\kappa(x-\b)} \sum_{n=1}^{\infty} 
\frac{e^{-\lambda_{n,k}^{-}T}}{\lambda_{n,k}^{-}}
\frac{M(-\frac{\lambda_{n,k}^{-}}{2|\kappa|} + 1,1 - \mu,|\kappa| x)}
{M_1(-\frac{\lambda_{n,k}^{-}}{2|\kappa|} + 1,1 - \mu,|\kappa| \b)}
, x\in (0,\b),
\\
\displaystyle 1 - {\Gamma(|\mu|, |\kappa|x) \over \Gamma(|\mu|, |\kappa|\b)} 
+ 2|\kappa|\bigg(\frac{\b}{x}\bigg)^{\mu}e^{\kappa(x-\b)} \sum_{n=1}^{\infty} 
\frac{e^{-\lambda_{n,k}^{+}T}}{\lambda_{n,k}^{+}}
\frac{U(-\frac{\lambda_{n,k}^{+}}{2|\kappa|} + 1,1 - \mu,|\kappa| x)}
{U_1(-\frac{\lambda_{n,k}^{+}}{2|\kappa|} + 1,1 - \mu,|\kappa| \b)}
, x\in (\b,\infty).
\end{cases}
\end{align}

The density of $g_\b(T)$ follows by Proposition 
\ref{last-passage-propn-time-t-spectral} with purely discrete spectrum. Computing $\hat\psi^\pm_{n}(\b) 
= {\partial \over \partial x}\psi_n^\pm(x;\b)|_{x= \b}$, where  
${\partial \over \partial z}M(\alpha,\beta,z) = {\alpha \over \beta}M(\alpha + 1,\beta + 1,z)$ and 
${\partial \over \partial z}U(\alpha,\beta,z) = -\alpha U(\alpha + 1,\beta + 1,z)$, gives
\begin{align}\label{eta_plus_CIR}
\eta^+(T-t;\b) 
&= {|\kappa| \over 1 - \mu} \sum_{n=1}^{\infty}e^{-\lambda_{n,\b}^- (T-t)} 
\bigg({2|\kappa|\over \lambda_{n,\b}^-} - 1\bigg)
\frac{M(-\frac{\lambda_{n,k}^{-}}{2|\kappa|} + 2,2 - \mu,|\kappa| \b)}
{M_1(-\frac{\lambda_{n,k}^{-}}{2|\kappa|} + 1,1 - \mu,|\kappa| \b)},
\\
\eta^-(T-t;\b) 
&= |\kappa| \sum_{n=1}^{\infty}e^{-\lambda_{n,\b}^+ (T-t)} 
\bigg(1 - {2|\kappa|\over \lambda_{n,\b}^+}\bigg)
\frac{U(-\frac{\lambda_{n,k}^{+}}{2|\kappa|} + 2,2 - \mu,|\kappa| \b)}
{U_1(-\frac{\lambda_{n,k}^{+}}{2|\kappa|} + 1,1 - \mu,|\kappa| \b)}.
\label{eta_minus_CIR}
\end{align}
Both endpoints are attracting, i.e.,  
$\mathcal{S}(0,\infty;\b) = \frac{1}{\mathcal{S}(0,\b]} + \frac{1}{\mathcal{S}[\b,\infty)}$. Using the above scale function, where ${1\over \m(\b)\s(\b)} = 2\b$, 
${1\over \m(\b)} = 2\b^{-\mu}e^{\kappa k}$, together with \eqref{eta_plus_CIR} and \eqref{eta_minus_CIR} within 
\eqref{last-passage-pdf-spectral} produces the density,
\begin{align}\label{last-passage-pdf-spectral_CIR}
f_{g_\b(T)}(t;x) = 2p(t;x,\b) 
&\bigg[{e^{\kappa k} \over (\b|\kappa|)^\mu}\bigg({1 \over \gamma(|\mu|, |\kappa|\b)} 
+ {1 \over \Gamma(|\mu|, |\kappa|\b)} \bigg)
\nonumber \\
&+  {\b|\kappa| \over 1 - \mu}\sum_{n=1}^{\infty}e^{-\lambda_{n,\b}^- (T-t)} 
\bigg(1 - {2|\kappa|\over \lambda_{n,\b}^-}\bigg)
\frac{M(-\frac{\lambda_{n,k}^{-}}{2|\kappa|} + 2,2 - \mu,|\kappa| \b)}
{M_1(-\frac{\lambda_{n,k}^{-}}{2|\kappa|} + 1,1 - \mu,|\kappa| \b)}
\nonumber \\
&+  \b|\kappa| \sum_{n=1}^{\infty}e^{-\lambda_{n,\b}^+ (T-t)} 
\bigg(1 - {2|\kappa|\over \lambda_{n,\b}^+}\bigg)
\frac{U(-\frac{\lambda_{n,k}^{+}}{2|\kappa|} + 2,2 - \mu,|\kappa| \b)}
{U_1(-\frac{\lambda_{n,k}^{+}}{2|\kappa|} + 1,1 - \mu,|\kappa| \b)}
\bigg],
\end{align}
$t\in (0,T)$, $x,\b\in (0,\infty)$, with $p(t;x,\b)$ given by \eqref{trans_PDF_spectral_CIR_4}. 

The ratio $U/ U_1$ in \eqref{eta_minus_CIR}--\eqref{last-passage-pdf-spectral_CIR} can be computed in  equivalent ways that avoid overflow for large 
$\frac{\lambda_{n,k}^{+}}{2|\kappa|}$. One way is to express the numerator in terms of its rescaled  function and the denominator in terms of $\widetilde{U}_1$ as in \eqref{equiv_rescaled_ratio}, where 
$\Gamma(\frac{\lambda_{n,k}^{+}}{2|\kappa|}) = (\frac{\lambda_{n,k}^{+}}{2|\kappa|} - 1)\Gamma(\frac{\lambda_{n,k}^{+}}{2|\kappa|} - 1)$, giving
\[
\frac{U(-\frac{\lambda_{n,k}^{+}}{2|\kappa|} + 2,2 - \mu,|\kappa| \b)}
{U_1(-\frac{\lambda_{n,k}^{+}}{2|\kappa|} + 1,1 - \mu,|\kappa| \b)} = 
{1\over \frac{\lambda_{n,k}^{+}}{2|\kappa|} - 1}
\frac{\widetilde{U}(-\frac{\lambda_{n,k}^{+}}{2|\kappa|} + 2,2 - \mu,|\kappa| \b)}
{\widetilde{U}_1(-\frac{\lambda_{n,k}^+}{2|\kappa|} + 1,1 - \mu,|\kappa| \b)}.
\]
Another way is to compute 
${\partial \over \partial x}U(-\frac{\lambda_{n,k}^{+}}{2|\kappa|} + 1,1 - \mu,|\kappa| x)|_{x= \b} = 
|\kappa|\Gamma(\frac{\lambda_{n,k}^{+}}{2|\kappa|})
\widetilde{U}^\prime(-\frac{\lambda_{n,k}^+}{2|\kappa|} + 1,1 - \mu,|\kappa| \b)$, where  
$\widetilde{U}^\prime(\alpha,\beta,z) \equiv {\partial \over \partial z}\widetilde{U}(\alpha,\beta,z)$ is computed by numerical differentiation. This quantity is then divided into 
$U_1(-\frac{\lambda_{n,k}^{+}}{2|\kappa|} + 1,1 - \mu,|\kappa| \b) = 
\Gamma(\frac{\lambda_{n,k}^{+}}{2|\kappa|})\widetilde{U}_1(-\frac{\lambda_{n,k}^+}{2|\kappa|} + 1,1 - \mu,|\kappa| \b)$ with $\Gamma(\frac{\lambda_{n,k}^{+}}{2|\kappa|})$ canceling out.

We now apply Proposition~\ref{joint_last-passage-propn-time-t-new-formula}. Using 
\eqref{CIR_Psi_plus_func} and \eqref{CIR_Psi_minus_func} within \eqref{joint_last_passage_pdf_spectral_1} and 
\eqref{joint_last_passage_pdf_spectral_2} gives 
\begin{align}\label{eta_plus_z_k_CIR}
&f^+(T-t,z;\b) 
= 2\vert\kappa\vert \bigg({\b\over z}\bigg)^{\mu}e^{\kappa(z-\b)}
\sum_{n=1}^{\infty}e^{-\lambda_{n,k}^- (T-t)} 
\frac{M(-\frac{\lambda_{n,k}^-}{2|\kappa|} + 1,1 - \mu,|\kappa| z)}
{M_1(-\frac{\lambda_{n,k}^-}{2|\kappa|} + 1,1 - \mu,|\kappa| \b)}, &  z\in (0,\b),
\\
&f^-(T-t,z;\b) 
= 2\vert\kappa\vert \bigg({\b\over z}\bigg)^{\mu}e^{\kappa(z-\b)}
\sum_{n=1}^{\infty}e^{-\lambda_{n,k}^+ (T-t)}  
\frac{U(-\frac{\lambda_{n,k}^+}{2|\kappa|} + 1,1 - \mu,|\kappa| z)}
{U_1(-\frac{\lambda_{n,k}^+}{2|\kappa|} + 1,1 - \mu,|\kappa| \b)}, &  z\in (\b,\infty).
\label{eta_minus_z_k_CIR}
\end{align}
Hence, we have the joint density
\begin{align}\label{joint_last_passage_pdf_CIR}
 f_{g_k(T),X_T}(t,z;x)  
    =  2|\kappa| p(t;x,\b) \left\{
             \begin{array}{lr}
\displaystyle \sum\limits_{n=1}^{\infty} e^{-\lambda_{n,k}^- (T-t)}   
\frac{M(- \frac{\lambda_{n,k}^-}{2|\kappa|}+1 ,1-\mu,|\kappa|z ) }
{M_1(- \frac{\lambda_{n,k}^-}{2|\kappa|}+1 ,1-\mu,|\kappa|\b )}
, &  z\in (0,\b), 
\nonumber \\
\displaystyle \sum\limits_{n=1}^{\infty}  e^{-\lambda_{n,k}^+ (T-t)} 
\frac{U(- \frac{\lambda_{n,k}^+}{2|\kappa|}+1 ,1-\mu,|\kappa|z ) }
{U_1(- \frac{\lambda_{n,k}^+}{2|\kappa|}+1 ,1-\mu,|\kappa|k )}  , &  z\in (\b,\infty),
\nonumber
             \end{array}
\right.\\
\end{align}
$t\in (0,T)$, $x,\b \in (0,\infty)$, with $p(t;x,\b)$ given by \eqref{trans_PDF_spectral_CIR_4}. 
Note the equivalence ${U\over U_1} = {\widetilde{U}\over \widetilde{U}_1}$ in \eqref{joint_last_passage_pdf_CIR}.

The partly discrete joint distribution follows by \eqref{prop_joint_last_time-t-2}. Hence, we simply employ the transition densities in \eqref{trans_PDF_spectral_CIR_b_below} and \eqref{trans_PDF_spectral_CIR_b_above} of Section \ref{subsect_CIR_kill_b}, with $t,b,y$ replaced by $T,\b,z$, respectively, giving
\begin{align}\label{prop_joint_last_time-t-2_CIR}
&\P_x(g_\b(T) = 0, X_T \in dz) 
\nonumber \\
&= {1\over 2} x^{-\mu} e^{\kappa x}
\begin{cases} 
\displaystyle  \sum_{n=1}^\infty  e^{-\lambda_{n,k}^- T} 
N_{n,(0,\b)}^2 \, M\big(\!\!-\!\frac{\lambda_{n,k}^-}{2\vert\kappa\vert} + 1, 1 - \mu,\vert\kappa\vert x\big)
 M\big(\!\!-\!\frac{\lambda_{n,k}^-}{2\vert\kappa\vert} + 1, 1 - \mu,\vert\kappa\vert z\big)\, dz \!\!\!\!&\!\!, x,z \in (0,\b),
\\
\displaystyle \sum_{n=1}^\infty  e^{-\lambda_{n,k}^+ T} 
N_{n,(\b,\infty)}^2 \,  U\big(\!-\!\frac{\lambda_{n,k}^+}{2\vert\kappa\vert} + 1, 1 - \mu,\vert\kappa\vert  x\big)
U\big(\!-\!\frac{\lambda_{n,k}^+}{2\vert\kappa\vert} + 1, 1 - \mu,\vert\kappa\vert  z\big)\, dz \!\!\!\!&\!\!, x,z \in (\b,\infty).
\end{cases}
\end{align}
The latter series is equally expressed in terms of the $\widetilde{U}$ functions as in the second series in \eqref{trans_PDF_spectral_CIR_b_above}.

For the jointly discrete portion of the distribution we implement \eqref{joint_last_doubly_defective_formula}. Since $r=\infty$ is a natural boundary, $\P_x(g_\b(T) = 0, X_T = \partial^\dagger) = 0$ for $x \in (\b,\infty)$. 
For $x \in (0,\b)$, $\P_x(g_\b(T) = 0, X_T = \partial^\dagger)$ is nonzero since $l=0$ is nonconservative, and is given by $\P_x(\Tau^-_0(\b) \le T)$. By the leading term asymptotic of the M function in \eqref{CIR_fundamental_funcs} we have $\varphi_{-\lambda_n}^+(x) \sim x^{-\mu}$, 
as $x\to 0+$, i.e., $\varphi^{+\,\prime}_{-\lambda_n}(0+)/\s(0+) = -\mu = |\mu|$. Combining this with \eqref{CIR_scale_func}, $\varphi_{-\lambda^-_{n,k}}^-(\b)$, \eqref{CIR_Psi_plus_func} and the above Wronskian $w_\lambda$ for $\lambda = -\lambda^-_{n,k}$, as well as using the identity $\Gamma(z + 1)=z\Gamma(z)$, within \eqref{joint_last_doubly_defective_formula}, gives
\begin{align}
&\P_x(g_\b(T) = 0, X_T = \partial^\dagger) 
= \displaystyle 1 - {\gamma(|\mu|, |\kappa|x) \over \gamma(|\mu|, |\kappa|\b)} 
\nonumber \\
&- 2 {|\kappa|^{1+|\mu|}\over \Gamma(|\mu|)}e^{\kappa x}x^{|\mu|}
\sum_{n=1}^\infty  e^{-\lambda_{n,k}^- T} 
\Gamma\big(\!-\! \frac{\lambda_{n,k}^-}{2|\kappa|}\big)
U\big(\!-\!\frac{\lambda_{n,k}^-}{2\vert\kappa\vert} + 1, 1 - \mu,\vert\kappa\vert  \b\big)
\frac{M(- \frac{\lambda_{n,k}^-}{2|\kappa|}+1 ,1-\mu,|\kappa|x ) }
{M_1(- \frac{\lambda_{n,k}^-}{2|\kappa|}+1 ,1-\mu,|\kappa|\b )},
\label{joint_last_doubly_defective_formula_CIR}
\end{align}
$x \in (0,\b)$. 
By using the Gamma reflection formula, this series is also re-expressed in terms of the $\widetilde{U}$ function, i.e., $\Gamma\big(\!-\! \frac{\lambda_{n,k}^-}{2|\kappa|}\big)
U\big(\!-\!\frac{\lambda_{n,k}^-}{2\vert\kappa\vert} + 1, 1 - \mu,\vert\kappa\vert  \b\big) 
= -{\pi \over \frac{\lambda_{n,k}^-}{2|\kappa|} \sin\big(\pi\frac{\lambda_{n,k}^-}{2|\kappa|} \big)}
\widetilde{U}\big(\!-\!\frac{\lambda_{n,k}^-}{2\vert\kappa\vert} + 1, 1 - \mu,\vert\kappa\vert  \b\big)$.

%
%
%
\subsection{Squared Radial Ornstein-Uhlenbeck process killed at either of two interior points}\label{subsect_CIR_kill_a_b}
Here we consider the Squared Radial Ornstein-Uhlenbeck process, $X_{(a,b),t}$, assuming   
$\kappa < 0$ and $\mu < 0$ as above, with imposed killing at either level $a$ or $b$, 
$0 < a < b < \infty$. Using \eqref{CIR_fundamental_funcs} within \eqref{phi_function} gives 
\begin{equation}\label{CIR_cylinder}
    \phi(x,y;\lambda) = (xy)^{-\mu} e^{\kappa(x+y)} S\big(\frac{\lambda}{2|\kappa|}+1,1-\mu, |\kappa|y, |\kappa|x\big)
\end{equation}
where we conveniently define the associated Kummer cylinder function 
\begin{equation}\label{CIR_S_cylinder}
S(\alpha,\beta;x,y) := M(\alpha,\beta,x)  U(\alpha,\beta,y) -U(\alpha,\beta,x) M(\alpha,\beta,y),\,\,\,x,y \in (0,\infty).
\end{equation}
Note the antisymmetry, $S(\alpha,\beta;y,x) = -S(\alpha,\beta;x,y)$. 
We also define  
$S_1(\alpha,\beta;x,y) := \frac{\partial}{\partial \alpha}S(\alpha,\beta;x,y)$ which can be computed in terms of $M,M_1,U,U_1$. Numerical differentiation can also be employed, as noted above for the $M_1$ and $U_1$ functions. 
As in Section \ref{subsect_CIR}, to avoid numerical overflow in a direct computation of the Kummer cylinder function and its derivatives for large negative values of its first argument, it is useful to define the rescaled cylinder function:
\begin{equation}\label{CIR_S_cylinder_rescaled}
\widetilde{S}(-\nu,\beta;x,y) := {S(-\nu,\beta;x,y) \over \Gamma(\nu + 1)} 
= M(-\nu,\beta,x) \widetilde{U}(-\nu,\beta,y) - \widetilde{U}(-\nu,\beta,x) M(-\nu,\beta,y)
\end{equation}
where $\widetilde{U}$ is defined by \eqref{rescaled_U}.

Using \eqref{CIR_cylinder}, \eqref{Delta_derivative} gives $\Delta(a,b;\lambda_n) = {(ab)^{-\mu}\over 2|\kappa|} e^{\kappa(a+b)} 
S_1\big(-\frac{\lambda_n}{2|\kappa|}+1,1-\mu, |\kappa|a, |\kappa|b\big)$, where  
the eigenvalues $\lambda_n \equiv \lambda_n^{(a,b)}$, $n\geq 1$, are the positive simple zeros solving $\phi(a,b;-\lambda_n) = 0$, i.e.,
\footnote{The eigenvalues in \eqref{CIR_eigen_ab} are readily computed by using a root finding (e.g., bisection) algorithm. The leading term asymptotics of the Kummer $M$ and $U$ functions within the cylinder function 
for large $\lambda_n$ can also be combined in terms of trigonometric functions whose zeros provide initial estimates for the eigenvalues. It can be shown that $\lambda_n$ grows approximately in proportion to $n^2$ for large $n$.
}
\begin{equation}\label{CIR_eigen_ab}
    S\big(\!-\!\frac{\lambda_n}{2|\kappa|}+1,1-\mu, |\kappa|a, |\kappa|b\big) = 0
\end{equation}
or equivalently $\widetilde{S}\big(-\frac{\lambda_n}{2|\kappa|}+1,1-\mu, |\kappa|a, |\kappa|b\big) = 0$, which follows since $S\big(\!-\!\frac{\lambda_n}{2|\kappa|}+1,1-\mu, |\kappa|a, |\kappa|b\big)
= \Gamma(\frac{\lambda_n}{2|\kappa|}) \widetilde{S}\big(-\frac{\lambda_n}{2|\kappa|}+1,1-\mu, |\kappa|a, |\kappa|b\big)$ where $\Gamma(\frac{\lambda_n}{2|\kappa|})>0$. Moreover, differentiating \eqref{CIR_S_cylinder_rescaled} w.r.t. its first argument (with $\nu = \frac{\lambda_n}{2|\kappa|}-1$, $\beta = 1-\mu$) and employing \eqref{CIR_eigen_ab} gives 
$S_1\big(\!-\!\frac{\lambda_n}{2|\kappa|}+1,1-\mu, |\kappa|a, |\kappa|b\big)
= \Gamma(\frac{\lambda_n}{2|\kappa|}) 
\widetilde{S}_1\big(-\frac{\lambda_n}{2|\kappa|}+1,1-\mu, |\kappa|a, |\kappa|b\big)$. We define  
$\widetilde{S}_1(\alpha,\beta;x,y) := \frac{\partial}{\partial \alpha}\widetilde{S}(\alpha,\beta;x,y)$ which can be computed by numerical differentiation or in terms of $M,M_1,\widetilde{U}, \widetilde{U}_1$ via \eqref{CIR_S_cylinder_rescaled}. 
Hence, using \eqref{spectral_3_product_eigen} within \eqref{u_spectral_3} produces a spectral series for the transition PDF expressed equally in terms of $S$, $S_1$ or $\widetilde{S},\widetilde{S}_1$ functions:
\begin{align}\label{CIR transition density killed}
p_{(a,b)}(t;x,y) &= {\vert\kappa\vert^{1 - \mu} x^{-\mu} e^{\kappa x}\over \Gamma\big(1 - \mu\big)} 
\sum_{n=1}^\infty  e^{-\lambda_n t} 
\frac{S(-{\lambda_n\over 2\vert\kappa\vert} + 1, 1 - \mu; \vert\kappa\vert a, \vert\kappa\vert x)
S(-{\lambda_n\over 2\vert\kappa\vert} + 1, 1 - \mu; \vert\kappa\vert b, \vert\kappa\vert y)}
{S_1(-{\lambda_n\over 2\vert\kappa\vert} + 1, 1 - \mu; \vert\kappa\vert a, \vert\kappa\vert b)
 / \Gamma\big(\!-\!\frac{\lambda_n}{2\vert\kappa\vert} + 1\big)}
\nonumber \\
&=  {\pi\vert\kappa\vert^{1 - \mu} x^{-\mu} e^{\kappa x}\over \Gamma\big(1 - \mu\big)} 
\sum_{n=1}^\infty  e^{-\lambda_n t} 
\frac{\widetilde{S}(-{\lambda_n\over 2\vert\kappa\vert} + 1, 1 - \mu; \vert\kappa\vert a, \vert\kappa\vert x)
\widetilde{S}(-{\lambda_n\over 2\vert\kappa\vert} + 1, 1 - \mu; \vert\kappa\vert b, \vert\kappa\vert y)}
{\sin(- {\lambda_n\over 2\vert\kappa\vert} \pi) \widetilde{S}_1(-{\lambda_n\over 2\vert\kappa\vert} + 1, 1 - \mu; \vert\kappa\vert a, \vert\kappa\vert b)},
\end{align}
where $\lambda_n \equiv \lambda_n^{(a,b)}$, $x,y\in (a,b)$, $t>0$. We note that the Gamma reflection formula was used in the second expression involving the rescaled cylinder functions.

From \eqref{FHT_eigenfunctions_ab} we have
\begin{align}\label{FHT_CIR_eigenfunctions_ab_plus}
\psi_n^+(x;a,b) &=  2|\kappa|\bigg({b\over x}\bigg)^\mu e^{\kappa(x - b)}
\frac{S(-{\lambda_n\over 2\vert\kappa\vert} + 1, 1 - \mu; \vert\kappa\vert a, \vert\kappa\vert x)}
{S_1(-{\lambda_n\over 2\vert\kappa\vert} + 1, 1 - \mu; \vert\kappa\vert a, \vert\kappa\vert b)},
\\
\psi_n^-(x;a,b) &=  2|\kappa|\bigg({a\over x}\bigg)^\mu e^{\kappa(x - a)}
\frac{S(-{\lambda_n\over 2\vert\kappa\vert} + 1, 1 - \mu; \vert\kappa\vert x, \vert\kappa\vert b)}
{S_1(-{\lambda_n\over 2\vert\kappa\vert} + 1, 1 - \mu; \vert\kappa\vert a, \vert\kappa\vert b)},
\label{FHT_CIR_eigenfunctions_ab_minus}
\end{align}
where $\lambda_n \equiv \lambda_n^{(a,b)}$. Note the equivalence of the ratio 
$S/S_1 = \widetilde{S} /\widetilde{S}_1$ for each respective arguments.

Using \eqref{FHT_CIR_eigenfunctions_ab_plus}--\eqref{FHT_CIR_eigenfunctions_ab_minus}, adapted to the respective intervals $(a,\b)$ and $(\b,b)$, within \eqref{prop_last_time-CDF_discrete_killed_ab}--\eqref{joint_last_ab_doubly_defective_new}, with scale function in \eqref{CIR_scale_func}, gives explicit spectral series for the nonzero discrete parts of the marginal and joint distributions (expressed below using the rescaled cylinder functions):
\begin{align}\label{last_time-CDF_discrete_killed_ab_CIR}
&\P_x(g^{(a,b)}_{\b}(T) = 0) 
\\
\nonumber 
&= 
\begin{cases}
\!\displaystyle {\gamma(|\mu|, |\kappa|\b)  - \gamma(|\mu|, |\kappa|x)
\over \gamma(|\mu|, |\kappa|\b) - \gamma(|\mu|, |\kappa|a)}  
+ 2|\kappa|\big({\b\over x}\big)^\mu \!e^{\kappa(x - \b)}
\sum_{n=1}^\infty {e^{- \lambda_n^{(a,\b)} T} \over \lambda_n^{(a,\b)}}
\frac{\widetilde{S}(-{\lambda_n^{(a,\b)}\over 2\vert\kappa\vert} + 1, 1 - \mu; \vert\kappa\vert a, \vert\kappa\vert x)}
{\widetilde{S}_1(-{\lambda_n^{(a,\b)}\over 2\vert\kappa\vert} + 1, 1 - \mu; \vert\kappa\vert a, \vert\kappa\vert \b)}
, x \in \!(a,\b),
\\
\!\displaystyle {\gamma(|\mu|, |\kappa|x)  - \gamma(|\mu|, |\kappa|\b)
\over \gamma(|\mu|, |\kappa|b) - \gamma(|\mu|, |\kappa|\b)}  
+ 2|\kappa|\big({\b\over x}\big)^\mu\! e^{\kappa(x - \b)}
\sum_{n=1}^\infty {e^{- \lambda_n^{(\b,b)} T} \over \lambda_n^{(\b,b)}}
\frac{\widetilde{S}(-{\lambda_n^{(\b,b)}\over 2\vert\kappa\vert} + 1, 1 - \mu; \vert\kappa\vert x, \vert\kappa\vert b)}
{\widetilde{S}_1(-{\lambda_n^{(\b,b)}\over 2\vert\kappa\vert} + 1, 1 - \mu; \vert\kappa\vert \b, \vert\kappa\vert b)}
, x \in \!(\b,b),
\end{cases}
\end{align}
\begin{align}\label{joint_last_ab_doubly_defective_CIR}
&\P_x(g^{(a,b)}_{\b}(T) = 0, X_{(a,b),T} = \partial^\dagger) 
\\
&= 
\begin{cases}
\!\displaystyle {\gamma(|\mu|, |\kappa|\b)  - \gamma(|\mu|, |\kappa|x)
\over \gamma(|\mu|, |\kappa|\b) - \gamma(|\mu|, |\kappa|a)}  
- 2|\kappa|\big({a\over x}\big)^\mu \!e^{\kappa(x - a)}
\sum_{n=1}^\infty {e^{- \lambda_n^{(a,\b)} T} \over \lambda_n^{(a,\b)}}
\frac{\widetilde{S}(-{\lambda_n^{(a,\b)}\over 2\vert\kappa\vert} + 1, 1 - \mu; \vert\kappa\vert x, \vert\kappa\vert \b)}
{\widetilde{S}_1(-{\lambda_n^{(a,\b)}\over 2\vert\kappa\vert} + 1, 1 - \mu; \vert\kappa\vert a, \vert\kappa\vert \b)}
, x \in \!(a,\b),
\nonumber \\
\!\displaystyle {\gamma(|\mu|, |\kappa|x)  - \gamma(|\mu|, |\kappa|\b)
\over \gamma(|\mu|, |\kappa|b) - \gamma(|\mu|, |\kappa|\b)}  
- 2|\kappa|\big({b\over x}\big)^\mu\! e^{\kappa(x - b)}
\sum_{n=1}^\infty {e^{- \lambda_n^{(\b,b)} T} \over \lambda_n^{(\b,b)}}
\frac{\widetilde{S}(-{\lambda_n^{(\b,b)}\over 2\vert\kappa\vert} + 1, 1 - \mu; \vert\kappa\vert \b, \vert\kappa\vert x)}
{\widetilde{S}_1(-{\lambda_n^{(\b,b)}\over 2\vert\kappa\vert} + 1, 1 - \mu; \vert\kappa\vert \b, \vert\kappa\vert b)}
, x \in \!(\b,b).
\end{cases}
\end{align}
The respective eigenvalues solve 
$\widetilde{S}\big(\!-\!\frac{\lambda_n^{(a,\b)}}{2|\kappa|}+1,1-\mu, |\kappa|a, |\kappa|\b\big) = 0$ and 
$\widetilde{S}\big(\!-\!\frac{\lambda_n^{(\b,b)}}{2|\kappa|}+1,1-\mu, |\kappa|\b, |\kappa|b\big) = 0$. 
Again we note that $\widetilde{S}/\widetilde{S}_1$ in \eqref{last_time-CDF_discrete_killed_ab_CIR}--\eqref{joint_last_ab_doubly_defective_CIR} are equal to the respective ratios $S/S_1$.

Using \eqref{CIR transition density killed}, for time $t=T$, within 
\eqref{prop_joint_last_time_kill_ab-t-3-zero} on the respective intervals gives the partly discrete portion of the distribution (expressed in terms of the rescaled cylinder functions):
\begin{align}\label{joint_last_ab_partly_defective_CIR}
&\P_x(g^{(a,b)}_{\b}(T) = 0, X_{(a,b),T} \in d z) / dz
= {\pi\vert\kappa\vert^{1 - \mu} x^{-\mu} e^{\kappa x}\over \Gamma\big(1 - \mu\big)}
\nonumber \\
&\times 
\begin{cases}
\displaystyle \sum_{n=1}^\infty \!e^{-\lambda_n^{(a,\b)} T} 
\frac{\widetilde{S}(-{\lambda_n^{(a,\b)}\over 2\vert\kappa\vert}\! +\! 1, 1 - \mu; \vert\kappa\vert a, \vert\kappa\vert x)
\widetilde{S}(-{\lambda_n^{(a,\b)}\over 2\vert\kappa\vert} \! +\! 1, 1 - \mu; \vert\kappa\vert \b, \vert\kappa\vert z)}
{\sin(- {\lambda_n^{(a,\b)}\over 2\vert\kappa\vert} \pi) 
\widetilde{S}_1(-{\lambda_n^{(a,\b)}\over 2\vert\kappa\vert} \! +\! 1, 1 - \mu; \vert\kappa\vert a, \vert\kappa\vert \b)}
, x,z\! \in \!(a,\b),
\\
\displaystyle \sum_{n=1}^\infty \!e^{-\lambda_n^{(\b,b)} T} 
\frac{\widetilde{S}(-{\lambda_n^{(\b,b)}\over 2\vert\kappa\vert}\! +\! 1, 1 - \mu; \vert\kappa\vert \b, \vert\kappa\vert x)
\widetilde{S}(-{\lambda_n^{(\b,b)}\over 2\vert\kappa\vert} \! +\! 1, 1 - \mu; \vert\kappa\vert b, \vert\kappa\vert z)}
{\sin(- {\lambda_n^{(\b,b)}\over 2\vert\kappa\vert} \pi) 
\widetilde{S}_1(-{\lambda_n^{(\b,b)}\over 2\vert\kappa\vert} \! +\! 1, 1 - \mu; \vert\kappa\vert \b, \vert\kappa\vert b)}
, x,z \in (\b,b).
\end{cases}
\end{align}

We now employ Proposition~\ref{joint_last-passage-propn-time-t_ab-new-version}, which requires  
$\hat\psi^+_{n}(a,\b)$ and $\hat\psi^-_{n}(\b,b)$. These follow directly by simply using 
$\Delta(a,\b;\lambda_n)$ and $\varphi^+_{-\lambda_n}$, for $\lambda_n = \lambda_n^{(a,\b)}$, and 
$\Delta(\b,b;\lambda_n)$ and $\varphi^+_{-\lambda_n}$, for $\lambda_n = \lambda_n^{(\b,b)}$, within \eqref{psi_hat_derivatives}:
\begin{align}\label{FHT_CIR_hat_plus_ak}
{\hat\psi^+_{n}(a,\b) \over \lambda_n} 
&= 
{\vert\kappa\vert^\mu \b^{2\mu} e^{-2\kappa \b} \Gamma(1 - \mu) \over 
{\lambda_n\over 2\vert\kappa\vert} \Gamma(-{\lambda_n\over 2\vert\kappa\vert} + 1)
S_1(-{\lambda_n\over 2\vert\kappa\vert} + 1, 1 - \mu; \vert\kappa\vert a, \vert\kappa\vert \b)}
\frac{M(- \frac{\lambda_n}{2|\kappa|}\! +\! 1 ,1 \!- \!\mu, |\kappa| a ) }
{M(- \frac{\lambda_n}{2|\kappa|}\! +\! 1 ,1 \!- \!\mu,|\kappa| \b )}
\nonumber \\
&= 
{\vert\kappa\vert^\mu \b^{2\mu} e^{-2\kappa \b} \Gamma(1 - \mu) \sin(\pi \frac{\lambda_n}{2|\kappa|})
\over 
\pi {\lambda_n\over 2\vert\kappa\vert} 
\widetilde{S}_1(-{\lambda_n\over 2\vert\kappa\vert} + 1, 1 - \mu; \vert\kappa\vert a, \vert\kappa\vert \b)}
\frac{M(- \frac{\lambda_n}{2|\kappa|}\! +\! 1 ,1 \!- \!\mu, |\kappa| a ) }
{M(- \frac{\lambda_n}{2|\kappa|}\! +\! 1 ,1 \!- \!\mu,|\kappa| \b )},\,\,\lambda_n \equiv \lambda_n^{(a,\b)},
\end{align}
\begin{align}\label{FHT_CIR_hat_minus_kb}
{\hat\psi^-_{n}(\b,b) \over \lambda_n} 
&= 
{\vert\kappa\vert^\mu \b^{2\mu} e^{-2\kappa \b} \Gamma(1 - \mu) \over 
{\lambda_n\over 2\vert\kappa\vert} \Gamma(-{\lambda_n\over 2\vert\kappa\vert} + 1)
S_1(-{\lambda_n\over 2\vert\kappa\vert} + 1, 1 - \mu; \vert\kappa\vert \b, \vert\kappa\vert b)}
\frac{M(- \frac{\lambda_n}{2|\kappa|}\! +\! 1 ,1 \!- \!\mu, |\kappa| b ) }
{M(- \frac{\lambda_n}{2|\kappa|}\! +\! 1 ,1 \!- \!\mu,|\kappa| \b )}
\nonumber \\
&= 
{\vert\kappa\vert^\mu \b^{2\mu} e^{-2\kappa \b} \Gamma(1 - \mu) \sin(\pi \frac{\lambda_n}{2|\kappa|})
\over 
\pi {\lambda_n\over 2\vert\kappa\vert} 
\widetilde{S}_1(-{\lambda_n\over 2\vert\kappa\vert} + 1, 1 - \mu; \vert\kappa\vert \b, \vert\kappa\vert b)}
\frac{M(- \frac{\lambda_n}{2|\kappa|}\! +\! 1 ,1 \!- \!\mu, |\kappa| b ) }
{M(- \frac{\lambda_n}{2|\kappa|}\! +\! 1 ,1 \!- \!\mu,|\kappa| \b )},\,\, \lambda_n \equiv \lambda_n^{(\b,b)}.
\end{align}
The second expressions in \eqref{FHT_CIR_hat_plus_ak}--\eqref{FHT_CIR_hat_minus_kb} arise by the Gamma reflection formula and the derivative of the rescaled cylinder function. 
Substituting the expressions in \eqref{FHT_CIR_hat_plus_ak}--\eqref{FHT_CIR_hat_minus_kb} into 
\eqref{last_passage_pdf_discrete_spec} gives the marginal density
\begin{align}\label{last_passage_pdf_explicit_ab_CIR}
&f_{g^{(a,b)}_{\b}(T)}(t;x) = 2p_{(a,b)}(t;x,\b) 
\bigg[ \b^{-\mu}e^{\kappa \b}\mathcal{S}(a,b;\b) 
+ {2\over \pi} |\kappa|^{1+\mu}\b^\mu e^{-\kappa \b}\Gamma(1 - \mu)
\nonumber \\
&\times
\sum_{n=1}^{\infty}
\bigg(
{e^{-\lambda_n^{(a,\b)} (T-t)} \over \lambda_n^{(a,\b)}}
{\sin(\pi \frac{\lambda_n^{(a,\b)}}{2|\kappa|})
\over 
\widetilde{S}_1(-{\lambda_n^{(a,\b)}\over 2\vert\kappa\vert} + 1, 1 - \mu; \vert\kappa\vert a, \vert\kappa\vert \b)}
\frac{M(- \frac{\lambda_n^{(a,\b)}}{2|\kappa|} + 1 ,1 - \mu, |\kappa| a ) }
{M(- \frac{\lambda_n^{(a,\b)}}{2|\kappa|} + 1 ,1 - \mu,|\kappa| \b )}
\nonumber \\
&\phantom{ \sum_{n=1}^{\infty}\bigg(}
+ {e^{-\lambda_n^{(\b,b)} (T-t)} \over \lambda_n^{(\b,b)}}
{\sin(\pi \frac{\lambda_n^{(\b,b)}}{2|\kappa|})
\over 
\widetilde{S}_1(-{\lambda_n^{(\b,b)}\over 2\vert\kappa\vert} + 1, 1 - \mu; \vert\kappa\vert \b, \vert\kappa\vert b)}
\frac{M(- \frac{\lambda_n^{(\b,b)}}{2|\kappa|} + 1 ,1 - \mu, |\kappa| b ) }
{M(- \frac{\lambda_n^{(\b,b)}}{2|\kappa|} + 1 ,1 - \mu,|\kappa| \b )}
\bigg)
\bigg],
\end{align}
$t\in (0,T)$, $x,\b\in (a,b)$, where $\mathcal{S}(a,b;\b) = \displaystyle{1\over |\kappa|}\bigg[{1 \over \gamma(|\mu|, |\kappa|\b) - \gamma(|\mu|, |\kappa|a)} 
+ {1 \over \gamma(|\mu|, |\kappa|b) - \gamma(|\mu|, |\kappa|\b)}\bigg]$. 
Using $\psi_n^+(z;a,\b)$ and $\psi_n^-(z;\b,b)$, i.e., \eqref{FHT_CIR_eigenfunctions_ab_plus} and \eqref{FHT_CIR_eigenfunctions_ab_minus} adapted to the respective intervals $(a,\b)$ and $(\b,b)$, within \eqref{joint_last_passage_pdf_explicit_ab_new} gives the joint PDF
\begin{align}\label{joint_last_passage_pdf_explicit_ab_CIR}
f_{\!g^{(a,b)}_{\b}(T), X_{(a,b),T}}\!(t,z;x) 
= 2 |\kappa| p_{(a,b)}(t;x,\b) \!
\left\{\!\!
\begin{array}{lr}
         \displaystyle  \sum_{n=1}^{\infty}e^{-\lambda_n^{(a,\b)} (T-t)} 
\frac{\widetilde{S}(-{\lambda_n^{(a,\b)} \over 2\vert\kappa\vert} + 1, 1 - \mu; \vert\kappa\vert a, \vert\kappa\vert z)}
{\widetilde{S}_1(-{\lambda_n^{(a,\b)}\over 2\vert\kappa\vert} + 1, 1 - \mu; \vert\kappa\vert a, \vert\kappa\vert \b)}, 
&\!\!\!\!\!\!  z\in (a,\b), 
\\
    \displaystyle \sum_{n=1}^{\infty} e^{-\lambda_n^{(\b,b)} (T-t)}
\frac{\widetilde{S}(-{\lambda_n^{(\b,b)} \over 2\vert\kappa\vert} + 1, 1 - \mu; \vert\kappa\vert z, \vert\kappa\vert b)}{\widetilde{S}_1(-{\lambda_n^{(\b,b)}\over 2\vert\kappa\vert} + 1, 1 - \mu; \vert\kappa\vert \b, \vert\kappa\vert b)}, &\!\!\!\!\!\!\!  z\in (\b,b), 
\end{array}
\right.
\end{align}
$t\in (0,T)$, $x,\b\in (a,b)$. We note that $p_{(a,b)}(t;x,\b)$ is given by \eqref{CIR transition density killed} and hence both \eqref{last_passage_pdf_explicit_ab_CIR} and \eqref{joint_last_passage_pdf_explicit_ab_CIR} also represent double series for the marginal and joint densities.

%
%
\subsection{Squared Radial Ornstein-Uhlenbeck process killed at one interior point}\label{subsect_CIR_kill_b}
We now consider the Squared Radial Ornstein-Uhlenbeck process, $X_{b,t}$, killed at $b\in (0,\infty)$, and derive distributions of $g^{b}_{\b}(T)$ and $(g^{b}_{\b}(T),X_{b,T})$ for the two cases: (i) $x,\b\in (0,b)$ and (ii) $x,\b\in  (b,\infty)$. As above, we assume $\kappa < 0$ and $\mu < 0$ and with $l=0$ specified as killing when 
$\mu\in(-1,0)$ for case (i). 
%
%

The transition PDF $p_b^-$ for $X_{b,t} \in (0,b)$, i.e., case (i), is given by \eqref{u_spectral_1} with product eigenfunction in \eqref{spectral_1_product_eigen} computed using \eqref{eigenfunc1}. In particular,
\begin{align}\label{trans_PDF_spectral_CIR_b_below}
p_b^-(t;x,y) = {1\over 2} x^{-\mu} e^{\kappa x} 
\sum_{n=1}^\infty  e^{-\lambda_{n,b}^- t} 
N_{n,(0,b)}^2 \, M\big(\!\!-\!\frac{\lambda_{n,b}^-}{2\vert\kappa\vert} + 1, 1 - \mu,\vert\kappa\vert x\big)
 M\big(\!\!-\!\frac{\lambda_{n,b}^-}{2\vert\kappa\vert} + 1, 1 - \mu,\vert\kappa\vert y\big),
\end{align}
$x,y\in (0,b)$, $t>0$, with squared normalization constant for eigenfunctions on $(0,b)$ given equivalently as
$
N_{n,(0,b)}^2 = -2\vert\kappa\vert^{1 - \mu}
\frac{\Gamma\big(-\frac{\lambda_{n,b}^-}{2\vert\kappa\vert} + 1\big)}{\Gamma\big(1 - \mu\big)} 
\frac{U\big(-\frac{\lambda_{n,b}^-}{2\vert\kappa\vert} + 1, 1 - \mu,\vert\kappa\vert b\big)}
{M_1\big(-\frac{\lambda_{n,b}^-}{2\vert\kappa\vert} + 1, 1 - \mu,\vert\kappa\vert b\big)}
= \frac{2\pi\vert\kappa\vert^{1 - \mu}}
{\Gamma\big(1 - \mu\big) \sin\big(-\pi\frac{\lambda_{n,b}^-}{2\vert\kappa\vert}\big)} 
\frac{\widetilde{U}\big(-\frac{\lambda_{n,b}^-}{2\vert\kappa\vert} + 1, 1 - \mu,\vert\kappa\vert b\big)}
{M_1\big(-\frac{\lambda_{n,b}^-}{2\vert\kappa\vert} + 1, 1 - \mu,\vert\kappa\vert b\big)}.
$
The latter expression uses the rescaled Kummer function in \eqref{rescaled_U} and arises simply by using the reflection formula for the Gamma function. 
The eigenvalues $\lambda_{n,b}^-$ solve $M\big(\!\!-\!\frac{\lambda_{n,b}^-}{2\vert\kappa\vert} + 1, 1 - \mu,\vert\kappa\vert b\big) = 0$.

By a similar derivation, the transition PDF $p_b^+$ for $X_{b,t} \in (b,\infty)$, i.e., case (ii), follows by \eqref{u_spectral_2} with product eigenfunction in \eqref{spectral_2_product_eigen} 
computed using \eqref{eigenfunc2}:
\begin{align}\label{trans_PDF_spectral_CIR_b_above}
p_b^+(t;x,y) &= {1\over 2} x^{-\mu} e^{\kappa x} 
\sum_{n=1}^\infty  e^{-\lambda_{n,b}^+ t} 
N_{n,(b,\infty)}^2 \,  U\big(\!-\!\frac{\lambda_{n,b}^+}{2\vert\kappa\vert} + 1, 1 - \mu,\vert\kappa\vert  x\big)
U\big(\!-\!\frac{\lambda_{n,b}^+}{2\vert\kappa\vert} + 1, 1 - \mu,\vert\kappa\vert  y\big)
\nonumber \\
&= {1\over 2} x^{-\mu} e^{\kappa x} 
\sum_{n=1}^\infty  e^{-\lambda_{n,b}^+ t} 
\frac{2\pi\vert\kappa\vert^{1 - \mu}}
{\Gamma\big(1 - \mu\big) \sin\big(\!-\!\pi\frac{\lambda_{n,b}^+}{2\vert\kappa\vert} \big)}
{M\big(\!-\!\frac{\lambda_{n,b}^+}{2\vert\kappa\vert} + 1, 1 - \mu,\vert\kappa\vert  b\big)
\over \widetilde{U}_1\big(\!-\!\frac{\lambda_{n,b}^+}{2\vert\kappa\vert} + 1, 1 - \mu,\vert\kappa\vert  b\big)}
\nonumber \\
&\phantom{{1\over 2} x^{-\mu} e^{\kappa x} \sum_{n=1}^\infty e^{-\lambda_{n,b}^+ t}}
\times \widetilde{U}\big(\!-\!\frac{\lambda_{n,b}^+}{2\vert\kappa\vert} + 1, 1 - \mu,\vert\kappa\vert  x\big)
\widetilde{U}\big(\!-\!\frac{\lambda_{n,b}^+}{2\vert\kappa\vert} + 1, 1 - \mu,\vert\kappa\vert  y\big),
\end{align}
$x,y\in (b,\infty)$, $t>0$.
In the first series, $N_{n,(b,\infty)}^2 = -2\vert\kappa\vert^{1 - \mu}
\frac{\Gamma\big(-\frac{\lambda_{n,b}^+}{2\vert\kappa\vert} + 1\big)}{\Gamma\big(1 - \mu\big)} 
\frac{M\big(-\frac{\lambda_{n,b}^+}{2\vert\kappa\vert} + 1, 1 - \mu,\vert\kappa\vert b\big)}
{U_1\big(-\frac{\lambda_{n,b}^+}{2\vert\kappa\vert} + 1, 1 - \mu,\vert\kappa\vert b\big)}$ is a squared normalization constant for eigenfunctions on $(b,\infty)$ and $\lambda_{n,b}^+$ solve  
$U\big(\!\!-\!\frac{\lambda_{n,b}^+}{2\vert\kappa\vert} + 1, 1 - \mu,\vert\kappa\vert b\big) = 0$ or equivalently $\widetilde{U}\big(\!\!-\!\frac{\lambda_{n,b}^+}{2\vert\kappa\vert} + 1, 1 - \mu,\vert\kappa\vert b\big) = 0$.

%
%
We first compute the defective distributions. 
For case (i) we use \eqref{prop_joint_last_b-below_pmf}--\eqref{doubly_defective_kill_up_b}. Hence, by direct use of the first line of \eqref{prop_last_time_discrete_CIR} and the second line of 
\eqref{last_time-CDF_discrete_killed_ab_CIR}, we have
\begin{align}\label{last_time-CDF_discrete_killed_below_b_CIR}
&\P_x(g^b_{\b}(T) = 0) 
\\
\nonumber
&=\begin{cases}
\displaystyle 1 - {\gamma(|\mu|, |\kappa|x) \over \gamma(|\mu|, |\kappa|\b)} 
+ 2|\kappa|\bigg(\frac{\b}{x}\bigg)^{\mu}e^{\kappa(x-\b)} \sum_{n=1}^{\infty} 
\frac{e^{-\lambda_{n,k}^{-}T}}{\lambda_{n,k}^{-}}
\frac{M(-\frac{\lambda_{n,k}^{-}}{2|\kappa|} + 1,1 - \mu,|\kappa| x)}
{M_1(-\frac{\lambda_{n,k}^{-}}{2|\kappa|} + 1,1 - \mu,|\kappa| \b)}
,\,\,\,\,\,\,\,\,\,\,\,\,\,\,\,\,\,\,\,\,\,\,\,\,\,\,\,\,\,\,\,\,\,\,x\in (0,\b),
\\
\displaystyle {\gamma(|\mu|, |\kappa|x)  - \gamma(|\mu|, |\kappa|\b)
\over \gamma(|\mu|, |\kappa|b) - \gamma(|\mu|, |\kappa|\b)}  
+ 2|\kappa|\big({\b\over x}\big)^\mu\! e^{\kappa(x - \b)}
\sum_{n=1}^\infty {e^{- \lambda_n^{(\b,b)} T} \over \lambda_n^{(\b,b)}}
\frac{\widetilde{S}(-{\lambda_n^{(\b,b)}\over 2\vert\kappa\vert} + 1, 1 - \mu; \vert\kappa\vert x, \vert\kappa\vert b)}
{\widetilde{S}_1(-{\lambda_n^{(\b,b)}\over 2\vert\kappa\vert} + 1, 1 - \mu; \vert\kappa\vert \b, \vert\kappa\vert b)}
, x \in \!(\b,b).
\end{cases}
\end{align}
By direct use of the time-$T$ transition densities $p_k^-(T;x,z)$ and $p_{(\b,b)}(T;x,z)$, i.e., \eqref{trans_PDF_spectral_CIR_b_below} with $b$ replaced by $\b$ and \eqref{CIR transition density killed} with $a$ replaced by $\b$, we have
\begin{align}\label{prop_joint_last_b-below_pmf_CIR}
&\P_x(g^{b}_{\b}(T) = 0, X_{b,T} \in dz) / dz
\\ \nonumber
&=\begin{cases}
 \displaystyle {1\over 2} x^{-\mu} e^{\kappa x} 
\sum_{n=1}^\infty  e^{-\lambda_{n,\b}^- T} 
N_{n,(0,\b)}^2 \, M\big(\!\!-\!\frac{\lambda_{n,\b}^-}{2\vert\kappa\vert} + 1, 1 - \mu,\vert\kappa\vert x\big)
 M\big(\!\!-\!\frac{\lambda_{n,\b}^-}{2\vert\kappa\vert} + 1, 1 - \mu,\vert\kappa\vert z\big), \,\,\,\,\,\,\,\,\, x,z \in (0,\b),
\\ \\
\displaystyle {\pi\vert\kappa\vert^{1 - \mu} x^{-\mu} e^{\kappa x}\over \Gamma\big(1 - \mu\big)} 
\sum_{n=1}^\infty  e^{-\lambda_n^{(\b,b)} T} 
\frac{\widetilde{S}(-{\lambda_n^{(\b,b)}\over 2\vert\kappa\vert}  \! +\! 1, 1 \!- \! \mu; \vert\kappa\vert \b, \vert\kappa\vert x)
\widetilde{S}(-{\lambda_n^{(\b,b)}\over 2\vert\kappa\vert}  \! +\! 1, 1 \!- \! \mu; \vert\kappa\vert b, \vert\kappa\vert z)}
{\sin(- {\lambda_n^{(\b,b)}\over 2\vert\kappa\vert} \pi) 
\widetilde{S}_1(-{\lambda_n^{(\b,b)}\over 2\vert\kappa\vert} + 1, 1 - \mu; \vert\kappa\vert \b, \vert\kappa\vert b)}, x,z \in (\b,b).
\end{cases}
\end{align}
The jointly discrete distribution is given by \eqref{doubly_defective_kill_up_b}. 
Hence, for $x \in (0,\b)$, $\P_x(g^{b}_{\b}(T) = 0, X_{b,T} = \partial^\dagger)=\P_x(\T^-_0 (\b) \le T)$ is given by \eqref{joint_last_doubly_defective_formula_CIR}. 
For $x \in (\b,b)$, $\P_x(g^{b}_{\b}(T) = 0, X_{b,T} = \partial^\dagger) = \P_x(\T^+_b (\b) \le T)$ is given by the second expression in \eqref{joint_last_ab_doubly_defective_CIR}.

%
%
%
For case (ii) the defective portions are given by \eqref{prop_joint_last_b-above_pmf}--\eqref{doubly_defective_kill_down_b}.  Direct use of the first line of \eqref{last_time-CDF_discrete_killed_ab_CIR}, with $a$ replaced by $b$, and the second line of 
\eqref{prop_last_time_discrete_CIR}, gives
\begin{align}\label{last_time-CDF_discrete_killed_above_b_CIR}
&\P_x(g^b_{\b}(T) = 0) 
\\ \nonumber
&=\begin{cases}
\!\displaystyle {\gamma(|\mu|, |\kappa|\b)  - \gamma(|\mu|, |\kappa|x)
\over \gamma(|\mu|, |\kappa|\b) - \gamma(|\mu|, |\kappa|b)}  
+ 2|\kappa|\big({\b\over x}\big)^\mu \!e^{\kappa(x - \b)}
\sum_{n=1}^\infty {e^{- \lambda_n^{(b,\b)} T} \over \lambda_n^{(b,\b)}}
\frac{\widetilde{S}(-{\lambda_n^{(b,\b)}\over 2\vert\kappa\vert} + 1, 1 - \mu; \vert\kappa\vert b, \vert\kappa\vert x)}
{\widetilde{S}_1(-{\lambda_n^{(b,\b)}\over 2\vert\kappa\vert} + 1, 1 - \mu; \vert\kappa\vert b, \vert\kappa\vert \b)}
, x \in (b,\b),
\\
\displaystyle 1 - {\Gamma(|\mu|, |\kappa|x) \over \Gamma(|\mu|, |\kappa|\b)} 
+ 2|\kappa|\bigg(\frac{\b}{x}\bigg)^{\mu}e^{\kappa(x-\b)} \sum_{n=1}^{\infty} 
\frac{e^{-\lambda_{n,k}^{+}T}}{\lambda_{n,k}^{+}}
\frac{\widetilde{U}(-\frac{\lambda_{n,k}^{+}}{2|\kappa|} + 1,1 - \mu,|\kappa| x)}
{\widetilde{U}_1(-\frac{\lambda_{n,k}^{+}}{2|\kappa|} + 1,1 - \mu,|\kappa| \b)}
,\,\,\,\,\,\,\,\,\,\,\,\,\,\,\,\,\,\,\,\,\,\,\,\,\,\,\,\,\,\,\,\, x\in (\b,\infty).
\end{cases}
\end{align}
By direct use of $p_{(b,\b)}(T;x,z)$ and $p_k^+(T;x,z)$, i.e., 
\eqref{CIR transition density killed} with interval $(a,b)$ replaced by $(b,\b)$ and  \eqref{trans_PDF_spectral_CIR_b_above} with $b$ replaced by $\b$, we have
\begin{align}\label{prop_joint_last_b-above_pmf_CIR}
&\P_x(g^{b}_{\b}(T) = 0, X_{b,T} \in dz) / dz
\\ \nonumber
&=\begin{cases}
\displaystyle {\pi\vert\kappa\vert^{1 - \mu} x^{-\mu} e^{\kappa x}\over \Gamma\big(1 - \mu\big)} 
\sum_{n=1}^\infty  e^{-\lambda_n^{(b,\b)} T} 
\frac{\widetilde{S}(-{\lambda_n^{(b,\b)}\over 2\vert\kappa\vert}  \! +\! 1, 1 \!- \! \mu; \vert\kappa\vert b, \vert\kappa\vert x)
\widetilde{S}(-{\lambda_n^{(b,\b)}\over 2\vert\kappa\vert}  \! +\! 1, 1 \!- \! \mu; \vert\kappa\vert \b, \vert\kappa\vert z)}
{\sin(- {\lambda_n^{(b,\b)}\over 2\vert\kappa\vert} \pi) 
\widetilde{S}_1(-{\lambda_n^{(b,\b)}\over 2\vert\kappa\vert} + 1, 1 - \mu; \vert\kappa\vert b, \vert\kappa\vert \b)}, x,z \in (b,\b),
\\
\displaystyle {1\over 2} x^{-\mu} e^{\kappa x} 
\sum_{n=1}^\infty  e^{-\lambda_{n,\b}^+ T} 
N_{n,(\b,\infty)}^2 \,  U\big(\!-\!\frac{\lambda_{n,\b}^+}{2\vert\kappa\vert} + 1, 1 - \mu,\vert\kappa\vert  x\big)
U\big(\!-\!\frac{\lambda_{n,\b}^+}{2\vert\kappa\vert} + 1, 1 - \mu,\vert\kappa\vert  z\big),
\,\,\,\,\,\,\,\,x,z \in (\b,\infty).
\end{cases}
\end{align}
The jointly discrete distribution is given by \eqref{doubly_defective_kill_down_b}. 
For $x\in (\b,\infty)$, $\P_x(g^{b}_{\b}(T) = 0, X_{b,T} = \partial^\dagger)=0$ since $r=\infty$ is a natural (conservative) boundary. For $x \in (b,\b)$, $\P_x(g^{b}_{\b}(T) = 0, X_{b,T} = \partial^\dagger) = \P_x(\T^-_b (\b) \le T)$, as given by the first expression in \eqref{joint_last_ab_doubly_defective_CIR} 
with $a$ replaced by $b$, i.e.,
\begin{align}\label{joint_last_b_doubly_defective_CIR}
&\P_x(g^{b}_{\b}(T) = 0, X_{b,T} = \partial^\dagger) 
\nonumber \\
&= 
\!\displaystyle {\gamma(|\mu|, |\kappa|\b)  - \gamma(|\mu|, |\kappa|x)
\over \gamma(|\mu|, |\kappa|\b) - \gamma(|\mu|, |\kappa|b)}  
- 2|\kappa|\big({b\over x}\big)^\mu \!e^{\kappa(x - b)}
\sum_{n=1}^\infty {e^{- \lambda_n^{(b,\b)} T} \over \lambda_n^{(b,\b)}}
\frac{\widetilde{S}(-{\lambda_n^{(b,\b)}\over 2\vert\kappa\vert} + 1, 1 - \mu; \vert\kappa\vert x, \vert\kappa\vert \b)}
{\widetilde{S}_1(-{\lambda_n^{(b,\b)}\over 2\vert\kappa\vert} + 1, 1 - \mu; \vert\kappa\vert b, \vert\kappa\vert \b)}.
\end{align}

We now apply Proposition~\ref{marginal_joint_last-passage-propn-kill-b} by directly using \eqref{CIR_scale_func},  \eqref{eta_plus_CIR},\eqref{eta_minus_CIR}, \eqref{eta_plus_z_k_CIR}, \eqref{eta_minus_z_k_CIR} and  \eqref{FHT_CIR_eigenfunctions_ab_plus}, \eqref{FHT_CIR_eigenfunctions_ab_minus}, \eqref{FHT_CIR_hat_plus_ak}, \eqref{FHT_CIR_hat_minus_kb} on the appropriate intervals $(\b,b)$ and $(b,\b)$. Note: $\ind_{\{E_0\}} = 
\ind_{\{E_\infty\}} = 1$ since both $l=0$ and $r=\infty$ are attracting. 
By combining \eqref{marginal_spectral_kill-b_1} and \eqref{marginal_spectral_kill-b_2} we have the marginal density
\begin{align}
&f_{g^{b}_\b(T)}(t;x) 
\nonumber \\
\!\!\!\!&= \!\begin{cases}
\displaystyle \! 2p^-_{b}(t;x,\b) 
\bigg[ R_-(b,\b) + {\b |\kappa| \over 1 - \mu}
\sum_{n=1}^\infty  e^{-\lambda^-_{n,\b}  (T-t)} \bigg(1 - {2|\kappa|\over \lambda^-_{n,\b}} \bigg)
\frac{M(-\frac{\lambda_{n,k}^{-}}{2|\kappa|} + 2,2 - \mu,|\kappa| \b)}
{M_1(-\frac{\lambda_{n,k}^{-}}{2|\kappa|} + 1,1 - \mu,|\kappa| \b)}
\\
\! + \displaystyle {2\over \pi} |\kappa|^{1+\mu}\b^\mu e^{-\kappa \b}\Gamma(1 - \mu)
\sum_{n=1}^{\infty}
{{e^{-\lambda_n^{(\b,b)} (T-t)} \over \lambda_n^{(\b,b)}}\sin(\pi \frac{\lambda_n^{(\b,b)}}{2|\kappa|}) 
\over 
 \widetilde{S}_1(-{\lambda_n^{(\b,b)}\over 2\vert\kappa\vert} \!+ \!1 ,1 \!-\! \mu; \vert\kappa\vert \b, \vert\kappa\vert b)}
\frac{M(- \frac{\lambda_n^{(\b,b)}}{2|\kappa|} \!+ \!1 ,1 \!-\! \mu, |\kappa| b ) }
{M(- \frac{\lambda_n^{(\b,b)}}{2|\kappa|} \!+ \!1 ,1 \!-\! \mu,|\kappa| \b )}\,
\bigg] 
 , x,\b \!\in \! (0,b),
\\
\displaystyle \! 2p^+_{b}(t;x,\b) 
\bigg[ R_+(b,\b) + 2\b \kappa^2
\sum_{n=1}^\infty  {e^{-\lambda^+_{n,\b}  (T-t)} \over \lambda_{n,k}^{+}}
\frac{\widetilde{U}(-\frac{\lambda_{n,k}^{+}}{2|\kappa|} + 2,2 - \mu,|\kappa| \b)}
{\widetilde{U}_1(-\frac{\lambda_{n,k}^{+}}{2|\kappa|} + 1,1 - \mu,|\kappa| \b)}
\nonumber \\
\! + \displaystyle {2\over \pi} |\kappa|^{1+\mu}\b^\mu e^{-\kappa \b}\Gamma(1 - \mu)
\sum_{n=1}^{\infty}
{{e^{-\lambda_n^{(b,\b)} (T-t)}\over \lambda_n^{(b,\b)}} \sin(\pi \frac{\lambda_n^{(b,\b)}}{2|\kappa|}) 
\over 
 \widetilde{S}_1(-{\lambda_n^{(b,\b)}\over 2\vert\kappa\vert} \!+ \!1 ,1 \!-\! \mu; \vert\kappa\vert b, \vert\kappa\vert \b)}
\frac{M(- \frac{\lambda_n^{(b,\b)}}{2|\kappa|} \!+ \!1 ,1 \!-\! \mu, |\kappa| b ) }
{M(- \frac{\lambda_n^{(b,\b)}}{2|\kappa|} \!+ \!1 ,1 \!-\! \mu,|\kappa| \b )}
\,\bigg]
, x,\!\b \!\in\! (b,\infty),
\end{cases}
\nonumber \\
\label{marginal_spectral_kill-b_CIR}
\end{align}
$t\in (0,T)$, $R_-(b,\b) :=\displaystyle {e^{k\kappa} \over (k|\kappa|)^\mu}
\bigg[{1 \over \gamma(|\mu|, |\kappa|\b)} 
+ {1 \over \gamma(|\mu|, |\kappa|b) - \gamma(|\mu|, |\kappa| \b)}\bigg]$,  
$R_+(b,\b) :=\displaystyle {e^{k\kappa} \over (k|\kappa|)^\mu}
\bigg[{1 \over \Gamma(|\mu|, |\kappa|\b)} 
+ {1 \over \Gamma(|\mu|, |\kappa|b) - \Gamma(|\mu|, |\kappa| \b)}\bigg]$. 
Note: $\gamma(|\mu|, \infty) = \Gamma(|\mu|)$ and 
$\Gamma(|\mu|) - \gamma(|\mu|, |\kappa| \b) = \Gamma(|\mu|, |\kappa| \b)$.

Combining \eqref{joint_last_passage_pdf_spectral_b_1} and \eqref{joint_last_passage_pdf_spectral_b_2} gives the joint density
\begin{align}\label{joint_last_passage_pdf_b_1_CIR}
f_{g^{b}_{\b}(T), X_{b,T}}(t,z;x) 
= 2|\kappa| p^-_{b}(t;x,\b) 
\begin{cases}
\displaystyle \sum\limits_{n=1}^{\infty} e^{-\lambda_{n,k}^- (T-t)}   
\frac{M(- \frac{\lambda_{n,k}^-}{2|\kappa|}+1 ,1-\mu,|\kappa|z ) }
{M_1(- \frac{\lambda_{n,k}^-}{2|\kappa|}+1 ,1-\mu,|\kappa|\b )}
&, z \in (0,\b),
\\
\displaystyle \sum_{n=1}^\infty e^{- \lambda_n^{(\b,b)} (T-t)} 
\frac{\widetilde{S}(-{\lambda_n^{(\b,b)}\over 2\vert\kappa\vert} + 1, 
1 - \mu; \vert\kappa\vert z, \vert\kappa\vert b)}
{\widetilde{S}_1(-{\lambda_n^{(\b,b)}\over 2\vert\kappa\vert} + 1, 
1 - \mu; \vert\kappa\vert \b, \vert\kappa\vert b)}
&, z \in (\b,b),
\end{cases}
\end{align}
for $x,\b\in(0,b)$, and 
\begin{align}\label{joint_last_passage_pdf_b_2_CIR}
f_{g^{b}_{\b}(T), X_{b,T}}(t,z;x) 
= 2|\kappa| p^+_{b}(t;x,\b) 
\begin{cases}
\displaystyle \sum_{n=1}^\infty e^{- \lambda_n^{(b,\b)} (T-t)} 
\frac{\widetilde{S}(-{\lambda_n^{(b,\b)}\over 2\vert\kappa\vert} + 1, 
1 - \mu; \vert\kappa\vert b, \vert\kappa\vert z)}
{\widetilde{S}_1(-{\lambda_n^{(b,\b)}\over 2\vert\kappa\vert} + 1, 
1 - \mu; \vert\kappa\vert b, \vert\kappa\vert \b)}
&, z \in (b,\b),
\\
\displaystyle \sum\limits_{n=1}^{\infty}  e^{-\lambda_{n,k}^+ (T-t)} 
\frac{\widetilde{U}(- \frac{\lambda_{n,k}^+}{2|\kappa|}+1 ,1-\mu,|\kappa|z ) }
{\widetilde{U}_1(- \frac{\lambda_{n,k}^+}{2|\kappa|}+1 ,1-\mu,|\kappa|k )}
&, z \in (\b,\infty),
\end{cases}
\end{align}
for $x,\b\in(b,\infty)$, $t\in (0,T)$. 
We note that using \eqref{trans_PDF_spectral_CIR_b_below} and \eqref{trans_PDF_spectral_CIR_b_above}
within \eqref{marginal_spectral_kill-b_CIR}--\eqref{joint_last_passage_pdf_b_2_CIR} gives double series representations for the marginal and joint densities.

%
%
%
%

\subsection{The Ornstein-Uhlenbeck (OU) Process}\label{subsect_OU}

The regular OU diffusion has state space $\{X_t, t \geq 0\} \in (-\infty, \infty)$ with linear  
drift $\alpha(x) = \gamma_0 -\gamma_1 x$ and constant diffusion coefficient $\nu(x) = \nu_0$, i.e., with SDE 
$d X_t = (\gamma_0 -\gamma_1 X_t)dt + \nu_0\, dW_t.$ 
In what follows we shall simply set $\gamma_0 = 0$. 
[The corresponding expressions for the fundamental solutions, transition PDFs, etc., for $\gamma_0 \ne 0$ follow by applying a simple translation of the spatial variables by an amount $\theta:= \gamma_0/\gamma_1$.]  
Hence, the scale density is $\s(x) = e^{\kappa  x^2/2}$ and the speed density is 
$\m(x) = {\kappa\over \gamma_1}e^{-\kappa  x^2/2}\equiv {2\over \nu_0^2}e^{-\kappa  x^2/2}$, with parameters $\kappa = 2 \gamma_1/ \nu_0^2$, $\nu_0 \ne 0$. Throughout, we shall assume the family of diffusions where $\kappa > 0$ 
($\gamma_1 > 0$). The case where $\kappa < 0$ ($\gamma_1 < 0$) can also be handled in a very similar fashion where a closely related set of fundamental solutions is employed. 
We have a two-parameter ($\kappa,\gamma_1$) family that reduces to a standard one-parameter family by 
setting $\nu_0 = 1$ (i.e., in the special case when $\kappa = 2\gamma_1$). In what follows, the OU diffusion we are specifically considering has generator
\begin{equation}\label{OU_generator}
\mathcal G f(x) := {\gamma_1\over \kappa}f''(x) -\gamma_1 x f'(x) \equiv 
{\nu_0^2\over 2}f''(x) -\gamma_1 x f'(x),\,\,x\in\R.
\end{equation}
Both endpoints, $l=-\infty$ and $r=\infty$, are non-attracting natural and hence conservative. 
Moreover, they are NONOSC, i.e., Spectral Category I, and therefore all spectral expansions with relevance to the OU diffusion are discrete (as in the case of the squared Radial OU diffusions). 

A pair of fundamental solutions to (\ref{eq:phi}) for $x\in\R$ is (e.g., see \cite{BS02,CM21})
\begin{equation}
\varphi^+_\lambda(x)= e^{{\kappa\over 4} x^2}\,D_{-{\lambda\over \gamma_1}}(-\sqrt{\kappa}\,x)\quad\mbox{and}\quad
\varphi^-_\lambda(x)=\varphi^+_\lambda(-x) = e^{{\kappa\over 4} x^2}\,D_{-{\lambda\over \gamma_1}}(\sqrt{\kappa}\,x)
\label{OU_fundamental_funcs}
\end{equation}
with Wronskian factor $w_\lambda = \frac{\sqrt{2\kappa \pi}}{\Gamma(\lambda/\gamma_1)}$. Throughout, $D_\nu(z)$ denotes Whittaker’s parabolic cylinder function of order $\nu$ (e.g., see \cite{AS72}). 
In terms of Kummer $M$ functions,
\begin{align}\label{parabolic_M_funcs}
     D_\nu(z) &= 2^{\frac{\nu}{2}} e^{-\frac{z^2}{4}} \sqrt{\pi} 
\Big[\frac{1}{\Gamma({-\frac{\nu}{2}+\frac{1}{2}})} M\Big(\!-\frac{\nu}{2},\frac{1}{2},\frac{z^2}{2}\Big) 
- \frac{\sqrt{2}z}{\Gamma(-\frac{\nu}{2})} 
M\Big(\!-\frac{\nu}{2}+\frac {1}{2},\frac{3}{2},\frac{z^2}{2}\Big)\Big].
\end{align}

Since $D_\nu(z)$ is entire in $\nu$, it follows that the pair $\varphi_\lambda^\pm(x)$ are entire functions 
in $\lambda$. 
In applying \eqref{u_spectral_4}--\eqref{eigenfunc_regular}, we have $w_{-\lambda_n} = 0 \implies 
\Gamma(-\frac{\lambda_n}{\gamma_1}) = \infty \implies \lambda_n = (n-1)\gamma_1, n=1,\ldots$. 
Hence, $\varphi^-_{-\lambda_n}(x) = e^{{\kappa\over 4} x^2}D_{n-1}(\sqrt{\kappa}\,x) = 
2^{-(n-1)/2}H_{n-1}(\sqrt{\kappa \over 2}\,x)$, where $H_m(z)$ denotes the Hermite polynomial of integer order $m\ge 0$ (e.g., see \cite{AS72}). Note that the symmetry property $H_m(-z) = (-1)^{m}H_m(z)$ leads to 
$\varphi^-_{-\lambda_n}(x) = (-1)^{n-1}\varphi^+_{-\lambda_n}(x)$, 
i.e., $A_n = (-1)^{n-1}$. Moreover, 
$C_n = {\sqrt{2\pi\kappa} \over \gamma_1}{d\over dx}{1\over \Gamma(x)}\vert_{x = -(n-1)} 
= {\sqrt{2\pi\kappa} \over \gamma_1}  (-1)^{n-1} (n-1)!$. Combining gives the known spectral series and its equivalent closed-form Gaussian expression for the transition PDF
\footnote{
The equivalence also follows by setting $z_1 = x\sqrt{\kappa /2}, z_2 = y\sqrt{\kappa /2}, 
w ={1\over 2}e^{-\gamma_1 t}$ in the summation identity 
$\sum_{m=0}^\infty {w^m \over m!} H_m(z_1) H_m(z_2) = {1\over \sqrt{1- 4w^2}}
\exp\big({2w(z_1^2 + z_2^2) - 2 z_1z_2 \over 2w - (2w)^{-1}}\big)$, $z_1,z_2 \in\R$, $\vert w\vert < {1\over 2}$.
}
\begin{align}
p(t;x,y) &= \sqrt{\kappa \over 2\pi}e^{-{\kappa\over 2} y^2}
\sum_{m=0}^\infty {e^{-m \gamma_1 t} \over 2^m m!}
H_{m}(x\sqrt{\kappa /2})H_{m}(y\sqrt{\kappa / 2})
\nonumber \\
&= \sqrt{\kappa \over 2\pi(1 - e^{-2\gamma_1 t})}\,
    \exp \left(- \frac{\kappa (y - e^{-\gamma_1 t} x)^2}{2 (1 - e^{-2\gamma_1 t})}\right), \,\,\,x,y\in \R, \, t>0. 
\label{trans_PDF_spectral_OU}
\end{align}

The discrete distribution of $g_\b(T)$, $X_0=x\in\R$, $\b\in \R$, is given by \eqref{prop_last_time-t-1-limit} where we apply Proposition \ref{prop_spec_first_hit_1}. 
The eigenvalues $\lambda_n \equiv \lambda_{n,\b}^-$, $n\ge 1$, are the positive simple zeros solving 
$\varphi^+_{-\lambda_n}(\b) = 0$, i.e.,
\begin{equation}\label{eigenvalues_OU_1}
D_{\lambda_n/ \gamma_1}(-\sqrt{\kappa} \b) = 0,
\end{equation} 
and $\lambda_n \equiv \lambda_{n,\b}^+$, $n\ge 1$, are the positive simple zeros solving 
$\varphi^-_{-\lambda_n}(\b) = 0$, i.e.,
\begin{equation}\label{eigenvalues_OU_2}
D_{\lambda_n/ \gamma_1}(\sqrt{\kappa} \b) = 0.
\end{equation}
Note the symmetry: $\lambda_{n,-\b}^+ = \lambda_{n,\b}^-$. 
All eigenvalues are readily numerically computed by a bisection method. The leading term asymptotic of $D_\nu(z)$ for large $\nu = \frac{\lambda_n}{\gamma_1}$ gives an initial approximation. 

For large values of $\nu$ (i.e., for large eigenvalues) the numerators and denominators in (\ref{parabolic_M_funcs}) become very large which can lead to numerical overflow. Hence, in a similar manner to \eqref{rescaled_U}, we avoid overflow in ratios of such functions and their derivatives by defining a rescaled function:
\begin{align}\label{parabolic_tilde_M_funcs}
\widetilde{D}_\nu(z) := {\sqrt{\pi}2^{-{\nu\over 2}} \over \Gamma({\nu\over 2})}D_\nu(z) 
&= e^{-\frac{z^2}{4}}
\Big[
{\Gamma({\nu + 1\over 2}) \over \Gamma({\nu\over 2})}
\cos\big({\pi\nu \over 2}\big)
M\Big(\!\!-\frac{\nu}{2},\frac{1}{2},\frac{z^2}{2}\Big) 
+ \frac{\nu}{\sqrt{2}} \sin\big({\pi\nu \over 2}\big) z 
M\Big(\!\!-\frac{\nu}{2}+\frac{1}{2},\frac{3}{2},\frac{z^2}{2}\Big) \Big].
\end{align}
The Gamma reflection formula was used twice in \eqref{parabolic_M_funcs}. 
The only term that can cause overflow is the Gamma function ratio. For smaller values  
(e.g., $\nu = \frac{\lambda_n}{\gamma_1} < 30$) we evaluate the ratio directly by computing each Gamma function without overflow. For larger argument (e.g., $\nu=\frac{\lambda_n}{\gamma_1} \ge 30$) we use the highly accurate Stirling  fomula, i.e., ${\Gamma({\nu + 1\over 2}) \over \Gamma({\nu\over 2})} \approx \sqrt{\nu\over 2e}(1 + {1\over \nu})^{\nu \over 2}$. 
Since the eigenvalues are positive and $2^{-\nu/2}/\Gamma({\nu\over 2}) >0$ for $\nu > 0$, it follows trivially that the above two sets of eigenvalues are equivalently given by 
$\widetilde{D}_{\lambda_{n,\b}^-/\gamma_1}(-\sqrt{\kappa} \b) = 0$ and 
$\widetilde{D}_{\lambda_{n,\b}^+/\gamma_1}(\sqrt{\kappa} \b) = 0$. 
These eigenvalue equations avoid overflow.

In what follows we denote the derivative of the Parabolic cylinder function w.r.t. its order by 
$D^{(1)}_\nu(z) \equiv {\partial \over \partial \nu}D_\nu(z)$ and 
$\widetilde{D}^{(1)}_\nu(z) \equiv {\partial \over \partial \nu}\widetilde{D}_\nu(z)$.
Either $D_\nu^{(1)}(z)$ or $\widetilde{D}_\nu^{(1)}(z)$ is accurately and most efficiently computed by a simple numerical differentiation, e.g., $\widetilde{D}_\nu^{(1)}(z)  
\simeq \frac{\widetilde{D}_{\nu+\delta}(z) - \widetilde{D}_\nu(z)}{\delta}$, with $\delta \simeq 10^{-6}$. 
Alternatively, we can analytically differentiate the expression in \eqref{parabolic_M_funcs} or \eqref{parabolic_tilde_M_funcs} in the same manner that lead to \eqref{Kummer_U_der} and \eqref{Kummer_rescaled_U_der}. In particular, differentiating \eqref{parabolic_tilde_M_funcs} w.r.t. $\nu$ gives
\begin{align}\label{re-scaled Parabolic Cylinder function derivative}
    \widetilde{D}_\nu^{(1)}(z) 
&= {1\over 2} e^{-\frac{z^2}{4}} 
\Bigg\{{\Gamma({\nu + 1\over 2}) \over \Gamma({\nu\over 2})}
\bigg[\bigg(\big[\big(\Psi({\nu + 1\over 2}) - \Psi({\nu\over 2})\big] \cos\big({\pi\nu \over 2}\big) 
-\pi \sin\big({\pi\nu \over 2}\big)\bigg) M\Big(\!\!-\frac{\nu}{2},\frac{1}{2},\frac{z^2}{2}\Big) 
\nonumber \\
&\phantom{{1\over 2} e^{-\frac{z^2}{4}} 
\Bigg\{\,\,\,\,}
- \cos\big({\pi\nu \over 2}\big) M_1\Big(\!\!-\frac{\nu}{2},\frac{1}{2},\frac{z^2}{2}\Big)\bigg]
+ 
\bigg(\sqrt{2}\sin\big({\pi\nu \over 2}\big) + {\pi\nu \over \sqrt{2}} \cos\big({\pi\nu \over 2}\big)
\bigg) z M\Big(\!\!-\frac{\nu}{2},\frac{3}{2},\frac{z^2}{2}\Big) 
\nonumber \\
&\phantom{{1\over 2} e^{-\frac{z^2}{4}} \Bigg\{\,\,\,\,}
- {\nu \over \sqrt{2}}\sin\big({\pi\nu \over 2}\big) z M_1\Big(\!\!-\frac{\nu}{2},\frac{3}{2},\frac{z^2}{2}\Big) 
\Bigg\}.
\end{align}  
Note that the ratio $\frac{\Gamma(\frac{\nu + 1}{2})}{\Gamma(\frac{\nu}{2})}$ is computed as stated above.
By using the definition in \eqref{parabolic_tilde_M_funcs}, differentiating w.r.t. $\nu$ and evaluating at 
$\nu = \lambda_n/\gamma_1$, while invoking the eigenvalue equation for each respective set of eigenvalues, we have:
$D^{(1)}_{\lambda_n/\gamma_1}(-\sqrt{\kappa} \b) = 
{1\over \sqrt{\pi}}\Gamma({\lambda_n\over 2\gamma_1}) 
2^{{\lambda_n \over 2\gamma_1}}\widetilde{D}^{(1)}_{\lambda_n/\gamma_1}(-\sqrt{\kappa} \b)$, 
for $\lambda_n = \lambda_{n,\b}^-$, 
and 
$D^{(1)}_{\lambda_n/\gamma_1}(\sqrt{\kappa} \b) = 
{1\over \sqrt{\pi}}\Gamma({\lambda_n\over 2\gamma_1}) 
2^{{\lambda_n\over 2\gamma_1}}
\widetilde{D}^{(1)}_{\lambda_n/\gamma_1}(\sqrt{\kappa} \b)$, for $\lambda_n = \lambda_{n,\b}^+$. From \eqref{parabolic_tilde_M_funcs}, for the respective eigenvalues 
$\lambda_n = \lambda_{n,\b}^\pm$, $D_{\lambda_n/\gamma_1}(\pm\sqrt{\kappa} x) 
= {1\over \sqrt{\pi}}\Gamma({\lambda_n\over 2\gamma_1}) 
2^{{\lambda_n \over 2\gamma_1}}\widetilde{D}_{\lambda_n/\gamma_1}(\pm\sqrt{\kappa} x)$. Hence, we have the equivalent ratios:
\begin{align}\label{equiv_rescaled_ratio_OU}
\frac{D_{\lambda_{n,\b}^\pm/\gamma_1}(\pm\sqrt{\kappa} x)}
{D^{(1)}_{\lambda_{n,\b}^\pm/\gamma_1}(\pm\sqrt{\kappa} \b)} 
= 
\frac{\widetilde{D}_{\lambda_{n,\b}^\pm/\gamma_1}(\pm\sqrt{\kappa} x)}
{\widetilde{D}^{(1)}_{\lambda_{n,\b}^\pm/\gamma_1}(\pm\sqrt{\kappa} \b)} .
\end{align}
The ratios in the rescaled functions are computed without numerical overflow, as described above. 
%
%

Using \eqref{OU_fundamental_funcs} within \eqref{FHT_eigenfunctions_1} and \eqref{FHT_eigenfunctions_2} gives
\begin{align}\label{OU_Psi_plus_func}
\psi_n^+(x;\b) &= -\gamma_1 e^{{\kappa\over 4}(x^2 - \b^2)}
\frac{D_{\lambda_{n,\b}^-/\gamma_1}(-\sqrt{\kappa} x)}
{D^{(1)}_{\lambda_{n,\b}^-/\gamma_1}(-\sqrt{\kappa} \b)} 
\equiv 
-\gamma_1 e^{{\kappa\over 4}(x^2 - \b^2)}
\frac{\widetilde{D}_{\lambda_{n,\b}^-/\gamma_1}(-\sqrt{\kappa} x)}
{\widetilde{D}^{(1)}_{\lambda_{n,\b}^-/\gamma_1}(-\sqrt{\kappa} \b)}
,
\\
\psi_n^-(x;\b) &= -\gamma_1 e^{{\kappa\over 4}(x^2 - \b^2)}
\frac{D_{\lambda_{n,\b}^+/\gamma_1}(\sqrt{\kappa} x)}
{D^{(1)}_{\lambda_{n,\b}^+/\gamma_1}(\sqrt{\kappa} \b)}
\,\,\,\,\equiv 
-\gamma_1 e^{{\kappa\over 4}(x^2 - \b^2)}
\frac{\widetilde{D}_{\lambda_{n,\b}^+/\gamma_1}(\sqrt{\kappa} x)}
{\widetilde{D}^{(1)}_{\lambda_{n,\b}^+/\gamma_1}(\sqrt{\kappa} \b)}.
\label{OU_Psi_minus_func}
\end{align} 
Hence, \eqref{FHT_prop1_2} gives
\begin{eqnarray}\label{OU_FHT_up_tail}
\P_{x}(T < \Tau^{+}_{\b} < \infty) = -\gamma_1 e^{{\kappa\over 4}(x^2 - \b^2)}
\sum_{n=1}^\infty 
{e^{- \lambda_{n,\b}^- T} \over  \lambda_{n,\b}^-}
\frac{D_{\lambda_{n,\b}^-/\gamma_1}(-\sqrt{\kappa} x)}
{D^{(1)}_{\lambda_{n,\b}^-/\gamma_1}(-\sqrt{\kappa} \b)},\,\,\,x\in (-\infty,\b),
\end{eqnarray}
and \eqref{FHT_prop2_2} gives
\begin{eqnarray}\label{OU_FHT_down_tail}
\P_{x}(T < \Tau^{-}_{\b} < \infty) = -\gamma_1 e^{{\kappa\over 4}(x^2 - \b^2)}
\sum_{n=1}^\infty 
{e^{- \lambda_{n,\b}^+ T} \over  \lambda_{n,\b}^+}
\frac{D_{\lambda_{n,\b}^+/\gamma_1}(\sqrt{\kappa} x)}
{D^{(1)}_{\lambda_{n,\b}^+/\gamma_1}(\sqrt{\kappa} \b)},\,\,\,x\in (\b,\infty).
\end{eqnarray}
Note: both $\pm\infty$ are non-attracting, i.e., 
$\P_x({\Tau}^+_\b = \infty) = \P_x({\Tau}^-_\b = \infty) = 0$. Hence, \eqref{prop_last_time-t-1-limit}  gives
\begin{align}\label{prop_last_time_discrete_OU}
\P_x(g_\b(T) = 0) = -\gamma_1 e^{{\kappa\over 4}(x^2 - \b^2)}
\begin{cases} 
\displaystyle \sum_{n=1}^\infty 
{e^{- \lambda_{n,\b}^- T} \over  \lambda_{n,\b}^-}
\frac{D_{\lambda_{n,\b}^-/\gamma_1}(-\sqrt{\kappa} x)}
{D^{(1)}_{\lambda_{n,\b}^-/\gamma_1}(-\sqrt{\kappa} \b)},&\,\,\,\,\, x\in (-\infty,\b),
\\
\displaystyle \sum_{n=1}^\infty 
{e^{- \lambda_{n,\b}^+ T} \over  \lambda_{n,\b}^+}
\frac{D_{\lambda_{n,\b}^+/\gamma_1}(\sqrt{\kappa} x)}
{D^{(1)}_{\lambda_{n,\b}^+/\gamma_1}(\sqrt{\kappa} \b)},& \,\,\,\,\,x\in (\b,\infty).
\end{cases}
\end{align}
Note that \eqref{equiv_rescaled_ratio_OU} provides the obvious equivalence for the expressions in 
\eqref{OU_FHT_up_tail}--\eqref{prop_last_time_discrete_OU} 
in terms of rescaled functions.

%
%
We now compute the density of $g_\b(T)$ using Proposition 
\ref{last-passage-propn-time-t-spectral} for the purely discrete spectrum. 
By the identity $\frac{\partial}{\partial z} \big(e^{z^2/4}D_{\nu} (z) \big) = \nu e^{z^2/4} D_{\nu-1}(z)$, we have 
$$
\frac{\partial}{\partial x} \bigg(e^{{\kappa \over 4} x^2}D_{\lambda_{n,\b}^\pm\over \gamma_1} (\pm\sqrt{\kappa} x) \bigg)\bigg\vert_{x=\b} = \pm {\sqrt{\kappa}\over \gamma_1}
\lambda_{n,\b}^\pm e^{{\kappa \over 4}\b^2}
D_{{\lambda_{n,\b}^\pm\over \gamma_1} - 1} (\pm\sqrt{\kappa} \b).
$$
Applying this to \eqref{OU_Psi_plus_func}--\eqref{OU_Psi_minus_func} gives the respective expressions
\begin{align}\label{eta_minus_plus_OU}
\eta^\pm(T-t;\b) 
=\pm \sqrt{\kappa} \sum_{n=1}^{\infty}e^{-\lambda_{n,\b}^\mp (T-t)} 
{D_{{\lambda_{n,\b}^\mp\over \gamma_1} - 1} (\mp\sqrt{\kappa} \b) 
\over D^{(1)}_{\lambda_{n,\b}^\mp\over \gamma_1}(\mp\sqrt{\kappa} \b)}.
\end{align}
Note that $\mathcal{S}(l,r;\b) \equiv \mathcal{S}(-\infty,\infty;\b) = 0$ since 
$\ind_{\{E_{-\infty}\}} = \ind_{\{E_\infty\}} = 0$ as both endpoints are non-attracting. 
Since ${1\over \m(\b)\s(\b)} = \gamma_1/\kappa$, \eqref{last-passage-pdf-spectral} gives the density:
\begin{align}\label{last-passage-pdf-spectral_OU}
f_{g_\b(T)}(t;x) = -{\gamma_1\over \sqrt{\kappa}} p(t;x,\b) 
\sum_{n=1}^{\infty}
\bigg[
e^{-\lambda_{n,\b}^+ (T-t)} 
{D_{{\lambda_{n,\b}^+\over \gamma_1} - 1} (\sqrt{\kappa} \b) 
\over D^{(1)}_{\lambda_{n,\b}^+\over \gamma_1}(\sqrt{\kappa} \b)}
+
e^{-\lambda_{n,\b}^- (T-t)} 
{D_{{\lambda_{n,\b}^-\over \gamma_1} - 1} (-\sqrt{\kappa} \b) 
\over D^{(1)}_{\lambda_{n,\b}^-\over \gamma_1}(-\sqrt{\kappa} \b)}
\bigg],
\end{align}
$t\in (0,T)$, $x,\b\in \R$, with $p(t;x,\b)$ given by \eqref{trans_PDF_spectral_OU}. 
Note that \eqref{last-passage-pdf-spectral_OU} is also equivalently expressible in terms of rescaled functions, 
${D_{{\lambda_{n,\b}^\pm \over \gamma_1} - 1} (\pm\sqrt{\kappa} \b) 
\over D^{(1)}_{\lambda_{n,\b}^\pm\over \gamma_1}(\pm\sqrt{\kappa} \b)} = 
{1\over\sqrt{2}}{\Gamma({\lambda_{n,\b}^\pm\over 2\gamma_1} - {1\over 2})\over 
\Gamma({\lambda_{n,\b}^\pm\over 2\gamma_1})}
{\widetilde{D}_{{\lambda_{n,\b}^\pm \over \gamma_1} - 1} (\pm\sqrt{\kappa} \b) 
\over \widetilde{D}^{(1)}_{\lambda_{n,\b}^\pm \over \gamma_1}(\pm\sqrt{\kappa} \b)}$, where 
${1\over\sqrt{2}}{\Gamma((\nu - 1)/2)\over \Gamma(\nu/2)}
\approx {\sqrt{e \nu} \over \nu -1}(1 - {1\over \nu})^{\nu \over 2}$ for large values of 
$\nu = {\lambda_{n,\b}^\pm/ \gamma_1}$.

%
%

The joint density now follows directly by Proposition~\ref{joint_last-passage-propn-time-t-new-formula} where only the discrete series contribute. We have the first hitting time densities
\begin{align}\label{eta_plus_minus_z_k_OU}
&f^\pm(T-t,z;\b) 
= -\gamma_1 e^{{\kappa \over 4}(z^2 - \b^2)}\sum_{n=1}^{\infty}e^{-\lambda_{n,k}^\mp (T-t)} 
{D_{{\lambda_{n,\b}^\mp\over \gamma_1}} (\mp\sqrt{\kappa} z) 
\over D^{(1)}_{\lambda_{n,\b}^\mp\over \gamma_1}(\mp\sqrt{\kappa} \b)}.
\end{align}
Hence, by \eqref{joint_last_passage_pdf_explicit} we have the joint density
\begin{align}\label{joint_last_passage_pdf_OU}
\!\!\!\!\!\!\!\! f_{g_k(T),X_T}(t,z;x)  
    =  -\gamma_1\,  p(t;x,\b) e^{{\kappa \over 4}(\b^2 - z^2)}
\begin{cases}
\displaystyle \sum\limits_{n=1}^{\infty} e^{-\lambda_{n,k}^- (T-t)}   
{D_{{\lambda_{n,\b}^-\over \gamma_1}} (-\sqrt{\kappa} z) 
\over D^{(1)}_{\lambda_{n,\b}^-\over \gamma_1}(-\sqrt{\kappa} \b)}
,  & z\in (-\infty,\b), 
\\
\displaystyle \sum\limits_{n=1}^{\infty}  e^{-\lambda_{n,k}^+ (T-t)} 
{D_{{\lambda_{n,\b}^+\over \gamma_1}} (\sqrt{\kappa} z) 
\over D^{(1)}_{\lambda_{n,\b}^+\over \gamma_1}(\sqrt{\kappa} \b)}
  ,   &z\in (\b,\infty),
\end{cases}
\end{align}
$t\in (0,T)$, $x,\b \in\R$, with $p(t;x,\b)$ given by \eqref{trans_PDF_spectral_OU}. 
Again, we note that \eqref{equiv_rescaled_ratio_OU} provides the equivalent expressions in terms of rescaled functions. The OU process is conservative on $\R$, i.e., 
\eqref{last_density_time-t-conserve_from_joint} holds. This is readily shown by integrating the joint PDF in \eqref{joint_last_passage_pdf_OU}, i.e., each term of the series in \eqref{last-passage-pdf-spectral_OU} is recovered:  
$\int_{-\infty}^\b e^{-{\kappa \over 4}z^2}D_{\lambda_{n,\b}^-/ \gamma_1} (-\sqrt{\kappa} z) dz 
= {1\over \sqrt{\kappa}}
e^{-{\kappa \over 4}\b^2}D_{{\lambda_{n,\b}^-/ \gamma_1} - 1} (-\sqrt{\kappa} \b)$
and 
$\int_\b^{\infty} e^{-{\kappa \over 4}z^2}D_{\lambda_{n,\b}^+/ \gamma_1} (\sqrt{\kappa} z) dz 
= {1\over \sqrt{\kappa}}e^{-{\kappa \over 4}\b^2}D_{{\lambda_{n,\b}^+/ \gamma_1} - 1} (\sqrt{\kappa} \b)$. Here we used the antiderivative 
$\int e^{-{x^2\over 4}}D_\nu (x)dx = -e^{-{x^2\over 4}}D_{\nu - 1} (x)$ and the asymptotic $D_{\nu-1}(x) \sim x^{\nu -1} e^{-{x^2\over 4}} \sim 0$, as $x\to\infty$, for $\nu = {\lambda_{n,\b}^\pm/ \gamma_1} > 0$.

By conservation, $\P_x(g_\b(T) = 0, X_T = \partial^\dagger) = 0$, $x\in \R$. 
The partly discrete joint distribution in \eqref{prop_joint_last_time-t-2} follows directly from \eqref{trans_PDF_spectral_OU_b_below} and \eqref{trans_PDF_spectral_OU_b_above}, written here in terms of the (unscaled) parabolic functions:
\begin{align}\label{prop_joint_last_time-t-2_OU}
&\P_x(g_\b(T) = 0, X_T \in dz) / dz = \sqrt{\kappa \over 2\pi}e^{\frac{\kappa}{4}(x^2-z^2)}
\nonumber \\
& \times
\begin{cases} 
\displaystyle  \sum_{n=1}^{\infty}e^{-\lambda_{n,\b}^- T}\,
\Gamma\Big(\!\!-{\lambda_{n,\b}^-\over \gamma_1}\Big)
\frac{D_\frac{\lambda_{n,\b}^-}{\gamma_1}(\sqrt{\kappa} \b)}
{D^{(1)}_\frac{\lambda_{n,\b}^-}{\gamma_1}(-\sqrt{\kappa} \b)}
D_\frac{\lambda_{n,\b}^-}{\gamma_1}(-\sqrt{\kappa}x)
D_\frac{\lambda_{n,\b}^-}{\gamma_1}(-\sqrt{\kappa}z),
 & x,z \in (-\infty,\b),
\\
\displaystyle  \sum_{n=1}^{\infty}e^{-\lambda_{n,\b}^+ T}\,
\Gamma\Big(\!\!-{\lambda_{n,\b}^+\over \gamma_1}\Big)
\frac{D_\frac{\lambda_{n,\b}^+}{\gamma_1}(-\sqrt{\kappa} \b)}
{D^{(1)}_\frac{\lambda_{n,\b}^+}{\gamma_1}(\sqrt{\kappa} \b)}
D_\frac{\lambda_{n,\b}^+}{\gamma_1}(\sqrt{\kappa}x)
D_\frac{\lambda_{n,\b}^+}{\gamma_1}(\sqrt{\kappa}z),
  & x,z \in (\b,\infty).
\end{cases}
\end{align}

%
%
%
\subsection{OU process killed at either of two interior points}\label{subsect_OU_kill_a_b}
Here we consider the Ornstein-Uhlenbeck process, $X_{(a,b),t}$ with imposed killing at either level $a$ or $b$, $-\infty < a < b < \infty$. As above, we assume $\kappa > 0$ and $\gamma_1 > 0$. 
Throughout, to make expressions more compact, we define the (one-parameter) cylinder function associated to the parabolic cylinder functions:
\begin{align}
S(\nu;x,y) :=   e^{(x^2+y^2)/4}\left[D_{-\nu}(x) D_{-\nu}(-y) - D_{-\nu}(y) D_{-\nu}(-x)\right].
\label{OU_param_func_1}
\end{align}
We also define 
$S_1(-\nu;x,y) := -{\partial \over \partial \nu}S(-\nu;x,y)$ which is most simply computed by numerical differentiation or by differentiating \eqref{OU_param_func_1}, 
\begin{equation*}
S_1(-\nu;x,y)  := e^{(x^2+y^2)/4}\left[D_{\nu}(-x) D^{(1)}_{\nu}(y) 
+ D_{\nu}(y) D^{(1)}_{\nu}(-x) - D_{\nu}(-y)D^{(1)}_{\nu}(x) - D_{\nu}(x)D^{(1)}_{\nu}(-y) 
\right].
\label{OU_param_func_1_der}
\end{equation*}
As in the previous section, to avoid overflow that can arise in the functions  
$D_\nu$ and $D^{(1)}_{\nu}$ for large positive values of $\nu$ (i.e., large eigenvalues) we introduce a rescaled cylinder function,
\begin{align}\label{parabolic_cylinder_rescaled_func}
\widetilde{S}(-\nu;x,y) := {\pi 2^{-\nu} \over \Gamma^2({\nu\over 2})}S(-\nu;x,y) 
= e^{(x^2+y^2)/4}\left[\widetilde{D}_{\nu}(x) \widetilde{D}_{\nu}(-y) 
- \widetilde{D}_{\nu}(y) \widetilde{D}_{\nu}(-x)\right],
\end{align}
where $\widetilde{D}_\nu$ is given by \eqref{parabolic_tilde_M_funcs}. Also of interest is the derivative 
$\widetilde{S}_1(-\nu;x,y) := -{\partial \over \partial \nu}\widetilde{S}(-\nu;x,y)$ which is efficiently computed by numerical differentiation or by differentiating \eqref{parabolic_cylinder_rescaled_func},
\begin{equation*}
\widetilde{S}_1(-\nu;x,y)  := e^{(x^2+y^2)/4}\left[\widetilde{D}_{\nu}(-x) \widetilde{D}^{(1)}_{\nu}(y) 
+ \widetilde{D}_{\nu}(y) \widetilde{D}^{(1)}_{\nu}(-x) 
- \widetilde{D}_{\nu}(-y)\widetilde{D}^{(1)}_{\nu}(x) - \widetilde{D}_{\nu}(x)\widetilde{D}^{(1)}_{\nu}(-y) 
\right].
\label{OU_rescaled_func_1_der}
\end{equation*}

Using \eqref{OU_fundamental_funcs} within \eqref{phi_function} gives 
$\phi(x,y;\lambda) = S({\lambda\over \gamma_1};\sqrt{\kappa}x,\sqrt{\kappa}y)$. The eigenvalues 
$\lambda_n \equiv \lambda_n^{(a,b)}$, $n\geq 1$, are the positive simple zeros solving 
$\phi(a,b;-\lambda_n) = 0$, i.e.,
\footnote{Again, we note that the eigenvalues in \eqref{OU_eigen_ab} are readily computed by using a root finding (e.g., bisection) algorithm. The leading term asymptotics of the parabolic cylinder functions 
for large $\lambda_n$ can also be combined in terms of trigonometric functions whose zeros provide initial estimates for the eigenvalues. It can be shown that $\lambda_n$ grows approximately in proportion to $n^2$ for large $n$.
}
\begin{equation}\label{OU_eigen_ab}
    S(-{\lambda_n\over \gamma_1};\sqrt{\kappa}a,\sqrt{\kappa}b) = 0
\end{equation}
or equivalently $\widetilde{S}(-{\lambda_n\over \gamma_1};\sqrt{\kappa}a,\sqrt{\kappa}b) = 0$. 
Using \eqref{Delta_derivative} gives $\Delta(a,b;\lambda_n) = -{1\over \gamma_1}
S_1(-{\lambda_n\over \gamma_1};\sqrt{\kappa}a,\sqrt{\kappa}b)$. 
Moreover, differentiating \eqref{parabolic_cylinder_rescaled_func} w.r.t. its first argument 
(with $\nu = \lambda_n/ \gamma_1$) and employing \eqref{OU_eigen_ab} gives 
$S_1(-{ \lambda_n\over \gamma_1};\sqrt{\kappa}a,\sqrt{\kappa}b) = {1\over \pi}
2^{\lambda_n/ \gamma_1} \Gamma^2({ \lambda_n\over 2\gamma_1})
\widetilde{S}_1(-{ \lambda_n\over \gamma_1};\sqrt{\kappa}a,\sqrt{\kappa}b)$. 
Hence, using \eqref{spectral_3_product_eigen} within \eqref{u_spectral_3} produces a spectral series for the transition PDF expressible in terms of the cylinder or rescaled cylinder functions:
\begin{align}\label{OU_transition_PDF_a_b}
p_{(a,b)}(t;x,y) &= \sqrt{\kappa \over 2\pi}e^{-\frac{\kappa}{2}y^2} 
\sum_{n=1}^{\infty}e^{-\lambda_n t}\,
\Gamma\Big(\!-\!{\lambda_n \over \gamma_1}\Big)
{S(-{\lambda_n\over \gamma_1};\sqrt{\kappa}a,\sqrt{\kappa}x)
S(-{\lambda_n\over \gamma_1};\sqrt{\kappa}y,\sqrt{\kappa}b)
\over S_1(-{\lambda_n\over \gamma_1};\sqrt{\kappa}a,\sqrt{\kappa}b)}
\nonumber \\
&= \sqrt{\kappa \over 2\pi}e^{-\frac{\kappa}{2}y^2} 
\sum_{n=1}^{\infty}e^{-\lambda_n t}
{2^{\frac{\lambda_n}{\gamma_1}} \, \Gamma^2({\lambda_n \over 2\gamma_1}) 
\over \sin(\pi \frac{\lambda_n}{\gamma_1})\Gamma({\lambda_n \over \gamma_1} + 1)}
{\widetilde{S}(-{\lambda_n\over \gamma_1};\sqrt{\kappa}a,\sqrt{\kappa}x)
\widetilde{S}(-{\lambda_n\over \gamma_1};\sqrt{\kappa}b,\sqrt{\kappa}y)
\over \widetilde{S}_1(-{\lambda_n\over \gamma_1};\sqrt{\kappa}a,\sqrt{\kappa}b)},
\end{align}
where $\lambda_n \equiv \lambda_n^{(a,b)}$, $x,y\in (a,b)$, $t>0$. The Gamma reflection formula and the antisymmetry property, $\widetilde{S}(-{\lambda_n\over \gamma_1};\sqrt{\kappa}y,\sqrt{\kappa}b) = -\widetilde{S}(-{\lambda_n\over \gamma_1};\sqrt{\kappa}b,\sqrt{\kappa}y)$, was used in the second expression. The rescaled functions avoid numerical overflow, where 
${2^\nu\Gamma^2({\nu \over 2}) \over \Gamma(\nu + 1)}$, $\nu = \frac{\lambda_n}{\gamma_1}$, is evaluated directly for smaller values, $\nu < 30$, and by using the very accurate asymptotic formula ${2^\nu\Gamma^2({\nu \over 2}) \over \Gamma(\nu + 1)} \simeq \sqrt{\pi}({\nu\over 2})^{-3/2}$, for larger values, $\nu \ge 30$.

From \eqref{FHT_eigenfunctions_ab} we have
\begin{align}\label{FHT_OU_eigenfunctions_ab_plus}
\psi_n^+(x;a,b) &= \gamma_1 
{S(-{\lambda_n\over \gamma_1};\sqrt{\kappa}a,\sqrt{\kappa}x)
\over S_1(-{\lambda_n\over \gamma_1};\sqrt{\kappa}a,\sqrt{\kappa}b)} 
\equiv 
\gamma_1 
{\widetilde{S}(-{\lambda_n\over \gamma_1};\sqrt{\kappa}a,\sqrt{\kappa}x)
\over \widetilde{S}_1(-{\lambda_n\over \gamma_1};\sqrt{\kappa}a,\sqrt{\kappa}b)},\,\,\,\,\,\,\lambda_n \equiv \lambda_n^{(a,b)},
\\
\psi_n^-(x;a,b) &=  \gamma_1 
{S(-{\lambda_n\over \gamma_1};\sqrt{\kappa}x,\sqrt{\kappa}b)
\over S_1(-{\lambda_n\over \gamma_1};\sqrt{\kappa}a,\sqrt{\kappa}b)}
\equiv 
\gamma_1 
{\widetilde{S}(-{\lambda_n\over \gamma_1};\sqrt{\kappa}x,\sqrt{\kappa}b)
\over \widetilde{S}_1(-{\lambda_n\over \gamma_1};\sqrt{\kappa}a,\sqrt{\kappa}b)}, \,\,\,\lambda_n \equiv \lambda_n^{(a,b)}.
\label{FHT_OU_eigenfunctions_ab_minus}
\end{align}
Using \eqref{FHT_OU_eigenfunctions_ab_plus}--\eqref{FHT_OU_eigenfunctions_ab_minus}, adapted to the respective intervals $(a,\b)$ and $(\b,b)$, within 
\eqref{prop_last_time-CDF_discrete_killed_ab}--\eqref{joint_last_ab_doubly_defective_new} produces explicit  series for the discrete parts of the marginal and joint distributions:
\begin{align}\label{last_time-CDF_discrete_killed_ab_OU}
\P_x(g^{(a,b)}_{\b}(T) = 0) 
=
\begin{cases}
\displaystyle R(x;a,\b) + \gamma_1 
\sum_{n=1}^\infty {e^{- \lambda_n^{(a,\b)} T} \over \lambda_n^{(a,\b)}}
\frac{S(-{\lambda_n^{(a,\b)}\over \gamma_1};\sqrt{\kappa}a,\sqrt{\kappa}x)}
{S_1(-{\lambda_n^{(a,\b)}\over \gamma_1};\sqrt{\kappa}a,\sqrt{\kappa}\b)}
, x \in \!(a,\b),
\\
\displaystyle R(x;\b,b) + \gamma_1 
\sum_{n=1}^\infty {e^{- \lambda_n^{(\b,b)} T} \over \lambda_n^{(\b,b)}}
\frac{S(-{\lambda_n^{(\b,b)}\over \gamma_1};\sqrt{\kappa}x,\sqrt{\kappa}b)}
{S_1(-{\lambda_n^{(\b,b)}\over \gamma_1};\sqrt{\kappa}\b,\sqrt{\kappa} b)}
, x \in \!(\b,b),
\end{cases}
\end{align}

\begin{align}\label{joint_last_ab_doubly_defective_OU}
\P_x(g^{(a,b)}_{\b}(T) = 0, X_{(a,b),T} = \partial^\dagger) 
= 
\begin{cases}
\displaystyle R(x;a,\b) - \gamma_1 
\sum_{n=1}^\infty {e^{- \lambda_n^{(a,\b)} T} \over \lambda_n^{(a,\b)}}
\frac{S(-{\lambda_n^{(a,\b)}\over \gamma_1};\sqrt{\kappa}x,\sqrt{\kappa}\b)}
{S_1(-{\lambda_n^{(a,\b)}\over \gamma_1};\sqrt{\kappa}a,\sqrt{\kappa}\b)}
, x \in \!(a,\b),
\\
\displaystyle R(x;\b,b) - \gamma_1 
\sum_{n=1}^\infty {e^{- \lambda_n^{(\b,b)} T} \over \lambda_n^{(\b,b)}}
\frac{S(-{\lambda_n^{(\b,b)}\over \gamma_1};\sqrt{\kappa}\b,\sqrt{\kappa}x)}
{S_1(-{\lambda_n^{(\b,b)}\over \gamma_1};\sqrt{\kappa}\b,\sqrt{\kappa} b)}
, x \in \!(\b,b),
\end{cases}
\end{align}
where $R(x;a,\b) := {\int_x^\b e^{{\kappa\over 2}y^2} dy \over \int_a^\b e^{{\kappa\over 2}y^2} dy}$, 
$R(x;\b,b) := {\int_\b^x e^{{\kappa\over 2}y^2} dy \over \int_\b^b e^{{\kappa\over 2}y^2} dy}$; 
$\lambda_n^{(a,\b)}$ solve $S(-{\lambda_n^{(a,\b)}\over \gamma_1};\sqrt{\kappa}a,\sqrt{\kappa}\b) = 0$, equivalently 
$\widetilde{S}(-{\lambda_n^{(a,\b)}\over \gamma_1};\sqrt{\kappa}a,\sqrt{\kappa}\b) = 0$, and 
$\lambda_n^{(\b,b)}$ solve $S(-{\lambda_n^{(\b,b)}\over \gamma_1};\sqrt{\kappa}\b,\sqrt{\kappa}b) = 0$, equivalently 
$\widetilde{S}(-{\lambda_n^{(\b,b)}\over \gamma_1};\sqrt{\kappa}\b,\sqrt{\kappa}b) = 0$. 
The ratios $S/S_1$ in \eqref{last_time-CDF_discrete_killed_ab_OU}--\eqref{joint_last_ab_doubly_defective_OU} are equally expressed using the respective ratios 
$\widetilde{S}/\widetilde{S}_1$. 
By using \eqref{OU_transition_PDF_a_b} on the respective intervals $(a,\b)$ and $(\b,b)$, 
\eqref{prop_joint_last_time_kill_ab-t-3-zero} gives us the partly discrete portion of the joint distribution 
(expressed more compactly here in terms of the cylinder functions $S$ and $S_1$):
\begin{align}\label{joint_last_ab_partly_defective_OU}
&\P_x(g^{(a,b)}_{\b}(T) = 0, X_{(a,b),T} \in d z) / dz
= \sqrt{\kappa \over 2\pi}e^{-\frac{\kappa}{2}z^2} 
\nonumber \\
&\times 
\begin{cases}
\displaystyle \sum_{n=1}^{\infty}e^{-\lambda_n^{(a,\b)} T}\,
\Gamma\big(\!-\!{\lambda_n^{(a,\b)} \over \gamma_1}\big)
{S(-{\lambda_n^{(a,\b)}\over \gamma_1};\sqrt{\kappa}a,\sqrt{\kappa}x)
S(-{\lambda_n^{(a,\b)}\over \gamma_1};\sqrt{\kappa}z,\sqrt{\kappa}\b)
\over S_1(-{\lambda_n^{(a,\b)}\over \gamma_1};\sqrt{\kappa}a,\sqrt{\kappa}\b)}, &\,\,\,x,z \in (a,\b),
\\
\displaystyle \sum_{n=1}^\infty \!e^{-\lambda_n^{(\b,b)} T} 
\Gamma\big(\!-\!{\lambda_n^{(\b,b)} \over \gamma_1}\big)
{S(-{\lambda_n^{(\b,b)}\over \gamma_1};\sqrt{\kappa}\b,\sqrt{\kappa}x)
S(-{\lambda_n^{(\b,b)}\over \gamma_1};\sqrt{\kappa}z,\sqrt{\kappa}b)
\over S_1(-{\lambda_n^{(\b,b)}\over \gamma_1};\sqrt{\kappa}\b,\sqrt{\kappa}b)}, &\,\,\,x,z \in (\b,b).
\end{cases}
\end{align}

We now employ Proposition~\ref{joint_last-passage-propn-time-t_ab-new-version}. By directly using 
$\Delta(a,\b;\lambda_n)$ and $\varphi^-_{-\lambda_n}$, for $\lambda_n = \lambda_n^{(a,\b)}$, and 
$\Delta(\b,b;\lambda_n)$ and $\varphi^-_{-\lambda_n}$, for $\lambda_n = \lambda_n^{(\b,b)}$, within \eqref{psi_hat_derivatives} gives
\begin{align}\label{FHT_OU_hat_plus_ak}
\hat\psi^+_{n}(a,\b) 
&= 
- \gamma_1 {\sqrt{2\pi\kappa} \over  \Gamma(-{\lambda_n \over \gamma_1})}
{e^{\frac{\kappa}{4}(a^2 - \b^2)} D_{\lambda_n \over \gamma_1}(\sqrt{\kappa}a) \over 
D_{\lambda_n \over \gamma_1}(\sqrt{\kappa}\b)
\,S_1(-{\lambda_n \over \gamma_1};\sqrt{\kappa}a,\sqrt{\kappa}\b)}
\nonumber \\
&= 
\gamma_1 \sqrt{2\pi\kappa} \sin\Big(\pi {\lambda_n \over \gamma_1}\Big)
{\Gamma({\lambda_n \over \gamma_1} + 1) \over 
2^{\frac{\lambda_n}{\gamma_1}} \, \Gamma^2({\lambda_n \over 2\gamma_1})}
{e^{\frac{\kappa}{4}(a^2 - \b^2)} \widetilde{D}_{\lambda_n \over \gamma_1}(\sqrt{\kappa}a) \over 
\widetilde{D}_{\lambda_n \over \gamma_1}(\sqrt{\kappa}\b)
\,\widetilde{S}_1(-{\lambda_n \over \gamma_1};\sqrt{\kappa}a,\sqrt{\kappa}\b)},\,\,\,\lambda_n \equiv \lambda_n^{(a,\b)},
\end{align}
\begin{align}\label{FHT_OU_hat_minus_kb}
\hat\psi^-_{n}(\b,b)  
&= 
- \gamma_1 {\sqrt{2\pi\kappa} \over  \Gamma(-{\lambda_n \over \gamma_1})}
{e^{\frac{\kappa}{4}(b^2 - \b^2)} D_{\lambda_n \over \gamma_1}(\sqrt{\kappa}b) \over 
D_{\lambda_n \over \gamma_1}(\sqrt{\kappa}\b)
\,S_1(-{\lambda_n \over \gamma_1};\sqrt{\kappa}\b,\sqrt{\kappa}b)}
\nonumber \\
&= 
\gamma_1 \sqrt{2\pi\kappa} \sin\Big(\pi {\lambda_n \over \gamma_1}\Big)
{\Gamma({\lambda_n \over \gamma_1} + 1) \over 
2^{\frac{\lambda_n}{\gamma_1}} \, \Gamma^2({\lambda_n \over 2\gamma_1})}
{e^{\frac{\kappa}{4}(b^2 - \b^2)} \widetilde{D}_{\lambda_n \over \gamma_1}(\sqrt{\kappa}b) \over 
\widetilde{D}_{\lambda_n \over \gamma_1}(\sqrt{\kappa}\b)
\,\widetilde{S}_1(-{\lambda_n \over \gamma_1};\sqrt{\kappa}\b,\sqrt{\kappa}b)},\,\,\,\lambda_n \equiv \lambda_n^{(\b,b)}.
\end{align}
In the above second expressions we used the Gamma reflection formula, the rescaled parabolic functions and the derivative of the rescaled cylinder function. These expressions avoid numerical overflow where  
${\Gamma(\nu + 1) \over 2^{\nu} \Gamma^2(\nu/2)}$, $\nu = \frac{\lambda_n}{\gamma_1}$, is computed as described above for its reciprocal. 
Substituting the first expressions in \eqref{FHT_OU_hat_plus_ak}--\eqref{FHT_OU_hat_minus_kb} into 
\eqref{last_passage_pdf_discrete_spec} gives the more compact expression for the marginal density:
\begin{align}\label{last_passage_pdf_explicit_ab_OU}
&f_{g^{(a,b)}_{\b}(T)}(t;x) = p_{(a,b)}(t;x,\b) 
\nonumber \\
&\times \bigg[ {\gamma_1 \over \kappa} e^{{\kappa\over 2} \b^2} \mathcal{S}(a,b;\b) 
 - \gamma_1^2 \sqrt{2\pi\over \kappa}e^{{\kappa\over 4} \b^2}
\bigg(e^{{\kappa\over 4} a^2}
\sum_{n=1}^{\infty} {e^{-\lambda_n^{(a,\b)} (T-t)} \over \lambda_n^{(a,\b)}}
{D_{\lambda_n^{(a,\b)} \over \gamma_1}(\sqrt{\kappa}a) \Big/ 
D_{\lambda_n^{(a,\b)} \over \gamma_1}(\sqrt{\kappa}\b) \over
\Gamma(-{\lambda_n^{(a,\b)} \over \gamma_1}) S_1(-{\lambda_n^{(a,\b)} \over \gamma_1};\sqrt{\kappa}a,\sqrt{\kappa}\b)}
\nonumber \\
&\phantom{\times \bigg[ {\gamma_1 \over \kappa} e^{{\kappa\over 2} \b^2} \mathcal{S}(a,b;\b) 
 - \gamma_1^2 \sqrt{2\pi\over \kappa} \bigg(}
\,\,\,\,\,\, + e^{{\kappa\over 4} b^2}
\sum_{n=1}^{\infty} {e^{-\lambda_n^{(\b,b)} (T-t)} \over \lambda_n^{(\b,b)}}
{D_{\lambda_n^{(\b,b)} \over \gamma_1}(\sqrt{\kappa}b) \Big/ 
D_{\lambda_n^{(\b,b)} \over \gamma_1}(\sqrt{\kappa}\b) \over
\Gamma(-{\lambda_n^{(\b,b)} \over \gamma_1}) S_1(-{\lambda_n^{(\b,b)} \over \gamma_1};\sqrt{\kappa}\b,\sqrt{\kappa} b)}
\bigg) \bigg],
\end{align}
$t\in (0,T)$, $x,\b\in (a,b)$, where 
$\mathcal{S}(a,b;\b) = \Big(\int_a^\b e^{{\kappa\over 2}y^2} dy\Big)^{-1} + 
\Big(\int_\b^b e^{{\kappa\over 2}y^2} dy\Big)^{-1}$.

By using \eqref{FHT_OU_eigenfunctions_ab_plus}--\eqref{FHT_OU_eigenfunctions_ab_minus}, for the respective intervals $(a,\b)$ and $(\b,b)$, \eqref{joint_last_passage_pdf_explicit_ab_new} gives the joint density: 
\begin{align}\label{joint_last_passage_pdf_explicit_ab_OU}
f_{\!g^{(a,b)}_{\b}(T), X_{(a,b),T}}\!(t,z;x) 
= p_{(a,b)}(t;x,\b)  e^{{\kappa\over 2}(\b^2 - z^2)} \gamma_1
\left\{
\begin{array}{lr}
         \displaystyle  \sum_{n=1}^{\infty}e^{-\lambda_n^{(a,\b)} (T-t)}
{S(-{\lambda_n^{(a,\b)}\over \gamma_1};\sqrt{\kappa}a,\sqrt{\kappa}z)
\over S_1(-{\lambda_n^{(a,\b)}\over \gamma_1};\sqrt{\kappa}a,\sqrt{\kappa}b)}, &  z\in (a,\b), 
\\
\displaystyle \sum_{n=1}^{\infty} e^{-\lambda_n^{(\b,b)} (T-t)}
{S(-{\lambda_n^{(\b,b)}\over \gamma_1};\sqrt{\kappa}z,\sqrt{\kappa}b)
\over S_1(-{\lambda_n^{(\b,b)}\over \gamma_1};\sqrt{\kappa}\b,\sqrt{\kappa} b)}, &  z\in (\b,b), 
\end{array}
\right.
\end{align}
$t\in (0,T)$, $x,\b\in (a,b)$. We note that $p_{(a,b)}(t;x,\b)$ is given by \eqref{OU_transition_PDF_a_b} and hence \eqref{last_passage_pdf_explicit_ab_OU} and \eqref{joint_last_passage_pdf_explicit_ab_OU} also represent double series for the marginal and joint densities. 

%
%
%
%
\subsection{OU process killed at one interior point}\label{subsect_OU_kill_b}
We now consider the OU process, $X_{b,t}$, killed at $b\in (0,\infty)$, 
for the two cases: (i) $x,\b\in (-\infty,b)$ and (ii) $x,\b\in (b,\infty)$, where $\kappa > 0$ and $\gamma_1 > 0$.
%
%
The transition PDF for case (i) is given by \eqref{u_spectral_1} with product eigenfunction in \eqref{spectral_1_product_eigen} computed using \eqref{eigenfunc1}. In particular, we have the equivalent series:
\begin{align}\label{trans_PDF_spectral_OU_b_below}
p_b^-(t;x,y) &= \sqrt{\kappa \over 2\pi}e^{\frac{\kappa}{4}(x^2-y^2)} 
\sum_{n=1}^{\infty}e^{-\lambda_{n,b}^- t}\,
\Gamma\Big(\!\!-{\lambda_{n,b}^-\over \gamma_1}\Big)
\frac{D_\frac{\lambda_{n,b}^-}{\gamma_1}(\sqrt{\kappa} b)}
{D^{(1)}_\frac{\lambda_{n,b}^-}{\gamma_1}(-\sqrt{\kappa} b)}
D_\frac{\lambda_{n,b}^-}{\gamma_1}(-\sqrt{\kappa}x)
D_\frac{\lambda_{n,b}^-}{\gamma_1}(-\sqrt{\kappa}y)
\nonumber \\
&= \sqrt{\kappa \over 2\pi}e^{\frac{\kappa}{4}(x^2-y^2)} 
\sum_{n=1}^{\infty}e^{-\lambda_{n,b}^- t}
{2^{\frac{\lambda_{n,b}^-}{\gamma_1}} \, \Gamma^2({\lambda_{n,b}^-\over 2\gamma_1}) 
\over \sin(-\pi \frac{\lambda_{n,b}^-}{\gamma_1})\Gamma({\lambda_{n,b}^-\over \gamma_1} + 1)}
\frac{\widetilde{D}_\frac{\lambda_{n,b}^-}{\gamma_1}\!(\sqrt{\kappa} b)}
{\widetilde{D}^{(1)}_\frac{\lambda_{n,b}^-}{\gamma_1}\!(-\sqrt{\kappa} b)}
\widetilde{D}_\frac{\lambda_{n,b}^-}{\gamma_1}\!(-\sqrt{\kappa}x)
\widetilde{D}_\frac{\lambda_{n,b}^-}{\gamma_1}\!(-\sqrt{\kappa}y),
\end{align}
$x,y\in (-\infty,b)$, $t>0$, where $\lambda_{n,b}^-$ solve 
$D_{\lambda_{n,b}^-/\gamma_1}(-\sqrt{\kappa}b)=0$ or equivalently 
$\widetilde{D}_{\lambda_{n,b}^-/\gamma_1}(-\sqrt{\kappa}b)=0$. 
The second expression arises from \eqref{equiv_rescaled_ratio_OU}, the definition in \eqref{parabolic_tilde_M_funcs} and the Gamma reflection formula. This avoids numerical overflow where the ratio 
${2^\nu\Gamma^2({\nu \over 2}) \over \Gamma(\nu + 1)}$, $\nu = \frac{\lambda_{n,b}^-}{\gamma_1}$, is evaluated as discussed above.

By a similar derivation, the transition PDF for case (ii) follows by \eqref{u_spectral_2}:
\footnote{We also observe the symmetry of the OU process when reflected about the origin. Namely, $p_b^+(t;x,y)$ obtains simply from $p_b^-(t;x,y)$ upon sending $x\to -x, y\to -y, b \to -b$ and where 
$\lambda_{n,-b}^- = \lambda_{n,b}^+$.
}
\begin{align}\label{trans_PDF_spectral_OU_b_above}
p_b^+(t;x,y) &= \sqrt{\kappa \over 2\pi}e^{\frac{\kappa}{4}(x^2-y^2)} 
\sum_{n=1}^{\infty}e^{-\lambda_{n,b}^+ t}\,
\Gamma\Big(\!\!-{\lambda_{n,b}^+\over \gamma_1}\Big)
\frac{D_\frac{\lambda_{n,b}^+}{\gamma_1}(-\sqrt{\kappa} b)}
{D^{(1)}_\frac{\lambda_{n,b}^+}{\gamma_1}(\sqrt{\kappa} b)}
D_\frac{\lambda_{n,b}^+}{\gamma_1}(\sqrt{\kappa}x)
D_\frac{\lambda_{n,b}^+}{\gamma_1}(\sqrt{\kappa}y)
\nonumber \\
&= \sqrt{\kappa \over 2\pi}e^{\frac{\kappa}{4}(x^2-y^2)} 
\sum_{n=1}^{\infty}e^{-\lambda_{n,b}^+ t}
{2^{\frac{\lambda_{n,b}^+}{\gamma_1}}\, \Gamma^2({\lambda_{n,b}^+\over 2\gamma_1}) 
\over \sin(-\pi \frac{\lambda_{n,b}^+}{\gamma_1})\Gamma({\lambda_{n,b}^+\over \gamma_1} + 1)}
\frac{\widetilde{D}_\frac{\lambda_{n,b}^+}{\gamma_1}\!(-\sqrt{\kappa} b)}
{\widetilde{D}^{(1)}_\frac{\lambda_{n,b}^+}{\gamma_1}\!(\sqrt{\kappa} b)}
\widetilde{D}_\frac{\lambda_{n,b}^+}{\gamma_1}\!(\sqrt{\kappa}x)
\widetilde{D}_\frac{\lambda_{n,b}^+}{\gamma_1}\!(\sqrt{\kappa}y),
\end{align}
$x,y\in (b,\infty)$, $t>0$. The eigenvalues $\lambda_{n,b}^+$ solve $D_{\lambda_{n,b}^+/\gamma_1}(\sqrt{\kappa}b)=0$ or equivalently 
$\widetilde{D}_{\lambda_{n,b}^+/\gamma_1}(\sqrt{\kappa}b)=0$. 

Consider first the defective portions of the distribution. 
For case (i), we use \eqref{prop_joint_last_b-below_pmf}--\eqref{doubly_defective_kill_up_b}. Hence, by the first line of \eqref{prop_last_time_discrete_OU} and the second line of 
\eqref{last_time-CDF_discrete_killed_ab_OU}, we have
\begin{align}\label{last_time-CDF_discrete_killed_below_b_OU}
\P_x(g^b_{\b}(T) = 0)
=\begin{cases}
\displaystyle  -\gamma_1 e^{{\kappa\over 4}(x^2 - \b^2)} \sum_{n=1}^\infty 
{e^{- \lambda_{n,\b}^- T} \over  \lambda_{n,\b}^-}
\frac{D_{\lambda_{n,\b}^-/\gamma_1}(-\sqrt{\kappa} x)}
{D^{(1)}_{\lambda_{n,\b}^-/\gamma_1}(-\sqrt{\kappa} \b)},& x\in (-\infty,\b),
\\
\displaystyle R(x;\b,b) + \gamma_1 
\sum_{n=1}^\infty {e^{- \lambda_n^{(\b,b)} T} \over \lambda_n^{(\b,b)}}
\frac{S(-{\lambda_n^{(\b,b)}\over \gamma_1};\sqrt{\kappa}x,\sqrt{\kappa}b)}
{S_1(-{\lambda_n^{(\b,b)}\over \gamma_1};\sqrt{\kappa}\b,\sqrt{\kappa} b)},
& x \in \!(\b,b).
\end{cases}
\end{align}
By the transition densities $p_k^-(T;x,z)$ and $p_{(\b,b)}(T;x,z)$, i.e., \eqref{trans_PDF_spectral_OU_b_below} with level $b$ replaced by $\b$ and \eqref{OU_transition_PDF_a_b} with $a$ replaced by $\b$, we have
\begin{align}
&\P_x(g^{b}_{\b}(T) = 0, X_{b,T} \in dz) / dz  = \sqrt{\kappa \over 2\pi}
\label{prop_joint_last_b-below_pmf_OU}
\\
&\times \begin{cases} 
\displaystyle e^{\frac{\kappa}{4}(x^2-z^2)}
\sum_{n=1}^{\infty}e^{-\lambda_{n,\b}^- T}\,
\Gamma\Big(\!\!-{\lambda_{n,\b}^-\over \gamma_1}\Big)
\frac{D_\frac{\lambda_{n,\b}^-}{\gamma_1}(\sqrt{\kappa} \b)}
{D^{(1)}_\frac{\lambda_{n,\b}^-}{\gamma_1}(-\sqrt{\kappa} \b)}
D_\frac{\lambda_{n,\b}^-}{\gamma_1}(-\sqrt{\kappa}x)
D_\frac{\lambda_{n,\b}^-}{\gamma_1}(-\sqrt{\kappa}z),
 & x,z \in (-\infty,\b),
\\
\displaystyle e^{-\frac{\kappa}{2}z^2} 
\sum_{n=1}^{\infty}e^{-\lambda_n^{(\b,b)} T}\,
\Gamma\Big(\!-\!{\lambda_n^{(\b,b)} \over \gamma_1}\Big)
{S(-{\lambda_n^{(\b,b)}\over \gamma_1};\sqrt{\kappa}\b,\sqrt{\kappa}x)
S(-{\lambda_n^{(\b,b)}\over \gamma_1};\sqrt{\kappa}z,\sqrt{\kappa}b)
\over S_1(-{\lambda_n^{(\b,b)}\over \gamma_1};\sqrt{\kappa}\b,\sqrt{\kappa}b)}, & x,z \in (\b,b).
\nonumber
\end{cases}
\end{align}
The jointly discrete distribution is given by \eqref{doubly_defective_kill_up_b}. 
For $x \in (-\infty,\b)$, $\P_x(g^{b}_{\b}(T) = 0, X_{b,T} = \partial^\dagger) = 0$ since the lower boundary $l= -\infty$ is conservative. For $x \in (\b,b)$, $\P_x(g^{b}_{\b}(T) = 0, X_{b,T} = \partial^\dagger) = \P_x(\T^+_b (\b) \le T)$ is given by the second expression in \eqref{joint_last_ab_doubly_defective_OU}.

For case (ii), the defective portions are given by \eqref{prop_joint_last_b-above_pmf}--\eqref{doubly_defective_kill_down_b}. Using the first line of 
\eqref{last_time-CDF_discrete_killed_ab_OU}, with $a$ replaced by $b$, and the second line of 
\eqref{prop_last_time_discrete_OU}, gives
\begin{align}\label{last_time-CDF_discrete_killed_above_b_OU}
\P_x(g^b_{\b}(T) = 0) 
=\begin{cases}
\displaystyle R(x;b,\b) + \gamma_1 
\sum_{n=1}^\infty {e^{- \lambda_n^{(b,\b)} T} \over \lambda_n^{(b,\b)}}
\frac{S(-{\lambda_n^{(b,\b)}\over \gamma_1};\sqrt{\kappa}b,\sqrt{\kappa}x)}
{S_1(-{\lambda_n^{(b,\b)}\over \gamma_1};\sqrt{\kappa}b,\sqrt{\kappa}\b)}, & x \in (b,\b),
\\
\displaystyle -\gamma_1 e^{{\kappa\over 4}(x^2 - \b^2)} 
\sum_{n=1}^\infty 
{e^{- \lambda_{n,\b}^+ T} \over  \lambda_{n,\b}^+}
\frac{D_{\lambda_{n,\b}^+/\gamma_1}(\sqrt{\kappa} x)}
{D^{(1)}_{\lambda_{n,\b}^+/\gamma_1}(\sqrt{\kappa} \b)}, & x\in (\b,\infty).
\end{cases}
\end{align}
Again, by direct use of $p_{(b,\b)}(T;x,z)$ and $p_k^+(T;x,z)$, i.e., 
\eqref{OU_transition_PDF_a_b} with interval $(a,b)$ replaced by $(b,\b)$ and  \eqref{trans_PDF_spectral_OU_b_above} with level $b$ replaced by $\b$, we have
\begin{align}
&\P_x(g^{b}_{\b}(T) = 0, X_{b,T} \in dz) / dz = \sqrt{\kappa \over 2\pi}
\nonumber \\
&\times \begin{cases}
\displaystyle e^{-\frac{\kappa}{2}z^2} 
\sum_{n=1}^{\infty}e^{-\lambda_n^{(b,\b)} T}\,
\Gamma\Big(\!-\!{\lambda_n^{(b,\b)} \over \gamma_1}\Big)
{S(-{\lambda_n^{(b,\b)}\over \gamma_1};\sqrt{\kappa}b,\sqrt{\kappa}x)
S(-{\lambda_n^{(b,\b)}\over \gamma_1};\sqrt{\kappa}z,\sqrt{\kappa}\b)
\over S_1(-{\lambda_n^{(b,\b)}\over \gamma_1};\sqrt{\kappa}b,\sqrt{\kappa}\b)} , x,z \in (b,\b),
\\
\displaystyle e^{\frac{\kappa}{4}(x^2-z^2)} 
\sum_{n=1}^{\infty}e^{-\lambda_{n,\b}^+ T}\,
\Gamma\Big(\!\!-{\lambda_{n,\b}^+\over \gamma_1}\Big)
\frac{D_\frac{\lambda_{n,\b}^+}{\gamma_1}(-\sqrt{\kappa} \b)}
{D^{(1)}_\frac{\lambda_{n,\b}^+}{\gamma_1}(\sqrt{\kappa} \b)}
D_\frac{\lambda_{n,\b}^+}{\gamma_1}(\sqrt{\kappa}x)
D_\frac{\lambda_{n,\b}^+}{\gamma_1}(\sqrt{\kappa}z), x,z \in (\b,\infty).
\label{prop_joint_last_b-above_pmf_OU}
\end{cases}
\end{align}
The jointly discrete distribution is given by \eqref{doubly_defective_kill_down_b}. 
For $x\in (\b,\infty)$, $\P_x(g^{b}_{\b}(T) = 0, X_{b,T} = \partial^\dagger)=0$ since $r=\infty$ 
is a conservative boundary. For $x \in (b,\b)$, $\P_x(g^{b}_{\b}(T) = 0, X_{b,T} = \partial^\dagger) = 
\P_x(\T^-_b (\b) \le T)$, as given by the first expression in \eqref{joint_last_ab_doubly_defective_OU} 
with $a$ replaced by $b$.

The continuous distributions are now simply given by applying 
Proposition~\ref{marginal_joint_last-passage-propn-kill-b} with the direct use of the scale function, 
\eqref{eta_minus_plus_OU}, \eqref{eta_plus_minus_z_k_OU}, 
\eqref{FHT_OU_eigenfunctions_ab_plus}, \eqref{FHT_OU_eigenfunctions_ab_minus}, and \eqref{FHT_OU_hat_plus_ak}--\eqref{FHT_OU_hat_minus_kb} on the appropriate intervals 
$(\b,b)$ and $(b,\b)$. Note: $\ind_{\{E_{-\infty}\}} = \ind_{\{E_\infty\}} = 0$. 
Combining \eqref{marginal_spectral_kill-b_1} and \eqref{marginal_spectral_kill-b_2} gives the marginal density:
\begin{align}\label{last_passage_pdf_explicit_b_OU}
&f_{g^b_{\b}(T)}(t;x) 
 \\ \nonumber
&= \begin{cases}
\displaystyle p_b^-(t;x,\b) {\gamma_1 \over \kappa}
\bigg[
{e^{{\kappa\over 2} \b^2} \over \mathcal{S}[\b,b]} 
 - \sqrt{\kappa}
\sum_{n=1}^{\infty} e^{-\lambda_{n,\b}^- (T-t)}
D_{{\lambda_{n,\b}^- \over \gamma_1} - 1}(-\sqrt{\kappa}\b) 
\Big/ D^{(1)}_{\lambda_{n,\b}^- \over \gamma_1}(-\sqrt{\kappa}\b)
 \\
\phantom{p_b^-(t;x,\b) {\gamma_1 \over \kappa}}
\displaystyle - e^{{\kappa\over 4} (b^2+\b^2)}\gamma_1\sqrt{2\pi\kappa}
\sum_{n=1}^{\infty} {e^{-\lambda_n^{(\b,b)} (T-t)} \over \lambda_n^{(\b,b)}}
{D_{\lambda_n^{(\b,b)} \over \gamma_1}(\sqrt{\kappa}b) \Big/ 
D_{\lambda_n^{(\b,b)} \over \gamma_1}(\sqrt{\kappa}\b) \over
\Gamma(-{\lambda_n^{(\b,b)} \over \gamma_1}) 
S_1(-{\lambda_n^{(\b,b)} \over \gamma_1};\sqrt{\kappa}\b,\sqrt{\kappa} b)}\bigg], &\!\! x,\b \in (-\infty,b),
\\
\displaystyle p_b^+(t;x,\b) {\gamma_1 \over \kappa}
\bigg[
{e^{{\kappa\over 2} \b^2} \over \mathcal{S}[b,\b]} 
 - \sqrt{\kappa}
\sum_{n=1}^{\infty} e^{-\lambda_{n,\b}^+ (T-t)}
D_{{\lambda_{n,\b}^+ \over \gamma_1} - 1}(\sqrt{\kappa}\b) 
\Big/ D^{(1)}_{\lambda_{n,\b}^+ \over \gamma_1}(\sqrt{\kappa}\b)
\\
\phantom{p_b^+(t;x,\b) {\gamma_1 \over \kappa}}
\displaystyle - e^{{\kappa\over 4} (b^2+\b^2)}\gamma_1\sqrt{2\pi\kappa}
\sum_{n=1}^{\infty} {e^{-\lambda_n^{(b,\b)} (T-t)} \over \lambda_n^{(b,\b)}}
{D_{\lambda_n^{(b,\b)} \over \gamma_1}(\sqrt{\kappa}b) \Big/ 
D_{\lambda_n^{(b,\b)} \over \gamma_1}(\sqrt{\kappa}\b) \over
\Gamma(-{\lambda_n^{(\b,b)} \over \gamma_1}) 
S_1(-{\lambda_n^{(b,\b)} \over \gamma_1};\sqrt{\kappa}b,\sqrt{\kappa} \b)}\bigg], &\!\!  x,\b \in (b,\infty),
\end{cases}
\end{align}
$t\in (0,T)$, with scale function $\mathcal{S}[x_1,x_2] = \int_{x_1}^{x_2} e^{{\kappa\over 2}y^2} dy$.   
Lastly, combining \eqref{joint_last_passage_pdf_spectral_b_1} and \eqref{joint_last_passage_pdf_spectral_b_2} gives the joint density:
\begin{eqnarray}\label{joint_last_passage_pdf_b_1_OU}
f_{g^{b}_{\b}(T), X_{b,T}}(t,z;x) 
= p^-_{b}(t;x,\b) e^{{\kappa\over 2} (\b^2 - z^2)}
\begin{cases}
\displaystyle \sum_{n=1}^{\infty}e^{-\lambda_{n,k}^- (T-t)} 
{D_{\lambda_{n,\b}^-/ \gamma_1} (-\sqrt{\kappa} z) 
\over D^{(1)}_{\lambda_{n,\b}^-/ \gamma_1}(-\sqrt{\kappa} \b)}  &, z \in (-\infty,\b),
\\
\displaystyle \gamma_1 \sum_{n=1}^\infty e^{- \lambda_n^{(\b,b)} (T-t)} 
{S(-{\lambda_n^{(\b,b)}\over \gamma_1};\sqrt{\kappa}z,\sqrt{\kappa}b)
\over S_1(-{\lambda_n^{(\b,b)} \over \gamma_1};\sqrt{\kappa}\b,\sqrt{\kappa}b)}
&, z \in (\b,b),
\end{cases}
\end{eqnarray}
for $x,\b\in(-\infty,b)$, and 
\begin{eqnarray}\label{joint_last_passage_pdf_b_2_OU}
f_{g^{b}_{\b}(T), X_{b,T}}(t,z;x) 
= p^+_{b}(t;x,\b) e^{{\kappa\over 2} (\b^2 - z^2)}
\begin{cases}
\displaystyle \gamma_1 \sum_{n=1}^\infty e^{- \lambda_n^{(b,\b)} (T-t)} 
{S(-{\lambda_n^{(b,\b)}\over \gamma_1};\sqrt{\kappa}b,\sqrt{\kappa}z)
\over S_1(-{\lambda_n^{(b,\b)} \over \gamma_1};\sqrt{\kappa}b,\sqrt{\kappa}\b)}
&, z \in (b,\b),
\\
\displaystyle \sum_{n=1}^{\infty}e^{-\lambda_{n,k}^+ (T-t)} 
{D_{\lambda_{n,\b}^+ / \gamma_1} (\sqrt{\kappa} z) 
\over D^{(1)}_{\lambda_{n,\b}^+ / \gamma_1}(\sqrt{\kappa} \b)}  &, z \in (\b,\infty),
\end{cases}
\end{eqnarray}
for $x,\b\in (b,\infty)$, $t\in (0,T)$. 
Using \eqref{trans_PDF_spectral_OU_b_below} and \eqref{trans_PDF_spectral_OU_b_above}
within \eqref{last_passage_pdf_explicit_b_OU}--\eqref{joint_last_passage_pdf_b_2_OU} gives a double series representation. Lastly, we again note that the series in \eqref{prop_joint_last_b-below_pmf_OU}--\eqref{joint_last_passage_pdf_b_2_OU} have their equivalent representations in terms of the respective rescaled functions $\widetilde{D}, \widetilde{D}^{(1)}, \widetilde{S}, \widetilde{S}_1$, as discussed above.

%
%
\section{Some Numerical Results} \label{sectNumerical}

We now present numerical calculations for only a few examples of the explicit spectral formulae derived in Section~\ref{subsect_spectral_formulae}. Additional numerical results can be found in \cite{YaodeThesis}. 
Our calculations serve to demonstrate how the distributions are efficiently and accurately computed by truncating the discrete spectral series, which typically have a rapid convergence for finite time horizon $T$. We note that for larger values of $T$, the series converge more rapidly since a smaller subset of the lowest eigenvalue terms are needed. This is a generic property of all spectral series. 

For cases with killing imposed at two interior endpoints, the CDF of the last hitting time $g^{(a,b)}_{\b}(T)$ is computed by using \eqref{last-passage-CDF-kill_a_b_time_0_to_t}, where the discrete portion is computed by truncating the single series in \eqref{prop_last_time-CDF_discrete_killed_ab} to the first $N$ terms and the integral term is computed by truncating the double series in \eqref{series_CDF_g_K_ab_T} to the first $N$ terms in the inner series and truncating to the first $M$ terms in the outer sum. The first computational step involves attaining an accurate convergence as $N$ is increased in the single series and then attaining accurate convergence for a sufficiently large number for $M$ in the double series. In most calculations presented below, an accurate convergence was achieved with $M\simeq 100$. As a second part of our calculations, we compute the joint PDF of the last hitting time and the process value $g^{(a,b)}_{\b}(T), X_{(a,b),T}$ by implementing \eqref{joint_last_passage_pdf_explicit_ab_new}. We perform such calculations on Brownian motion (BM) and the SQB processes. As an example of a process without imposed killing, we calculate the CDF of $g_{\b}(T)$ and joint PDF of $g_{\b}(T), X_T$ for the OU process.

For killed BM on $(a,b)$ with drift $\mu$, we compute the CDF of $g^{(a,b)}_{\b}(T)$ by adding the truncated ($N$-term) series in \eqref{drifted_BM_prop_last_time_discrete_killed_ab} to the truncated double series in \eqref{series_CDF_g_K_ab_T} which uses the truncated series ($N$ terms for the inner sum and $M=100$ for the outer sum) in \eqref{drifted_BM_last_passage_pdf_spec}. For the case of zero drift we simply use the expressions for $\mu=0$. 
Here, we set the killing levels \(a = 1\) and \(b = 11\), the last hitting level \(\b = 5\), the initial value \(X_0 = x = 3\), and the length of the time interval \(T = 20\).  Figure \ref{fig:bm_conv} overlays computed CDF curves with $\mu=0$ for increasing values of $N$. An accurate convergence (with nearly overlapping curves) is already achieved for $N=8$. Although the rate of convergence has a dependence on the interval length $b-a$, the quadratic growth of the eigenvalues generally leads to a rapid series convergence. Observe that the CDF is positive at time $t=0$ which is given by $\P_x(g^{(a,b)}_{\b}(T) = 0) > 0$. The CDF climbs to unity at $t=T$, as required. Figure \ref{fig:dbm_conv} is a repeat of the calculations for nonzero drift \(\mu = 0.05\). Since the process starts relatively close to the lower killing level, and below the last hitting level with upper killing level being relatively further above 
($a=1, b =11, \b = 5, x=3$), a slightly positive drift decreases the probability that the hitting level is never attained within time $T$, i.e., we see that $\P_x(g^{(a,b)}_{\b}(T) = 0)$ decreases from about 0.5 to 0.45. 
For drifted BM with killing, we further conduct sensitivity analysis on parameter impacts using the numerically converged CDFs. 
Figure \ref{fig:dbm_x} shows that decreasing the starting value $x$ increases the value of $\P_x(g^{(a,b)}_{\b}(T) = 0)$. Starting values that are close to the upper killing level $b$ lead to higher probabilities of being killed before hitting $\b$ and this is more so for positive drift 
$\mu = 0.1$. 
In Figure \ref{fig:dbm_mu}, \(\mu \) is varied from -0.25 to 1. A negative \(\mu = -0.25\) gives the highest value of $\P_x(g^{(a,b)}_{\b}(T) = 0) = 0.73$, i.e., the process is more likely to be killed at the lower level $a$. As $\mu$ increases, $\P_x(g^{(a,b)}_{\b}(T) = 0)$ decreases, and the CDF curve grows faster towards unity.

\begin{figure}[h!]
\centering
\begin{tabular}{@{\hspace{-2mm}}c@{\hspace{-4mm}}c@{\hspace{-2mm}}}
\raisebox{0mm}{\begin{minipage}{0.45\textwidth}
\centering
\includegraphics[width=\linewidth, height=3.5cm]{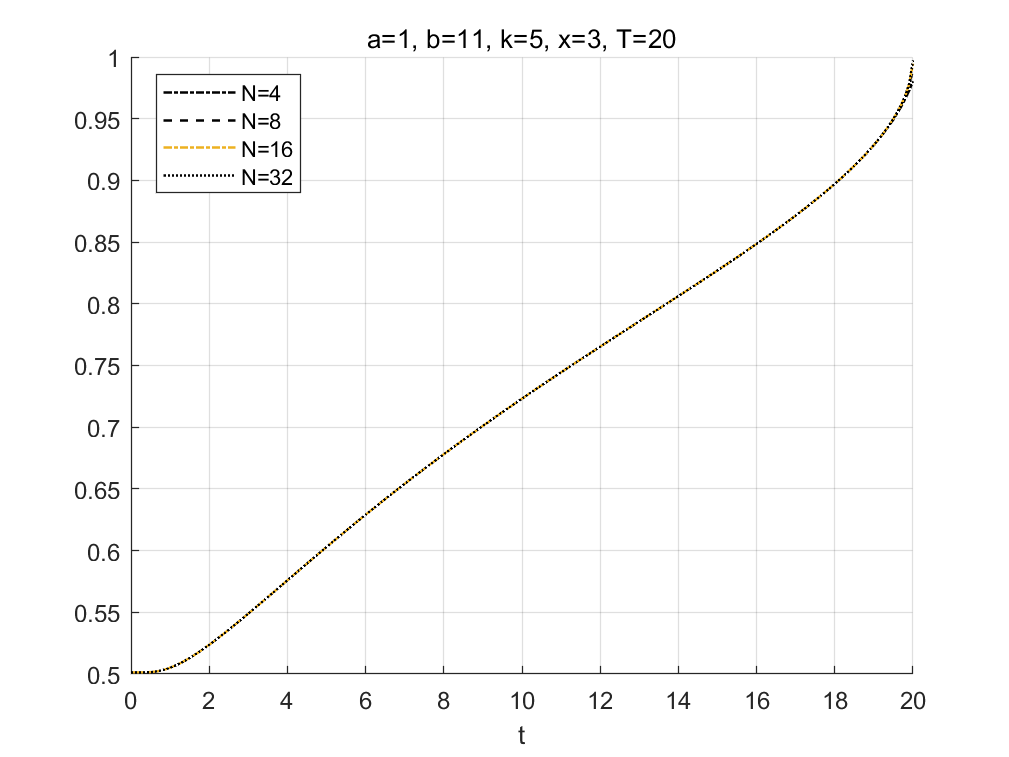} \\
\vspace{-3.5mm}\caption{\small CDF convergence for killed BM.}
\label{fig:bm_conv}
\end{minipage}} &
\raisebox{0mm}{\begin{minipage}{0.45\textwidth}
\centering
\includegraphics[width=\linewidth, height=3.5cm]{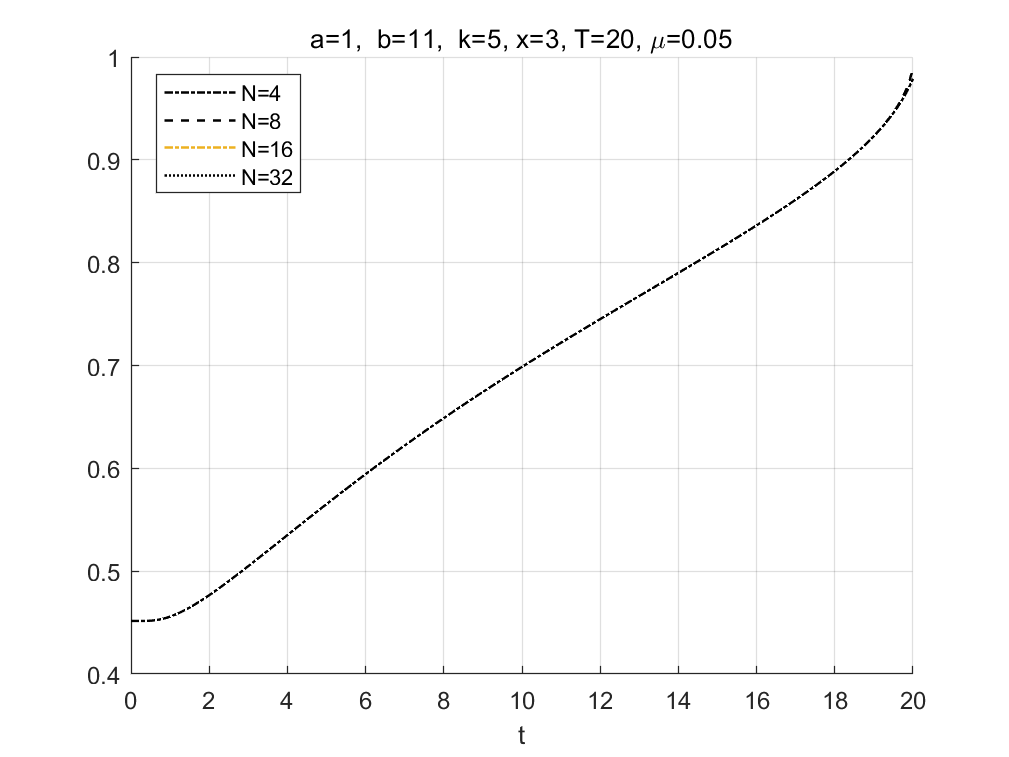} \\
\vspace{-3.5mm}\caption{\small CDF convergence for drifted killed BM.}
\label{fig:dbm_conv}
\end{minipage}} \\

\raisebox{0mm}{\begin{minipage}{0.48\textwidth}
\centering
\includegraphics[width=\linewidth, height=3.5cm]{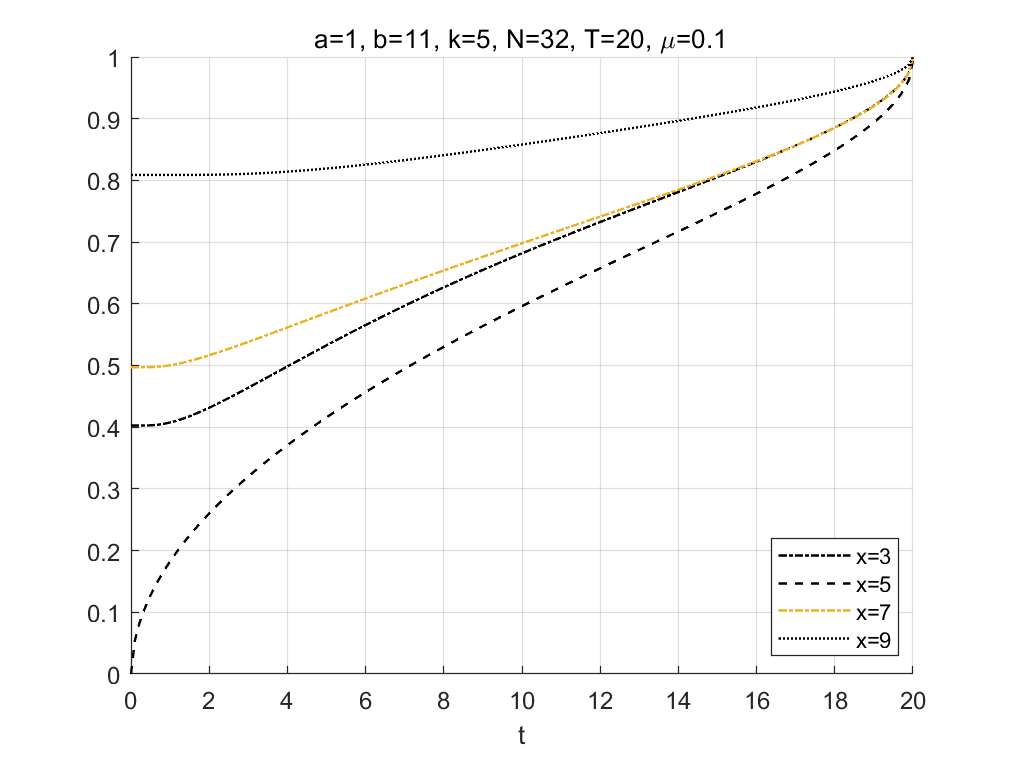} \\
\vspace{-3.5mm}\caption{\small CDF with varying $X_0=x$.}
\label{fig:dbm_x}
\end{minipage}} &
\raisebox{0mm}{\begin{minipage}{0.48\textwidth}
\centering
\includegraphics[width=\linewidth, height=3.5cm]{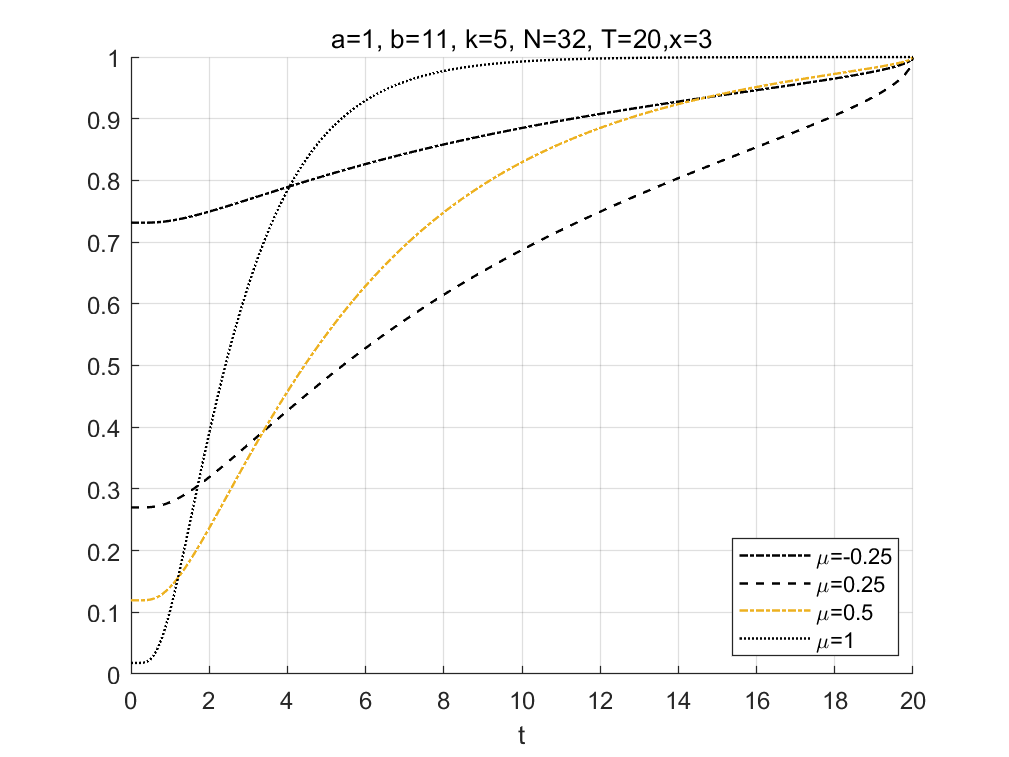} \\
\vspace{-3.5mm}\caption{\small CDF with varying drift $\mu$.}
\label{fig:dbm_mu}
\end{minipage}} \\
\end{tabular}
\end{figure}

We now compute the joint PDF of the last hitting time and process value for zero-drift and drifted BM with killing on $(a,b)$ by utilizing \eqref{drifted_BM_kill_a_b_joint_last_passage} with each single series truncated to 100 terms. 
We set $x=5, a = 1, \b = 10$ and the upper killing boundary to $b = 20$. For drifted BM we set 
$\mu = 0.1$. 
Figure \ref{fig:BM_ab_joint} demonstrates that, for zero drift, the joint PDF is nearly symmetric about the zero density line $z=\b$. This is expected for killing levels chosen relatively far from the starting value $x$. We note that, as the respective killing levels are taken 
to $\pm\infty$, the expression in \eqref{drifted_BM_kill_a_b_joint_last_passage} converges to the (symmetric about $z=\b$) known joint PDF in \eqref{last_hitting_joint_PDF_BM} for standard BM with no imposed killing. Figure \ref{fig:driftedBM_ab_joint}, for BM with positive drift 
$\mu=0.1$, has a similar overall shape. However, due to the drift term, the joint PDF values in the upper interval $z\in(\b,b)$ are notably higher with a more prominent peak.

\begin{figure}[h]
\centering
\begin{minipage}{0.32\textwidth}
\centering
\includegraphics[width=1.02\linewidth, height=3.5cm]{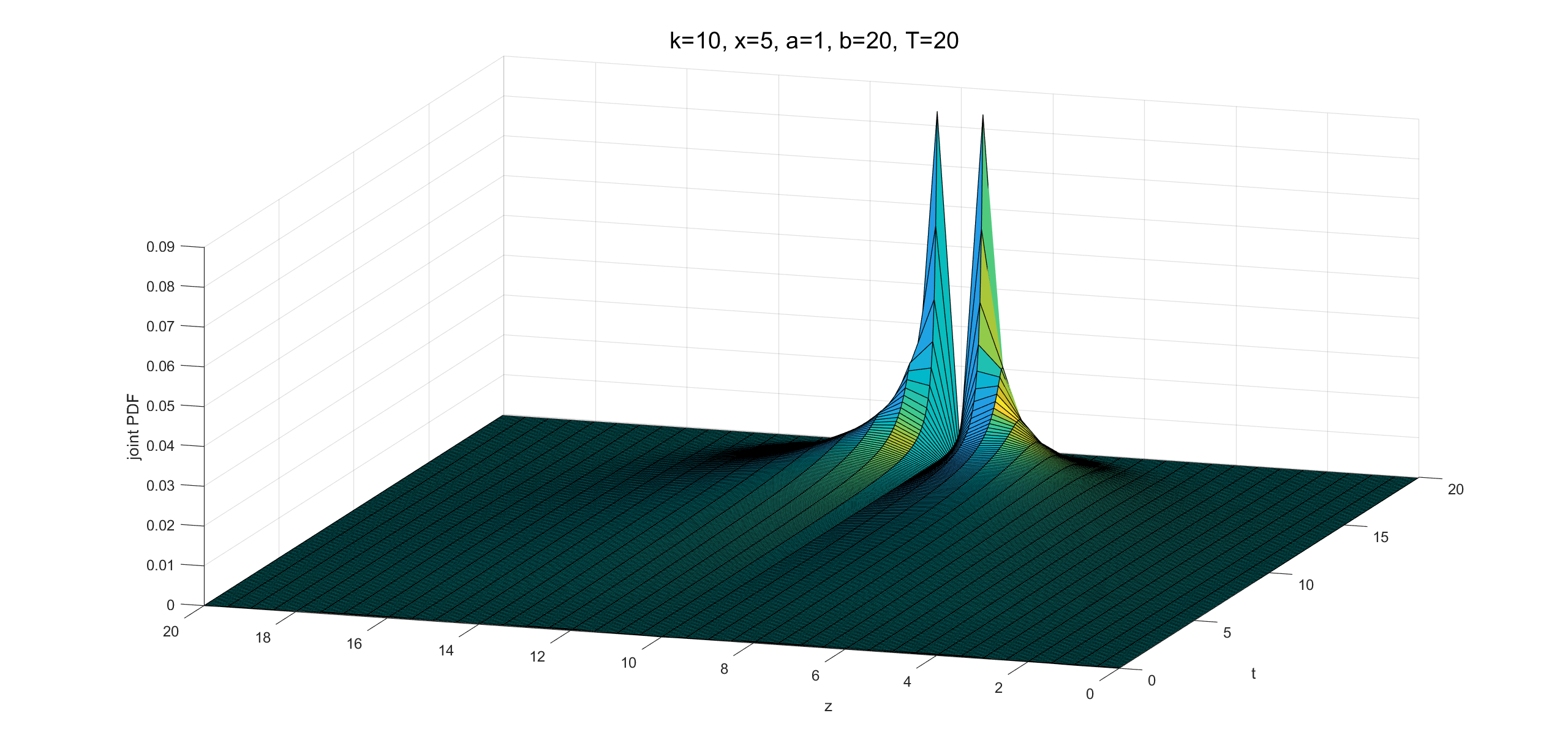} 
\vspace{-7mm}
\caption{\small Joint PDF for killed BM.}
\label{fig:BM_ab_joint}
\end{minipage}\hfill
\begin{minipage}{0.32\textwidth}
\centering
\includegraphics[width=1.02\linewidth,height=3.5cm]{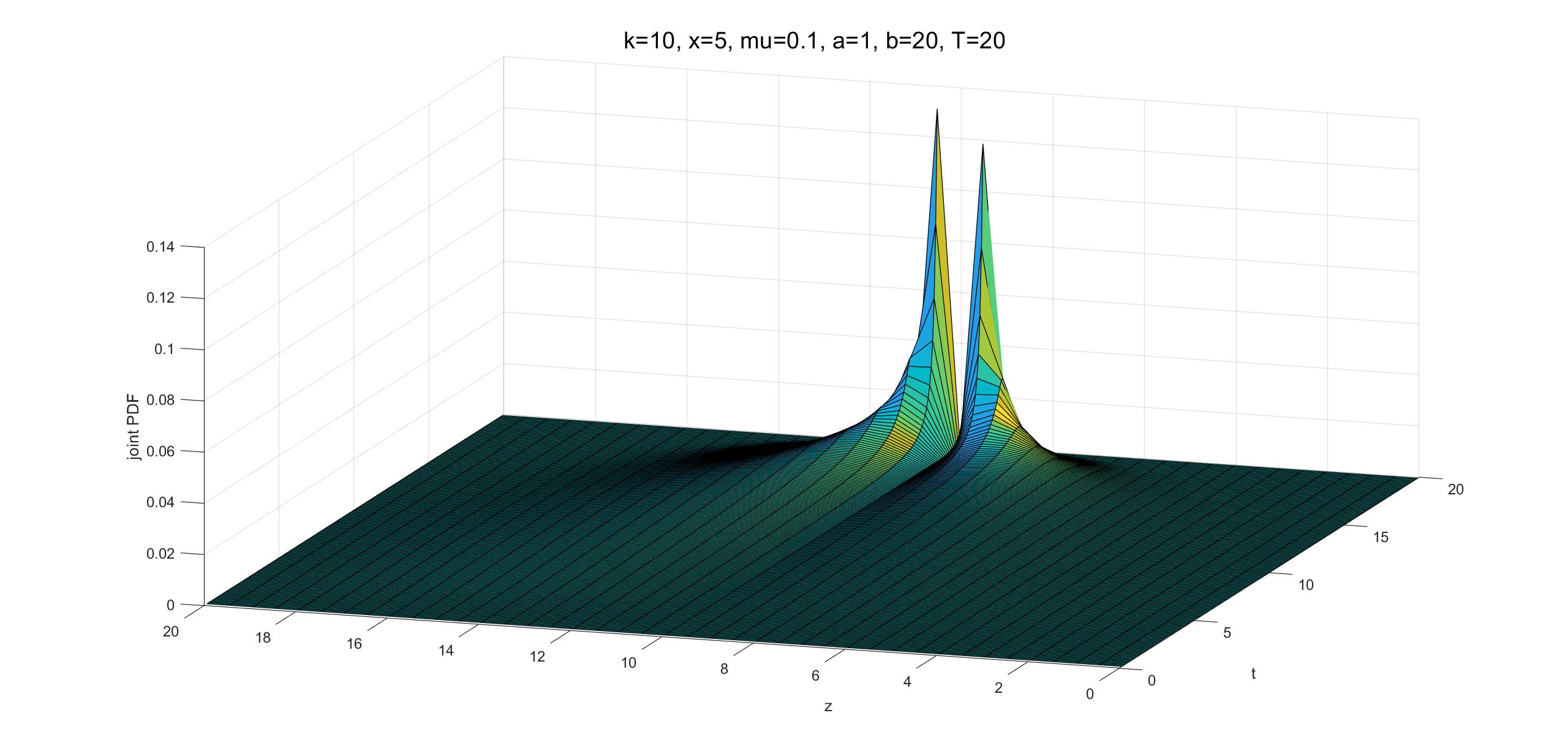}
\vspace{-7mm}
\caption{\small Joint PDF for killed drifted BM.}
\label{fig:driftedBM_ab_joint}
\end{minipage}\hfill
\begin{minipage}{0.32\textwidth}
\centering
\includegraphics[width=1.02\linewidth,height=3.5cm]{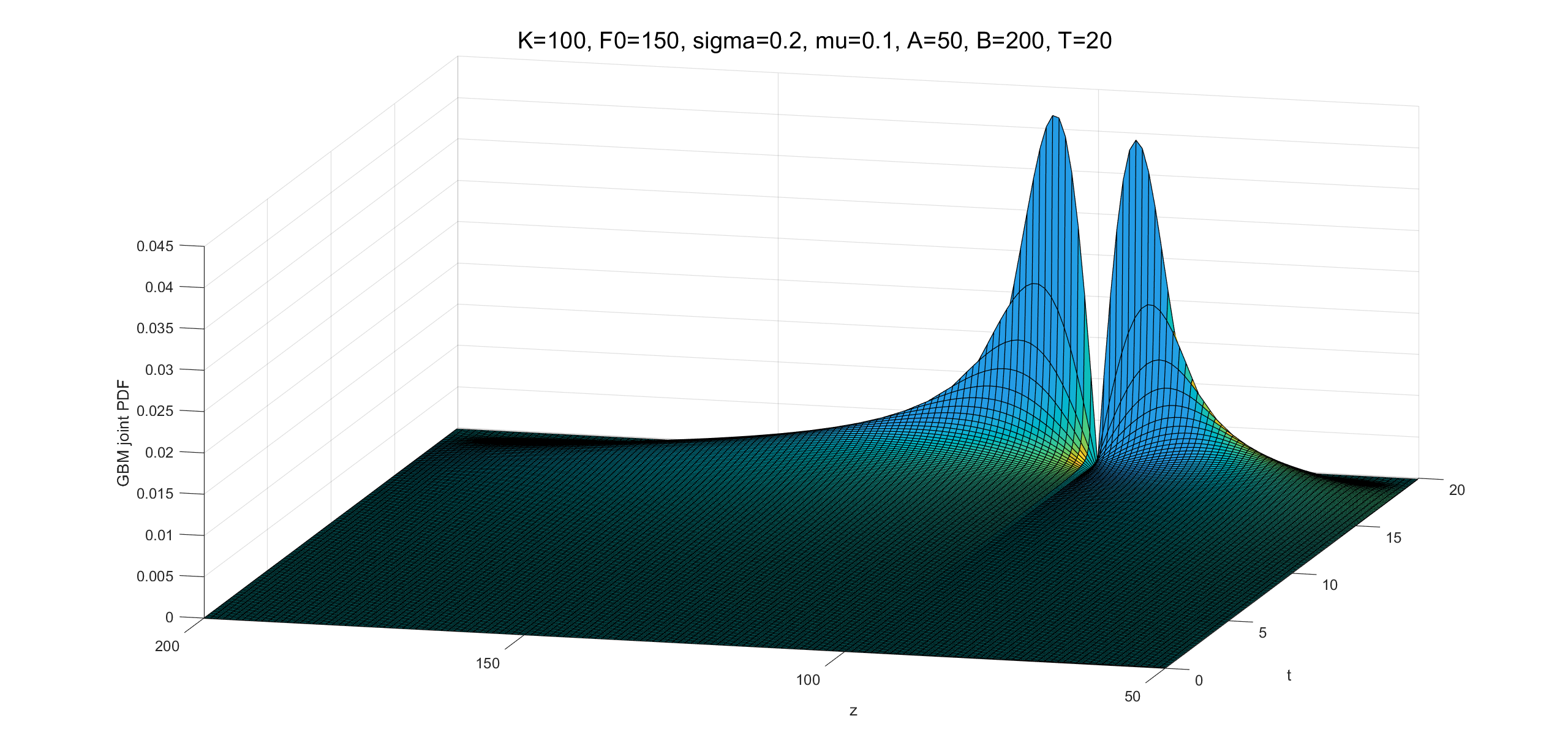}
\vspace{-7mm}
\caption{\small  Joint PDF for killed GBM.}
\label{fig:joint_GBM}
\end{minipage}
\end{figure}

The CDF and joint PDF for GBM on a positive interval $(A,B)$, with killing at the endpoints, follow directly from those for killed drifted BM  using \eqref{last_passage_CDF_X_to_F_kill_A_B}--\eqref{joint_last_PDF_X_to_F_kill_A_B} with exponential mapping defined in Section \ref{subsect_GBM}. As a numerical example, we set the parameters as: $\sigma=0.2$, $A=50$, $B=200$, $K=100$, $F_0 =150$, $\mu=0.1$ and $T=20$. The joint PDF is shown in Figure \ref{fig:joint_GBM} where we observe the same characteristic pattern as in drifted BM.

We next consider the SQB process on the positive interval $(a,b)$ with killing at the endpoints. The CDF of $g^{(a,b)}_{\b}(T)$ is computed using \eqref{last-passage-CDF-kill_a_b_time_0_to_t}, where the discrete portion is computed by truncating the series in 
\eqref{last_time-CDF_discrete_killed_ab_SQB} to $N$ terms and adding it to the resulting double series obtained by integrating the marginal density in \eqref{last_passage_pdf_explicit_ab_SQB}. The CDF curves in Figure \ref{fig:SQB_CDF_convergence} again demonstrate very rapid convergence of the series for the CDF. Figure \ref{fig:SQB_diff_x} demonstrates the change in CDF with the starting point $X_0=x$. The joint PDF of $g^{(a,b)}_{\b}(T), X_{(a,b),T}$ in Figure \ref{fig:SQB_joint}
was computed by truncating the series in \eqref{joint_last_passage_pdf_explicit_ab_SQB} with $N$ terms in the sum over $n$ and $M$ terms for the spectral series in \eqref{trans_PDF_spectral_SQB_ab} for $p_{(a,b)}$.
\begin{figure}[h]
\centering
\begin{minipage}{0.32\textwidth}
\centering
\includegraphics[width=\linewidth,height=3.5cm]{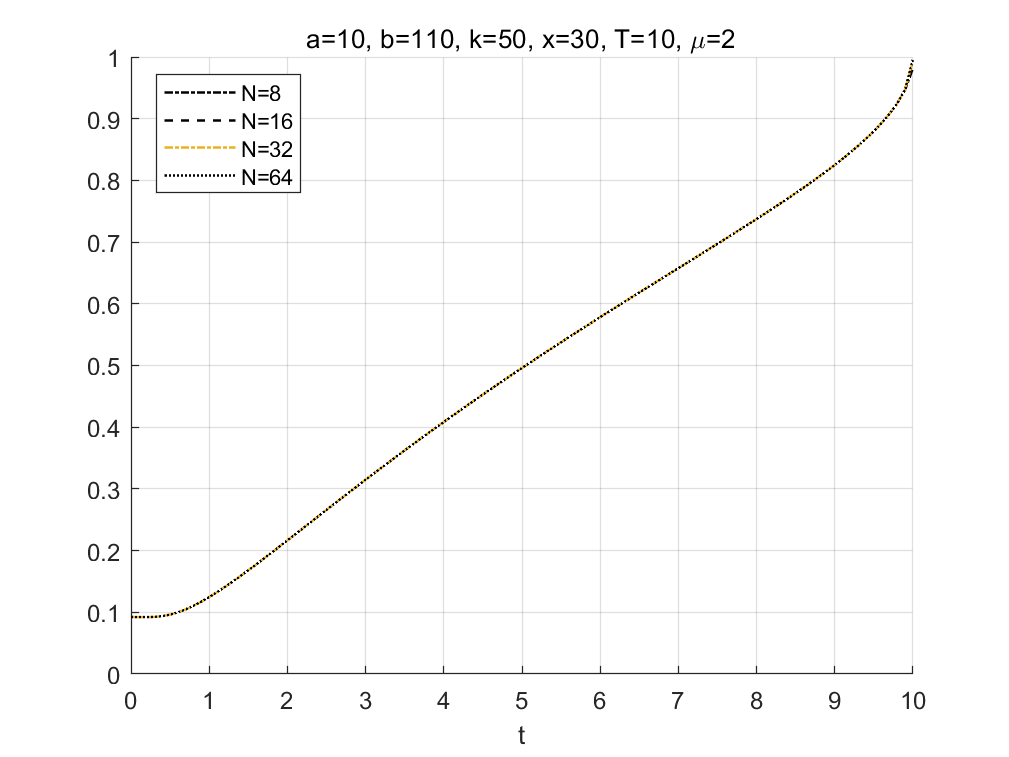}
\vspace{-7mm}
\caption{\small CDF convergence for killed SQB.}
\label{fig:SQB_CDF_convergence}
\end{minipage}\hfill
\begin{minipage}{0.32\textwidth}
\centering
\includegraphics[width=\linewidth,height=3.5cm]{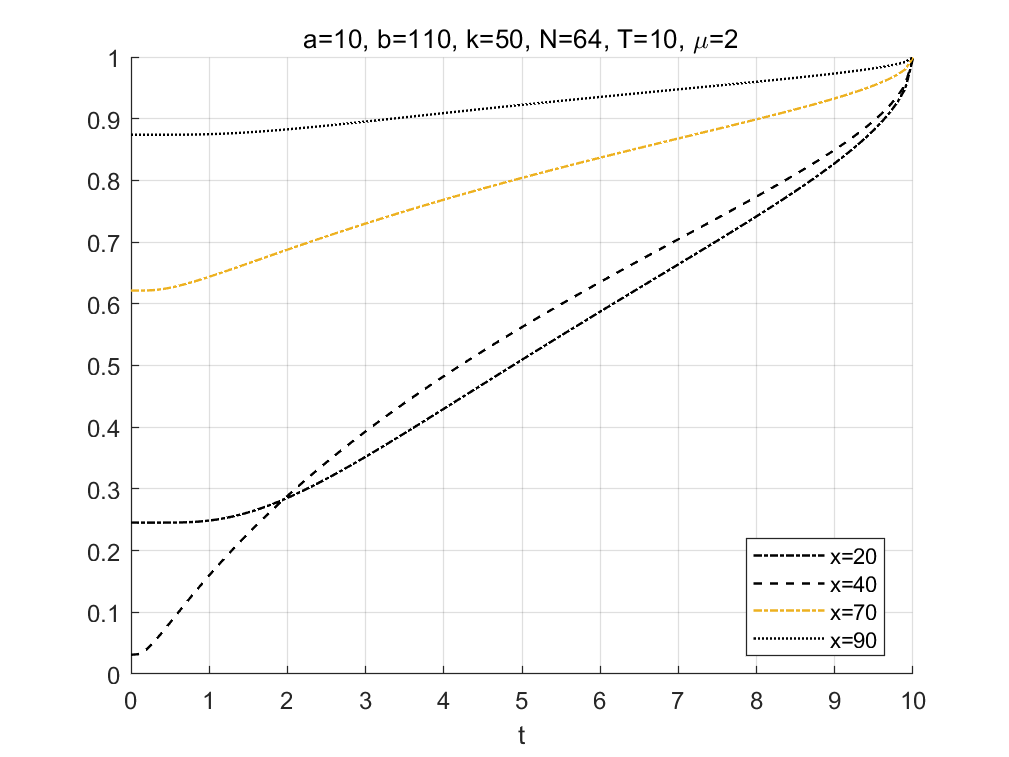}
\vspace{-7mm}
\caption{\small Converged CDF curves for killed SQB with varying $X_0=x$.}
\label{fig:SQB_diff_x}
\end{minipage}\hfill
\begin{minipage}{0.32\textwidth}
\centering
\includegraphics[width=\linewidth,height=3.5cm]{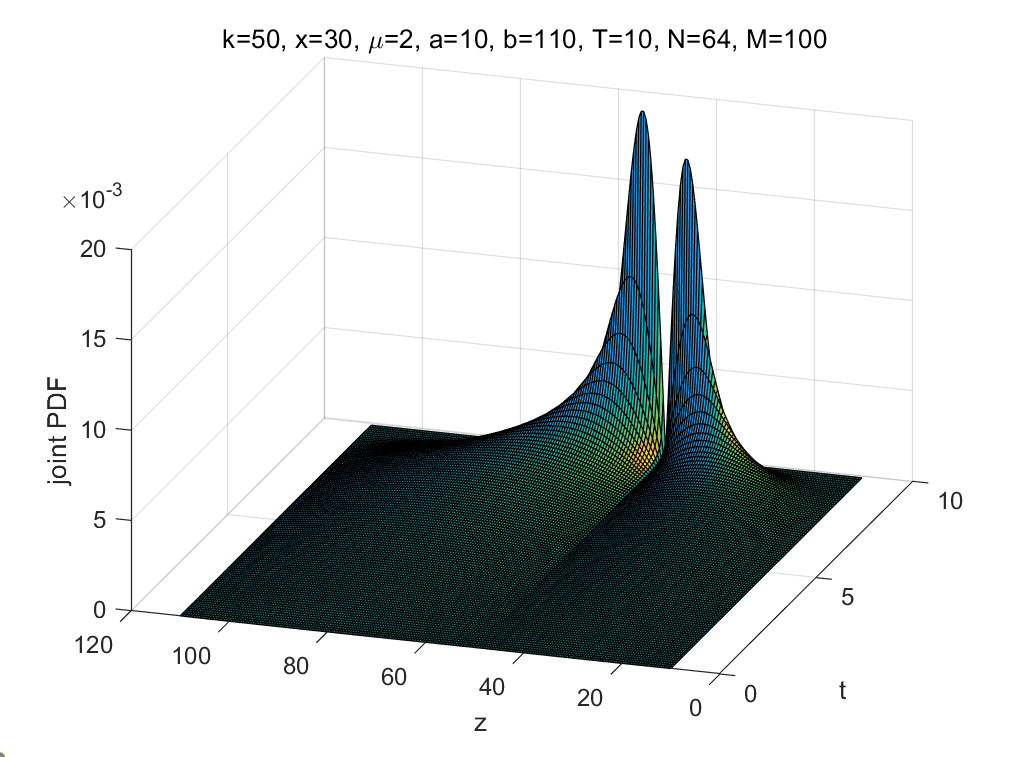}
\vspace{-7mm}
\caption{\small Converged joint PDF for killed SQB.}
\label{fig:SQB_joint}
\end{minipage}
\end{figure}

%
%
As a last example we consider the OU process on $\R$, i.e., without killing, where the distribution of $g_\b(T)$ is given by combining \eqref{prop_last_time_discrete_OU} and \eqref{last-passage-pdf-spectral_OU} and the joint PDF of $g_\b(T), X_T$ is given by \eqref{joint_last_passage_pdf_OU} which only involves a single series truncated to $N$ terms. Each single series makes use of the same eigenvalue set for given parameters $\kappa, \gamma_1$. 
The CDF is calculated via \eqref{last-passage-CDF-time_0_to_t}, where the series in \eqref{prop_last_time_discrete_OU} is truncated to $N$ terms and the continuous portion is computed by termwise integration of the PDF which results in an explicit double series upon using the Hermite polynomial series in \eqref{trans_PDF_spectral_OU} truncated to $M=100$ terms. 
To avoid any overflow we make use of the re-scaled parabolic cylinder functions. For instance, with parameters \(\kappa = 1\), \(x = 0.5\), \(\gamma_1 = 2\), \(k = 7\), \(T = 10\) we obtain a rapid convergence for both the CDF and joint PDF. 
Figure \ref{fig:OU_no_kill_PDF} shows the sensitive dependence of the CDF of on the initial value of the OU process relative to the last hitting level. Figure \ref{fig:OU_no_kill_CDF} displays the calculated joint PDF (plotted here for positive values of the endpoint value $z$).  This displays similar basic characteristics as the other processes. However, the relative peaks and shapes of the distribution for $z <\b$ and $z > \b$ vary significantly with changing parameters $\kappa, \gamma_1$ and initial value $X_0=x$.

\begin{figure}[h]
    \centering

    \begin{minipage}{0.4\textwidth} 
        \centering
        \includegraphics[width=5.5cm, height=3.5cm]{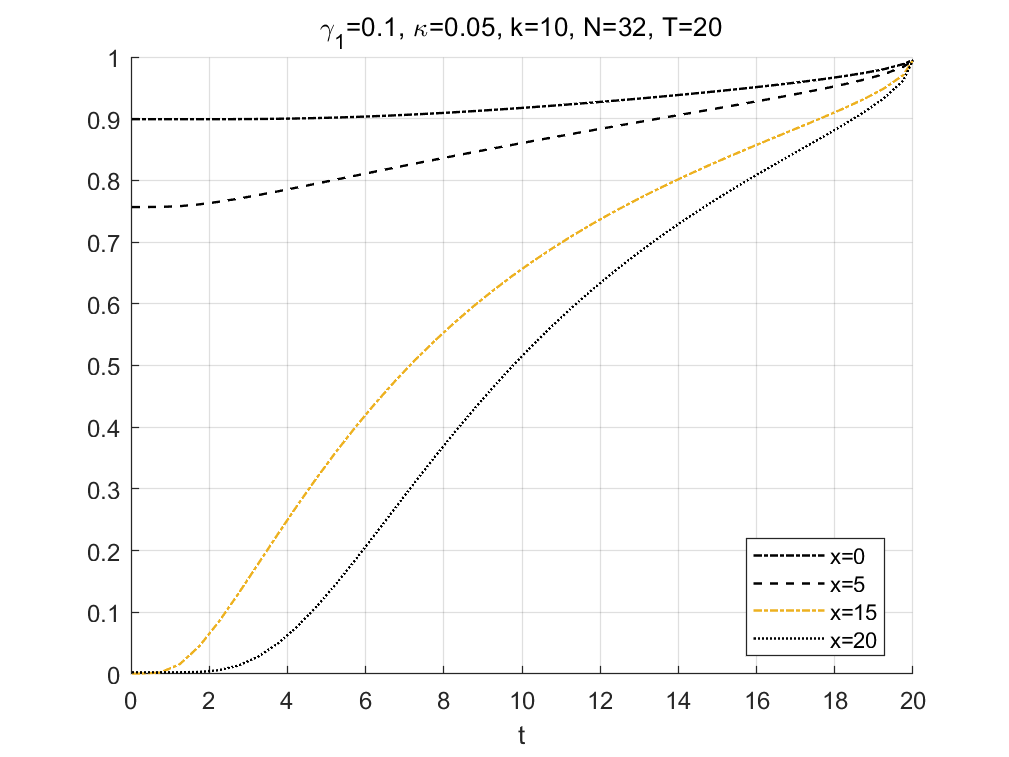} 
        \vspace{-7mm}
        \caption{\small OU process: CDF of last hitting time $g_\b(T)$ with varying initial value $X_0=x$.}
        \label{fig:OU_no_kill_PDF}
    \end{minipage}%
    \hspace{0.05\textwidth} 
    \begin{minipage}{0.4\textwidth}
        \centering
        \includegraphics[width=5.5cm, height=3.5cm]{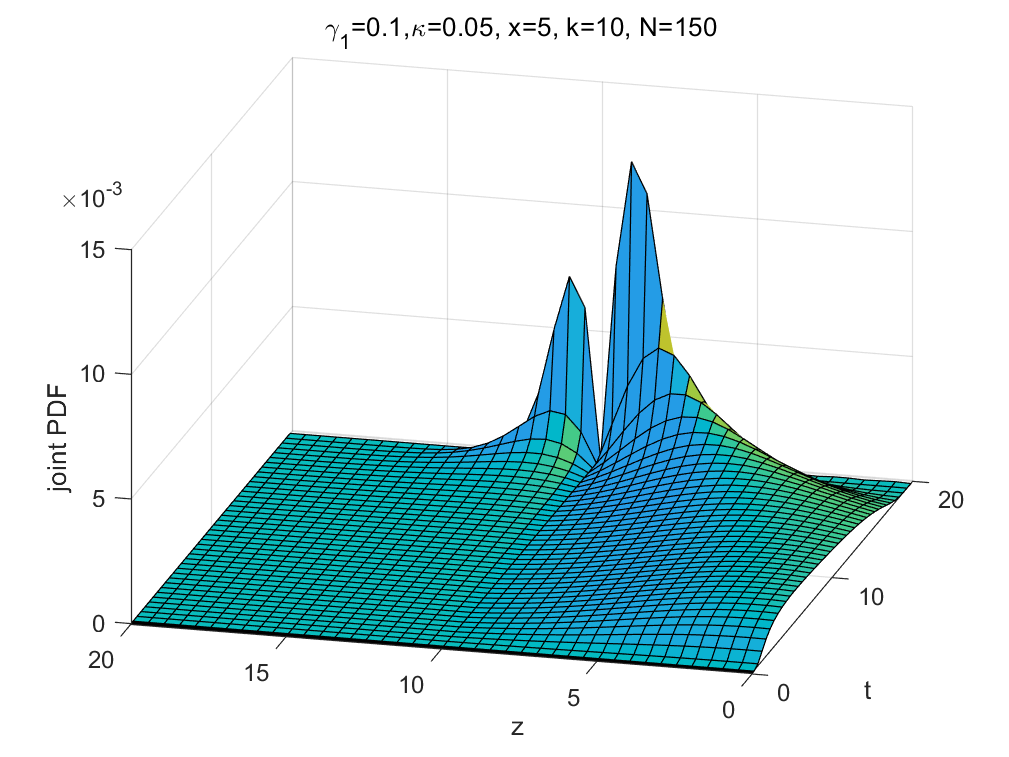}
        \vspace{-7mm}
        \caption{\small Joint PDF of $g_\b(T), X_T$ for OU process.}
        \label{fig:OU_no_kill_CDF}
    \end{minipage}
\end{figure}

\newpage


\appendix

\section{Appendix: Proofs} \label{sect_A1}

%
%
%

\subsection{Proof of Lemma \ref{Lemma_hitting_time}} \label{sect_Lemma1_proof}
Using the tower property, while conditioning on the natural filtration at time 
$t$, ${\mathcal F}^X_t$, with $\ind_{\{t < \Tau^{\pm}_{b} < \infty\}} = 
\ind_{\{\Tau^{\pm}_{b} > t\}}\ind_{\{\Tau^{\pm}_{b} < \infty\}}$ and  
$\ind_{\{\Tau^{\pm}_{b} > t\}}$ as ${\mathcal F}^X_t$-measurable, 
together with the Markov property on the inner expectation, gives \eqref{X_hit_lemma_1}:
\begin{eqnarray*}
\P_{x}(t < \Tau^{\pm}_{b} < \infty) 
&=& \E_{x}\big[\ind_{\{\Tau^{\pm}_{b} > t\}} \, \E\big[\ind_{\{\Tau^{\pm}_{b} < \infty\}}
\vert {\mathcal F}^X_t\big]\big]
= \E_{x}\big[\ind_{\{\Tau^{\pm}_{b} > t\}} \, \P_{X_t}(\Tau^{\pm}_{b} < \infty)\big]
\nonumber \\
&=& \E_{x}\big[\ind_{\{\Tau^{\pm}_{b}> t\}} \, \Phi^{\pm}_b(X_t)\big]
= \E_{x}\big[\Phi^{\pm}_b(X_{b,t})\big]
= \int_{\I_b^{\mp}} \Phi^{\pm}_b(y)\, \P_x(X_{b,t}\in d y)\,,
\end{eqnarray*}
where 
$\P_x(X_{b,t}\in d y) = p_b^\mp(t;x,y) d y$ for $x,y \in \I_b^{\mp}$, respectively. 

We now prove (\ref{fhit_up_down_pdf}) for $f^{+}(t;x,b)$. The proof for $f^{-}(t;x,b)$ follows by similar steps with left boundary $l$ replaced by right boundary $r$. 
By definition, $f^{+}(t;x,b)=-{\partial\over \partial t}\P_{x}(t < \Tau^+_b < \infty)$, and 
using the forward
\footnote{
Both the backward (in $t,x$) and forward (in $t,y$) Kolmogorov PDEs are readily shown to hold for all transition PDFs given by any of the spectral categories. In particular, since 
${\mathcal G}\varphi^\pm(x) = \lambda \varphi^\pm(x)$ we have 
${\mathcal G}_x G(\lambda;x,y) = \hat{\mathcal G}_y G(\lambda;x,y) = \lambda G(\lambda;x,y)$ for any of the Green functions.
}
 Kolmogorov PDE, in the variables $(t,y)$, for the transition PDF of the diffusion killed at upper level $b$, 
${\partial \over \partial t}p^-_b(t;x,y) = \hat{\mathcal G}_y \,p^-_b(t;x,y)$, 
$\hat{\mathcal G}_y\, h(y):= \frac{\partial}{\partial y}(\frac{1}
{\s(y)}\frac{\partial}{\partial y}(\frac{h(y)}{\m(y)}))$, while
differentiating (\ref{X_hit_lemma_1}) and integrating by parts:
\begin{equation*}
f^{+}(t;x,b)
= - \int_l^b \Phi^{+}_b(y)\, \hat{\mathcal G}_y\, p^-_b(t;x,y) \, d y
= -\frac{\Phi^{+}_b(y)}{\s(y)}\frac{\partial}{\partial y}\bigg(\frac{p^-_b(t;x,y)}{\m(y)}\bigg)\bigg|_{y=l+}^{y=b-}
+\, \chi(l,b]\,\frac{p^-_b(t;x,y)}{\m(y)}\bigg|_{y=l+}^{y=b-}
\end{equation*}
where $\chi(l,b] = {1\over {\mathcal S}(l,b]}$
if $l$ is attracting and non-reflecting, and is otherwise zero. The second limit vanishes by the boundary condition
$p^-_b(t;x,b)/\m(b)=0$ and $p^-_b(t;x,l+)/\m(l+)=0$ for attracting and non-reflecting 
(i.e. exit-not-entrance, killing or attracting natural) boundary $l$. In the first limit expression, we note that $\Phi^{+}_b(b) = 1$. Moreover,
$\Phi^{+}_b(l+) = 1$ when the left boundary $l$ is non-attracting or reflecting, and equals zero otherwise,
whereas $\frac{1}{\s(y)}\frac{\partial}{\partial y}\big(\frac{p^-_b(t;x,y)}{\m(y)}\big)\big\vert_{y=l+}=0$ 
when $l$ is non-attracting or reflecting and is otherwise bounded. 
This reduces the above expression to $f^{+}(t;x,b) = -\frac{1}{\s(b)}\frac{\partial}{\partial y}
\big(\frac{p^-_b(t;x,y)}{\m(y)}\big)\big\vert_{y=b-}$, which proves the first equality in \eqref{fhit_up_down_pdf}. 
The second equality in \eqref{fhit_up_down_pdf} follows simply from the symmetry of the transition PDFs w.r.t. the speed measure, i.e., 
$q_b^\mp(t;x,y) := {p_b^\mp(t;x,y) \over \m(y)} = q_b^\mp(t;y,x)$. Hence, 
$\frac{\partial}{\partial y}\big(\frac{p^\mp_b(t;x,y)}{\m(y)}\big)\big\vert_{y=b\mp}
=\frac{\partial}{\partial y}q_b^\mp(t;x,y)\vert_{y=b\mp} = 
\frac{\partial}{\partial y}q_b^\mp(t;y,x)\vert_{y=b\mp} = {1\over \m(x)}
\frac{\partial}{\partial y}p_b^\mp(t;y,x)\vert_{y=b\mp}$.

Remark: Taking the Laplace transform w.r.t. $t$ on both sides of \eqref{fhit_up_down_pdf}, 
where ${\mathcal L}_t\{p^-(t;x,y)\}(\lambda) = G_b^-(\lambda;x,y)$, and then differentiating, w.r.t. $y$, the Green function in (\ref{greenfunc_up}) evaluated at $y=b$ (note that $x < b$, i.e., we have $x\wedge y = x, x \vee y = y$ since $y\nearrow b$) gives
\begin{align*}
{\mathcal L}_t\{f^+(t;x,b)\}(\lambda) 
&= 
 - \displaystyle\frac{1}{\s(b)} \frac{\partial}{\partial y}
\left( \frac{G_b^-(\lambda;x,y)}{\m(y)}\right)\bigg|_{y=b-}
= \displaystyle\frac{1}{\s(b)}
\frac{\varphi_\lambda^+(x )}{w_\lambda \varphi_\lambda^+(b)}
\,{\partial \over \partial y}\phi(b,y;\lambda)\big|_{y=b}
\\
&= \frac{\varphi_\lambda^+(x )}{\varphi_\lambda^+(b)}
\displaystyle\frac{1}{w_\lambda \s(b)}
\,{\partial \over \partial y}\left[\varphi_\lambda^-(b) \varphi_\lambda^+(y)
- \varphi_\lambda^+(b) \varphi_\lambda^-(y) \right]\big|_{y=b}
\\
&= \frac{\varphi_\lambda^+(x )}{\varphi_\lambda^+(b)}
\displaystyle\frac{1}{w_\lambda \s(b)}
W[\varphi_\lambda^-,\varphi_\lambda^+](b) = \frac{\varphi_\lambda^+(x )}{\varphi_\lambda^+(b)}.
\end{align*}
In the last two equation lines we used (\ref{phi_function}) and (\ref{wronskian}). In fact, we have the identity
\begin{eqnarray}\label{Laplace_FHT_Green_derivative}
\mp \displaystyle\frac{1}{\s(b)} \frac{\partial}{\partial y}
\left( \frac{G_b^\mp(\lambda;x,y)}{\m(y)}\right)\bigg|_{y=b\mp} = {\varphi_\lambda^\pm(x) \over \varphi_\lambda^\pm(b)}.
\end{eqnarray}
The same steps follow for $f^-$ where 
the use of $G_b^+(\lambda;x,y)$ in (\ref{greenfunc_down}) produces ${\mathcal L}_t\{f^-(t;x,b)\}(\lambda) = \varphi_\lambda^-(x) /\varphi_\lambda^-(b)$ for $x > b$. 
Hence, we have  shown
\begin{eqnarray}\label{Laplace_FHT_density}
{\mathcal L}_t\{f^\pm(t;x,b)\}(\lambda) = \varphi_\lambda^\pm(x) /\varphi_\lambda^\pm(b)
\end{eqnarray}
which is consistent with the known formula for the Laplace transform of the first hitting time PDF in terms of the  fundamental solutions $\varphi_\lambda^\pm$ to \eqref{eq:phi}, e.g., see \cite{BS02}.

%
%

\subsection{Proof of Proposition \ref{prop_spec_first_hit_1}}\label{sect_FHT_Prop1_proof}
Since both boundaries $l$ and $b$ are NONOSC,
the SL problem has a simple purely discrete positive spectrum $\{\lambda_n \equiv \lambda^-_{n,b}\}_{n=1}^\infty$. By Spectral Category I, the transition density $p_b^-$ for the $X$-diffusion on $\I_b^-$ 
is given by (\ref{u_spectral_1}) with eigenvalues solving 
$\varphi^+_{-\!\lambda_n}(b)=0$. Hence, using (\ref{fhit_up_down_pdf}) for $p_b^-$, while inserting the product eigenfunctions in (\ref{spectral_1_product_eigen}) and then differentiating the series term-by-term w.r.t. $y=b$, as well as invoking the function $\psi_n^+(x;b)$ in (\ref{FHT_eigenfunctions_1}), gives
\begin{eqnarray*}
f^{+}(t;x,b) &=& \sum_{n=1}^\infty e^{-\lambda_n t}\psi_n^+(x;b)
\frac{\varphi^-_{-\lambda_n}(b) \varphi^{+\prime}_{-\lambda_n}(b)}{\s(b) w_{_{-\lambda_n}}}.
\end{eqnarray*}
This is equivalent to (\ref{FHT_prop1_1}) since the above ratio term is identically one. This is easily seen from the zero boundary condition, i.e., the eigenvalue equation $\varphi^+_{-\!\lambda_n}(b)=0$, 
and hence the Wronskian in (\ref{wronskian}) implies
$\frac{\varphi^-_{-\lambda_n}(b) \varphi^{+\prime}_{-\lambda_n}(b)}{\s(b) w_{-\lambda_n}} 
= \frac{W[\varphi^-_{-\lambda_n},\varphi^+_{-\lambda_n}](b)}{\s(b) w_{-\lambda_n}} \equiv 1$. 

Equation (\ref{FHT_prop1_2}) follows by termwise integration, with respect to time, of the density in (\ref{FHT_prop1_1}), i.e., $\P_{x}(t < \Tau^{+}_{b} < \infty) = \int_t^\infty f^{+}(u;x,b) du$, where 
$\int_t^\infty e^{- \lambda_n u} du = {e^{- \lambda_n t} \over  \lambda_n}$. This completes the proof.

We remark that there are (at least) two other alternative proofs of Proposition~\ref{prop_spec_first_hit_1}. 
One alternate proof 
\footnote{
This approach is essentially equivalent to the proofs of the tail probabilities given in 
\cite{Linetsky2004b} for both Propositions \ref{prop_spec_first_hit_1}-\ref{prop_spec_first_hit_2}, where the conditional expectation in (\ref{X_hit_lemma_1}) is evaluated by applying the
spectral theorem for semigroups of self-adjoint contractions in the Hilbert space of real-valued functions
that are square-integrable with respect to the speed measure. This alternate proof requires the extra condition that 
$\Phi_b^{\mp}$ be square-integrable (w.r.t. $\m$) on $\I_b^\mp$ in the respective cases that $l$ (or $r$) is a natural boundary.
}
 is to firstly derive (\ref{FHT_prop1_2}) by inserting the series in (\ref{u_spectral_1}) 
into the integral in (\ref{X_hit_lemma_1}), giving rise to a series of integrals: 
$\P_{x}(t < \Tau^{+}_{b} < \infty) = \sum_{n=1}^\infty e^{- \lambda_n t}\phi_n(x)\,{c}_{n,b}$, where
${c}_{n,b} = \int_l^b \phi_n(y)\Phi_b^{+}(y) \m(y) dy$. 
Each of these integrals is simplified by writing the eigenfunction $\phi_n(y) = -(\lambda_n)^{-1}{\mathcal G}\phi_n(y)$, with generator ${\mathcal G}$, and then integrating by parts, while applying appropriate boundary conditions as $x\to l+$ and using the fact that ${1\over \s(y)}{\partial \over \partial y}\Phi_b^{+}(y)$ is constant w.r.t. $y$. 
This gives the coefficient
${c}_{n,b} = -(\lambda_n)^{-1}{\phi_n'(b)\over \s(b)}$. The rest of the steps follow from the Wronskian relation as in the above proof.
Then, (\ref{FHT_prop1_1}) follows by termwise differentiation of (\ref{FHT_prop1_2}).

A second (and simpler) alternative proof involves the Laplace inversion of the above relation, i.e., 
${\mathcal L}_t\{f^+(t;x,b)\}(\lambda) = \frac{\varphi_\lambda^+(x )}{\varphi_\lambda^+(b)}$. Hence, 
$f^+(t;x,b) = {\mathcal L}_\lambda^{-1}\big\{\frac{\varphi_\lambda^+(x )}{\varphi_\lambda^+(b)}\big\}(t)$. Since we are in Spectral Category I, $e^{\lambda t}\frac{\varphi_\lambda^+(x )}{\varphi_\lambda^+(b)}$ is a ratio of analytic functions (i.e., a meromorphic function) of $\lambda$ with simple poles 
at $\lambda = -\lambda_{n,b}^-,\, n =1,2, \ldots,$ which are the zeros of 
$\varphi_\lambda^+(b)$. By the standard procedure of closing the Bromwich contour to the left and applying the Residue Theorem, we recover (\ref{FHT_prop1_1}):
\begin{equation*}
f^+(t;x,b) = \sum_{n=1}^\infty 
\text{Res}\left[ e^{\lambda t}\frac{\varphi_\lambda^+(x )}{\varphi_\lambda^+(b)} ; 
\lambda = - \lambda_{n,\b}^-\right]  
= \sum_{n=1}^\infty e^{-\lambda_{n,b}^- t} {\varphi^+_\lambda (x) \over {\partial \over \partial \lambda}\,
\varphi^+_{\lambda}(b)}\bigg\vert_{\lambda=-\lambda_{n,b}^-}.
\end{equation*}

The proofs for the hitting time down mirror those given above, 
wherein the right boundary now plays the role of the left. For completeness, we summarize the steps as follows. 
Since both boundaries $b$ and $r$ are NONOSC, 
the SL problem has a simple purely discrete positive spectrum given by the eigenvalues 
$\{\lambda_n \equiv \lambda^+_{n,b}\}_{n=1}^\infty$. We are in Spectral Category I with transition density $p_b^+$ for the $X$-diffusion on $\I_b^+$ given by (\ref{u_spectral_2}) with eigenvalues solving (\ref{eigen_trans_pdf_2}). Hence, using (\ref{fhit_up_down_pdf}) for $p_b^+$, while inserting the product eigenfunctions in (\ref{spectral_2_product_eigen}) and then differentiating the series term-by-term w.r.t. $y=b$, as well as invoking the function $\psi_n^-(x;b)$ in (\ref{FHT_eigenfunctions_2}), gives
\begin{eqnarray*}
f^{-}(t;x,b) &=& \sum_{n=1}^\infty e^{-\lambda_n t}\psi_n^-(x;b)
\frac{-\varphi^+_{-\lambda_n}(b) \varphi^{-\prime}_{-\lambda_n}(b)}{\s(b) w_{_{-\lambda_n}}}.
\end{eqnarray*}
This is equivalent to (\ref{FHT_prop2_1}) since the eigenvalue equation gives $\varphi^-_{-\!\lambda_n}(b)=0$, and hence the Wronskian in (\ref{wronskian}) implies 
$-\frac{\varphi^+_{-\lambda_n}(b) \varphi^{-\prime}_{-\lambda_n}(b)}{\s(b) w_{-\lambda_n}} = 
\frac{W[\varphi^-_{-\lambda_n},\varphi^+_{-\lambda_n}](b)}{\s(b) w_{-\lambda_n}} \equiv 1$. 
Equation (\ref{FHT_prop2_2}) follows directly by termwise integration. We also remark that alternative proofs follow analogously to those stated above. In particular, we note the Laplace transform relation ${\mathcal L}_t\{f^-(t;x,b)\}(\lambda) = \frac{\varphi_\lambda^-(x )}{\varphi_\lambda^-(b)}$ where  
$f^-(t;x,b) = {\mathcal L}_\lambda^{-1}\big\{\frac{\varphi_\lambda^-(x )}{\varphi_\lambda^-(b)}\big\}(t)$ leads to (\ref{FHT_prop2_1}).

%
%

\subsection{Proof of Proposition \ref{prop_spec_first_hit_2}}\label{sect_FHT_Prop2_proof}
The densities in \eqref{FHT_prop_ONO1_1} and \eqref{FHT_prop_ONO2_1} follow by simply applying (\ref{fhit_up_down_pdf}) to the respective spectral expansions of $p_b^\pm$ in (\ref{u_spectral_2_b}), which hold for Spectral Category II. The respective summation terms follow in identical fashion as in the above proof in  \ref{sect_FHT_Prop1_proof}. The respective integral terms in the first equation line in \eqref{FHT_prop_ONO1_1} and \eqref{FHT_prop_ONO2_1} arise by dividing the Green function by $\m(y)$ while moving the derivative w.r.t. $y$ inside the first integral in (\ref{u_spectral_2_b}) and using \eqref{Laplace_FHT_Green_derivative} with $\lambda = \epsilon e^{-i\pi}$, i.e.,
\begin{align*}
\mp \displaystyle\frac{1}{\s(b)} \frac{\partial}{\partial y}
\left[  \frac{1}{\m(y)}\textup{Im} \{ G_b^\mp(\epsilon e^{-i\pi};x,y)\}\right]_{y=b\mp}
&=\!  
\textup{Im} \bigg\{\! \mp \displaystyle\frac{1}{\s(b)}
\frac{\partial}{\partial y}\left[ {G_b^\mp(\epsilon e^{-i\pi};x,y) \over \m(y)}\right]_{y=b\mp}\!
\bigg\}
= 
\textup{Im} 
\bigg\{{\varphi^\pm_{\epsilon e^{-i\pi}} (x) \over \varphi^\pm_{\epsilon e^{-i\pi}} (b)}\!\bigg\}.
\end{align*}
The respective integral terms in the second equation line of \eqref{FHT_prop_ONO1_1} and \eqref{FHT_prop_ONO2_1} follow after dividing out $\m(y)$ and directly differentiating the integrand w.r.t. $y$ within the second line of 
 (\ref{u_spectral_2_b}), i.e., ${\partial \Psi_b(y,\epsilon) \over \partial y}\vert_{y=b\mp} = 
{\partial \Psi_b(y,\epsilon) \over \partial y}\vert_{y=b} \equiv \Psi_b^\prime(b,\epsilon)$. Note that since  
$\Psi_b(y,\epsilon) = C \phi(b,y;-\epsilon)$, for any constant $C\ne 0$ with $\phi(b,y;-\epsilon)$ given by \eqref{phi_function} for $\lambda = -\epsilon$, it is differentiable at any $y=b\in (l,r)$. 
Finally, the expressions in \eqref{FHT_prop_ONO1_2} and \eqref{FHT_prop_ONO2_2} now follow simply by termwise integration of \eqref{FHT_prop_ONO1_1} and \eqref{FHT_prop_ONO2_1}.

%
%
\subsection{Proof of Lemma \ref{Lemma_hitting_time_ab}} \label{sect_Lemma2_proof}
The proof follows closely that in \ref{sect_Lemma1_proof}. Here we only prove  
\eqref{hit_ab_lemma_1} and \eqref{hit_ab_lemma_3}, since \eqref{hit_ab_lemma_2} and \eqref{hit_ab_lemma_4} 
follow similarly. By the tower and Markov properties we derive \eqref{hit_ab_lemma_1}:
\begin{eqnarray}
\P_{x}(t < {\T}^{+}_{b}\!(a) < \infty) &=& \E_{x}\big[\ind_{\{{\T}^{+}_{b}\!(a) > t\}} \,
\E\big[\ind_{\{{\T}^{+}_{b}\!(a) < \infty\}}\vert {\mathcal F}_t\big]\big]
= \E_{x}\big[\ind_{\{{\T}^{+}_{b}\!(a) > t \}} \,
\E\big[\ind_{\{\T_b < {\T}_a\}}\vert {\mathcal F}_t\big]\big]
\nonumber \\
&=& \E_{x}\big[\ind_{\{m_t > a,\, M_t < b \}} \, \P_{X_t}(\T_b < \T_a )\big]
= \E_{x}\big[\ind_{\{m_t > a,\, M_t < b \}} \, \Phi^{+}_b(X_t \vert a )\big]
\nonumber \\
&=& \E_{x}\big[\Phi^{+}_b(X_{(a,b),t}\,\vert\, a )\big] = \int_{a}^{b}  \Phi^{+}_b(y\vert a) \,p_{(a,b)}(t;x,y) \, dy.
\label{Lemma_ab_cond_expect}
\end{eqnarray}
The density in \eqref{hit_ab_lemma_3} arises by using the forward 
Kolmogorov PDE,  
${\partial \over \partial t}p_{(a,b)}(t;x,y) = \hat{\mathcal G}_y \,p_{(a,b)}(t;x,y)$, 
$\hat{\mathcal G}_y\, h(y):= \frac{\partial}{\partial y}(\frac{1}
{\s(y)}\frac{\partial}{\partial y}(\frac{h(y)}{\m(y)}))$, while  
differentiating \eqref{hit_ab_lemma_1} and integrating by parts, with $\s(y) = {\partial\over \partial y} \mathcal{S}[a,y]$ and $\Phi^{+}_b(y\vert a) = {\mathcal{S}[a,y] \over \mathcal{S}[a,b]}$):
\[
f^{+}(t;x,b\vert a) = - {1 \over \mathcal{S}[a,b]}
\bigg[\frac{\mathcal{S}[a,y]}{\s(y)}\frac{\partial}{\partial y}
\bigg(\frac{p_{(a,b)}(t;x,y)}{\m(y)}\bigg)
- \frac{p_{(a,b)}(t;x,y)}{\m(y)}\bigg]_{y=a+}^{y=b-}\,.
\]
Since $p_{(a,b)}(t;x,a) = p_{(a,b)}(t;x,b) = 0$, the second term vanishes. Moreover, 
$\mathcal{S}[a,y]\to 0$, while all other terms
are finite at $y=a$ in the first expression, so we have the first expression in \eqref{hit_ab_lemma_3}. 
The equivalent second expression in \eqref{hit_ab_lemma_3} follows simply by the symmetry 
${p_{(a,b)}(t;x,y) \over \m(y)} = {p_{(a,b)}(t;y,x) \over \m(x)}$.

%
%

\subsection{Proof of Proposition \ref{prop_spec_first_hit_ab}} \label{sect_Prop_FHT_ab_proof}
This follows the same steps as in \ref{sect_FHT_Prop1_proof}. 
We only prove \eqref{FHT_prop3_up_1} and \eqref{FHT_prop3_up_2} since \eqref{FHT_prop3_down_1} and \eqref{FHT_prop3_down_2} follow in the obvious similar manner. By using $p_{(a,b)}$, with product eigenfunctions in \eqref{spectral_3_product_eigen}, within \eqref{hit_ab_lemma_3} and differentiating the series termwise (w.r.t. $y$ at $b$) gives
\begin{equation*}
f^+(t;x,b\vert a) = \sum_{n=1}^\infty e^{-\lambda_n t} 
\bigg[-\frac{\phi(a,x;-\lambda_n)}{\Delta(a,b;\lambda_n)} \bigg] 
{\frac{\partial}{\partial y} \phi(b,y=b;-\lambda_n) \over w_{_{-\lambda_n}}\s(b) }
= \sum_{n=1}^\infty e^{-\lambda_n t} \psi_n^+(x;a,b).
\end{equation*}
In the last equality we used the definition in \eqref{FHT_eigenfunctions_ab} and the Wronskian relation in \eqref{wronskian} together with the definition of the cylinder function, i.e., 
$\frac{\partial}{\partial y} \phi(b,y=b;-\lambda_n) = 
W[\varphi^-_{-\lambda_n},\varphi^+_{-\lambda_n}](b) = w_{_{-\lambda_n}}\s(b)$. 
The series in \eqref{FHT_prop3_up_2} now follows by termwise integration, i.e., 
$\P_{x}(t < \Tau^{+}_{b}(a) < \infty) = \int_t^\infty f^+(u;x,b\vert a) \,du$.

We remark that one alternative proof is to first derive \eqref{FHT_prop3_up_2} by inserting the spectral expansion for 
$p_{(a,b)}$ within \eqref{hit_ab_lemma_1} and integrate termwise, where $\Phi^{+}_b(y\vert a)$ is 
square-integrable (w.r.t. $\m(y)$) for $y\in (a,b)$. The rest of the derivation follows by standard manipulations as mentioned in \ref{sect_FHT_Prop1_proof}. Another alternate proof follows from the known Laplace transform relation: 
${\mathcal L}_t\{f^+(t;x,b\vert a)\}(\lambda) = {\phi(a,x;\lambda) \over \phi(a,b;\lambda)}$. The Laplace inverse is simply computed from the fact that ${\phi(a,x;\lambda) \over \phi(a,b;\lambda)}$ is a ratio of analytic functions of 
$\lambda$ with simple poles at $\lambda = -\lambda_n^{(a,b)}$, where 
$\lambda_n\equiv \lambda_n^{(a,b)}$ are given by \eqref{eigen_trans_pdf_3}. 
The Residue Theorem recovers \eqref{FHT_prop3_up_1}.

%
%

\subsection{Proof of Theorem \ref{last-passage-propn-time-t}} \label{sect_a1}

\begin{proof} Applying $\lambda{\mathcal L}_T\{\cdot\}(\lambda)$ to 
both sides of (\ref{prop_last_time-t-1-prime}) gives
\begin{align*}
\lambda{\mathcal L}_T\{\P_x(g_\b(T) \le t) \}(\lambda)
= 1 - \int_l^r \lambda{\mathcal L}_T\{\P_y(\T_\b \le T - t)\}(\lambda)\,  p(t;x,y) dy\,.
\end{align*}
Using the Laplace transform of the first hitting time density, 
${\mathcal L}_t\{f^\pm(u;y,\b)\}(\lambda) = {\varphi_\lambda^\pm(y) \over \varphi_\lambda^\pm(\b)}$, and from standard properties of the Laplace transform, we have
\begin{align}
\lambda{\mathcal L}_T\{\P_y(\T_\b \le T - t)\}(\lambda) 
&= e^{-\lambda t}\lambda{\mathcal L}_u\{\P_y(\T_\b \le u)\}(\lambda) 
\nonumber \\
&= e^{-\lambda t}\left[{\mathcal L}_u\{f^+(u;y,\b)\}(\lambda) \,\mathbb {I}_{\{ y < \b\}} 
+ {\mathcal L}_u\{f^-(u;y,\b)\}(\lambda)\, \mathbb {I}_{\{ y \ge \b\}}
\right]
\nonumber \\
&= e^{-\lambda t}\left[{\varphi_\lambda^+(y) \over \varphi_\lambda^+(\b)} \mathbb {I}_{\{ y < \b\}} 
+ {\varphi_\lambda^-(y) \over \varphi_\lambda^-(\b)} \mathbb {I}_{\{ y \ge \b\}} \right]\,.
\end{align}
Substituting this expression into the above integral gives 
\begin{align*}
\lambda{\mathcal L}_T\{\P_x(g_\b(T) \le t) \}(\lambda) = 1 - e^{-\lambda t}
 \left[\int_l^\b {\varphi_\lambda^+(y) \over \varphi_\lambda^+(\b)}   \,  p(t;x,y) dy
+ \int_\b^r {\varphi_\lambda^-(y) \over \varphi_\lambda^-(\b)} \,  p(t;x,y) dy \right] \,.
\end{align*}
Differentiating this expression w.r.t. $t$ gives 
$\lambda {\mathcal L}_T\{f_{g_\b(T)}(t;x) \}(\lambda)$ as a sum of two integral terms
\begin{align}\label{last_passage_two_integrals}
\lambda {\mathcal L}_T\{f_{g_\b(T)}(t;x) \}(\lambda) &= e^{-\lambda t}
\left[{1 \over \varphi_\lambda^+(\b)} \int_l^\b \left(\lambda \varphi_\lambda^+(y)  \,  p(t;x,y) 
- \varphi_\lambda^+(y) \,  {\partial \over \partial t}p(t;x,y)\right) dy \right.
\nonumber \\
&\left. +
{1 \over \varphi_\lambda^-(\b)} \int_\b^r  \left(\lambda \varphi_\lambda^-(y)  \,  p(t;x,y) 
- \varphi_\lambda^-(y) \,  {\partial \over \partial t}p(t;x,y)\right) dy \right]\,.
\end{align}

We now simplify these integrals by using the forward Kolmogorov PDE, 
${\partial \over \partial t}p(t;x,y) = \hat{\mathcal G}_y p(t;x,y) \equiv \frac{\partial}{\partial y}(\frac{1}
{\s(y)}\frac{\partial}{\partial y}(q(t;x,y)))$, where $q(t;x,y) := p(t;x,y) / \m(y)$. 
Applying an integration by parts within the first integral gives
\begin{align*}
\int_l^\b  \varphi_\lambda^+(y) \,  {\partial \over \partial t}p(t;x,y) \,dy
&= \int_l^\b  \varphi_\lambda^+(y) {\partial \over \partial y}\left({1 \over \s(y)} {\partial \over \partial y} q(t;x,y) \right)  \,dy
\nonumber \\
&= {\varphi_\lambda^+(y) \over \s(y)} {\partial \over \partial y}q(t;x,y)\bigg\vert_{y=l+}^{y = \b}
- \int_l^\b  {\partial \over \partial y}q(t;x,y)\left({1 \over \s(y)} {d \over d y}\varphi_\lambda^+(y)\right) dy.
\end{align*}
Applying another integration by parts on the second integral and combining terms gives
\begin{align*}
\int_l^\b  \varphi_\lambda^+(y) \,  {\partial \over \partial t}p(t;x,y) \,dy
&= \left[{\varphi_\lambda^+(y) \over \s(y)} {\partial \over \partial y}q(t;x,y) 
- q(t;x,y){1 \over \s(y)} {d \over d y}\varphi_\lambda^+(y) \right]_{y=l+}^{y = \b}
\nonumber \\
&\,\,\,+ \int_l^\b  \lambda \varphi_\lambda^+(y) p(t;x,y)  dy\,.
\end{align*}
The last term follows since ${\mathcal G}\varphi_\lambda^+(y) 
\equiv {1 \over \m(y)} {d \over d y}\left({1 \over \s(y)}{d \over d y}\varphi_\lambda^+(y)\right) 
= \lambda \varphi_\lambda^+(y)$. The left limit in the above first term is zero since  
$\varphi_\lambda^+(l+) = 0$ and $q(t;x,l+) =0$ (with bounded ${1 \over \s(l+)} {d \over d y}\varphi_\lambda^+(l+)$ and 
${1 \over \s(l+)} {\partial \over \partial y}q(t;x,l+)$) if the left boundary $l$ is regular killing, natural or exit-not-entrance;   otherwise ${1 \over \s(l+)} {d \over d y}\varphi_\lambda^+(l+) = 0$ and 
${1 \over \s(l+)} {\partial \over \partial y}q(t;x,l+) = 0$ (with bounded $\varphi_\lambda^+(l+)$ and $q(t;x,l+)$) if 
$l$ is regular reflecting or entrance-not-entrance. Hence, we only have the term with $y=\b$:
\begin{align*}
\int_l^\b \!\!\left(\!\lambda \varphi_\lambda^+(y)  \,  p(t;x,y) 
- \varphi_\lambda^+(y) \,  {\partial \over \partial t}p(t;x,y)\!\!\right) \!dy 
=  {1 \over \s(\b)}\!\left[q(t;x,\b) {d \varphi_\lambda^+(\b)\over d \b} - 
\varphi_\lambda^+(\b) {\partial \over \partial \b}q(t;x,\b) \right]\!.
\end{align*}
Applying similar steps and arguments as above to the integral on $(\b,r)$ in (\ref{last_passage_two_integrals}) gives
\begin{align*}
\int_\b^r \!\! \left(\!\lambda \varphi_\lambda^-(y)  \,  p(t;x,y) 
- \varphi_\lambda^-(y) \,  {\partial \over \partial t}p(t;x,y)\!\!\right) \!dy 
= {1 \over \s(\b)}\left[ \varphi_\lambda^-(\b) {\partial \over \partial \b}q(t;x,\b) 
- q(t;x,\b) {d \varphi_\lambda^-(\b)\over d \b}\right]\!.
\end{align*}
Combining these two expressions into (\ref{last_passage_two_integrals}), while canceling two terms and invoking 
(\ref{wronskian}) and (\ref{greenfunc}) gives the result in (\ref{last-passage-density-time-t-Laplace}): 
\begin{align*}
\lambda {\mathcal L}_T\{f_{g_\b(T)}(t;x) \}(\lambda) &= e^{-\lambda t}q(t;x,\b)
 {1 \over \varphi_\lambda^+(\b)\varphi_\lambda^-(\b)}\cdot {W[\varphi_\lambda^-, \varphi_\lambda^+](\b) \over  \s(\b)}
\nonumber \\
&= e^{-\lambda t}{p(t;x,\b) \over w_\lambda^{-1}\m(\b)\varphi_\lambda^+(\b)\varphi_\lambda^-(\b)}
= e^{-\lambda t}{p(t;x,\b) \over G(\lambda;\b,\b)}.
\end{align*}
Equations (\ref{last-passage-density-time-t-1})--(\ref{last-passage-density-phi}) follow directly by Laplace inversion. 

Equation (\ref{last-passage-density-phi-explicit}), and hence (\ref{last-passage-pdf-explicit}), can be proven in multiple ways. A simple proof follows by showing 
${\mathcal L}_u\{\xi(u;\b)\}(\lambda) := {1 \over \lambda G(\lambda;\b,\b)}$. In particular, using the Laplace transforms ${\mathcal L}_u\{\P_{y}({\Tau}^\pm_\b \le u)\}(\lambda) = {1\over \lambda}
{\varphi_\lambda^\pm(y) \over \varphi_\lambda^\pm(\b)}$ and moving the derivative w.r.t. the initial value $y$ gives
\begin{align*}
{\mathcal L}_u\big\{{\partial \over \partial y} \P_y(\T^+_\b \le u)\big\}(\lambda) 
= {\partial \over \partial y} {\mathcal L}_u\{\P_y(\T^+_\b \le u)\}(\lambda) 
= {1\over \lambda} {\varphi_\lambda^{+\,\prime} (y) \over \varphi_\lambda^+(\b)}\,,\,y\in(\ell,\b).
\end{align*}
Hence, ${\mathcal L}_u\big\{{\partial \over \partial y} \P_y(\T^+_\b \le u)\vert_{y=\b-}\big\}(\lambda)
 = {1\over \lambda} {\varphi_\lambda^{+\,\prime} (\b) \over \varphi_\lambda^+(\b)}$. 
By using similar steps, for $y\in (\b,r)$, we have 
${\mathcal L}_u\big\{{\partial \over \partial y} \P_y(\T^-_\b \le u)\vert_{y=\b+}\big\}(\lambda)
 = {1\over \lambda} {\varphi_\lambda^{-\,\prime} (\b) \over \varphi_\lambda^-(\b)}$. Combining the two expressions gives the Laplace transform of the right-hand side of (\ref{last-passage-density-phi-explicit}):
\begin{align*}
{1 \over \m(\b)\s(\b)}{\mathcal L}_u
&\left\{{\partial \over \partial y} \P_y(\T^+_\b \le u)\vert_{y=\b-} 
- {\partial \over \partial y} \P_y(\T^-_\b \le u)\vert_{y=\b+} \right\}(\lambda) 
\\
&= {1 \over \m(\b)\s(\b)}{1\over \lambda}\left[ {\varphi_\lambda^{+\,\prime} (\b) \over \varphi_\lambda^+(\b)}
-  {\varphi_\lambda^{-\,\prime} (\b) \over \varphi_\lambda^-(\b)} \right]
\\
&=  {1 \over \lambda}{1 \over \m(\b)\varphi_\lambda^+(\b)\varphi_\lambda^-(\b)}{W[\varphi_\lambda^-, \varphi_\lambda^+](\b) \over  \s(\b)} = {1 \over \lambda G(\lambda;\b,\b)}.
\end{align*}
\end{proof}
\noindent {\it Remark}: An alternative direct proof of the density in (\ref{last-passage-pdf-explicit}) follows from (\ref{prop_last_time-t-1-prime}), where 
$$\int_l^r \P_y(\T_\b \le T - t)\,  p(t;x,y) dy = \int_l^\b \P_y(\T^+_\b \le T - t) p(t;x,y) dy + 
\int_\b^r \P_y(\T^-_\b \le T - t) p(t;x,y) dy.$$
Differentiating w.r.t. time $t$ gives
\begin{align*}
f_{g_\b(T)}(t;x) = - \int_l^\b {\partial \over \partial t}\left[\P_y(\T^+_\b \le T - t)\,  p(t;x,y) \right]dy 
 - \int_\b^r {\partial \over \partial t}\left[\P_y(\T^-_\b \le T - t)\,  p(t;x,y) \right]dy.
\end{align*}
We now evaluate the first integral over $(l,\b)$ by using the forward Kolmogorov PDE for $p(t;x,y)$ and the fact that 
$\P_y(\T^+_\b < T - t)$ satisfies the backward Kolmogorov PDE in the variables $(t,y)$. Hence,
\begin{align*}
- \int_l^\b {\partial \over \partial t}\left[\P_y(\T^+_\b \le T - t)\,  p(t;x,y) \right]dy 
&= - \int_l^\b \P_y(\T^+_\b \le T - t) 
{\partial \over \partial y}\left({1 \over \s(y)} {\partial \over \partial y} q(t;x,y) \right) dy
\\
 &
+ \int_l^\b q(t;x,y) 
{\partial \over \partial y}\left({1 \over \s(y)} {\partial \over \partial y}  \P_y(\T^+_\b \le T - t) \right) dy
\end{align*}
Applying integration by parts on both integrals, while cancelling integral terms, gives
\begin{align*}
- \int_l^\b {\partial \over \partial t}\left[\P_y(\T^+_\b \le T - t)\,  p(t;x,y) \right]dy 
&= - \P_y(\T^+_\b \le T - t) \cdot {1 \over \s(y)} {\partial \over \partial y} q(t;x,y) \bigg\vert_{y=l+}^{y=\b-}
\\
 &
+ q(t;x,y) 
\cdot {1 \over \s(y)} {\partial \over \partial y}  \P_y(\T^+_\b \le T - t)\bigg\vert_{y=l+}^{y=\b-}\,.
\end{align*}
If $l$ is regular killing, natural or exit-not-entrance, then we have ${1 \over \s(y)} {\partial \over \partial y} q(t;x,y) < \infty$ and 
$\P_y(\T^+_\b \le T - t) = 0$ in the limit ${y=l+}$. Otherwise, if $l$ is regular reflecting or entrance-not-exit, we have 
$0 < \P_y(\T^+_\b \le T - t) \le 1$ and ${1 \over \s(y)} {\partial \over \partial y} q(t;x,y) = 0$ 
when ${y=l+}$. Hence, for all the boundary cases we have $\P_y(\T^+_\b \le T - t) \cdot {1 \over \s(y)} {\partial \over \partial y} q(t;x,y) \bigg\vert_{y=l+} = 0$. In the limit ${y=\b-}$, we have $\P_y(\T^+_\b \le T - t) = 1$ and 
by continuity ${\partial \over \partial y} q(t;x,y=\b-) = {\partial \over \partial y} q(t;x,y=\b)$. By similar boundary arguments, we have 
$q(t;x,y) \cdot {1 \over \s(y)} {\partial \over \partial y}  \P_y(\T^+_\b \le T - t)\bigg\vert_{y=l+} = 0$. Hence, 
\begin{align*}
- \int_l^\b {\partial \over \partial t}\left[\P_y(\T^+_\b \le T - t)\,  p(t;x,y) \right]dy 
&= - {1 \over \s(\b)} {\partial \over \partial y} q(t;x,y=\b)  
\\
&+ 
q(t;x,\b) 
\cdot {1 \over \s(\b)} {\partial \over \partial y}  \P_y(\T^+_\b \le T - t)\bigg\vert_{y=\b-}\,.
\end{align*}
Repeating the above steps and analysis to the second integral on $(\b,r)$ gives
\begin{align*}
- \int_\b^r {\partial \over \partial t}\left[\P_y(\T^-_\b \le T - t)\,  p(t;x,y) \right]dy 
&= {1 \over \s(\b)} {\partial \over \partial y} q(t;x,y=\b)  
\\
&- 
q(t;x,\b) 
\cdot {1 \over \s(\b)} {\partial \over \partial y}  \P_y(\T^-_\b \le T - t)\bigg\vert_{y=\b+}\,.
\end{align*}
Adding these two terms produces the formula in (\ref{last-passage-pdf-explicit}).

%
%

\subsection{Proof of Proposition~\ref{last-passage-propn-time-t-spectral}}\label{sect_proof_last_hitting_spectral}
The formulae in \eqref{last-passage-pdf-spectral} follows directly by using \eqref{X_full_cdf} within 
\eqref{last-passage-pdf-explicit}, where 
$\s(\b) = {d \over d y} \mathcal{S}(l,y]\vert_{y=\b} = - {d \over d y} \mathcal{S}[y,r)\vert_{y=\b}$, and where termwise differentiation of the series in \eqref{FHT_prop1_2}, \eqref{FHT_prop2_2}, \eqref{FHT_prop_ONO1_2} and \eqref{FHT_prop_ONO2_2}, in the respective cases that $l$ or $r$ is NONOSC or O-NO, produces the respective series in \eqref{last_passage_pdf_spectral_1}--\eqref{last_passage_pdf_spectral_2} given by 
$\eta^\pm(T-t;\b) \equiv  {\partial \over \partial y} \P_y(T-t < \Tau^\pm_\b < \infty)\vert_{y=\b\mp}$. 
We note that the integrals for the O-NO cases arise simply by differentiating the integrands, where 
${\partial \over \partial y} \textup{Im} {\varphi^{\pm}_\lambda (y) \over \varphi^\pm_\lambda (\b)}\big\vert_{y=\b}
= \textup{Im} {\varphi^{\pm\prime}_\lambda (\b) \over \varphi^\pm_\lambda (\b)}$.

%
%
%

\subsection{Derivation of the formula in \eqref{joint_last_doubly_defective_formula}}\label{Proof_doubly_defective} 
To employ \eqref{joint_last_doubly_defective_1} we use \eqref{FHT_cdf_complement} with \eqref{prop_joint_last_time-t-2} and combine the integral terms to give
\begin{eqnarray*}
\P_x(g_\b(T) = 0, X_T = \partial^\dagger) 
= \begin{cases} 
{\mathcal{S}[x,\b] \over \mathcal{S}(l,\b]} 
- {1 \over \mathcal{S}(l,\b]}\int_l^\b  \mathcal{S}[y,\b] p_{\b}^-(T;x,y)  d y&, x < \b,
\\
{\mathcal{S}[\b,x] \over \mathcal{S}[\b,r)} 
- {1 \over \mathcal{S}[\b,r)}\int_\b^r  \mathcal{S}[\b,y] p_{\b}^+(T;x,y)  d y&, x > \b.
\end{cases}
\end{eqnarray*}
It suffices to derive the spectral series for $x \le \b$ since the same steps apply for $x \ge \b$ with $l$ replaced by $r$. 
First note that $l$ is NONOSC, i.e., nonconservative (regular killing or exit-not-entrance). Hence, 
we can directly substitute the discrete spectral series for $p_{\b}^-(T;x,y)$, using \eqref{u_spectral_1}, i.e., 
\eqref{spectral_1_product_eigen} and \eqref{FHT_eigenfunctions_1}, and integrate termwise to obtain
\begin{eqnarray*}
-\int_l^\b  \mathcal{S}[y,\b] p_{\b}^-(T;x,y)  d y 
= \sum_{n=1}^\infty e^{-\lambda_n T} \psi^+_n(x;\b)  
{\varphi^-_{-\lambda_n}\!(\b) \over w_{_{\!-\lambda_n}}} 
\int_l^\b \mathcal{S}[y,\b] \varphi^+_{-\lambda_n}\!(y) \m(y) dy
\end{eqnarray*}
where $\lambda_n = \lambda^-_{n,\b}$. 
Since ${\mathcal G}\varphi^+_{-\lambda_n}\!(y) = -\lambda_n\varphi^+_{-\lambda_n}\!(y)$, then 
$\varphi^+_{-\lambda_n}\!(y) \m(y) = -{1\over \lambda_n}{d \over d y}
\big( {\varphi^+_{-\lambda_n}\!(y) \over s(y)}\big)$. 
Using this relation within the above integral and integrating by parts, 
where ${d \over d y} \mathcal{S}[y,\b] = -\s(y)$, gives
\begin{eqnarray*}
\int_l^\b \mathcal{S}[y,\b] \varphi^+_{-\lambda_n}\!(y) \m(y) dy
= {1\over \lambda_n} 
\bigg[\mathcal{S}(l,\b] {\varphi^{+ \prime}_{-\lambda_n}\!(l+) \over \s(l+)}
- \varphi^{+}_{-\lambda_n}\!(\b) + \varphi^{+}_{-\lambda_n}\!(l+) \bigg] = 
{1\over \lambda_n} \mathcal{S}(l,\b] {\varphi^{+ \prime}_{-\lambda_n}\!(l+) \over \s(l+)}.
\end{eqnarray*}
Substituting this into the above summation, with $\mathcal{S}(l,\b]$ canceling, gives the series in 
\eqref{joint_last_doubly_defective_formula} for $x \le \b$. Note: the eigenvalue equation gives 
$\varphi^{+}_{-\lambda_n}\!(\b) = 0$ and $\varphi^{+}_{-\lambda_n}\!(l+)=0$ since $l$ is either 
regular killing or exit-not-entrance.

An alternate derivation based on \eqref{joint_last_doubly_defective_2} is also instructive. In particular, we have 
\begin{eqnarray*}
\P_x(g_\b(T) = 0, X_T = \partial^\dagger) 
=
\begin{cases} 
\P_x(\T_l^-\!(\b) \!< \!\infty)  - \P_x(T < \T_l^-\!(\b) < \infty) 
= 
{\mathcal{S}[x,\b] \over \mathcal{S}(l,\b]} - \int_T^\infty \! f^-(t;x,l\vert \b) dt, x < \b,&
\\
\P_x(\T_r^+\!(\b) \!< \!\infty) - \P_x(T < \T_r^+\!(\b) < \infty) 
=
{\mathcal{S}[\b,x] \over \mathcal{S}[\b,r)} - \!\int_T^\infty\! f^+(t;x,r\vert \b) dt, x > \b.&
\end{cases} 
\end{eqnarray*} 
The densities $f^-(t;x,l\vert \b):= \frac{\partial}{\partial t} \P_x(\T_l^-\!(\b) \le t)$ 
and $f^+(t;x,r\vert \b):= \frac{\partial}{\partial t} \P_x(\T_r^+\!(\b) \le t)$, 
for first hitting down at $l$ before $\b$ and up at $r$ before $\b$, are given by
\begin{equation*}
f^-(t;x,l\vert \b) = \displaystyle\frac{1}{\s(l+)} \frac{\partial}{\partial y}
\left( \frac{p_\b^-(t;x,y)}{\m(y)}\right)\bigg|_{y=l+}
\,\,\,\,\,\text{and}\,\,\,\,\,
f^+(t;x,r\vert \b)= -\displaystyle\frac{1}{\s(r-)} \frac{\partial}{\partial y}
\left( \frac{p_\b^+(t;x,y)}{\m(y)}\right)\bigg|_{y=r-}.
\end{equation*}
We now only derive the spectral series for $x \le \b$, as the same steps apply for $x \ge \b$ with $l$ replaced by $r$. 
By substituting the series for $p_{\b}^-(t;x,y)$, again given by \eqref{u_spectral_1}, and differentiating termwise at $y=l+$, gives
\begin{equation*}
f^-(t;x,l\vert \b) = -\sum_{n=1}^\infty e^{-\lambda_n t} \psi^+_n(x;\b)  
{\varphi^-_{-\lambda_n}\!(\b) \over w_{_{\!-\lambda_n}}} {\varphi^{+ \prime}_{-\lambda_n}\!(l+) \over \s(l+)}
\end{equation*}
where $\lambda_n = \lambda^-_{n,\b}$ and $\psi^+_n(x;\b)$ given by \eqref{FHT_eigenfunctions_1}. 
Integrating this series termwise, using $\int_T^\infty e^{-\lambda_n t}dt = {1\over \lambda_n}e^{-\lambda_n T}$, gives the series in \eqref{joint_last_doubly_defective_formula} for $x\le \b$.

%
%
%
\subsection{Proof of Theorem \ref{joint_last-passage-propn-time-t}} \label{sect_a2}
One way to prove (\ref{joint_last_passage_pdf_explicit}) is to take ${\partial \over \partial t}$ of 
(\ref{prop_joint_last_time-t-1_alt}) giving:
\begin{align*}
f_{g_\b(T), X_T}(t,z;x)  = 
\begin{cases}
 \int_l^\b  {\partial \over \partial t}\left[ p_{\b}^-(T-t;y,z)\,p(t;x,y) \right] dy
&, z \in (l,\b),
\\
\int_\b^r {\partial \over \partial t}\left[ p_{\b}^+(T-t;y,z) \,p(t;x,y) \right] dy 
&, z \in (\b,r).
\end{cases}
\end{align*}
Consider the case when $z \in (l,\b)$. By using the product derivative rule, the forward Kolmogorov PDE 
for $p(t;x,y)$ and the backward Kolmogorov PDE for $p_{\b}^-(T-t;y,z)$, in the variables $(t,y)$:
\begin{align*}
\int_l^\b  {\partial \over \partial t}\left[ p_{\b}^-(T-t;y,z)\,p(t;x,y) \right] dy
&= \int_l^\b p_{\b}^-(T-t;y,z) {\partial \over \partial y}\left({1 \over \s(y)}
{\partial \over \partial y}  q(t;x,y) \right) dy
\\
& - \int_l^\b q(t;x,y) {\partial \over \partial y}\left({1 \over \s(y)}
{\partial \over \partial y}  p_{\b}^-(T-t;y,z) \right) dy
\end{align*}
where throughout we define $q(t;x,y) \equiv {p(t;x,y) \over \m(y)}$. Applying integration by parts on both integrals, while cancelling integral terms, gives
\begin{align*}
\int_l^\b  {\partial \over \partial t}\!\left[ p_{\b}^-(T-t;y,z)\,p(t;x,y) \right] \! dy 
&= {p_{\b}^-(T-t;y,z) \over \s(y)} {\partial \over \partial y} q(t;x,y) \bigg\vert_{y=l+}^{y=\b-}
- {q(t;x,y) \over \s(y)} {\partial \over \partial y} p_{\b}^-(T-t;y,z)\bigg\vert_{y=l+}^{y=\b-}
\\
&= - {q(t;x,\b) \over \s(\b)} {\partial \over \partial y} p_{\b}^-(T-t;y,z)\bigg\vert_{y=\b-}\,.
\end{align*}
The last equation follows since the upper limit in the first term is zero, i.e., $p_{\b}^-(T-t;\b-,z) = p_{\b}^-(T-t;\b,z) = 0$ 
and ${\partial \over \partial y} q(t;x,\b-)={\partial \over \partial y} q(t;x,\b)$ is bounded, and the two left limits at $y=l+$ are zero for similar reasons as in the proof of Theorem \ref{last-passage-propn-time-t}. In particular, if $l$ is regular killing, natural or exit-not-entrance then 
${1 \over \s(l+)} {\partial \over \partial y} p_{\b}^-(T-t;l+,z)$ and
${1 \over \s(l+)} {\partial \over \partial y} q(t;x,l+)$ are bounded (zero if natural) and $p_{\b}^-(T-t;l+,z)=q(t;x,l+)=0$. If $l$ is regular reflecting or entrance-not-exit then we have 
${1 \over \s(l+)} {\partial \over \partial y} p_{\b}^-(T-t;l+,z) = {1 \over \s(l+)} {\partial \over \partial y} q(t;x,l+) = 0$, with bounded $p_{\b}^-(T-t;l+,z)$ and $q(t;x,l+)$. 
Hence, this proves the formula in (\ref{joint_last_passage_pdf_explicit}) for $z \in (l,\b)$.

For $z \in (\b,r)$, we use similar steps as above where we now have
\begin{align*}
\int_\b^r  {\partial \over \partial t}\!\left[ p_{\b}^+(T-t;y,z)\,p(t;x,y) \right]\!  dy 
&= {p_{\b}^+(T-t;y,z) \over \s(y)} {\partial \over \partial y} q(t;x,y) \bigg\vert_{y=\b+}^{y=r-}
- {q(t;x,y) \over \s(y)} {\partial \over \partial y} p_{\b}^+(T-t;y,z)\bigg\vert_{y=\b+}^{y=r-}
\\
&= {q(t;x,\b) \over \s(\b)} {\partial \over \partial y} p_{\b}^+(T-t;y,z)\bigg\vert_{y=\b+}\,.
\end{align*}
Again, the last equation follows by similar arguments as above where the left boundary $l+$ limit is replaced by the right boundary $r-$. In particular, the lower limit in the first expression and the two limit terms at $y=r-$ vanish. This completes the proof.
\vskip 0.1in
\noindent {\it Remark}: An alternative and instructive proof is to consider the Laplace transform w.r.t. time $T$ (as in the proof of Theorem~\ref{last-passage-propn-time-t}). 
Taking ${\mathcal L}_T\{\cdot\}(\lambda)$ on both sides of the first equation in the above proof gives:
\begin{align*}
&{\mathcal L}_T\{f_{g_\b(T), X_T}(t,z;x) \}(\lambda) = 
\begin{cases}
 \int_l^\b  G_{\b}^-(\lambda;y,z)\,{\partial \over \partial t}\left[ e^{-\lambda t} p(t;x,y) \right] dy
&, z \in (l,\b),
\\
\int_\b^r G_{\b}^+(\lambda;y,z) \,{\partial \over \partial t}\left[ e^{-\lambda t} p(t;x,y) \right] dy 
&, z \in (\b,r),
\end{cases}
\\
&= e^{-\lambda t} 
\begin{cases}
 \int_l^\b  \left[ G_{\b}^-(\lambda;y,z){\partial \over \partial t}p(t;x,y) - \lambda G_{\b}^-(\lambda;y,z) p(t;x,y) \right] dy
&, z \in (l,\b),
\\
\int_\b^r \left[ G_{\b}^+(\lambda;y,z){\partial \over \partial t}p(t;x,y) - \lambda G_{\b}^+(\lambda;y,z) p(t;x,y) \right] dy 
&, z \in (\b,r).
\end{cases}
\end{align*}
Here we note that ${\mathcal L}_T\{p_{\b}^\pm(T-t;y,z)\}(\lambda) = e^{-\lambda t}G_{\b}^\pm(\lambda;y,z)$, with  Green functions $G_{\b}^\pm(\lambda;y,z)$ on the respective intervals $\I_\b^\pm$ (see (\ref{greenfunc_up})-(\ref{greenfunc_down})). Consider the case when $z \in (l,\b)$. By using the forward Kolmogorov PDE for $p(t;x,y)$ and then an integration by parts in the above first integral term gives (note that the integral over $y\in (l,\b)$ is equivalently split into integrals over $y\in(l,z)$ and $y\in(z,\b)$)
\begin{align*}
\int_l^\b G_{\b}^-(\lambda;y,z){\partial \over \partial t}p(t;x,y) dy &= 
\int_l^\b G_{\b}^-(\lambda;y,z) {\partial \over \partial y}\left( {1 \over \s(y)} {\partial \over \partial y}q(t;x,y) \right) dy
\\
&=  {G_{\b}^-(\lambda;y,z) \over \s(y)} {\partial \over \partial y}q(t;x,y)\bigg\vert_{y=l+}^{y=\b} 
- \int_l^\b {1 \over \s(y)} {\partial \over \partial y}q(t;x,y) {\partial \over \partial y}G_{\b}^-(\lambda;y,z)  dy\,.
\end{align*}
Now applying another integration by parts on the last integral,
\begin{align*}
\int_l^\b {1 \over \s(y)} {\partial \over \partial y}q(t;x,y) {\partial \over \partial y}G_{\b}^-(\lambda;y,z)  dy
&= {q(t;x,y) \over \s(y)} {\partial \over \partial y} G_{\b}^-(\lambda;y,z)\bigg\vert_{y=l+}^{y=\b} 
- \int_l^\b \lambda G_{\b}^-(\lambda;y,z) p(t;x,y) dy\,,
\end{align*}
where we used $\mathcal{G}_yG_{\b}^-(\lambda;y,z) \equiv 
{1\over \m(y)} {\partial \over \partial y}\left( {1 \over \s(y)} {\partial \over \partial y}G_{\b}^-(\lambda;y,z) \right) = \lambda G_{\b}^-(\lambda;y,z)$, $y\ne z$. Substituting the above expression into the previous equation gives
\begin{align*}
&\int_l^\b  \!\left[ G_{\b}^-(\lambda;y,z){\partial \over \partial t}p(t;x,y) - \lambda G_{\b}^-(\lambda;y,z) p(t;x,y) \right] dy
\\
&= \! \left[ {G_{\b}^-(\lambda;y,z) \over \s(y)} {\partial \over \partial y}q(t;x,y) 
- {q(t;x,y) \over \s(y)}{\partial \over \partial y}G_{\b}^-(\lambda;y,z)  \right]_{y=l+}^{y=\b-} 
\!= - {q(t;x,\b) \over \s(\b)}{\partial \over \partial y}G_{\b}^-(\lambda;y,z)\bigg\vert_{y=\b-} .
\end{align*}
The last expression follows since $G_{\b}^-(\lambda;\b,z) = 0$ and 
${1 \over \s(\b)}{\partial \over \partial y}q(t;x,\b)$ is bounded. 
By the same arguments given in the proof of Theorem~\ref{last-passage-propn-time-t}, 
the two limit terms at $y=l+$ are also zero. Hence, for $z \in (l,\b)$,
\begin{align*}
{\mathcal L}_T\{f_{g_\b(T), X_T}(t,z;x) \}(\lambda) =  - e^{-\lambda t} {q(t;x,\b) \over \s(\b)}{\partial \over \partial y}G_{\b}^-(\lambda;y,z)\bigg\vert_{y=\b-}.
\end{align*}
By employing similar steps as above, for $z \in (\b,r)$, we obtain
\begin{align*}
{\mathcal L}_T\{f_{g_\b(T), X_T}(t,z;x) \}(\lambda) =  e^{-\lambda t} {q(t;x,\b) \over \s(\b)}{\partial \over \partial y}G_{\b}^+(\lambda;y,z)\bigg\vert_{y=\b+}.
\end{align*}
Finally, a Laplace inversion of the above two expressions recovers the joint PDF in (\ref{joint_last_passage_pdf_explicit}) since
\begin{align*}
{\mathcal L}^{-1}_\lambda\{ e^{-\lambda t} {\partial \over \partial y}G_{\b}^\pm(\lambda;y,z) \}(T) 
= {\partial \over \partial y}{\mathcal L}^{-1}_\lambda\{ G_{\b}^\pm(\lambda;y,z) \}(T-t) 
= {\partial \over \partial y}p_\b^\pm(T- t;y,z).
\end{align*}

%
%
\subsection{Proof of Proposition \ref{joint_last-passage-propn-time-t-new-formula}}\label{spectral_prop_joint_last_proof}
For $z\in (\b,r)$, where $r$ is NONOSC, we make use of \eqref{u_spectral_2} and \eqref{spectral_2_product_eigen} 
for $p_{\b}^+(T-t;y,z)$, $y \in (\b,r)$, and differentiate the series termwise w.r.t. $y=\b+$:
\begin{align*}
{1\over \s(\b)}{\partial \over \partial y} p_{\b}^+(T-t;y,z)\bigg\vert_{y=\b+}
= -\m(z)\sum_{n=1}^\infty e^{-\lambda_n t}\psi_n^-(z;\b)
\frac{\varphi^+_{-\lambda_n}(\b) \varphi^{-\,\prime}_{-\lambda_n}(\b)}{\s(\b) w_{_{-\lambda_n}}}
= \m(z)\sum_{n=1}^\infty e^{-\lambda_n t}\psi_n^-(z;\b).
\end{align*}
Here we used \eqref{FHT_eigenfunctions_2} and combined \eqref{wronskian} with the eigenvalue equation, 
$\varphi^-_{-\lambda_n}(\b) = 0$, i.e., $\lambda_n \equiv \lambda^+_{n,\b}$, which implies $-\frac{\varphi^+_{-\lambda_n}(\b) \varphi^{-\,\prime}_{-\lambda_n}(\b)}{\s(\b) w_{_{-\lambda_n}}}= \frac{W[\varphi^-_{-\lambda_n},\varphi^+_{-\lambda_n}](\b)}{\s(\b) w_{-\lambda_n}} \equiv 1$. 

Similarly, for $z\in (l,\b)$ and $r$ as NONOSC, we use \eqref{u_spectral_1} and \eqref{spectral_1_product_eigen} 
for $p_{\b}^-(T-t;y,z)$, $y\in (l,\b)$, and differentiate termwise w.r.t. $y=\b-$:
\begin{align*}
-{1\over \s(\b)}{\partial \over \partial y} p_{\b}^-(T-t;y,z)\bigg\vert_{y=\b-}
= \m(z)\sum_{n=1}^\infty e^{-\lambda_n t}\psi_n^+(z;\b)
\frac{\varphi^-_{-\lambda_n}(\b) \varphi^{+\,\prime}_{-\lambda_n}(\b)}{\s(\b) w_{_{-\lambda_n}}}
= \m(z)\sum_{n=1}^\infty e^{-\lambda_n t}\psi_n^+(z;\b),
\end{align*}
where we used \eqref{FHT_eigenfunctions_1} and
$\frac{\varphi^-_{-\lambda_n}(\b) \varphi^{+\,\prime}_{-\lambda_n}(\b)}{\s(\b) w_{_{-\lambda_n}}} 
= \frac{W[\varphi^-_{-\lambda_n},\varphi^+_{-\lambda_n}](\b)}{\s(\b) w_{-\lambda_n}} \equiv 1$ since 
$\varphi^+_{-\lambda_n}(\b) = 0$, i.e., $\lambda_n \equiv \lambda^-_{n,\b}$.

For the O-NO cases, we make use of \eqref{u_spectral_2_b}. The derivation of the discrete summation portion (if nonempty) follows exactly as above. Hence, we have left to derive the integral portions. 
The respective expressions for $\pm{1\over \s(\b)}{\partial \over \partial y} p_{\b}^\pm(T-t;y=\b\pm,z)$ 
corresponding to only the integral portions of $p_{\b}^\pm(T-t;y,z)$, given by \eqref{u_spectral_2_b} for killing level $k$ and time $T-t$, 
follow immediately from the identity:
\begin{eqnarray}\label{Green_derivative_2}
\pm \displaystyle\frac{1}{\s(\b)} \frac{\partial}{\partial y}
G_\b^\pm(\lambda;y=\b\pm,z) = \m(z){\varphi_\lambda^\mp(z) \over \varphi_\lambda^\mp(\b)}.
\end{eqnarray}
This is derived in the same fashion as \eqref{Laplace_FHT_Green_derivative} above. 
Employing \eqref{Green_derivative_2} for $\lambda = \epsilon e^{-i\pi}$ within the respective integrals over 
$[\Lambda_\pm,\infty)$ then gives the integral portions in \eqref{joint_last_passage_pdf_spectral_1}--\eqref{joint_last_passage_pdf_spectral_2}. This completes the derivation.

%
%

\subsection{Derivation of the formula in \eqref{joint_last_ab_doubly_defective_new}}\label{Proof_doubly_defective_ab} 
The derivation follows similar steps as in \ref{Proof_doubly_defective}. 
To employ the first equation line in \eqref{joint_last_ab_doubly_defective_new}, we identify  
$\P_x(g^{(a,b)}_{\b}(T) = 0)$ as $\P_x(\T^+_\b (a) > T)$ or $\P_x(\T^-_\b (b) > T)$, 
as in \eqref{prop_last_time-CDF_discrete_killed_ab}, and then adopt  
\eqref{FHT_cdf_complement_tau_ab_1}, \eqref{FHT_cdf_complement_tau_ab_2}, \eqref{hit_ab_lemma_1} and \eqref{hit_ab_lemma_2}, in combination with \eqref{prop_joint_last_ab_pmf}. Combining integrals gives
\begin{eqnarray*}
\P_x(g^{(a,b)}_{\b}(T) = 0, X_{(a,b), T} = \partial^\dagger) 
= \begin{cases} 
{\mathcal{S}[x,\b] \over \mathcal{S}[a,\b]} 
- {1 \over \mathcal{S}[a,\b]}\int_a^\b  \mathcal{S}[y,\b] p_{(a,\b)}(T;x,y)  d y&, x \in (a,\b),
\\
{\mathcal{S}[\b,x] \over \mathcal{S}[\b,b]} 
- {1 \over \mathcal{S}[\b,b]}\int_\b^b  \mathcal{S}[\b,y] p_{(\b,b)}(T;x,y)  d y&, x \in (\b,b).
\end{cases}
\end{eqnarray*}
We only derive the series for $x \in (a,\b)$ since the series for $x \in (\b,b)$ follows by similar steps. By adopting \eqref{u_spectral_3} for the spectral series of 
$p_{(a,\b)}(T;x,y)$ and integrating termwise gives
\begin{eqnarray*}
-\int_a^\b  \mathcal{S}[y,\b] p_{(a,\b)}(T;x,y)  d y 
= \sum_{n=1}^\infty e^{-\lambda_n T} \psi^+_n(x;a,\b)  
{1 \over w_{_{\!-\lambda_n}}} 
\int_a^\b \mathcal{S}[y,\b] \phi(\b,y;{-\lambda_n}) \m(y) dy
\end{eqnarray*}
where $\lambda_n \equiv \lambda^{(a,\b)}_n$ and $\psi^+_n(x;a,\b)$ defined in \eqref{FHT_eigenfunctions_ab}. 

Since $\phi(\b,y;{-\lambda_n})$ is an eigenfunction, i.e., 
${\mathcal G}\phi(\b,y;{-\lambda_n}) = -\lambda_n\phi(\b,y;{-\lambda_n})$, then 
$\phi(\b,y;{-\lambda_n}) \m(y) = -{1\over \lambda_n}{\partial \over \partial y}
\big( {1 \over s(y)} {\partial \over \partial y}\phi(\b,y;{-\lambda_n})\big)$. Using this 
relation within the above last integral and integrating by parts, 
where ${\partial \over \partial y} \mathcal{S}[y,\b] = -\s(y)$, $\mathcal{S}[\b,\b] = \phi(\b,\b;{-\lambda_n}) = 0$ and  
$\phi(\b,a;{-\lambda_n}) = -\phi(a,\b;{-\lambda_n}) = 0$ (by the eigenvalue equation), gives
\begin{eqnarray*}
\int_a^\b \mathcal{S}[y,\b] \phi(\b,y;{-\lambda_n}) \m(y) dy
= - {1\over \lambda_n} {1 \over \s(a)} \phi^\prime(a,\b;{-\lambda_n})
\cdot \mathcal{S}(a,\b],
\end{eqnarray*}
where $\phi^\prime(a,\b;{-\lambda_n}) \equiv {\partial \over \partial y}\phi(y,\b;{-\lambda_n})\big\vert_{y=a} 
= \varphi^+_{-\lambda_n}(\b)\varphi^{-\,\prime}_{-\lambda_n}(a) 
- \varphi^-_{-\lambda_n}(\b)\varphi^{+\,\prime}_{-\lambda_n}(a)$. 
Substituting this into the above series gives
\begin{eqnarray*}
\P_x(g^{(a,b)}_{\b}(T) = 0, X_{(a,b), T} = \partial^\dagger) 
 = {\mathcal{S}[x,\b] \over \mathcal{S}[a,\b]} 
- \sum_{n=1}^\infty {e^{-\lambda_n T} \over \lambda_n}   
{\psi^+_n(x;a,\b) \phi^\prime(a,\b;{-\lambda_n}) \over w_{_{\!-\lambda_n}} \s(a)}.
\end{eqnarray*}
This series is simplified to its equivalent one in \eqref{joint_last_ab_doubly_defective_new}. In particular, 
using \eqref{FHT_eigenfunctions_ab} and the eigenvalue equation, i.e.,  
$\varphi^+_{-\lambda_n}(\b) 
= \varphi^{+}_{-\lambda_n}(a) {\varphi^{-}_{-\lambda_n}(\b) \over \varphi^{-}_{-\lambda_n}(a)}$,  
$\varphi^+_{-\lambda_n}(a) 
= \varphi^{+}_{-\lambda_n}(\b) {\varphi^{-}_{-\lambda_n}(a) \over \varphi^{-}_{-\lambda_n}(\b)}$, 
$\varphi^-_{-\lambda_n}(\b) 
= \varphi^{-}_{-\lambda_n}(a) {\varphi^{+}_{-\lambda_n}(\b) \over \varphi^{+}_{-\lambda_n}(a)}$ and 
$\varphi^-_{-\lambda_n}(a) 
= \varphi^{-}_{-\lambda_n}(\b) {\varphi^{+}_{-\lambda_n}(a) \over \varphi^{+}_{-\lambda_n}(\b)}$, while factoring and canceling out these ratios in the product gives
\begin{eqnarray*}
{\psi^+_n(x;a,\b) \phi^\prime(a,\b;{-\lambda_n}) \over w_{_{\!-\lambda_n}} \s(a)}
= -{\phi(x,\b;{-\lambda_n}) \over \Delta(a,\b;\lambda_n)}
{W[\varphi^-_{-\lambda_n},\varphi^+_{-\lambda_n}](a) \over w_{_{\!-\lambda_n}} \s(a)}
= -{\phi(x,\b;{-\lambda_n}) \over \Delta(a,\b;\lambda_n)} = \psi^-_n(x;a,\b).
\end{eqnarray*}

%
%
%

\subsection{Proof of Theorem \ref{last-passage-propn-time-t-kill-ab}} \label{sect_a3}
The proof follows closely that of Theorem \ref{last-passage-propn-time-t}. 
In particular, taking the Laplace transform w.r.t. $T$ on both sides of (\ref{prop_last_time-CDF_2_killed_ab}) gives
\begin{align*}
\lambda {\cal L}_T \{ \P_x(g^{(a,b)}_{\b}(T) \le t) \}(\lambda) 
= 1 &- \int_a^\b \lambda {\cal L}_T \{ \P_y(\T^+_\b (a) \le T - t)\}(\lambda) \,  p_{(a,b)}(t;x,y) dy 
\nonumber \\
&- \int_\b^b \lambda {\cal L}_T \{ \P_y(\T^-_\b (b) \le T - t) \}(\lambda) \,  p_{(a,b)}(t;x,y) dy \,.
\end{align*}
The above Laplace transforms are given simply by using the known relations for the Laplace transforms of the CDFs of the first hitting times $\T^+_\b (a)$ and $\T^-_\b (b)$ as ratios of cylinder functions, i.e., 
\begin{align}\label{Laplace_last_ab_1}
\lambda {\cal L}_T \{ \P_y(\T^+_\b (a) \le T - t)\}(\lambda) 
&= e^{-\lambda t}\lambda{\mathcal L}_u\{\P_y(\T^+_\b (a)  \le u)\}(\lambda) 
= e^{-\lambda t} {\phi(a,y;\lambda) \over \phi(a,\b;\lambda)}\,,
\\
\label{Laplace_last_ab_2}
\lambda {\cal L}_T \{ \P_y(\T^-_\b (b) \le T - t) \}(\lambda) 
&= e^{-\lambda t}\lambda{\mathcal L}_u\{\P_y(\T^-_\b (b) \le u)\}(\lambda) 
= e^{-\lambda t} {\phi(y,b;\lambda) \over \phi(\b,b;\lambda)}\,.
\end{align}
Substituting these expressions into the above integrals gives
\begin{align*}
\lambda {\cal L}_T \{ \P_x(g^{(a,b)}_{\b}(T) \le t) \}(\lambda) 
= 1 &- {e^{-\lambda t} \over \phi(a,\b;\lambda)}\int_a^\b  \phi(a,y;\lambda) \,  p_{(a,b)}(t;x,y) dy 
\nonumber \\
&
- {e^{-\lambda t} \over \phi(\b,b;\lambda)}\int_\b^b  \phi(y,b;\lambda) \, p_{(a,b)}(t;x,y) dy \,. 
\end{align*}
Differentiating w.r.t. $t$, where ${\partial \over \partial t}\P_x(g^{(a,b)}_{\b}(T) \le t) = f_{g^{(a,b)}_{\b}(T)}(t;x)$, gives the analogue of (\ref{last_passage_two_integrals}):
\begin{align}\label{Laplace_last_marg_ab}
\lambda {\cal L}_T \{ f_{g^{(a,b)}_{\b}(T)}(t;x) \}(\lambda) 
&= e^{-\lambda t}
\left[{1 \over \phi(a,\b;\lambda)} \int_a^\b \left(\lambda \phi(a,y;\lambda)  \,  p_{(a,b)}(t;x,y) 
- \phi(a,y;\lambda) \,  {\partial \over \partial t}p_{(a,b)}(t;x,y)\right)\! dy \right.
\nonumber \\
&\left. +
{1 \over \phi(\b,b;\lambda)} \int_\b^r  \left(\lambda \phi(y,b;\lambda)  \,  p_{(a,b)}(t;x,y) 
- \phi(y,b;\lambda) \,  {\partial \over \partial t}p_{(a,b)}(t;x,y)\right)\! dy \right]\,.
\end{align}
As in the proof of Theorem~\ref{last-passage-propn-time-t}, we use the forward Kolmogorov PDE within the two integrals involving ${\partial \over \partial t}p_{(a,b)}(t;x,y)$ and apply integration by parts. For the first integral in 
(\ref{Laplace_last_marg_ab}) we have
\begin{align*}
&\int_a^\b  \phi(a,y;\lambda) \,  {\partial \over \partial t}p_{(a,b)}(t;x,y) \,dy 
= \int_a^\b  \phi(a,y;\lambda) {\partial \over \partial y}\left({1 \over \s(y)} {\partial \over \partial y} q_{(a,b)}(t;x,y) \right)  \,dy
\nonumber \\
&= {\phi(a,y;\lambda) \over \s(y)} {\partial \over \partial y}q_{(a,b)}(t;x,y)\bigg\vert_{y=a}^{y = \b}
- \int_a^\b  {\partial \over \partial y}q_{(a,b)}(t;x,y)\left({1 \over \s(y)} {\partial \over \partial y}\phi(a,y;\lambda) \right) dy
\end{align*}
where $q_{(a,b)}(t;x,y) := p_{(a,b)}(t;x,y) / \m(y)$. 
At this point we use another integration by parts on the latter integral while using the boundary condition $\phi(a,a;\lambda) = 0$ and the fact that $\phi(a,y;\lambda)$ satisfies (\ref{eq:phi}) in $y$, i.e., ${1 \over \m(y)}{\partial \over \partial y}\bigg({1 \over \s(y)} {\partial \over \partial y}\phi(a,y;\lambda)\bigg) = \lambda \phi(a,y;\lambda)$. Combining all terms gives the first integral in \eqref{Laplace_last_marg_ab}:
\begin{align*}
&\int_a^\b \left(\lambda \phi(a,y;\lambda)  \, p_{(a,b)}(t;x,y) 
- \phi(a,y;\lambda) \,  {\partial \over \partial t}p_{(a,b)}(t;x,y)\right) dy
\nonumber \\
&= {1 \over \s(\b)} q_{(a,b)}(t;x,\b) {\partial \over \partial y}\phi(a,y=\b-;\lambda)
-  {\phi(a,\b;\lambda) \over \s(\b)}{\partial \over \partial \b} q_{(a,b)}(t;x,\b)\,.
\end{align*}
Applying similar steps to the second integral in (\ref{Laplace_last_marg_ab}), where $\phi(b,b;\lambda)=0$ and 
$\phi(y,b;\lambda)$ satisfies (\ref{eq:phi}) in $y$, gives
\begin{align*}
&\int_\b^b \left(\lambda \phi(y,b;\lambda)  \,  p_{(a,b)}(t;x,y) 
- \phi(y,b;\lambda) \, {\partial \over \partial t}p_{(a,b)}(t;x,y)\right) dy 
\nonumber \\
&= {\phi(\b,b;\lambda) \over \s(\b)}{\partial \over \partial \b} q_{(a,b)}(t;x,\b)
- {1 \over \s(\b)} q_{(a,b)}(t;x,\b) {\partial \over \partial y}\phi(y=\b+,b;\lambda) \,.
\end{align*} 
Substituting the above two expressions into (\ref{Laplace_last_marg_ab}), and canceling terms in 
${\partial \over \partial \b} q_{(a,b)}(t;x,\b)$, gives
\begin{align*}
&{\cal L}_T \{ f_{g^{(a,b)}_{\b}(T)}(t;x) \}(\lambda) 
= {q_{(a,b)}(t;x,\b) \over \s(\b)} {e^{-\lambda t} \over \lambda}
\bigg[{{\partial \over \partial y}\phi(a,y=\b-;\lambda) \over \phi(a,\b;\lambda)} 
- {{\partial \over \partial y}\phi(y=\b+,b;\lambda) \over \phi(\b,b;\lambda)} \bigg]\,.
\end{align*}
Note that ${\partial \over \partial y}\phi(y=\b+,b;\lambda) 
= {\partial \over \partial \b}\phi(\b,b;\lambda)$ and ${\partial \over \partial y}\phi(a,y=\b-;\lambda) = {\partial \over \partial \b}\phi(a,\b;\lambda)$ since the derivatives of the cylinder functions at any interior point $\b\in (l,r)$ are continuous. 
Now using (\ref{Laplace_last_ab_1}) and (\ref{Laplace_last_ab_2}) we have
\begin{align*}\label{Laplace_last_ab_1_derivative}
&{\cal L}_T \{ {\partial \over \partial y} \P_y(\T^+_\b (a) \le T - t)\}(\lambda) \big\vert_{y=\b-}
= {e^{-\lambda t} \over \lambda} {{\partial \over \partial \b}\phi(a,\b;\lambda) \over \phi(a,\b;\lambda)},
\nonumber \\
&{\cal L}_T \{ {\partial \over \partial y} \P_y(\T^-_\b (b) \le T - t)\}(\lambda) \big\vert_{y=\b+}
= {e^{-\lambda t} \over \lambda} {{\partial \over \partial \b}\phi(\b,b;\lambda) \over \phi(\b,b;\lambda)}.
\end{align*}
Hence, (\ref{last-passage-pdf-explicit-kill-ab}) follows. Moreover, the first line in (\ref{last-passage-density-phi-killed}) involving the Green function holds by a simple manipulation of the above expression in ${\cal L}_T \{ f_{g^{(a,b)}_{\b}(T)}(t;x) \}(\lambda)$ where
\begin{align*}
{\cal L}_u\{\xi_{(a,b)}(u;\b)\}(\lambda) 
&= {1 \over \m(\b) \s(\b)}{1 \over \lambda}
\bigg[{{\partial \over \partial \b}\phi(a,\b;\lambda) \over \phi(a,\b;\lambda)} 
- {{\partial \over \partial \b}\phi(\b,b;\lambda) \over \phi(\b,b;\lambda)} \bigg]
\nonumber \\
&= {1 \over \m(\b) \s(\b)}{1 \over \lambda}
{\phi(\b,b;\lambda){\partial \over \partial \b}\phi(a,\b;\lambda) 
- \phi(a,\b;\lambda) {\partial \over \partial \b}\phi(\b,b;\lambda) \over \phi(a,\b;\lambda) \phi(\b,b;\lambda)}
\nonumber \\
&= {1 \over \lambda}{w_\lambda \phi(a,b;\lambda) \over \m(\b)\phi(a,\b;\lambda) \phi(\b,b;\lambda)}
= {1 \over \lambda \,G_{(a,b)}(\lambda;\b,\b)}\,.
\end{align*}
The last line follows from \eqref{greenfunc_double}, where we used the Wronskian:  
$W[\phi(\cdot,b;\lambda), \phi(a,\cdot;\lambda)](\b) = w_\lambda \s(\b)\phi(a,b;\lambda)$.

%
%

\subsection{Proof of Theorem \ref{joint_last-passage-propn-time-t_ab}} \label{sect_a4}
The proof follows the same steps as in Theorem~\ref{joint_last-passage-propn-time-t}. 
Noting \eqref{last-passage-CDF-time_0_to_t_kill_ab} and taking ${\partial \over \partial t}$ on both sides of (\ref{prop_joint_last_time_kill_ab-t-3}), with l.h.s. given by (\ref{joint_PDF_last_ab_defn}), gives
\begin{align*}
f_{g^{(a,b)}_{\b}(T), X_{(a,b),T}}(t,z;x) = 
\begin{cases}
\int_a^\b {\partial \over \partial t}\left[p_{(a,\b)}(T-t;y,z)  p_{(a,b)}(t;x,y) \right] dy\,,\,\, z \in (a,\b),
\\
\int_\b^b {\partial \over \partial t}\left[p_{(\b,b)}(T-t;y,z)  p_{(a,b)}(t;x,y) \right] dy\,,\,\, z \in (\b,b).
\end{cases}
\end{align*}
Consider the case when $z \in (a,\b)$. By using the forward Kolmogorov PDE 
for $p_{(a,b)}(t;x,y)$ and the backward Kolmogorov PDE for $p_{(a,\b)}(T-t;y,z)$, in the variables $(t,y)$:
\begin{align*}
\int_a^\b {\partial \over \partial t}\!\left[p_{(a,\b)}(T-t;y,z)  p_{(a,b)}(t;x,y) \right] \!dy
&= \int_a^\b p_{(a,\b)}(T-t;y,z) {\partial \over \partial y}\!\left({1 \over \s(y)}
{\partial \over \partial y}  q_{(a,b)}(t;x,y) \!\!\right) \! dy
\\
& - \int_a^\b q_{(a,b)}(t;x,y) {\partial \over \partial y}\!\left({1 \over \s(y)}
{\partial \over \partial y}  p_{(a,\b)}(T-t;y,z) \!\!\right)\! dy
\end{align*}
where $q_{(a,b)}(t;x,y) \equiv {p_{(a,b)}(t;x,y) \over \m(y)}$. 
Applying integration by parts on both integrals, while cancelling identical integral terms, gives
\begin{align*}
\int_a^\b {\partial \over \partial t}\left[p_{(a,\b)}(T-t;y,z)  p_{(a,b)}(t;x,y) \right] dy 
&= {p_{(a,\b)}(T-t;y,z) \over \s(y)} {\partial \over \partial y} q_{(a,b)}(t;x,y) \bigg\vert_{y=a+}^{y=\b-}
\\
&\,\,\,- {q_{(a,b)}(t;x,y) \over \s(y)} {\partial \over \partial y} p_{(a,\b)}(T-t;y,z)\bigg\vert_{y=a+}^{y=\b-}
\\
&= - {q_{(a,b)}(t;x,\b) \over \s(\b)} {\partial \over \partial y} p_{(a,\b)}(T-t;y,z)\bigg\vert_{y=\b-}\,.
\end{align*}
This is the expression in (\ref{joint_last_passage_pdf_explicit_ab}) for $z \in (a,\b)$. 
The reduction to the last equation line follows by applying the killing boundary conditions, 
$p_{(a,\b)}(T-t;a,z) = p_{(a,\b)}(T-t;\b,z) = 0$ and $q_{(a,b)}(t;x,a) = 0$, and the boundedness of the derivatives at $y=a+$ and $y=\b$. The derivation of the expression in (\ref{joint_last_passage_pdf_explicit_ab}) for $z \in (\b,b)$ follows using similar steps with killing boundary $b$ in the place of $a$.

%
%

\subsection{Proof of Proposition \ref{joint_last-passage-propn-time-t_ab-new-version}} \label{proof_joint_last_spec_ab}
The series in \eqref{joint_last_passage_pdf_explicit_ab_new} follows directly from \eqref{joint_last_passage_pdf_explicit_ab} in Theorem~\ref{joint_last-passage-propn-time-t_ab} upon using \eqref{FHT_prop3_up_1}--\eqref{FHT_prop3_down_1}. Alternatively, from \eqref{joint_last_passage_pdf_explicit_ab}, we simply differentiate termwise the series for $p_{(a,\b)}(T-t;y,z)$, using \eqref{u_spectral_3} and the first expression in \eqref{spectral_3_product_eigen}, where 
$\lambda_n\equiv \lambda_n^{(a,\b)}$:
\begin{align*}
-{1\over \s(\b)}{\partial \over \partial y} p_{(a,\b)}(T-t;y,z)\bigg\vert_{y=\b-}
= \m(z)\sum_{n=1}^\infty e^{-\lambda_n (T-t)}
{\varphi^+_{-\lambda_n}\!(\b){\partial\over \partial \b}\phi(a,\b;-\lambda_n) 
\over w_{_{-\lambda_n}} \varphi^+_{-\lambda_n}\!(a)\s(\b)}
\bigg[{-\phi(a,z;-\lambda_n) \over \Delta(a,\b;\lambda_n)}\bigg],
\end{align*}
where throughout we write ${\partial\over \partial \b}\phi(a,\b;-\lambda_n)
\equiv {\partial\over \partial y}\phi(a,y=\b;-\lambda_n)$. 
This is equivalent to the series in \eqref{joint_last_passage_pdf_explicit_ab_new} for $z\in (a,\b)$ upon using the definition for $\psi_n^+(z;a,\b)$ in \eqref{FHT_eigenfunctions_ab} and since 
${\varphi^+_{-\lambda_n}\!(\b){\partial\over \partial \b}\phi(a,\b;-\lambda_n) \over 
w_{_{-\lambda_n}} \varphi^+_{-\lambda_n}\!(a)\s(\b)}=1$. 
The latter relation follows by the Wronskian w.r.t. $\b$, $W[\varphi^+_{\lambda}(\b), \phi(a,\b;\lambda)] = w_{_{\lambda}} \varphi^+_{\lambda}\!(a)\s(\b)$ 
and then setting $\lambda = -\lambda_n$. Moreover, for $\lambda = -\lambda_n$, the Wronskian is also equivalent to    
$\varphi^+_{-\lambda_n}\!(\b){\partial\over \partial \b}\phi(a,\b;-\lambda_n)$ due to the eigenvalue equation, 
$\phi(a,\b;-\lambda_n)=0$.

For $z\in (\b,b)$, we differentiate termwise the spectral series for $p_{(\b,b)}(T-t;y,z)$, where $\lambda_n\equiv \lambda_n^{(\b,b)}$: 
\begin{align*}
{1\over \s(\b)}{\partial \over \partial y} p_{(\b,b)}(T-t;y,z)\bigg\vert_{y=\b+}
= \m(z)\sum_{n=1}^\infty e^{-\lambda_n (T-t)}
{\varphi^+_{-\lambda_n}\!(\b){\partial\over \partial \b}\phi(\b,b;-\lambda_n) 
\over w_{_{-\lambda_n}} \varphi^+_{-\lambda_n}\!(b)\s(\b)}
{\phi(z,b;-\lambda_n) \over \Delta(\b,b;\lambda_n)}.
\end{align*}
This recovers \eqref{joint_last_passage_pdf_explicit_ab_new} upon using the definition for $\psi_n^-(z;\b,b)$ in \eqref{FHT_eigenfunctions_ab} and ${\varphi^+_{-\lambda_n}\!(\b){\partial\over \partial \b}\phi(\b,b;-\lambda_n) \over w_{_{-\lambda_n}} \varphi^+_{-\lambda_n}\!(b)\s(\b)} = -1$. This follows from the Wronskian w.r.t. $\b$, $W[\varphi^+_{\lambda}(\b), \phi(\b,b;\lambda)] 
= - w_{_{\lambda}} \varphi^+_{\lambda}\!(b)\s(\b)$ and setting $\lambda = -\lambda_n$. Due to the eigenvalue equation, $\phi(\b,b;-\lambda_n)=0$, the Wronskian is also equal to 
$\varphi^+_{-\lambda_n}\!(\b){\partial\over \partial \b}\phi(\b,b;-\lambda_n)$.

The series in \ref{last_passage_pdf_discrete_spec} follows by appropriately adopting \eqref{X_full_cdf_tau_ab_1}, \eqref{X_full_cdf_tau_ab_2}, \eqref{FHT_prop3_up_2} and \eqref{FHT_prop3_down_2} within 
\eqref{last-passage-pdf-explicit-kill-ab}, where 
${\partial \over \partial y}\mathcal{S}[a,y]\vert_{y=\b} = - {\partial \over \partial y}\mathcal{S}[y,b]\vert_{y=\b} = \s(\b)$, and applying termwise differentiation of the series which give rise to \eqref{psi_hat_derivatives}. In particular, by the definition in \eqref{FHT_eigenfunctions_ab},
\begin{align*}
\hat\psi^+_{n}(a,\b) &= {1 \over \s(\b)}{1 \over \Delta(a,\b;\lambda_n)}
{\partial \over \partial y}\phi(a,y; -\lambda_n)\vert_{y=\b} 
= {1 \over \Delta(a,\b;\lambda_n)}
{\varphi^-_{-\lambda_n}(a) \varphi^{+\,\prime}_{-\lambda_n}(\b) - 
\varphi^+_{-\lambda_n}(a) \varphi^{-\,\prime}_{-\lambda_n}(\b) \over \s(\b)}
\\
&= {1 \over \Delta(a,\b;\lambda_n)}{\varphi^+_{-\lambda_n}\!(a) \over \varphi^+_{-\lambda_n}\!(\b)}
{W[\varphi^-_{-\lambda_n},\varphi^+_{-\lambda_n}](\b) \over \s(\b)}
= { w_{-\lambda_n}\over \Delta(a,\b; \lambda_n)}{\varphi^+_{-\lambda_n}\!(a) \over \varphi^+_{-\lambda_n}\!(\b)}.
\end{align*}
In the second line we used the eigenvalue equation, i.e.,  
$\varphi^-_{-\lambda_n}(a) = \varphi^{-}_{-\lambda_n}(\b) 
{\varphi^{+}_{-\lambda_n}(a) \over \varphi^{+}_{-\lambda_n}(\b)}$ and 
$\varphi^+_{-\lambda_n}(a) = \varphi^{+}_{-\lambda_n}(\b) 
{\varphi^{-}_{-\lambda_n}(a) \over \varphi^{-}_{-\lambda_n}(\b)}$, where 
${\varphi^{-}_{-\lambda_n}(a) \over \varphi^{-}_{-\lambda_n}(\b)}
= {\varphi^{+}_{-\lambda_n}(a) \over \varphi^{+}_{-\lambda_n}(\b)}$, and the Wronskian relation. A similar analysis produces the expression for $\hat\psi^-_{n}(\b,b)$ in \eqref{psi_hat_derivatives}.

Alternatively, \eqref{last_passage_pdf_discrete_spec} can be shown by performing the Laplace inverse in 
\eqref{last-passage-density-phi-killed}. Let $H(\lambda):= \m(\b){1 \over \lambda G_{(a,b)}(\lambda;\b,\b)}$. Using 
\eqref{greenfunc_double} with $x=y=\b$ gives
\begin{equation*}
{\mathcal L}_\lambda^{-1}\left\{H(\lambda)\right\}(u) = 
{\mathcal L}_\lambda^{-1}\left\{{1 \over \lambda} 
\frac{w_\lambda \phi(a,b;\lambda)}{\phi(a, \b;\lambda)  \phi(\b,b;\lambda)} \right\}(u).
\end{equation*}
From the analytic properties of $G_{(a,b)}$, $H(\lambda)$ is meromorphic in $\lambda$ with a simple pole at $\lambda=0$ and simple poles at the two sets of eigenvalues, 
$\lambda = -\lambda^{(a,\b)}_n$ and $\lambda = -\lambda^{(\b,b)}_n$, corresponding to the zeros in $\lambda$ of 
$\phi(a, \b;\lambda)$ and $\phi(\b,b;\lambda)$, respectively. For $\lambda=0$ we have 
$
{1\over \m(\b)} G_{(a,b)}(0;\b,\b) = {\mathcal{S}[a,\b]\mathcal{S}[\b,b] \over \mathcal{S}[a,b]} 
$, 
which follows simply from the scale function solution to \eqref{eq:phi} when $\lambda=0$. 
Hence, the residue at $\lambda=0$ is
\begin{equation*}
{\rm Res}\{H(\lambda);\lambda = 0\} = {1\over G_{(a,b)}(0;\b,\b)} 
= {\mathcal{S}[a,b] \over \mathcal{S}[a,\b]\mathcal{S}[\b,b]}.
\end{equation*}
For $\lambda = -\lambda^{(a,\b)}_n$:
\begin{align*}
{\rm Res}\{H(\lambda);\lambda = -\lambda^{(a,\b)}_n\} 
&= \frac{w_{-\lambda^{(a,\b)}_n}}{-\lambda^{(a,\b)}_n}
{ \phi(a,b;-\lambda^{(a,\b)}_n) \over \phi(\b,b;-\lambda^{(a,\b)}_n) 
{\partial \over \partial \lambda}\phi(a,\b;\lambda)\vert_{\lambda=-\lambda^{(a,\b)}_n} } 
\\
&= \frac{1}{\lambda^{(a,\b)}_n}
\frac{w_{-\lambda^{(a,\b)}_n}}{\Delta(a,\b;\lambda^{(a,\b)}_n)}
{ \phi(a,b;-\lambda^{(a,\b)}_n) \over \phi(\b,b;-\lambda^{(a,\b)}_n) } 
= \frac{1}{\lambda^{(a,\b)}_n} \hat\psi^+_{n}(a,\b).
\end{align*}
The equivalence of $\hat\psi^+_{n}(a,\b)$ with the expression given above follows since 
${\phi(a,b;-\lambda_n) \over \phi(\b,b;-\lambda_n) } = 
{\varphi^+_{-\lambda_n}\!(a) \over \varphi^+_{-\lambda_n}\!(\b)}$ when $\lambda_n = \lambda^{(a,\b)}_n$. 
The latter holds since ${\phi(a,b;-\lambda_n) \over \phi(\b,b;-\lambda_n) } \cdot  
{\varphi^+_{-\lambda_n}\!(\b)\over \varphi^+_{-\lambda_n}\!(a) } = 1$, 
upon multiplying $\varphi^+_{-\lambda_n}\!(\b)$ and $\varphi^+_{-\lambda_n}\!(a)$ with cylinder functions and then using the eigenvalue equation for $\lambda_n=\lambda^{(a,\b)}_n$.

\noindent A similar analysis for $\lambda = -\lambda^{(\b,b)}_n$ gives
\begin{align*}
{\rm Res}\{H(\lambda);\lambda = -\lambda^{(\b,b)}_n\} 
&= \frac{w_{-\lambda^{{(\b,b)}}_n}}{-\lambda^{(\b,b)}_n}
{ \phi(a,b;-\lambda^{(\b,b)}_n) \over \phi(a,\b;-\lambda^{(\b,b)}_n) 
{\partial \over \partial \lambda}\phi(\b,b;\lambda)\vert_{\lambda=-\lambda^{(\b,b)}_n} } 
\\
&= \frac{1}{\lambda^{(\b,b)}_n}
\frac{w_{-\lambda^{(\b,b)}_n}}{\Delta(\b,b;\lambda^{(\b,b)}_n)}
{ \phi(a,b;-\lambda^{(\b,b)}_n) \over \phi(a,\b;-\lambda^{(\b,b)}_n) } 
= \frac{1}{\lambda^{(\b,b)}_n} \hat\psi^-_{n}(\b,b).
\end{align*}
Hence, summing all residue contributions in the Laplace inverse gives the series representation: 
\begin{align}\label{xi_func_series}
\xi_{(a,b)}(u;\b) ={1\over \m(\b)} {\mathcal L}_\lambda^{-1}\{H(\lambda)\}(u) 
= 
{1 \over \m(\b)} \!
\left[ {\mathcal{S}[a,b] \over \mathcal{S}[a,\b]\mathcal{S}[\b,b]} 
+ 
\sum_{n=1}^{\infty}
\bigg(
{e^{-\lambda_n^{(a,\b)}u} \over \lambda_n^{(a,\b)} }
\hat\psi^+_{n}(a,\b) 
+  
{e^{-\lambda_n^{(\b,b)} u} \over \lambda_n^{(\b,b)} } 
\hat\psi^-_{n}(\b,b)\!
\bigg)\!\right].
\end{align}

\section*{Acknowledgements}
This research was funded by the Natural Sciences and Engineering Research 
Council of Canada (NSERC) discovery grant number 2018-06176.



\end{document}